\documentclass[reqno]{amsart}

\usepackage{amsfonts}
\usepackage{amssymb}
\usepackage{amsmath}
\usepackage{mathtools}

\usepackage{bigints}
\usepackage[mathcal]{euscript}

\usepackage{enumitem}
\usepackage{xcolor}
\usepackage{todonotes}

\setenumerate{label={\rm (\alph{*})}}
\usepackage{bbm,bm,euscript,mathrsfs}
\usepackage{paper_diening}
\usepackage{graphicx}
\usepackage[top=1in, bottom=1.25in, left=1.10in, right=1.10in]{geometry}
\usepackage{relsize}
\usepackage{hyperref} 
\usepackage{cases, color,amsthm} 
\hypersetup{linkcolor=blue, colorlinks=true,citecolor = red}

\usepackage[nameinlink, noabbrev, capitalize]{cleveref}

\Crefname{section}{Section}{Sections}
\crefname{appendix}{Appendix}{Appendices}
\Crefname{appendix}{Appendix}{Appendices}

\allowdisplaybreaks

\numberwithin{equation}{section}

\usepackage{tikz-cd}
\usetikzlibrary{lindenmayersystems}
\usetikzlibrary{decorations.pathreplacing}
\usepackage[dvipsnames]{xcolor}

\DeclareMathOperator{\Div}{div}

\newcommand{\R}{\mathbb R}

\newcommand{\dd}{\mathrm d}
\newcommand{\dt}{\,\mathrm{d} t}

\newcommand{\bx}{\mathbf{x}}
\newcommand{\by}{\mathbf{y}}

\newcommand{\bq}{\mathbf{q}}
\newcommand{\bu}{\mathbf{u}}

\newcommand{\bv}{\mathbf{v}}
\newcommand{\bn}{\mathbf{n}}

\newcommand{\dy}{\, \mathrm{d}\mathbf{y}}

\newcommand{\dx}{\, \mathrm{d} \mathbf{x}}

\newcommand{\divx}{\mathrm{div} }

\newcommand{\nabx}{\nabla }
\newcommand{\naby}{\nabla_{\mathbf{y}}}

\newcommand{\Delx}{\Delta }
\newcommand{\Dely}{\Delta_{\mathbf{y}}}

\newcommand{\dH}{\,\mathrm{d}\mathbf{\mathcal{H}}^2}

\newcommand{\ds}{\,\mathrm{d}s}

\newcommand{\Oeta}{\Omega_\eta}

\newtheorem{theorem}{Theorem}[section]
\newtheorem{lemma}{Lemma}[section]
\newtheorem{proposition}{Proposition}[section]
\newtheorem{corollary}{Corollary}[section]
\newtheorem{remark}{Remark}[section]

\newtheorem{definition}{Definition}[section]
\theoremstyle{definition}

\begin{document}

\hypersetup{
  urlcolor     = MidnightBlue, 
  linkcolor    = Bittersweet, 
  citecolor   = Cerulean
}

\title[Blow-Up Criteria  and Weak--Strong Uniqueness for FSI ]{Blow-Up Criteria  and\! Weak--Strong Uniqueness  for Compressible Fluid--Viscoelastic Shell Interactions}

\author{Prince Romeo Mensah}
\address{Faculty of Mathematics, University of Duisburg-Essen, Thea-Leymann-Straße 9, 45127 Essen, Germany}
\email{prince.mensah@uni-due.de}

\author{Pierre Marie Ngougoue Ngougoue}
\address{Faculty of Mathematics, University of Duisburg-Essen, Thea-Leymann-Straße 9, 45127 Essen, Germany}
\email{pierre.ngougouengougoue@uni-due.de}


\date{\today}


\begin{abstract}
Existence and uniqueness of  strong solutions to a  barotropic compressible fluid--viscoelastic shell interaction system have recently been established on a finite time interval. A natural question is whether such solutions can be continued globally. In this work, we derive a continuation criterion for this coupled system.  Our analysis is based on an energy estimate at the level of material acceleration, derived  under Serrin-type and Beale--Kato--Majda-type control assumptions.  While in the incompressible  setting, such control is  sufficient to prevent finite-time blow-up, in the compressible regime it does not by itself ensure propagation of  the full regularity required for  strong solutions. To obtain a genuine continuation criterion, we impose a Beale--Kato--Majda  Lipschitz-type control on the density and  velocity gradients with stronger time integrability.  In combination with the control framework underlying the acceleration estimate,  we close a higher-order energy estimate and thereby prevent loss of strong-solution regularity. Consequently, the solution can be extended beyond a potential blow-up time, provided that the corresponding control norms remain finite. We further  establish a weak-strong uniqueness principle for the system under the above conditional regularity criterion.

\end{abstract}

\maketitle

\section{Introduction}
We consider the motion of a viscous compressible fluid interacting with a deformable elastic structure. At each time $t \in I := (0, T)$, the fluid occupies a moving domain $\Omega_{\eta(t)} \subset \R^{3}$, whose deformation is described by the  displacement $\eta(t,\cdot)$ of an elastic structure anchored to part of the boundary of a reference domain $\Omega \subset \R^3.$     The time evolution of the density and the velocity of the fluid in $\Omega_{\eta}$ is governed by the compressible Navier--Stokes equations: 
\begin{equation}\label{eq:ContMomentEq}
\left\{\begin{aligned}
&\partial_t \rho  + \divx(\rho\mathbf{v} \big)
= 
0 &\text{ in }  I\times\Omega_\eta,\\
&\partial_t(\rho \bv)+\Div(\rho\bv\otimes\bv)
= 
\mu\Delx \bv +(\lambda+\mu)\nabx\Div \bv-\nabx p(\rho) &\text{ in }  I\times\Omega_\eta,\\
&\rho(0,\bx)=\rho_0(\bx), \qquad
(\rho\bv)(0,\bx)=\bq_0(\bx) & \forall\; \bx \in  \Omega_{\eta_0},
\end{aligned}\right.
\end{equation}
where 
\[\mathbf{v}:(t, \mathbf{x})\in I \times \Oeta \mapsto  \mathbf{v}(t, \mathbf{x}) \in \mathbb{R}^3  \quad \text{and} \quad  \rho:(t, \mathbf{x})\in I \times \Oeta \mapsto  \rho(t, \mathbf{x}) \in \mathbb{R},  \]
denote the fluid velocity and density, respectively,  and the pressure $p$ is given by the isentropic state equation 
\[p(\rho)=a\rho^\gamma \quad a>0,~ \gamma>1. \]
The structure $\eta:(t, \by)\in I \times \omega \mapsto   \eta(t,\by)\in \R,$ modeled as a viscoelastic shell defined on a fixed reference surface $\omega \subset \R^2, $ satisfies \
\begin{equation}\label{eq:ShellEq}
\left\{\begin{aligned}
& \partial_t^2\eta - \partial_t\Dely \eta + \Dely^2\eta=-\bn^\intercal(\bm{\tau}\bn_\eta)\circ\bm{\varphi}_\eta
\mathrm{det}(\naby \bm{\varphi}_{\eta}) 
&\text{ in }   I\times\omega
,\\
&\eta(0,\by)=\eta_0(\by), \quad (\partial_t\eta)(0, \by)=\eta_*(\by)
&\forall\; \by\in\omega,
\end{aligned}\right.
\end{equation}
with periodic boundary conditions in space. Here $\bfvarphi_\eta :\omega\to \partial \Omega_\eta$ is a parametrisation of the deformed boundary, and $\mathrm{det}(\naby \bm{\varphi}_{\eta})$ denotes the associated surface Jacobian, that is, 
\[
\mathrm{det}(\naby \bm{\varphi}_{\eta})  = \bigl| \partial_1\bm{\varphi}_{\eta} \times \partial_2 \bm{\varphi}_{\eta} \bigr|.
\] 
The vectors $\bn$ and $\bn_\eta$ are the outer normal vectors  of the reference and  deformed boundary, respectively.  $\bm{\tau}$ denotes the Cauchy stress of the fluid given by {\em Newton's rheological law}, that is
\[\bm{\tau} = \mathbb{S}(\nabla\bv)-p(\rho)\mathbb{I}_{3\times 3},\]
with the viscous stress tensor
\[\mathbb{S}(\nabla\bv)=2\mu\left(\frac{1}{2}\big(\nabla\bv+(\nabla\bv)^\intercal\big)-\frac{1}{3}\Div \bv\mathbb{I}_{3\times 3}\right)+\left(\lambda+\frac{2}{3}\mu\right)\Div\bv\mathbb{I}_{3\times 3}. \]
The shear and bulk viscosity coefficients $\mu$ and $\lambda, $ satisfy the physical restrictions 
\[\mu>0, \quad \lambda+\frac{2}{3}\mu\geq0.\]
The coupling at the interface $\bfvarphi_{\eta}(\omega) $  is expressed through the kinematic boundary condition 
\begin{align}\label{eq:interfaceCond}
\bv\circ \bfvarphi_\eta= \big(\partial_t\eta\big)\bn \quad\text{ in }  I\times \omega.
\end{align} 

The mathematical analysis of such fluid--structure interaction (FSI) systems has been initiated in the weak-solution framework by \cite{breit2018compressible}.  In their model, the elastic response of the shell is described by an operator of the form 
\[K'(\eta) = m\Dely^2 \eta + B\eta,\]
which corresponds to a linearisation of the Koiter shell model. Here $m > 0$ depends on the shell material and $B$ is a second-order differential operator. 
Their approach is based on a four-layer approximation scheme involving artificial pressure, artificial viscosity, boundary regularisation and Galerkin discretisation. Passing to the limit through compactness arguments, they obtained finite-energy weak-solutions for adiabatic exponents $\gamma > \frac{12}{7} $ -- valid up to the first self-intersection of the moving interface.  
This result provided a rigorous weak solution theory for compressible fluids interacting with a linear elastic shell, but left opened the question of higher regularity and uniqueness.  Later on \cite{ngougoue2025local} investigated a special case of this coupling, notably \eqref{eq:ContMomentEq}--\eqref{eq:interfaceCond}, corresponding to a viscoelastic shell model in which 
\[B = -\partial_t \Dely, \quad  m =1, \quad \text{so that } \quad K'(\eta) = \Dely^2 \eta -\partial_t \Dely\eta. \]  
In this setting, the author proved  the local-in-time existence and uniqueness of strong solutions, providing a precise functional framework in which the fluid and the structure possess the regularity required to interpret \eqref{eq:ContMomentEq}--\eqref{eq:interfaceCond}  pointwise almost everywhere in space-time. 

\noindent We recall below the notion of  a strong solution introduced in  \cite{ngougoue2025local}, and the corresponding well-posedness theorem, which will serve as the starting point for our analysis.   

\begin{definition}\label{def:StrongSol}
Let the initial data $(\rho_0,\bv_0,\eta_0,\eta_*)$ satisfy 
\begin{align}
&\rho_0 \in W^{3,2}(\Omega_{\eta_0}),\qquad \bv_0 \in W^{3,2}(\Omega_{\eta_0}), \qquad\eta_0\in W^{5,2}(\omega), \label{eq:InitialCondSpace}
\\
& \eta_*\in W^{3,2}(\omega),\qquad\Vert\eta_0\Vert_{L^\infty(\omega)}<L, \qquad\bv_0\circ \bm{\varphi}_{\eta_0}=\eta_*\bn \text{ on } \omega,   \label{eq:InitialCondInterface}
\\
&  \exists  \, m >0 \; \colon\;   m \leq \rho_0 (\bx)   \;\text{ a.e. in } \Omega_{\eta_0}, \label{eq:PositivityAssumption}
\end{align}
where $L > 0$ is the tubular neighbourhood radius stemming from the Hanzawa transform. \\
We call the triple $(\rho, \bv, \eta)$ a strong solution of \eqref{eq:ContMomentEq}--\eqref{eq:interfaceCond} provided that the following conditions hold:

\begin{itemize}
\item[(a)] $\rho \in   W^{1,\infty} \big(I; W^{2,2}(\Omega_\eta ) \big)\cap L^\infty\big(I;W^{3,2}(\Omega_\eta )  \big)$; \\[-0.2cm]

\item[(b)] The structure displacement $\eta $ satisfies 
\begin{align*}
&\eta \in L^2\big(I; W^{6,2}(\omega)\big)  
\cap W^{3,2}\big(I ; L^{2}(\omega)\big)
  \cap W^{2,\infty}\big(I; W^{1,2}(\omega)\big),
\\
&\partial_t\eta \in L^\infty\big(I; W^{3,2}(\omega)\big) \cap L^{2}\big(I; W^{4,2}(\omega)\big); 
\end{align*}

\item[(c)] The velocity field $\bv$ satisfies 
\begin{align*}
&\bv\in L^{2}\big(I; W^{4,2}(\Omega_\eta)\big)
\cap W^{2,2}\big(I; L^{2}(\Omega_\eta)\big). 
\end{align*}

\item[(c)]  The fluid--structure system \eqref{eq:ContMomentEq}--\eqref{eq:interfaceCond}   holds a.e. in $I\times \Omega_\eta$. 

\end{itemize}
\end{definition}

\noindent The following result ensures the local-in-time well-posedness of the system \eqref{eq:ContMomentEq}--\eqref{eq:interfaceCond}. 

\begin{theorem}\label{theo:Existenceresult}
Assume the initial data $(\rho_0,\bv_0,\eta_0,\eta_*)$ satisfy \eqref{eq:InitialCondSpace} and  \eqref{eq:InitialCondInterface}.
Then there  exists   $T_* \in I $ such that  \eqref{eq:ContMomentEq}--\eqref{eq:interfaceCond}  admits a unique strong solution  $(\rho,\bv,\eta)$ on $I_* := (0, T_*)$ satisfying 
\begin{align*}
&\rho\in L^{\infty}\big(I_*;W^{3,2}(\Omega_\eta)\big)\cap W^{1,\infty}\big(I_*;W^{2,2}(\Omega_\eta)\big),
\\
&\bv\in L^{2}\big(I_*;W^{4,2}(\Omega_\eta)\big)
\cap W^{2,2}\big(I_*;L^{2}(\Omega_\eta)\big),
\\
&\eta \in L^2\big(I_*;W^{6,2}(\omega)\big)  
\cap W^{3,2}\big(I_*;L^{2}(\omega)\big)
  \cap W^{2,\infty}\big(I_*;W^{1,2}(\omega)\big),
\\
&\partial_t\eta \in L^\infty\big(I_*;W^{3,2}(\omega)\big) \cap L^{2}\big(I_*;W^{4,2}(\omega)\big).
\end{align*}
\end{theorem} 
The relatively high Sobolev regularity required in \cref{theo:Existenceresult}, is a consequence of the moving-domain nature of the problem and the analytic requirements needed to construct strong solutions. In contrast to the  classical compressible Navier--Stokes system, the fluid occupies  a moving domain determined by the viscoelastic shell, and the equations are rewritten on the fixed reference configuration using the Hanzawa transform.  In particular, in  \cite{ngougoue2025local},  the continuity equation is solved by the method of characteristics in the reference configuration, where the effective transport field depends on both the fluid velocity and the geometry of the interface. Propagating the regularity of the density in this setting requires strong control of the associated flow map, which in turn imposes higher regularity on the velocity field and the structure displacement.  A similar phenomenon occurs in FSI models, where the structure is described  by a plate or wave equation (see e.g., \cite{maityroy2021existence}). Thus, the increase in regularity is a general feature of moving-domains problems, although the precise functional setting depends on the analytical approach. A result comparable to \cite{cho2004unique} for FSI problems is currently open.

A fundamental question that naturally  arises from this local well-posedness theory, concerns the  continuation of strong solutions: Under which conditions can the solution be extended beyond its maximal time of existence? In other words, one seeks to determine whether a finite-time breakdown can occur, and which analytical or geometric quantities control the possible loss of regularity.  

In the absence of coupling with a structure, continuation and conditional regularity criteria for the incompressible Navier--Stokes equations have been extensively studied since the classical work of \cite{serrin1962interior, serrin1963initial}, who proved conditional uniqueness of weak solutions assuming that the velocity satisfies 
\begin{equation}\label{eq:SerrinCond}
\bv \in L^\mathtt{s}\left(I; L^\mathtt{r}(\Omega)\right), \quad \frac{2}{\mathtt{s}} + \frac{3}{\mathtt{r}} \leq 1, \quad   3 < \mathtt{r} \leq \infty.   
\end{equation} 
Shortly thereafter, \cite{ladyzhenskaya1969uniqueness} and  \cite{prodi1959teorema}  independently established that this same integrability assumption implies conditional regularity of weak solutions. The resulting \emph{Ladyzhenskaya--Prodi--Serrin} condition therefore provides a precise analytic threshold ensuring that Leray--Hopf weak solutions of the incompressible Navier--Stokes equations are, in fact, smooth and unique.   \\
\noindent For compressible flows, several analogues of these criteria have been developed. The authors in \cite{sun2011beale} obtained a Beale--Kato--Majda blow-up criterion for the three-dimensional Navier--Stokes system. Under the  assumption $7\mu > \lambda,$ they proved that the boundedness of the density,
\[\sup\limits_{0\leq t < T} \Vert \rho \Vert_{L^\infty(\Omega)} < \infty, \]
suffices to continue strong solutions beyond the maximal time of existence $T_*.$    Their argument combines estimates on the effective viscous flux with logarithmic bounds for the Lam\'e operator.  Subsequently,  \cite{huang2011serrin} extended both \cite{sun2011beale} and Serrin's incompressible theory to the viscous compressible case in $\R^3.$ Indeed, for the Cauchy problem, the authors proved that if \eqref{eq:SerrinCond} holds and either the density stays bounded in $L^\infty$ or $\Div\bv \in L^1\left(I; L^\infty(\Omega)\right),$ then no blow-up occurs. Moreover, in regimes with sufficiently large shear viscosity, that is, $7\mu > \lambda$ or without vacuum, that is, $\min\limits_{\R^3}\rho_0 > 0$, the explicit Serrin condition   \eqref{eq:SerrinCond} can be dropped from the criterion. For more recent results on blow-up criteria in related compressible models, see \cite{feireisl2024nash}.

For incompressible FSI, continuation requires geometric control due to the moving interface. In \phantom{par} particular, \cite{breit2023ladyzhenskaya} proved a  \emph{Ladyzhenskaya--Prodi--Serrin}--type conditional regularity and uniqueness result for a three-dimensional incompressible fluid coupled with a viscoelastic shell \eqref{eq:ShellEq}. Assuming a Serrin bound on the fluid velocity, and a uniform control on the deformation -- specifically, 
\[\eta \in L^\infty\left(I; C^1(\omega) \right),\]
they derived an acceleration estimate that controls second-order quantities of both the fluid and structure, and yields continuation and weak-strong uniqueness as long as the geometry stays regular.  Indeed, the uniform control assumption  prevents degeneracy of the parametrisation (that is, $\partial_1\bm{\varphi}_\eta \times \partial_2\bm{\varphi}_\eta \neq \bm{0}$) and preserves a well-defined normal field $\bn_\eta$.  In particular, it guarantees that the fluid domain retains sufficient regularity to define the kinematic boundary condition, and the boundary stress. Without such control, the interface may lose regularity or develop singularities, and the analytical framework for the FSI model fails to be well-posed.

The aim of this paper is to derive a continuation criterion for the compressible FSI system \eqref{eq:ContMomentEq}--\eqref{eq:interfaceCond}. Our approach follows the same general strategy as in the incompressible setting of \cite{breit2023ladyzhenskaya}, namely testing the  momentum equation with a material-derivative type field.   In the compressible regime, however, the presence of the continuity equation and of the nonlinear pressure requires a renormalised formulation in order to convert the pressure work into internal energy and to retain a close estimate on moving domains. Importantly, this is coupled with an interface-adapted velocity extension, which preserves boundary compatibility.  This material-derivative testing procedure leads to an acceleration estimate that prevents degeneration of the fluid--structure configuration. To formulate this a priori estimate precisely, we isolate the analytic and geometric conditions under which the argument closes.

\begin{enumerate}[label={(A\arabic*)}]
    \item \label{A1}
    (\textbf{Serrin-type control of the momentum.})
    The velocity satisfies the integrability condition
    \[
        \rho^{1/2}\bv \in L^\mathtt{s}\big(I_*; L^\mathtt{r}(\Omega_\eta)\big), 
        \quad  \text{with }\;\;  \frac{2}{\mathtt{s}} + \frac{3}{\mathtt{r}} \le 1,\quad \mathtt{r} \in(3,\infty],\; \mathtt{s}\in[2,\infty).
    \]

    \item \label{A2}
    (\textbf{Control of compressibility effects.})
    At least one of the following conditions holds: \\[-0.25cm]
    \begin{enumerate}
        \item \label{A2a}
        $ \|\Div \bv\|_{L^1\big(I_*; L^\infty(\Omega_\eta)\big)} < \infty$; \\[-0.15cm]
        \item \label{A2b} 
        $ \Vert \rho\Vert_{L^\infty\big(I_*; L^\infty(\Omega_\eta)\big)} < \infty$.  \\[-0.15cm]
    \end{enumerate}

    \item \label{A3}
    (\textbf{Geometric regularity of the structure.})
    The shell displacement satisfies
    \[
        \eta \in L^\infty\left(I_*; C^1(\omega)\right),
    \]
    and the fluid--structure interface stays nondegenerate.
\end{enumerate}

Under Assumptions \ref{A1}--\ref{A3}, we obtain the following acceleration estimate -- a key preliminary result of this work.

\begin{theorem}\label{theo:MainResult}
 Let $(\rho, \bv, \eta) $ be a strong solution of \eqref{eq:ContMomentEq}--\eqref{eq:interfaceCond} in the sense of \cref{def:StrongSol}.
 Suppose that Assumptions \ref{A1}--\ref{A3} hold, 
 then the following  acceleration estimate holds:
 
 \begin{equation}\label{eq:AccelEstimate}
\begin{aligned}
{\mathlarger{\mathtt{E}}}_{\mathrm{acc}} :=&  \sup\limits_{I_*} \int_\omega \left( |\partial_t\naby\eta|^2  +   |\naby\Dely\eta|^2\right) \dy   + \frac{\mu}{2}\sup\limits_{I_*} \int_{\Omega_\eta} |\nabla \bv|^2 \dx + \frac{\lambda +\mu}{2}\sup\limits_{I_*} \int_{\Omega_\eta} |\Div \bv|^2 \dx  
\\
&\quad +  \int_{I_*}\int_\omega \left( |\partial^2_{t}\eta|^2 + |\partial_t\Dely\eta|^2  \right) \dy\dt    +  \int_{I_*}\int_{\Omega_\eta} \left( |\nabla^2 \bv|^2 +  \rho|\partial_t \bv|^2   + |\nabla p|^2  \right) \dx\dt
\\
&\lesssim  \int_\omega \left( |\naby\eta_*|^2 +  |\naby\Dely\eta_0|^2  \right) \dy   + \int_{\Omega_{\eta_0}} \left( |\nabla \bv_0|^2 + |\Div \bv_0|^2\right) \dx,
\end{aligned}
\end{equation}
where the implicit constant depends only on  $\Vert \rho^{1/2}\bv \Vert_{L^\mathtt{s}\left(I; L^\mathtt{r}(\Omega_\eta)\right)}, \; \Vert \Div\bv \Vert_{L^{1}\left(I_*; L^\infty(\Omega_\eta)  \right) }, \;  \Vert \eta \Vert_{L^{\infty}\left( I_* ; C^{1}(\omega) \right) } , $ and on the constant  $C_0 > 0$ coming  from the a priori energy estimate available for weak   solutions (cf.~\eqref{eq:BasicEnergy}). 
\end{theorem}

In contrast with the incompressible setting of \cite{breit2023ladyzhenskaya}, the acceleration estimate \eqref{eq:AccelEstimate}, does not, in the  compressible regime, recover the full regularity required of a strong solution in the sense of  \cref{def:StrongSol}.  Nevertheless, it provides uniform control of the natural energy associated with the acceleration of the system, which is sufficient to prevent the formation of singularities in the space-time norms controlled by the energy functional ${\mathlarger{\mathtt{E}}}_{\mathrm{acc}}$. 
This behaviour is consistent with the compressible setting. Indeed, even for compressible Navier--Stokes equations on fixed domains, continuation of strong solutions typically requires additional control on both the density and the velocity field, such as boundedness of the density in $L^\infty$ together with Serrin-type integrability or Lipschitz control of the velocity (see, e.g., \cite{wang2022continuation, cho2004unique}).  In the present moving-boundary configuration, the situation is further complicated by the dependence of the domain on the structure, as well as the compressibility nature of the fluid. Consequently, \cref{theo:MainResult} should  be  interpreted as a continuation criterion for a reduced notion of strong solution. While this is necessary to prevent loss of regularity, it is not sufficient to propagate the  full strong-solution  regularity.  
In this regard, the derivation of a genuine continuation criterion requires an additional  hypothesis under which higher-order regularity of the strong solution can be propagated.  
More precisely, we assume:

\begin{enumerate}[label={(B)}]
 				\item \label{B}
    ($\bm{L^2}$-\textbf{in-time Beale--Kato--Majda-type control.})
    The following conditions hold:
 \[  \int_{I_*} \Vert \nabla\bv \Vert_{L^\infty(\Omega_{\eta})}^2 \dt < \infty , \quad \text{and } \quad   \int_{I_*} \Vert \nabla\rho \Vert_{L^\infty(\Omega_{\eta})}^2 \dt < \infty .\]   

\end{enumerate}
Importantly, such Lipschitz-type control is classical in the analysis of transport and hyperbolic equations, where it ensures stability of characteristics and propagation of regularity (see, e.g.,  \cite{BahouriCheminDanchin2011, majda2012compressible}). Moreover, similar assumptions arise in the analysis of viscous compressible Navier--Stokes equations on fixed \phantom{do} domains, where continuation of strong solutions requires control of the gradient of the velocity and density (see \cite{feireisl2014regularity}).
Given Assumption \ref{B}, we can now control the regularity of the strong-solution beyond the level of the  acceleration energy, leading to the following key result: 

\begin{theorem}\label{theo:MainResult2}
 Let $(\rho, \bv, \eta) $ be a strong solution of \eqref{eq:ContMomentEq}--\eqref{eq:interfaceCond} in the sense of \cref{def:StrongSol}. Under Assumptions \ref{A1}--\ref{A3}, \ref{B} and the compatibility condition 
 \begin{equation}\label{eq:CC}
 \Big[\partial_t\bv    + \bv\cdot \nabla\bv \Big]\circ\bm{\varphi}_\eta  \raisebox{-1.6ex}{$\Big|_{t=0}$}  = \left( \partial^2_t \eta \right)\bn\raisebox{-1.6ex}{$\Big|_{t=0}$} \qquad \text{ on } I \times \omega ,  \tag{CC}
 \end{equation}
 the following a priori estimate holds:
 \begin{equation}\label{eq:HigherEstimate}
\begin{aligned}
{\mathlarger{\mathtt{E}}}_{\mathrm{high}} :=& \sup\limits_{I_*} \int_\omega \left( |\partial^2_t\naby\eta|^2  +   |\partial_t\naby\Dely\eta|^2\right) \dy   + \frac{\mu}{2}\sup\limits_{I_*} \int_{\Omega_\eta} |\partial_t\nabla \bv|^2 \dx 
\\
&\quad + \frac{\lambda +\mu}{2}\sup\limits_{I_*} \int_{\Omega_\eta} |\partial_t\Div \bv|^2 \dx  +  \int_{I_*}\int_\omega \left( |\partial^3_{t}\eta|^2 + |\partial^2_t\Dely\eta|^2 + |\partial_t\Dely^2\eta|^2 + |\Dely^3\eta|^2  \right) \dy\dt    
\\
& \quad +  \int_{I_*}\int_{\Omega_\eta} \left( |\partial_t \nabla^2 \bv|^2 +  \rho|\partial^2_t \bv|^2 \right) \dx\dt
\\
& \quad + \int_{I_*} \int_{\Omega_\eta} \left( |\nabla^4\bv|^2  +  |\nabla^3 p|^2         \right) \dx\dt   + \sup\limits_{I_*} \int_{\Omega_\eta} \left( |\nabla^3 \bv|^2   + |\partial_t\nabla^2\rho|^2            \right) \dx
\\
&\lesssim  \Vert \eta_0 \Vert_{W^{5,2}(\omega)}^2 + \Vert \eta_* \Vert_{W^{3,2}(\omega)}^2   + \Vert \bv_0  \Vert_{W^{3,2}(\Omega_{\eta_0})}^2  + \Vert \rho_0  \Vert_{W^{3,2}(\Omega_{\eta_0})}^2,
\end{aligned}
\end{equation}
where the hidden constant depends solely on the quantities specified in \ref{B}, \ref{A1}--\ref{A3}, on  $T_*$, and on  the acceleration energy ${\mathlarger{\mathtt{E}}}_{\mathrm{acc}}$ defined in \eqref{eq:AccelEstimate}.
\end{theorem}
The preliminary acceleration estimate of \cref{theo:MainResult}, together with the higher-order control obtained in  \cref{theo:MainResult2}, ensures that all strong-solution norms remain bounded.  
We therefore arrive at the following main continuation criterion for the compressible FSI system \eqref{eq:ContMomentEq}--\eqref{eq:interfaceCond}. 

\begin{theorem}\label{theo:MainResultFinal}
 Let $(\rho, \bv, \eta) $ be a strong solution of \eqref{eq:ContMomentEq}--\eqref{eq:interfaceCond} in the sense of \cref{def:StrongSol}. Moreover, assume that Assumptions \ref{A1}--\ref{A3}, \ref{B}  and \eqref{eq:CC} hold. Then,  it holds that 
 \begin{equation}\label{eq:HigherEstimate}
\begin{aligned}
{\mathlarger{\mathtt{E}}}_{\mathrm{acc}}  + {\mathlarger{\mathtt{E}}}_{\mathrm{high}}  \leq C \left( \Vert \eta_0 \Vert_{W^{5,2}(\omega)}^2 + \Vert \eta_* \Vert_{W^{3,2}(\omega)}^2   + \Vert \bv_0  \Vert_{W^{3,2}(\Omega_{\eta_0})}^2   +  \Vert \rho_0  \Vert_{W^{3,2}(\Omega_{\eta_0})}^2   \right), 
\end{aligned}
\end{equation}
where the constant $C > 0$ depends only on   \ref{A1}--\ref{A3}, \ref{B}, $T_*$ and $C_0$.

\end{theorem}

\begin{corollary}[Continuation criterion]\label{cor:Continuation}
Let   $(\rho, \bv, \eta) $ be a maximal strong solution of \eqref{eq:ContMomentEq}--\eqref{eq:interfaceCond} on $I_*$ in the sense of  \cref{def:StrongSol}.  Assume the compatibility condition \eqref{eq:CC} holds.   If Assumptions \ref{A1}--\ref{A3} and \ref{B} remain valid, then the solution extends beyond $T_*$. 

\end{corollary}
The proofs of \cref{theo:MainResult} and \cref{theo:MainResult2} are given in  \cref{sec:Proof1} and \cref{sec:Proof2}  respectively.  

In \cref{sec:WeakStrong}, we further establish a weak--strong uniqueness principle for the compressible FSI system.  Under the continuation hypotheses of  \cref{theo:MainResultFinal},   any finite--energy weak solution in the sense of \cite{breit2018compressible}, with the same initial data, coincides with the strong solution constructed in \cite{ngougoue2025local} on their common interval of existence.  More  precisely, we prove the following result:
\begin{theorem}\label{theo:WeakStrongUniqueMain}
Let the initial data $(\rho_0, \bv_0, \eta_0, \eta_*)$ satisfy \eqref{eq:InitialCondSpace}--\eqref{eq:PositivityAssumption}, and  let $(\rho, \bv, \eta)$ be a finite-energy weak solution of \eqref{eq:ContMomentEq}--\eqref{eq:interfaceCond} in the sense of  \cite[Section 2, Definition 2.1]{macha2022existence}, with adiabatic exponent $\gamma > \frac{12}{7}$. Moreover, assume  that Assumptions  \ref{A1}--\ref{A3}, \ref{B}, and  \eqref{eq:CC} hold on every finite interval $I = (0, T)$. \\
Then  $(\rho, \bv, \eta)$ is a strong solution on $I$ in the sense of \cref{def:StrongSol}.  In particular, $(\rho, \bv, \eta)$ is unique in the class of weak solutions with deformation in $L^{\infty}\left( I; C^{1}(\omega) \right)$.

\end{theorem}



\section{Preliminaries}\label{sec:Prelim}
For later arguments, we shall need to extend boundary data on $\omega$ as vector field on the moving domain $\Omega_\eta . $ We therefore recall the construction of a boundary extension operator suitable for the moving geometry. \\
Let $\bm{\varphi} \colon \omega \to \partial\Omega $ denote the reference boundary parametrisation, and let the deformation of the fluid domain $ \Omega_\eta $ be described -- for each fixed time $t \in I_* $ -- by the usual Hanzawa transform 
\[\bfPsi_\eta(t, \cdot) \colon \overline{\Omega} \to \overline{\Omega_\eta} , \]
which coincides with identity away from a fixed tubular neighbourhood  of $\partial\Omega$ and satisfies 
\[\bfPsi_\eta\left(t, \bm{\varphi}(\by)\right) = \bm{\varphi}(\by) + \eta(t, \by)\bn \quad \text{ for } \: \by \in \omega  \text{ (see e.g., \cite[Section 1.2]{ngougoue2025local})}.  \]
Following the construction used in \cite[Section 2.5]{breit2023ladyzhenskaya}, we make use of the classical boundary extension operator $\mathcal{E}_\Omega$ acting on the sufficiently smooth reference geometry $\Omega , $ that is,  
\[\mathcal{E}_\Omega \colon W^{\sigma, \mathtt{p}}(\partial\Omega) \to W^{\sigma + 1/\mathtt{p}, \mathtt{p}}(\R^3), \quad \text{with } \;\; \mathcal{E}_\Omega (v)\Big|_{\partial\Omega} = v,   \]
for all $\mathtt{p} \in [1, \infty] $ and all $\sigma > 0 .$ \\
Transporting this extension through the Hanzawa transform maps elements on $\omega$  into the moving domain 
\[\mathcal{E}_\eta (b\bn) = \mathcal{E}_\Omega \left(   (b\bn)\circ \bm{\varphi}^{-1}    \right)\circ \bfPsi_\eta^{-1}, \quad b \in W^{\sigma, \mathtt{p}}(\omega) . \]
As shown in \cite[Section 2.5, Lemma 2.2]{breit2023ladyzhenskaya},  for sufficiently smooth deformation $\eta, \; \mathcal{E}_\eta $  behaves like a classical extension  with $\big(\mathcal{E}_\eta (b\bn)\big)\circ \bm{\varphi}_\eta = b\bn $ on $\omega $ for all  $ b \in  W^{\sigma, \mathtt{p}}(\omega). $

Furthermore, in order to refer unambiguously to the a priori estimate available for the coupled fluid-structure system under consideration, we also recall the basic energy estimate satisfied by any weak solution $(\rho, \bv, \eta) $ to \eqref{eq:ContMomentEq}--\eqref{eq:interfaceCond}.  More precisely, if $(\rho, \bv, \eta) $ is a weak solution of \eqref{eq:ContMomentEq}--\eqref{eq:interfaceCond}, with data $(\rho_0, \bv_0, \eta_0, \eta_*) , $ then its total mechanical energy is uniformly bounded on $I_*, $ notably 
\begin{equation}\label{eq:BasicEnergy}
\begin{aligned}
&\sup\limits_{I_*} \left(  \int_{\Omega_\eta} \left(  \frac{1}{2} \rho|\bv|^2  + H(\rho) \right)\dx    +  \frac{1}{2} \int_{\omega} \left( |\partial_t \eta|^2 + |\Dely\eta|^2 \right) \dy  \right)
\\
& \quad + \int_{I_*} \int_{\Omega_\eta} \left( \mu |\nabla\bv|^2 + (\lambda + \mu)|\Div\bv|^2 \right)\dx\dt  +  \int_{I_*} \int_{\omega} |\partial_t\naby\eta|^2 \dy\dt 
\\
&\leq  C_0,
\end{aligned}
\end{equation} 
where 
\[C_0  = \dfrac{1}{2} \left(   \Vert \rho_0^{1/2}\bv_0\Vert_{L^2(\Omega_{\eta_0})}^2 +   \Vert \eta_*\Vert_{L^2(\omega)}^2    + \Vert \Dely\eta_0\Vert_{L^2(\omega)}^2          \right)    + \int_{\Omega_{\eta_0}}  H(\rho_0) \dx,         \]
and 
\[H(\rho) = \frac{a\rho^\gamma}{\gamma -1} . \]
$H$  denotes the pressure potential.



\section{Proof of  \cref{theo:MainResult} } \label{sec:Proof1}

The proof follows the strategy developed in the  incompressible setting in \cite[Section 4]{breit2023ladyzhenskaya}, but additional arguments are required here to handle the time-dependent density.  The central  step is  to test the fluid momentum equation $\eqref{eq:ContMomentEq}_{2}$ with a boundary-compatible material-derivative-type test function 
\[\bm{\psi} :=  \partial_t \bv + \mathcal{E}_\eta (\partial_t\eta\bn)\cdot \nabla\bv .\]  
This choice ensures that the trace of $\bm{\psi}$ matches the structural acceleration on the interface, namely   
\[\bm{\psi}\circ\bm{\varphi}_{\eta} = \left( \partial^{2}_t \eta\right) \!\bn \quad \text{on} \;\;  I_{*}\times \omega .   \] 
Thus, we obtain for $t \in I_*$
\begin{align}
& \int_{0}^{t} \int_{\Omega_\eta}  \rho (\partial_t\bv + \bv\cdot\nabla\bv)\cdot \big(  \partial_t \bv + \mathcal{E}_\eta (\partial_t\eta\bn)\cdot \nabla\bv \big) \dx\ds    \nonumber
\\
& = \int_{0}^{t} \int_{\Omega_\eta}  \Div\bm{\tau} \cdot \big(  \partial_t \bv + \mathcal{E}_\eta (\partial_t\eta\bn)\cdot \nabla\bv \big) \dx\ds   \nonumber  
\\
& = \int_{0}^{t} \int_{\partial\Omega_\eta}  \Big( {\bm{\tau}}^\intercal \big(  \partial_t \bv + \mathcal{E}_\eta (\partial_t\eta\bn)\cdot \nabla\bv \big) \Big) \cdot \bn_{\eta}  \dH\ds   - \int_{0}^{t} \int_{\Omega_\eta}   \bm{\tau}\colon \nabla\Big( \big(  \partial_t \bv + \mathcal{E}_\eta (\partial_t\eta\bn)\cdot \nabla\bv \big) \Big)\dx\ds    \nonumber 
\\
& = \int_{0}^{t} \int_{\omega}     {\mathlarger{\mathtt{F}}}_\eta \partial^{2}_t\eta \dy\ds 
\nonumber
\\
&\quad - \frac{\mu}{2} \int_{0}^{t} \dfrac{\dd}{\ds} \int_{\Omega_\eta}  |\nabla\bv|^2 \dx\ds + \frac{\mu}{2} \int_{0}^{t} \int_{\partial\Omega_\eta} (\partial_t\eta \bn)\circ\bm{\varphi}_{\eta}^{-1}\cdot \bn_\eta |\nabla\bv|^2 \dH\ds
\nonumber
\\
&\quad - \mu \int_{0}^{t} \int_{\Omega_\eta} \nabla\bv \colon \nabla\big( \mathcal{E}_\eta (\partial_t\eta\bn)\cdot \nabla\bv  \big)\dx\ds  - \mu \int_{0}^{t} \int_{\partial\Omega_\eta}\Big(\nabla\bv  \big(\partial_t\bv + \mathcal{E}_\eta (\partial_t\eta\bn)\cdot \nabla\bv \big)\Big) \cdot \bn_\eta \dH\ds 
\nonumber
\\
&\quad +  \mu \int_{0}^{t} \int_{\partial\Omega_\eta} (\Div\bv) \big(\partial_t\bv + \mathcal{E}_\eta (\partial_t\eta\bn)\cdot \nabla\bv \big)\cdot \bn_\eta \dH\ds 
\nonumber
\\
&\quad - \frac{\lambda + \mu}{2} \int_{0}^{t} \dfrac{\dd}{\ds} \int_{\Omega_\eta}  |\Div\bv|^2 \dx\ds + \frac{\lambda + \mu}{2}  \int_{0}^{t} \int_{\partial\Omega_\eta} (\partial_t\eta \bn)\circ\bm{\varphi}_{\eta}^{-1}\cdot \bn_\eta |\Div\bv|^2 \dH\ds
\nonumber
\\
&\quad  - (\lambda +\mu ) \int_{0}^{t} \int_{\Omega_\eta} (\Div\bv) \Div\big( \mathcal{E}_\eta (\partial_t\eta\bn)\cdot \nabla\bv \big) \dx\ds 
\nonumber
\\
& \quad + \int_{0}^{t} \int_{\Omega_\eta} p(\rho)\Div\big(\partial_t\bv + \mathcal{E}_\eta (\partial_t\eta\bn)\cdot \nabla\bv \big) \dx\ds , \label{eq:MomEqTest}
\end{align}
with
\[{\mathlarger{\mathtt{F}}}_\eta  := \bn^\intercal(\bm{\tau}\bn_\eta)\circ\bm{\varphi}_\eta
\mathrm{det}(\naby \bm{\varphi}_{\eta}) .\]
Note that in deriving \eqref{eq:MomEqTest}, we used the Reynold's transport theorem, as well as the vector identity \[\Div(\nabla\bv)^\intercal = \nabla(\Div\bv).\]

Testing the shell equation \eqref{eq:ShellEq} with $\partial^{2}_t\eta$, we obtain from standard integration by parts 
\begin{align*}
&\int_{I_*}\int_\omega |\partial^{2}_t \eta|^2 \dy\dt + \frac{1}{2} \int_{I_*}\dfrac{\dd}{\dt}  \int_\omega |\partial_t \naby\eta|^2 \dy\dt + \int_{I_*}\int_\omega \Dely^2\eta \cdot \partial^{2}_t \eta \dy\dt   
\\
& = -\int_{I_*}\int_\omega \bn^\intercal(\bm{\tau}\bn_\eta)\circ\bm{\varphi}_\eta
\mathrm{det}(\naby \bm{\varphi}_{\eta})\partial^{2}_t\eta \dy\dt. 
\end{align*}
However, observe that 
\begin{equation}\label{eq:identityEta}
- \int_{I_*}\int_\omega \Dely^{2}\eta \cdot \partial^{2}_{t}\eta \dy\dt  = \int_{I_*}\int_\omega \partial_{t}\big( \naby\Dely\eta \cdot \partial_{t}\naby\eta \big) \dy\dt  +   \int_{I_*}\int_\omega |\partial_{t}\Dely\eta|^2 \dy\dt .
\end{equation}
Hence,
\begin{align}
&\int_{I_*}\int_\omega |\partial^{2}_t \eta|^2 \dy\dt + \sup\limits_{I_*} \int_\omega |\partial_t \naby\eta|^2 \dy  \nonumber
\\
&\leq  C \left(  \int_\omega |\naby\eta_*|^2 \dy  + \sup\limits_{I_*} \int_\omega |\naby\Dely\eta|^2 \dy   + \int_{I_*}\int_\omega |\partial_t\Dely\eta|^2 \dy\dt  \right)  -\int_{I_*}\int_\omega {\mathlarger{\mathtt{F}}}_\eta \partial^{2}_t\eta \dy\dt,   \label{eq:ShellEqTest}
\end{align}
for some constant $C > 0. $ Adding \eqref{eq:MomEqTest} and \eqref{eq:ShellEqTest}, we exploit the trivial identity 
$\bv\cdot\nabla\bv - \bv\cdot\nabla\bv = 0,$  to rewrite the test function $\bm{\psi}$ as
\[\bm{\psi} = \dot{\bv} + \left( \mathcal{E}_\eta (\partial_t\eta\bn) - \bv\right)\cdot \nabla\bv, \quad \text{that is, } \;\; \dot{\bv} = \partial_t\bv +  \bv \cdot \nabla \bv .\]
 This  purely algebraic step is conceptually important. Indeed, by isolating the material derivative $\dot{\bv}, $ we single out the quantity $\Vert \rho^{1/2}\dot{\bv} \Vert_{L^2(\Omega_\eta)}, $ which in subsequent analysis, will allow us to control $\nabla^2\bv. $
Thus, using Young's inequality, we deduce that 
\begin{align}
&\sup\limits_{I_*} \int_{\Omega_\eta} |\nabla\bv|^2\dx +  \sup\limits_{I_*} \int_{\Omega_\eta} |\Div\bv|^2\dx  + \int_{I_*}\int_{\Omega_\eta} \rho |\dot{\bv}|^2 \dx\dt + \int_{I_*}\int_\omega |\partial^{2}_t \eta|^2 \dy\dt + \sup\limits_{I_*} \int_\omega |\partial_t \naby\eta|^2 \dy
\nonumber
\\
& \lesssim  \int_{I_*}\int_{\Omega_\eta}  \rho |\bv\cdot\nabla\bv|^2 \dx\dt +  \int_{I_*}\int_{\partial\Omega_\eta} (\partial_t\eta \bn)\circ\bm{\varphi}_{\eta}^{-1}\cdot \bn_\eta |\nabla\bv|^2 \dH\dt + \int_{I_*}\Vert \rho^{1/2}\mathcal{E}_\eta (\partial_t\eta\bn) \cdot\nabla\bv \Vert_{L^{2}(\Omega_\eta)}^2 \dt
\nonumber
\\
&\quad  - \int_{I_*} \int_{\Omega_\eta} \nabla\bv \colon \nabla\big( \mathcal{E}_\eta (\partial_t\eta\bn)\cdot \nabla\bv  \big)\dx\dt  - \int_{I_*} \int_{\partial\Omega_\eta}\big(\dot{\bv} \cdot \nabla\bv \big) \cdot \bn_\eta \dH\dt  + \int_{I_*} \int_{\partial\Omega_\eta} (\Div\bv) \dot{\bv}\cdot \bn_\eta \dH\dt 
\nonumber
\\
&\quad + \int_{I_*} \int_{\partial\Omega_\eta} (\partial_t\eta \bn)\circ\bm{\varphi}_{\eta}^{-1}\cdot \bn_\eta |\Div\bv|^2 \dH\dt - \int_{I_*}\int_{\Omega_\eta} (\Div\bv) \Div\big(\mathcal{E}_\eta (\partial_t\eta\bn)\cdot \nabla\bv \big) \dx\dt   \label{eq:AccFirstEstimate}
\\
&\quad +  \int_{I_*} \int_{\Omega_\eta} p(\rho)\Div\big( \mathcal{E}_\eta (\partial_t\eta\bn)\cdot \nabla\bv \big) \dx\dt + \int_{I_*} \int_{\Omega_\eta} p(\rho)\partial_t\Div(\bv) \dx\dt
\nonumber
\\
&\quad + \int_{\Omega_{\eta_0}} |\nabla\bv_0|^2\dx + \int_{\Omega_{\eta_0}} |\Div\bv_0|^2\dx  + \int_\omega |\naby\eta_*|^2 \dy  + \sup\limits_{I_*} \int_\omega |\naby\Dely\eta|^2 \dy   + \int_{I_*}\int_\omega |\partial_t\Dely\eta|^2 \dy \dt 
\nonumber
\\
& \quad =: \sum\limits_{\mathsf{k}=1}^{15} {\mathlarger{\mathfrak{R}}}_\mathsf{k}. 
\nonumber
\end{align}
We now estimate each term ${\mathlarger{\mathfrak{R}}}_\mathsf{k}. $ The purpose of the forthcoming term-by-term analysis is to show that, under the assumptions of  \cref{theo:MainResult}, each contribution ${\mathlarger{\mathfrak{R}}}_\mathtt{k} $  is either dissipative, can be absorbed into the left-hand side through an $\varepsilon-$splitting,  or is of  lower-order and therefore controlled by the standard energy estimate (see \eqref{eq:BasicEnergy}).

\noindent We start with the convective contribution 
\[{\mathlarger{\mathfrak{R}}}_1   = \int_{I_*}\int_{\Omega_\eta}  \rho |\bv\cdot\nabla\bv|^2 \dx\dt. \]
Using H\"older's inequality with $\mathtt{q} := \dfrac{2\mathtt{r}}{\mathtt{r} -2} $ and the Gargliardo--Nirenberg interpolation 
\begin{equation}\label{eq:GargliadoNirenberg}
\Vert \nabla\bv\Vert_{L^{\mathtt{q}}(\Omega_\eta)}^2 \lesssim \Vert \bv\Vert_{W^{2,2}(\Omega_\eta)}^{\frac{6}{\mathtt{r}}}  \Vert \bv\Vert_{W^{1,2}(\Omega_\eta)}^{\frac{2\mathtt{r} - 6}{\mathtt{r}}} ,
\end{equation}
we obtain 
\begin{align*}
{\mathlarger{\mathfrak{R}}}_1 & \leq \int_{I_*} \Vert \rho^{1/2}\bv\Vert_{L^{\mathtt{r}}(\Omega_\eta)}^2  \Vert \nabla\bv\Vert_{L^{\mathtt{q}}(\Omega_\eta)}^2   \dt  
 \lesssim  \int_{I_*} \Vert \rho^{1/2}\bv\Vert_{L^{\mathtt{r}}(\Omega_\eta)}^2  \Vert \bv\Vert_{W^{2,2}(\Omega_\eta)}^{\frac{6}{\mathtt{r}}}  \Vert \bv\Vert_{W^{1,2}(\Omega_\eta)}^{\frac{2\mathtt{r} - 6}{\mathtt{r}}} \dt.
\end{align*}
Applying Young's inequality with $\mathtt{s} := \dfrac{2\mathtt{r}}{\mathtt{r}-3} \in [2, \infty) $  yields
\begin{align}\label{eq:R1}
{\mathlarger{\mathfrak{R}}}_1 & \lesssim \varepsilon  \int_{I_*} \Vert \bv\Vert_{W^{2,2}(\Omega_\eta)}^{2} \dt + c(\varepsilon) \int_{I_*} \Vert \rho^{1/2}\bv\Vert_{L^{\mathtt{r}}(\Omega_\eta)}^{\mathtt{s}}  \Vert \bv\Vert_{W^{1, 2}(\Omega_\eta)}^{2} \dt.
\end{align}
The first term on the right-hand side will be absorbed  into  the left-hand side of \eqref{eq:AccelEstimate} for $\varepsilon > 0$ sufficiently small, while the remaining term will  then be controlled by applying Gr\"onwall's inequality  to the differential inequality satisfied  by $\Vert \nabla\bv\Vert_{L^{2}(\Omega_\eta)}^2, $  using in particular  the Serrin-type bound on $\rho^{1/2}\bv.$

For the boundary transport of $|\nabla\bv|^2$, that is,
\[{\mathlarger{\mathfrak{R}}}_2 = \int_{I_*} \int_{\partial\Omega_\eta} (\partial_t\eta \bn)\circ\bm{\varphi}_{\eta}^{-1}\cdot \bn_\eta |\nabla\bv|^2 \dH\dt,  \]
it follows from H\"older's inequality and Sobolev embeddings
\[W^{1/2, 2}(\partial\Omega_\eta) \hookrightarrow L^{4}(\partial\Omega_\eta), \quad  W^{1/4, 2}(\partial\Omega_\eta) \hookrightarrow L^{8/3}(\partial\Omega_\eta)\footnote{Of note, $\partial\Omega_\eta$ is uniformly Lipschitz in time, with a contant controlled by $\sup_{I_*}\Vert \eta\Vert_{C^1(\omega)}.$ }, \]
that 
\[{\mathlarger{\mathfrak{R}}}_2 \lesssim \int_{I_*} \Vert \nabla\bv \Vert_{W^{1/2, 2}(\partial\Omega_\eta)} \Vert \nabla\bv \Vert_{W^{1/4, 2}(\partial\Omega_\eta)} \Vert \partial_t\eta\circ\bm{\varphi}_{\eta}^{-1} \Vert_{W^{1/4, 2}(\partial\Omega_\eta)} \dt . \]
Using the trace embedding, the estimate
\begin{align*}
\Vert \partial_t\eta\circ\bm{\varphi}_{\eta}^{-1} \Vert_{W^{1/4, 2}(\partial\Omega_\eta)}  \lesssim \Vert \partial_t\eta\Vert_{W^{1/4, 2}(\omega)} 
\end{align*}
 and the characterisation of fractional Sobolev spaces as interpolation spaces, namely 
\begin{align}\label{eq:SobolevInterpo}
W^{3/4, 2}(\Omega_\eta) = \Big[ L^2(\Omega_\eta), W^{1,2}(\Omega_\eta)    \Big]_{3/4}, \quad W^{1/4, 2}(\Omega_\eta) = \Big[ L^2(\Omega_\eta), W^{1,2}(\Omega_\eta)    \Big]_{1/4}   ,
\end{align}
we obtain 
\[{\mathlarger{\mathfrak{R}}}_2 \lesssim \int_{I_*} \Vert \nabla\bv \Vert_{W^{1, 2}(\Omega_\eta)}^{7/4} \Vert \nabla\bv \Vert_{L^{ 2}(\Omega_\eta)}^{1/4} \Vert \partial_t\eta \Vert_{L^{2}(\omega)}^{3/4} \Vert \partial_t\eta \Vert_{W^{1,2}(\omega)}^{1/4} \dt . \]
The uniform-in-time  boundedness of  $\Vert \partial_t\eta \Vert_{L^{2}(\omega)}$  ( a consequence of \eqref{eq:BasicEnergy}), together with Young's inequality further imply that 
\begin{align}\label{eq:R2}
{\mathlarger{\mathfrak{R}}}_2 \lesssim \varepsilon \int_{I_*} \Vert \nabla\bv \Vert_{W^{1, 2}(\Omega_\eta)}^2 \dt + c(\varepsilon) \int_{I_*} \Vert \partial_t\eta \Vert_{W^{1,2}(\omega)}^2 \Vert \nabla\bv \Vert_{L^{ 2}(\Omega_\eta)}^2 \dt.
\end{align}

We now consider 
\[{\mathlarger{\mathfrak{R}}}_3 = \int_{I_*} \Vert \rho^{1/2}\mathcal{E}_\eta (\partial_t\eta\bn) \cdot\nabla\bv \Vert_{L^{2}(\Omega_\eta)}^2 \dt. \]
Using H\"older's inequality, the embedding $W^{3/4,2}(\Omega_\eta) \hookrightarrow    L^{4}(\Omega_\eta) ,  $ Gagliardo--Nirenberg interpolation 
\[\Vert \nabla\bv\Vert_{L^{4}(\Omega_\eta)}^2 \lesssim \Vert \bv\Vert_{W^{2,2}(\Omega_\eta)}^{3/2}  \Vert \bv\Vert_{W^{1, 2}(\Omega_\eta)}^{1/2} , \] 
and the uniform boundedness of the density, we get  
\[{\mathlarger{\mathfrak{R}}}_3  \lesssim \int_{I_*} \Vert \mathcal{E}_\eta (\partial_t\eta\bn) \Vert_{W^{3/4,2}(\Omega_\eta) }^2  \Vert \bv\Vert_{W^{1, 2}(\Omega_\eta)}^{1/2}   \Vert \bv\Vert_{W^{2, 2}(\Omega_\eta)}^{3/2} \dt.          \]
However, the extension operator $\mathcal{E}_\eta $ satisfies the trace property 
\begin{align*}
 \Vert \mathcal{E}_\eta (\partial_t\eta\bn) \Vert_{W^{3/4,2}(\Omega_\eta) }  \lesssim  \Vert \partial_t\eta\Vert_{W^{1/4, 2}(\omega)}.  
\end{align*}
Thus, using once again \eqref{eq:SobolevInterpo}, the uniform-in-time boundedness of $\Vert \partial_t\eta \Vert_{L^{2}(\omega)}$ and Young's inequality yield
\begin{align}\label{eq:R3}
{\mathlarger{\mathfrak{R}}}_3 \lesssim \varepsilon \int_{I_*} \Vert \bv \Vert_{W^{2, 2}(\Omega_\eta)}^2 \dt + c(\varepsilon) \int_{I_*} \Vert \partial_t\eta \Vert_{W^{1,2}(\omega)}^2 \Vert \bv \Vert_{W^{ 1, 2}(\Omega_\eta)}^2 \dt.
\end{align}

Moving on to the bilinear bulk term ${\mathlarger{\mathfrak{R}}}_4, $ a product rule yields the decomposition
\[{\mathlarger{\mathfrak{R}}}_4 = - \int_{I_*} \int_{\Omega_\eta} \nabla\bv \colon \nabla\big( \mathcal{E}_\eta (\partial_t\eta\bn)\cdot \nabla\bv  \big)\dx\dt =: {\mathlarger{\mathfrak{R}}}_{4,1} + {\mathlarger{\mathfrak{R}}}_{4,2},   \]
where 
\[{\mathlarger{\mathfrak{R}}}_{4,1}  = - \int_{I_*} \int_{\Omega_\eta} \nabla\bv \colon \mathcal{E}_\eta (\partial_t\eta\bn)^\intercal \nabla^2\bv \dx\dt, \quad   {\mathlarger{\mathfrak{R}}}_{4,2}  = - \int_{I_*} \int_{\Omega_\eta} \nabla\bv \colon \nabla\mathcal{E}_\eta (\partial_t\eta\bn) (\nabla\bv)^\intercal \dx\dt. \]
Both terms are handled by the same sequence of arguments used in the estimate of ${\mathlarger{\mathfrak{R}}}_3 .$ The only difference is the choice of the Lebesgue exponent for the extension field in the H\"older step: in ${\mathlarger{\mathfrak{R}}}_{4,1}$  the extension field is taken in $L^{4}(\Omega_\eta),$ whereas in ${\mathlarger{\mathfrak{R}}}_{4,2}$ it is taken in $L^{2}(\Omega_\eta) , $ in accordance with the structure of each integrand.
We therefore, do not repeat the details and conclude that
\begin{align}\label{eq:R4}
{\mathlarger{\mathfrak{R}}}_{4,\mathsf{j}} \lesssim \varepsilon \int_{I_*} \Vert \bv \Vert_{W^{2, 2}(\Omega_\eta)}^2 \dt + c(\varepsilon) \int_{I_*} \Vert \partial_t\eta \Vert_{W^{1,2}(\omega)}^2 \Vert \bv \Vert_{W^{1,  2}(\Omega_\eta)}^2 \dt, \qquad \forall \; \mathsf{j} \in {1,2}. 
\end{align} 

We next address the term 
\[{\mathlarger{\mathfrak{R}}}_5 =   -  \int_{I_*} \int_{\partial\Omega_\eta}\big(\dot{\bv} \cdot \nabla\bv \big) \cdot \bn_\eta \dH\dt  . \] 
Applying H\"older's inequality together with  the interpolation identity 
\[L^{2}(\partial\Omega_\eta) = \big[ W^{-1/2,2}(\partial\Omega_\eta), W^{1/2,2}(\partial\Omega_\eta) \big]_{1\slash2}, \]
we obtain -- after identifying the material derivative with the acceleration of the shell 
\[{\mathlarger{\mathfrak{R}}}_5 \lesssim \int_{I_*} \Vert \bv \Vert_{W^{1/2,2}(\partial\Omega_\eta)}^{1/2} \Vert \nabla\bv \Vert_{W^{1/2,2}(\partial\Omega_\eta)}^{1/2}  \Vert \partial^{2}_t \eta \Vert_{L^{2}(\omega)} \dt ,\]
 which after applying the trace inequality and Young's inequality, further reduces to 
 \begin{align}\label{eq:R5}
 {\mathlarger{\mathfrak{R}}}_5 \lesssim \varepsilon \int_{I_*} \left( \Vert \nabla\bv \Vert_{W^{1, 2}(\Omega_\eta)}^2 + \Vert \partial^{2}_t \eta \Vert_{L^{2}(\omega)}^2  \right) \dt +  c(\varepsilon) \int_{I_*} \Vert \bv \Vert_{W^{1, 2}(\Omega_\eta)}^{2} \dt.
 \end{align} 

Considering the term 
\[{\mathlarger{\mathfrak{R}}}_6 = \int_{I_*} \int_{\partial\Omega_\eta} (\Div\bv) \dot{\bv}\cdot \bn_\eta \dH\dt , \]
its structure is identical to that of ${\mathlarger{\mathfrak{R}}}_5 , $ with $\nabla \bv$ replaced by $\Div\bv.$  Therefore, the argument is identical to that of ${\mathlarger{\mathfrak{R}}}_5, $ whence
\begin{align}\label{eq:R6}
 {\mathlarger{\mathfrak{R}}}_6 \lesssim \varepsilon \int_{I_*} \left( \Vert \nabla\bv \Vert_{W^{1, 2}(\Omega_\eta)}^2 + \Vert \partial^{2}_t \eta \Vert_{L^{2}(\omega)}^2  \right) \dt +  c(\varepsilon) \int_{I_*} \Vert \bv \Vert_{W^{1, 2}(\Omega_\eta)}^{2} \dt.
 \end{align} 

The term 
\[{\mathlarger{\mathfrak{R}}}_7 = \int_{I_*} \int_{\partial\Omega_\eta} (\partial_t\eta \bn)\circ\bm{\varphi}_{\eta}^{-1}\cdot \bn_\eta |\Div\bv|^2 \dH\dt,  \]
is the divergence analogue of ${\mathlarger{\mathfrak{R}}}_2. $  Consequently, the same trace--interpolation mechanism used for ${\mathlarger{\mathfrak{R}}}_2$ applies identically. Hence, 
\begin{align}\label{eq:R7}
{\mathlarger{\mathfrak{R}}}_7 \lesssim \varepsilon \int_{I_*} \Vert \nabla\bv \Vert_{W^{1, 2}(\Omega_\eta)}^2 \dt + c(\varepsilon) \int_{I_*} \Vert \partial_t\eta \Vert_{W^{1,2}(\omega)}^2 \Vert \Div\bv \Vert_{L^{ 2}(\Omega_\eta)}^2 \dt.
\end{align}

Furthermore, observe that 
\[{\mathlarger{\mathfrak{R}}}_8 = - \int_{I_*} \int_{\Omega_\eta} (\Div\bv)  \Div\big( \mathcal{E}_\eta (\partial_t\eta\bn)\cdot \nabla\bv  \big)\dx\dt =: {\mathlarger{\mathfrak{R}}}_{8,1} + {\mathlarger{\mathfrak{R}}}_{8,2},   \]
with 
\[{\mathlarger{\mathfrak{R}}}_{8,1}  = - \int_{I_*} \int_{\Omega_\eta} (\Div\bv) \Div(\nabla\bv)^\intercal \cdot \mathcal{E}_\eta (\partial_t\eta\bn) \dx\dt, \quad   {\mathlarger{\mathfrak{R}}}_{8,2}  = - \int_{I_*} \int_{\Omega_\eta} (\Div\bv) (\nabla\bv)^\intercal \colon \nabla\mathcal{E}_\eta (\partial_t\eta\bn) \dx\dt,  \]
mirrors ${\mathlarger{\mathfrak{R}}}_4. $ Proceeding exactly as in the estimate for ${\mathlarger{\mathfrak{R}}}_4, $ we obtain 
\begin{align}\label{eq:R8}
{\mathlarger{\mathfrak{R}}}_{8,\mathsf{j}} \lesssim \varepsilon \int_{I_*} \Vert \bv \Vert_{W^{2, 2}(\Omega_\eta)}^2 \dt + c(\varepsilon) \int_{I_*} \Vert \partial_t\eta \Vert_{W^{1,2}(\omega)}^2 \Vert \Div\bv \Vert_{L^{ 2}(\Omega_\eta)}^2 \dt, \qquad \forall \; \mathsf{j} \in {1,2}. 
\end{align}  

We proceed to estimate the first pressure-driven contribution
\[{\mathlarger{\mathfrak{R}}}_9 = \int_{I_*} \int_{\Omega_\eta} p(\rho)\Div\big( \mathcal{E}_\eta (\partial_t\eta\bn)\cdot \nabla\bv \big) \dx\dt =:  {\mathlarger{\mathfrak{R}}}_{9,1} + {\mathlarger{\mathfrak{R}}}_{9, 2}, \]
where
\[{\mathlarger{\mathfrak{R}}}_{9, 1} =  \int_{I_*} \int_{\Omega_\eta} p(\rho) \Div(\nabla\bv)^\intercal \cdot \mathcal{E}_\eta (\partial_t\eta\bn)\dx\dt, \qquad {\mathlarger{\mathfrak{R}}}_{9, 2} =  \int_{I_*} \int_{\Omega_\eta} p(\rho) (\nabla\bv)^\intercal\colon \nabla\big(\mathcal{E}_\eta (\partial_t\eta\bn)\big) \dx\dt.  \]
To estimate ${\mathlarger{\mathfrak{R}}}_{9, 1} $ we apply H\"older's inequality together with the Sobolev embedding 
\[W^{3/2,2}(\Omega_\eta) \hookrightarrow    L^{4}(\Omega_\eta), \]
 which yields 
\[{\mathlarger{\mathfrak{R}}}_{9, 1}  \lesssim \int_{I_*} \Vert \nabla\bv \Vert_{W^{1, 2}(\Omega_\eta)}  \Vert p(\rho) \Vert_{L^{4}(\Omega_\eta)} \Vert \mathcal{E}_\eta (\partial_t\eta\bn)\Vert_{W^{3/2, 2}(\Omega_\eta)} \dt.  \] 
Using the uniform boundedness of $\rho$ (and hence of $p(\rho)$), together with  the trace property  
\begin{align}\label{eq:Embed32}
\Vert \mathcal{E}_\eta (\partial_t\eta\bn) \Vert_{W^{3/2,2}(\Omega_\eta) }  \lesssim  \Vert \partial_t\eta\Vert_{W^{1, 2}(\omega)}, 
\end{align}
we obtain, after an application of Young's inequality  
\begin{align}\label{eq:R91}
{\mathlarger{\mathfrak{R}}}_{9,1} \lesssim \varepsilon \int_{I_*} \Vert \nabla\bv \Vert_{W^{1, 2}(\Omega_\eta)}^2 \dt + c(\varepsilon) \int_{I_*} \Vert \partial_t\eta \Vert_{W^{1,2}(\omega)}^2  \dt. 
\end{align} 

For the term ${\mathlarger{\mathfrak{R}}}_{9,2}, $ the uniform boundedness of $\rho$ and H\"older's inequality yield
\[{\mathlarger{\mathfrak{R}}}_{9,2} \lesssim \int_{I_*} \Vert \nabla\bv \Vert_{L^{ 2}(\Omega_\eta)}  \Vert p(\rho) \Vert_{L^{\infty}(\Omega_\eta)} \Vert \mathcal{E}_\eta (\partial_t\eta\bn)\Vert_{W^{1, 2}(\Omega_\eta)} \dt.  \]  
Using the embedding $W^{3/2,2}(\Omega_\eta) \hookrightarrow    W^{1,2}(\Omega_\eta) $  and  \eqref{eq:Embed32}, it follows from Young's inequality that 
\begin{align}\label{eq:R92}
{\mathlarger{\mathfrak{R}}}_{9,2} \lesssim \int_{I_*} \left( \Vert \nabla\bv \Vert_{L^{2}(\Omega_\eta)}^2  + \Vert \partial_t\eta \Vert_{W^{1,2}(\omega)}^2  \right) \dt. 
\end{align} 

The pressure-divergence term 
\[ {\mathlarger{\mathfrak{R}}}_{10} = \int_{I_*} \int_{\Omega_\eta} p(\rho)\partial_t\Div(\bv) \dx\dt \]
requires careful attention, since it couples a time derivative of the volumetric strain with the nonlinear pressure, and therefore cannot be estimated directly by H\"older or interpolation. Its analysis instead exploit the structure of the continuity equation  $\eqref{eq:ContMomentEq}_{1}$ which yields a renormalised continuity equation
\begin{align}\label{eq:RenormContEq}
\partial_t p(\rho) + \Div\big(p(\rho)\bv\big) + (\gamma -1)p(\rho)\Div\bv     = 0. 
\end{align}
Using \eqref{eq:RenormContEq}, we derive that  
\begin{align*}
{\mathlarger{\mathfrak{R}}}_{10}  &= \int_{I_*} \dfrac{\dd}{\dt} \int_{\Omega_\eta} p(\rho) \Div\bv \dt - \int_{I_*} \int_{\Omega_\eta} p(\rho)\bv\cdot \nabla\Div\bv  \dx\dt   + (\gamma - 1)\int_{I_*} \int_{\Omega_\eta} p(\rho) \left(\Div\bv\right)^2\dx\dt.
\end{align*}
From Young's inequality and the uniform boundedness of the density, we obtain  
\begin{align*}
{\mathlarger{\mathfrak{R}}}_{10} \lesssim \varepsilon \sup\limits_{I_{*}} \Vert \nabla\bv\Vert_{L^2(\Omega_\eta)}^2 + c(\varepsilon) + \int_{I_*}\Vert \nabla\bv\Vert_{L^2(\Omega_\eta)}^2  \dt - \int_{I_*} \int_{\Omega_\eta} p(\rho)\bv\cdot \nabla\Div\bv  \dx\dt.  
\end{align*}
To estimate the last term, we introduce the vector field 
\[ \mathbf{u} = \bv - \mathcal{E}_\eta (\partial_t\eta\bn),  \;\; \text{so that} \;\; \mathbf{u} = 0 \quad \text{on}\;\; \partial\Omega_\eta. \]
This yields the decomposition
\[-\int_{I_*} \int_{\Omega_\eta} p(\rho)\bv\cdot \nabla\Div\bv  \dx\dt =: {\mathlarger{\mathfrak{R}}}_{10}^a + {\mathlarger{\mathfrak{R}}}_{10}^b, \]
where 
\[{\mathlarger{\mathfrak{R}}}_{10}^\mathtt{a} = -\int_{I_*} \int_{\Omega_\eta} p(\rho)\mathbf{u}\cdot \nabla\Div\bv  \dx\dt, \quad {\mathlarger{\mathfrak{R}}}_{10}^\mathtt{b} = -\int_{I_*} \int_{\Omega_\eta} p(\rho)\mathcal{E}_\eta (\partial_t\eta\bn)\cdot \nabla\Div\bv  \dx\dt . \] 
Since $\bu $ vanishes on the boundary, Poincar\'e inequality and the trace extension bound  \eqref{eq:Embed32} give
\begin{align}\label{eq:ZeroEstimate}
\Vert \bu \Vert_{L^2(\Omega_\eta)} \lesssim \Vert \nabla\bv \Vert_{L^2(\Omega_\eta)}  + \Vert \partial_t\eta \Vert_{W^{1,2}(\omega)} .
\end{align}
Hence, using  the $L^\infty-$bound for $\rho, $   H\"older's inequality,  \eqref{eq:ZeroEstimate} and Young's inequality, we obtain
\begin{align}\label{eq:R10a}
{\mathlarger{\mathfrak{R}}}_{10}^\mathtt{a} \lesssim \varepsilon \int_{I_*} \Vert \nabla^2\bv \Vert_{L^{2}(\Omega_\eta)}^2 \dt + c(\varepsilon) \int_{I_*} \left( \Vert \partial_t\eta \Vert_{W^{1,2}(\omega)}^2 + \Vert \nabla\bv \Vert_{L^{2}(\Omega_\eta)}^2 \right) \dt .
\end{align} 
A similar argument yields
\begin{align}\label{eq:R10b}
{\mathlarger{\mathfrak{R}}}_{10}^\mathtt{b} \lesssim \varepsilon \int_{I_*} \Vert \nabla^2\bv \Vert_{L^{2}(\Omega_\eta)}^2 \dt + c(\varepsilon) \int_{I_*}  \Vert \partial_t\eta \Vert_{W^{1,2}(\omega)}^2 \dt .
\end{align}

Collecting the estimates \eqref{eq:R1}, \eqref{eq:R2}--\eqref{eq:R8}, \eqref{eq:R91}--\eqref{eq:R92} and \eqref{eq:R10a}--\eqref{eq:R10b}, we obtain 
\begin{align}\label{eq:R1R10}
&\sup\limits_{I_*} \int_{\Omega_\eta} |\nabla\bv|^2\dx +  \sup\limits_{I_*} \int_{\Omega_\eta} |\Div\bv|^2\dx  + \int_{I_*}\int_{\Omega_\eta} \rho |\dot{\bv}|^2 \dx\dt + \int_{I_*}\int_\omega |\partial^{2}_t \eta|^2 \dy\dt + \sup\limits_{I_*} \int_\omega |\partial_t \naby\eta|^2 \dy 
\nonumber
\\[0.4em]
& \lesssim  \varepsilon \bigg( \sup\limits_{I_{*}} \Vert \nabla\bv\Vert_{L^2(\Omega_\eta)}^2  +  \int_{I_*} \left( \Vert \bv \Vert_{W^{ 2, 2}(\Omega_\eta)}^2 + \Vert \partial^{2}_t \eta \Vert_{L^{2}(\omega)}^2  \right) \dt \bigg)  
\nonumber
\\[0.4em]
&\quad + c(\varepsilon) \int_{I_*} \left( \Vert \rho^{1/2}\bv\Vert_{L^{\mathtt{r}}(\Omega_\eta)}^{\mathtt{s}} + \Vert \partial_t\eta \Vert_{W^{1,2}(\omega)}^2 \right) \Vert \nabla\bv \Vert_{L^{ 2}(\Omega_\eta)}^2  \dt  
\nonumber
\\[0.4em]
&\quad + c(\varepsilon) \int_{I_*} \Vert \partial_t\eta \Vert_{W^{1,2}(\omega)}^2 \Vert \Div\bv \Vert_{L^{ 2}(\Omega_\eta)}^2   \dt + c(\varepsilon) \int_{I_*} \left( \Vert \partial_t\eta \Vert_{W^{1,2}(\omega)}^2 + \Vert \nabla\bv \Vert_{L^{2}(\Omega_\eta)}^2 \right) \dt
\\
&\quad + \int_{\Omega_{\eta_0}} |\nabla\bv_0|^2\dx + \int_{\Omega_{\eta_0}} |\Div\bv_0|^2\dx + \int_\omega |\naby\eta_*|^2 \dy  \nonumber
\\[0.4em]
&\quad + \sup\limits_{I_*} \int_\omega |\naby\Dely\eta|^2 \dy   + \int_{I_*}\int_\omega |\partial_t\Dely\eta|^2 \dy \dt \nonumber.
\end{align}
The last two terms can be estimated by testing the shell equation \eqref{eq:ShellEq}  with $\partial_t\Dely\eta ,  $ which yields
\begin{equation}\label{eq:EstimateR1415}
\begin{aligned}
&\frac{1}{2} \sup\limits_{I_*} \int_{\omega} |\partial_t \naby\eta|^2 \dy + \int_{I_*}\int_\omega |\partial_t \Dely\eta|^2\dy\dt + \frac{1}{2} \sup\limits_{I_*} \int_{\omega} | \naby\Dely\eta|^2 \dy
\\
& = \frac{1}{2} \int_\omega |\naby\eta_*|^2\dy + \frac{1}{2} \int_\omega |\naby\Dely\eta_0|^2\dy + \int_{I_*}\int_\omega {\mathlarger{\mathtt{F}}}_\eta \partial_t\Dely\eta \dy\dt. 
\end{aligned}
\end{equation}
However, we have 
\begin{align*}
 \int_{I_*}\int_\omega {\mathlarger{\mathtt{F}}}_\eta \partial_t\Dely\eta \dy\dt & \leq   \int_{I_*} \Vert {\mathlarger{\mathtt{F}}}_\eta \Vert_{W^{1/2,2}(\omega)} \Vert \partial_t \Dely\eta \Vert_{W^{-1/2, 2}(\omega)} \dt
 \\
 & \leq  \int_{I_*} \Vert \bm{\tau}\Vert_{W^{1/2,2}(\partial\Omega_\eta)} \Vert \partial_t \eta \Vert_{W^{3/2, 2}(\omega)} \dt
 \\
 & \lesssim  \int_{I_*} \big( \Vert \nabla^2\bv\Vert_{L^2(\Omega_\eta)} + \Vert \nabla p\Vert_{L^2(\Omega_\eta)}  \big) \Vert \partial_t \eta \Vert_{W^{3/2, 2}(\omega)} \dt
 \\
 & \lesssim \varepsilon \int_{I_*} \left( \Vert \nabla^2\bv\Vert_{L^2(\Omega_\eta)}^2 +  \Vert \nabla p\Vert_{L^2(\Omega_\eta)}^2 + \Vert \partial_t\Dely\eta\Vert_{L^2(\omega)}^2   \right) \dt  + c(\varepsilon) \int_{I_*} \Vert \partial_t\eta\Vert_{W^{1,2}(\omega)}^2 \dt,
\end{align*}
where we have used the Sobolev  interpolation inequality
\begin{equation}\label{eq:SobolevEta}
\Vert \partial_t \eta \Vert_{W^{3/2, 2}(\omega)}^2 \leq \Vert \partial_t \eta \Vert_{W^{1, 2}(\omega)}\Vert \partial_t \eta \Vert_{W^{2, 2}(\omega)}.
\end{equation}
Hence, it follows from \eqref{eq:EstimateR1415} that
\begin{equation}\label{eq:R1415}
\begin{aligned}
& \sup\limits_{I_*} \int_{\omega} |\partial_t \naby\eta|^2 \dy + \int_{I_*}\int_\omega |\partial_t \Dely\eta|^2\dy\dt + \sup\limits_{I_*} \int_{\omega} |\partial_t \naby\Dely\eta|^2 \dy
\\
& \lesssim  \varepsilon \int_{I_*} \left( \Vert \nabla^2\bv\Vert_{L^2(\Omega_\eta)}^2 +  \Vert \nabla p\Vert_{L^2(\Omega_\eta)}^2 \right) \dt  + c(\varepsilon) \int_{I_*} \Vert \partial_t\eta\Vert_{W^{1,2}(\omega)}^2 \dt + \int_\omega |\naby\eta_*|^2\dy 
\\
&\quad   +  \int_\omega |\naby\Dely\eta_0|^2\dy. 
\end{aligned}
\end{equation}
Combining \eqref{eq:R1415} and \eqref{eq:R1R10}, we deduce that 
\begin{align}\label{eq:R1R15}
&\sup\limits_{I_*} \int_{\Omega_\eta} |\nabla\bv|^2\dx +  \sup\limits_{I_*} \int_{\Omega_\eta} |\Div\bv|^2\dx  + \int_{I_*}\int_{\Omega_\eta} \rho |\dot{\bv}|^2 \dx\dt + \int_{I_*}\int_\omega |\partial^{2}_t \eta|^2 \dy\dt + \sup\limits_{I_*} \int_\omega |\partial_t \naby\eta|^2 \dy 
\nonumber
\\
&  \quad + \int_{I_*}\int_\omega |\partial_t \Dely\eta|^2\dy\dt + \sup\limits_{I_*} \int_{\omega} |\partial_t \naby\Dely\eta|^2 \dy
\nonumber
\\
& \lesssim \varepsilon \bigg( \sup\limits_{I_{*}} \Vert \nabla\bv\Vert_{L^2(\Omega_\eta)}^2  +  \int_{I_*} \left( \Vert \nabla^2\bv \Vert_{L^{ 2}(\Omega_\eta)}^2 + \Vert \nabla p\Vert_{L^2(\Omega_\eta)}^2 + \Vert \partial^{2}_t \eta \Vert_{L^{2}(\omega)}^2  \right) \dt \bigg)  
\nonumber
\\
&\quad + c(\varepsilon) \int_{I_*} \left( \Vert \rho^{1/2}\bv\Vert_{L^{\mathtt{r}}(\Omega_\eta)}^{\mathtt{s}} + \Vert \partial_t\eta \Vert_{W^{1,2}(\omega)}^2 \right) \Vert \nabla\bv \Vert_{L^{ 2}(\Omega_\eta)}^2  \dt  
\nonumber
\\
&\quad + c(\varepsilon) \int_{I_*} \Vert \partial_t\eta \Vert_{W^{1,2}(\omega)}^2 \Vert \Div\bv \Vert_{L^{ 2}(\Omega_\eta)}^2   \dt + c(\varepsilon) \int_{I_*} \left( \Vert \partial_t\eta \Vert_{W^{1,2}(\omega)}^2 + \Vert \nabla\bv \Vert_{L^{2}(\Omega_\eta)}^2 \right) \dt
\\
&\quad + \int_{\Omega_{\eta_0}} |\nabla\bv_0|^2\dx + \int_{\Omega_{\eta_0}} |\Div\bv_0|^2\dx + \int_\omega |\naby\eta_*|^2\dy  +  \int_\omega |\naby\Dely\eta_0|^2\dy \nonumber.
\end{align}
To close the acceleration estimate \eqref{eq:R1R15} and thereby complete the proof, it remains to control 
\[\Vert \bv \Vert_{W^{2,  2}(\Omega_\eta)} \quad \text{and}\quad   \Vert \nabla p\Vert_{L^2(\Omega_\eta)} .\]
However, such a control cannot be obtained directly through the time-evolution energy method used thus far, and must instead be deduced from the elliptic structure inherent in the momentum equation $\eqref{eq:ContMomentEq}_2 .$ 
For this purpose, we freeze the time variable and regard the pair $(\bv, p) $ as the solution of a stationary Stokes-type system with prescribed boundary data for $\bv .$ The required control of $\nabla^2\bv $ and $p$ then follows from the next lemma.


\begin{lemma}\label{lem:MRlowerOrderVp}
Let $(\bv, p) $ be the unique solution of the boundary value problem 
\begin{equation}\label{eq:BVP}
\left\{\begin{aligned}
&\mathcal{A} (\bv, p)^\intercal
= 
\big( \rho\dot{\bv} , (2\mu + \lambda)^{-1}{\mathlarger{\mathfrak{F}}} \big)^\intercal &&\text{ in }  \Omega_\eta,\\[0.4em]
&\mathcal{B}(\bv, p)^\intercal
= 
\left(\partial_t\eta\bn\right)\circ\bfvarphi_{\eta}^{-1} &&\text{ on }  \partial\Omega_\eta,
\end{aligned}\right.
\end{equation}
where the interior operator is
\[\mathcal{A}  := \begin{pmatrix}
\mu\Delta \mathbb{I}_{3\times3} + (\mu+ \lambda)\nabla\Div &   & -\nabla \\\\
\Div &  & -(2\mu + \lambda)^{-1}
\end{pmatrix},
\]
the boundary operator is
\[\mathcal{B}  := \begin{pmatrix}
 \mathbb{I}_{3\times3}  &   & \mathbf{0}_3 \end{pmatrix},
\]
and 
\[ {\mathlarger{\mathfrak{F}}} := (2\mu + \lambda)\Div\bv - p(\rho) \]
denotes the effective viscous flux. 
\medskip

Then, for every fixed time $t \in I_*$,  there exists a constant 
\[
C = C\left(\mu, \lambda, \Omega_\eta, \Vert\rho\Vert_{L^{\infty}\left(I_* ; L^\infty( \Omega_\eta)\right) } \right) >0 ,
\]
such that the elliptic estimate
\begin{equation}\label{eq:EllipticRegularity}
\Vert \bv \Vert_{W^{2, 2}(\Omega_\eta)} + \Vert  p\Vert_{W^{1,2}(\Omega_\eta)} \leq C \left( \Vert \rho^{1/2}\dot{\bv} \Vert_{L^2(\Omega_\eta)} + \Vert \partial_t \eta\Vert_{W^{3/2,2}(\omega)} \right)
\end{equation}
holds. 

\end{lemma}


\begin{proof}\phantom{go to next line}\\
We start by observing that the momentum equation implies -- at the level of differential identities -- that, at each fixed time, the effective viscous flux ${\mathlarger{\mathfrak{F}}}  $ solves  a scalar elliptic problem of Poisson type. More precisely, taking the divergence of the momentum equation $\eqref{eq:ContMomentEq}_2 ,$ we obtain that ${\mathlarger{\mathfrak{F}}}  $ satisfies 
\[\Delta{\mathlarger{\mathfrak{F}}}   = \Div(\rho\dot{\bv}) \quad \text{in } \; \Omega_\eta, \]
with natural Neumann boundary condition 
\[\partial_{\bn_\eta}{\mathlarger{\mathfrak{F}}}  =  (\rho\dot{\bv})\cdot \bn_\eta                 \quad \text{on } \; \partial\Omega_\eta . \]  
Importantly, due to the interface condition \eqref{eq:interfaceCond},  
\[
\dot{\bv}  = \left(\partial^2_t \eta \bn \right)\circ \bm{\varphi}^{-1}_\eta  \qquad \text{ on } \;  \omega. 
\]

Since ${\mathlarger{\mathfrak{F}}}  $ is determined only up to an additive constant, we fix $ \int_{\Omega_\eta}{\mathlarger{\mathfrak{F}}} \dx = 0.  $ This renormalisation is completely harmless as the pressure $p $ is itself uniquely determined modulo constants. Thus, the weak formulation yields the following a priori estimate 
\begin{equation}\label{eq:EstimateFlux}
\Vert \nabla {\mathlarger{\mathfrak{F}}}   \Vert_{L^{2}(\Omega_\eta)} \leq C  \Vert \rho^{1/2} \dot{\bv}  \Vert_{L^{2}(\Omega_\eta)},
\end{equation}
for some constant $C = C\left(\mu, \lambda, \Omega_\eta, \Vert\rho\Vert_{L^{\infty}\left(I_* ; L^\infty( \Omega_\eta)\right) } \right) >0 .$\\

\noindent Let 
\[\mathcal{A}^{(0)}(\bm\xi)  := \begin{pmatrix}
-\mu|\bm{\xi}|^2 \mathbb{I}_{3\times3} - (\mu+ \lambda)\bm{\xi}\otimes\bm{\xi} &   & -i\bm{\xi}  \\\\
i\bm{\xi}^\intercal &  & -(2\mu + \lambda)^{-1}
\end{pmatrix}, \quad \bm{\xi} \in \R^3 ,\]
be the principal symbol of the operator $\mathcal{A}, $ with respect to the Douglis--Nirenberg (DN) weights $\mathbf{t} = (2, 1)$ for the unknowns and $\mathbf{s} = (0, -1)$ for the two interior equations.  We denote by 
\[\mathcal{B}^{(0)}(\bm{\xi})  := \begin{pmatrix}
 \mathbb{I}_{3\times3}  &   & \mathbf{0}_3 \end{pmatrix},\]
 the principal symbol of the operator $\mathcal{B} , $ with boundary weights $\mathbf{r} = (-2, -1). $\\

\medskip

\noindent
{\bf{Claim.}} The boundary value problem \eqref{eq:BVP} is DN elliptic in the sense that \\[-0.2cm]

\begin{itemize}

\item[(1)]  The principal symbol $\mathcal{A}^{(0)}(\bm{\xi}) $ is DN elliptic (injective for every $\bm{\xi} \neq \mathbf{0}_3$);  \\[-0.1cm]

\item[(2)] The boundary operator $\mathcal{B} $ satisfies the Shapiro--Lopatinskii (SL) condition. In the general form of the theory, the SL condition requires that, for each nonzero tangential frequency $\bm{\xi}' \in \R^2\setminus \{\mathbf{0}_2\} , $ the following initial value half space problem   
\begin{equation}\label{eq:SL}
\left\{\begin{aligned}
&\mathcal{A}^{(0)} (\bm{\xi}', -i\partial_{x_3}) (\mathbf{w}, \pi)^\intercal (x_3)
= 
0 & \text{ for }  x_3 > 0,\\[0.15cm]
&\mathcal{B}^{(0)}(\bm{\xi}', -i\partial_{x_3}) (\mathbf{w}, \pi)^\intercal (x_3) \Big|_{x_3 = 0}
= 
\mathbf{d},  & 
\end{aligned}\right.
\end{equation}
has a unique solution $(\mathbf{w}, \pi) \in {\mathlarger{\mathfrak{M}}}_+ $ for all $\mathbf{d} \in \mathbb{C}^3, $ where $\partial_{x_3} := \dfrac{\dd}{\dd x_3} $ and 
\[{\mathlarger{\mathfrak{M}}}_+ := \Big\{ (\mathbf{w}, \pi)  \colon \R \to \mathbb{C}^3 \times \mathbb{C}\;   \text{    which tends to zero as } \;  x_3 \to \infty                   \Big\}       .\]

\end{itemize}

\medskip

\noindent
Indeed, it is immediate that $\mathcal{A}^{(0)}(\bm{\xi}) $ is invertible for all $\bm{\xi} \neq \mathbf{0}_3 .$ Thus DN ellipticity holds for the interior operator. Moreover, using the second interior equation of the system \eqref{eq:SL} to eliminate the unknown $\pi, $ the SL condition reduces to the requirement that the ODE
\begin{equation}\label{eq:SLODE}
\left\{\begin{aligned}
&\mu\left(|\bm{\xi}'|^2  - \partial^{2}_{x_3} \right)\mathbf{w} - \mu (\bm{\xi}\otimes\bm{\xi}) \mathbf{w}
= 
0 & \text{ in }  \{ x_3 > 0\},\\[0.15cm]
&\mathbf{w}(\mathbf{0}_3) 
= 
\mathbf{d},  & 
\end{aligned}\right.
\end{equation}
with $\bm{\xi} = (\bm{\xi}' , -i\partial_{x_3})^{\intercal} = (\xi_1, \xi_2 , -i\partial_{x_3})^{\intercal} ,  $ has a unique solution which decays as $x_3 \to \infty. $ \\[-0.8em]

\medskip

\noindent 
Assume that  $(\mathbf{w}, \pi)$ is a solution of the homogeneous  initial value problem  \eqref{eq:SLODE}  on $\{x_3 > 0\}$ such that 
\[\mathbf{w}(\mathbf{0}_3) = 0, \quad (\mathbf{w},\pi)(x_3) \to 0 \;\; \text{ as } x_3 \to \infty. \]
we show in the sequel that necessarily $\mathbf{w} \equiv \mathbf{0}_3 $ and $ \pi \equiv 0. $ 

Indeed, multiplying \eqref{eq:SLODE} by $\overline{\mathbf{w}} = (\overline{w}_1, \overline{w}_2, \overline{w}_3) $ (complex conjugate) and  integrating in $x_3 \in (0, \infty) , $ we obtain  -- thanks to the decay at $\infty $ and the homogeneous boundary condition,
\[\mu \int_{0}^{\infty} \left( |\partial_{x_3}\mathbf{w}|^2 + |\bm{\xi}'|^2|\mathbf{w}|^2      \right)\dd x_3 - \mu \int_{0}^{\infty} \left( \xi_1^2 |w_1|^2 + \xi_2^2 |w_2|^2  + 2\xi_2 \xi_1 \text{Re}(\overline{w}_2 w_1)  - \partial^2_3 w_3 \overline{w}_3      \right) \dd x_3                       =0             , \]
where  $\text{Re}(z) $ denotes the real part of  $z \in \mathbb{C}. $
However, using the inequality 
\[2\xi_2 \xi_1 \text{Re}(\overline{w}_2 w_1) \leq \xi_2^2 |w_1|^2    + \xi_1^2 |w_2|^2 , \]
it follows that 
\begin{align*}
&\mu \int_{0}^{\infty} \left( |\partial_{x_3}\mathbf{w}|^2 + |\bm{\xi}'|^2|\mathbf{w}|^2      \right)\dd x_3 - \mu \int_{0}^{\infty} \left( \xi_1^2 |w_1|^2 + \xi_2^2 |w_2|^2  + 2\xi_2 \xi_1 \text{Re}(\overline{w}_2 w_1)  - \partial^2_3 w_3 \overline{w}_3      \right) \dd x_3 
\\
&\quad \geq \mu  \int_{0}^{\infty} \left(  |\partial_{x_3}w_1|^2 + |\partial_{x_3}w_2|^2 + |\bm{\xi}'|^2 |w_3|^2   \right) \dd x_3
\\
&\quad \geq 0. 
\end{align*}
Thus, $w_3 = \partial_{x_3}w_1 = \partial_{x_3}w_2 =   0 . $ The boundary condition then implies 
\[\mathbf{w} \equiv \mathbf{0}_3, \quad \text{hence} \;\; \pi \equiv 0.  \]
The homogeneous half-space initial value problem has no nontrivial decaying solution and consequently the SL condition \eqref{eq:SL} holds. 

Since the boundary value problem \eqref{eq:BVP} is DN elliptic, it follows from \cite[Section 4, Theorem 4.1]{krupchyk2006sharpiro}
that the associated boundary value operator 
\[\mathcal{O} \colon  W^{2,2}(\Omega_\eta) \times W^{1,2}(\Omega_\eta) \to L^{2}(\Omega_\eta) \times W^{1,2}(\Omega_\eta) \times W^{3/2,2}(\partial\Omega_\eta); \quad \mathcal{O} = (\mathcal{A}, \mathcal{B}) ,  \]
is a  Fredholm operator and that the following  a priori estimate holds  
\begin{equation}\label{eq:VPclosed}
\Vert \bv \Vert_{W^{2, 2}(\Omega_\eta)} +  \Vert \nabla p\Vert_{L^2(\Omega_\eta)}  \leq C \left(  \Vert \rho^{1/2}\dot{\bv} \Vert_{L^{ 2}(\Omega_\eta)}      +  \Vert \nabla{\mathlarger{\mathfrak{F}}} \Vert_{L^2(\Omega_\eta)}   + \Vert \partial_t \eta \circ\bm{\varphi}_{\eta}^{-1}  \Vert_{W^{3/2,2}(\partial\Omega_\eta)} \right),
\end{equation}
for  some constant $C = C\left(\mu, \lambda, \Omega_\eta, \Vert\rho\Vert_{L^{\infty}\left(I_* ; L^\infty( \Omega_\eta)\right) } \right) >0 .$\\
Combining \eqref{eq:VPclosed} and   \eqref{eq:EstimateFlux} with the boundary estimate 
\[\Vert \partial_t \eta \circ\bm{\varphi}_{\eta}^{-1}  \Vert_{W^{3/2,2}(\partial\Omega_\eta)} \lesssim \Vert \partial_t \eta  \Vert_{W^{3/2,2}(\omega)} \]
yields the desired inequality and completes the proof.

\end{proof}



\section{Proof of  \cref{theo:MainResult2} } \label{sec:Proof2}

The proof proceeds through  a sequence of conditional a priori estimates for the density $\rho, $ the structural displacement $\eta $ and the fluid velocity $\bv. $ None of these estimates is closed on its own; each depends on quantities that are estimated in subsequent parts of the argument.  \\
For later use, we first recall  the maximal regularity estimate associated with the shell equation \eqref{eq:ShellEq}. 
\begin{remark}\label{rem:HigherOrderEta}
Rewriting \eqref{eq:ShellEq} as a first-order in time evolution equation and  using the invariance of maximal $L^2-$regularity under shifts of the fractional domain scale   (cf. e.g.,  \cite[Section 2, Theorem 2.2 ]{denk2007optimal}) we obtain  $\forall \, \sigma \in [0, 2]$
\begin{equation}\label{eq:MRforLemma41prior}
\begin{aligned}
& \int_{I_*} \left( \Vert \partial_t^2 \eta \Vert_{W^{\sigma, 2}(\omega)}^2 + \Vert \partial_t \eta \Vert_{W^{\sigma + 2, 2}(\omega)}^2 + \Vert  \eta \Vert_{W^{\sigma + 4, 2}(\omega)}^2  \right) \dt 
\\[0.4em]
& \quad \lesssim \int_{I_*} \left( \Vert \bv \Vert_{W^{\sigma + 3/2,2}(\Omega_\eta)}^2 + \Vert p \Vert_{W^{\sigma + 1/2,2}(\Omega_\eta)}^2 \right) \dt + \Vert \eta_0 \Vert_{W^{5,2}(\omega)}^2 + \Vert \eta_* \Vert_{W^{3,2}(\omega)}^2.
\end{aligned}
\end{equation}
Moreover, using the continuous embedding (cf.  \cite[Section 3, Theorem 3.1]{LionsMagenes1})
\begin{equation}\label{eq:GeneralEmbedding}
W^{1,2}\left(I_*; W^{\sigma, 2}(\omega) \right) \cap  L^{2}\left( I_*; W^{\theta, 2}(\omega)  \right) \hookrightarrow  C\left(\overline{I_*}; W^{(\sigma+\theta)/2, 2}(\omega)\right) \qquad \forall \, 0 < \sigma < \theta, 
\end{equation}
we deduce from \eqref{eq:MRforLemma41prior} that 
\begin{equation}\label{eq:MRforLemma41}
\begin{aligned}
& \int_{I_*} \left( \Vert \partial_t^2 \eta \Vert_{W^{\sigma, 2}(\omega)}^2 + \Vert \partial_t \eta \Vert_{W^{\sigma+2, 2}(\omega)}^2 + \Vert  \eta \Vert_{W^{\sigma + 4, 2}(\omega)}^2  \right) \dt  +  \sup\limits_{I_*} \left(  \Vert \partial_t \eta \Vert_{W^{\sigma + 1, 2}(\omega)}^2  + \Vert \eta \Vert_{W^{\sigma + 3, 2}(\omega)}^2 \right)
\\[0.4em]
& \quad \lesssim \int_{I_*} \left( \Vert \bv \Vert_{W^{\sigma +3/2,2}(\Omega_\eta)}^2 + \Vert p \Vert_{W^{\sigma +1/2,2}(\Omega_\eta)}^2 \right) \dt + \Vert \eta_0 \Vert_{W^{5,2}(\omega)}^2 + \Vert \eta_* \Vert_{W^{3,2}(\omega)}^2.
\end{aligned}
\end{equation}
\end{remark}
We now proceed with the conditional higher-order estimate for the density. \\


\medskip

\noindent {\textbf{Step 1: Conditional estimate for the density.} } \\
Consider the continuity equation $\eqref{eq:ContMomentEq}_1, $ that is, 
\[
\partial_t\rho + \nabla\rho\cdot \bv + \rho\Div\bv = 0 .
\] 
We have  $\forall \, \mathtt{k}  \in \{1,2\}$
\begin{align*}
\Vert \partial_t \rho \Vert_{W^{ \mathtt{k},2}(\Omega_\eta)}^2 &\lesssim \Vert \Div(\rho \bv) \Vert_{W^{ \mathtt{k},2}(\Omega_\eta)}^2 
\\
& \lesssim \Vert \rho \Vert_{W^{ \mathtt{k} + 1,2}(\Omega_\eta)}^2 \Vert \bv \Vert_{W^{ \mathtt{k} + 1,2}(\Omega_\eta)}^2. 
\end{align*}
Hence, 
\begin{equation} \label{eq:DensityTime1}
\sup\limits_{I_*} \Vert \partial_t \rho \Vert_{W^{ \mathtt{k}, 2}(\Omega_\eta)}^2 \lesssim \sup\limits_{I_*} \Vert \rho \Vert_{W^{ \mathtt{k} + 1,2}(\Omega_\eta)}^2 \sup\limits_{I_*}\Vert \bv \Vert_{W^{ \mathtt{k} + 1,2}(\Omega_\eta)}^2.
\end{equation}
However, from the continuous embedding  
\[
W^{1,2}\big(I_*; W^{\mathtt{k}, 2}(\Omega_\eta)\big) \cap L^2\big( I_*; W^{ \mathtt{k} + 2,2}(\Omega_\eta)   \big)  \hookrightarrow  C\left(\overline{I_*}; W^{ \mathtt{k} +1 ,2}(\Omega_\eta) \right),
\] 
the estimate \eqref{eq:DensityTime1} becomes 
\begin{equation}\label{eq:DensityTime2}
\sup\limits_{I_*} \Vert \partial_t \rho \Vert_{W^{ \mathtt{k},2}(\Omega_\eta)}^2 \lesssim \sup\limits_{I_*} \Vert \rho \Vert_{W^{ \mathtt{k} +1,2}(\Omega_\eta)}^2 \left(  \Vert \bv \Vert_{L^{2}\left( I_*; W^{ \mathtt{k} +2, 2}(\Omega_\eta)\right)}^2 + \Vert \partial_t \bv \Vert_{L^{2}\left( I_*; W^{\mathtt{k}, 2}(\Omega_\eta)\right)}^2     \right).
\end{equation}
Recall from the method of characteristics,  the identity
\[\rho(t,\bx) =  \rho_0\big(\mathbf{\Phi}(0,t,\bx)\big) \exp{ \bigg( -\int_0^t \Div\bv\big(s, \mathbf{\Phi}(s,t,\bx)\big)\dd s \bigg)} \qquad \forall\; \bx \in \Omega_\eta.\]  
Moreover,  $\forall\; t \in I_*, $ the pointwise estimate 
\begin{equation}\label{eq:DensityP}
\Vert \rho  \Div\bv \Vert_{W^{ \mathtt{k} +1, 2}(\Omega_\eta)} \lesssim \Vert \Div\bv \Vert_{W^{ \mathtt{k} + 1,2}(\Omega_\eta)} \Vert \rho \Vert_{W^{ \mathtt{k} +1, 2}(\Omega_\eta)}
\end{equation}
holds, and,  by  the regularity of the fluid velocity $\bv$, we have  $\Vert \Div\bv \Vert_{W^{ \mathtt{k} +1, 2}(\Omega_\eta)}  \in L^1(I_*). $  Thus, applying   \cite[Chapter 3, Theorem 3.14 \& Remark 3.17]{BahouriCheminDanchin2011}, we deduce that, for some constant $c > 0 $ independent of time,  
\begin{equation}\label{eq:Density1}
\sup\limits_{I_*} \Vert \rho \Vert_{W^{\mathtt{k} + 1,2}(\Omega_\eta)}^2 \lesssim \Vert \rho_0 \Vert_{W^{\mathtt{k} +1,2}(\Omega_\eta)}^2 \exp{\bigg( c\int_{I_*} \Vert \nabla\bv \Vert_{W^{\mathtt{k} +1 ,2}(\Omega_\eta)} \dt  \bigg)}.
\end{equation}
Therefore, combining \eqref{eq:DensityTime2} and \eqref{eq:Density1} yields the conditional estimate
\begin{equation}\label{eq:DensityCond}
\begin{aligned}
&\sup\limits_{I_*} \big( \Vert \partial_t \rho \Vert_{W^{\mathtt{k}, 2}(\Omega_\eta)}^2  + \Vert \rho \Vert_{W^{\mathtt{k} +1,2}(\Omega_\eta)}^2 \big) 
\\[0.4em]
&\quad  \lesssim \Vert \rho_0 \Vert_{W^{\mathtt{k} +1,2}(\Omega_\eta)}^2 \!\!\left(  1 +   \Vert \bv \Vert_{L^{2}\left( I_*; W^{\mathtt{k} +2, 2}(\Omega_\eta)\right)}^2 + \Vert \partial_t\bv \Vert_{L^{2}\left( I_*; W^{\mathtt{k}, 2}(\Omega_\eta)\right)}^2     \!\right) \exp{\bigg(\! c\int_{I_*}\! \Vert \nabla\bv \Vert_{W^{\mathtt{k} + 1,2}(\Omega_\eta)} \dt  \!\bigg)}.
\end{aligned}
\end{equation}

\begin{remark}\label{rem:UniformBoundDensity}
By \eqref{eq:Density1}, it follows from H\"older's inequality that , 
\begin{equation}\label{eq:Density1Final}
\sup\limits_{I_*} \Vert \rho \Vert_{W^{\mathtt{k} +1,2}(\Omega_\eta)}^2 \lesssim \Vert \rho_0 \Vert_{W^{\mathtt{k}  +1,2}(\Omega_\eta)}^2 \exp{ \left( c\sqrt{T_*} \bigg( \int_{I_*} \Vert \bv \Vert_{W^{\mathtt{k} +2,2}(\Omega_\eta)}^2 \dt  \bigg)^{1/2} \right) }.
\end{equation}
\end{remark} 
In particular, for $\mathtt{k} = 2$, \eqref{eq:Density1Final} yields an exponential dependence on $\Vert \bv \Vert_{L^{2}\left( I_*; W^{4,2}(\Omega_\eta)  \right)}$ -- a quantity  to be controlled. Such  a dependence is not suitable for the subsequent analysis, since it prevents closing the estimates. This therefore, motivates the derivation of an alternative bound.  
\begin{proposition}\label{prop:DensityW32BoundV}
It holds that 
\begin{equation}\label{eq:eq:Density1Finalk2-New}
\begin{aligned}
\sup\limits_{I_*} \Vert \rho \Vert_{W^{3,2}(\Omega_\eta)}^2   & \lesssim \Vert \rho_0 \Vert_{W^{3,2}(\Omega_\eta)}^2 + \kappa \int_{I_*} \Vert \bv \Vert_{W^{4,2}(\Omega_\eta)}^2 \dt   + c(\kappa)\int_{I_*}\Vert \rho \Vert_{L^{\infty}(\Omega_\eta)}^2 
\Vert \rho \Vert_{W^{3,2}(\Omega_\eta)}^2  \dt
\\[0.4em]
& \quad +  c(\kappa) \left(   \int_{I_*} \Vert \nabla\bv \Vert_{L^{\infty}(\Omega_\eta)}^2  \Vert \rho \Vert_{W^{3,2}(\Omega_\eta)}^2  \dt 
+  \int_{I_*} \Vert \nabla\rho \Vert_{L^{\infty}(\Omega_\eta)}^2   \Vert \bv \Vert_{W^{3,2}(\Omega_\eta)}^2   
 \dt \right),
\end{aligned}
\end{equation}
for arbitrary $\kappa > 0$.

\end{proposition}

\begin{proof}
Consider the continuity equation $\eqref{eq:ContMomentEq}_1$,  that is, 
\begin{equation*}
\partial_t\rho + \nabla\rho\cdot \bv + \rho\Div\bv = 0.
\end{equation*}
Taking the spatial derivative $\partial^{\bm{\alpha}}_\bx$ of $\eqref{eq:ContMomentEq}_1$ for all multi-index $\bm{\alpha} = (\alpha_1, \alpha_2, \alpha_3) \in \mathbb{N}^3$ such that \phantom{such that } $| \bm{\alpha} | := \alpha_1 + \alpha_2 + \alpha_3 \leq 3$, we obtain 
\begin{equation}\label{eq:ContEqAlphaD}
\partial_t \partial^{\bm{\alpha}}_\bx \rho + \left(\bv\cdot \nabla \right)\partial^{\bm{\alpha}}_\bx \rho = - \left(\partial^{\bm{\alpha}}_\bx \rho \right)\Div\bv +  {\mathlarger{\mathtt{C}_1 }}  +  {\mathlarger{\mathtt{C}_2 }},
\end{equation}
with commutators 
\[
 {\mathlarger{\mathtt{C}_1 }}  = \left(\bv\cdot \nabla \right)\partial^{\bm{\alpha}}_\bx \rho - \partial^{\bm{\alpha}}_\bx \big( (\bv \cdot\nabla)\rho \big) \qquad \text{and } \quad  {\mathlarger{\mathtt{C}_2 }}  =  \left(\partial^{\bm{\alpha}}_\bx \rho \right)\Div\bv - \partial^{\bm{\alpha}}_\bx  \big( \rho \Div\bv \big).
\]
Testing  \eqref{eq:ContEqAlphaD} with $\partial^{\bm{\alpha}}_\bx \rho$ and using Reynold's transport theorem, we obtain for all  $ t \in I_*$
\begin{equation}\label{eq:ContEqTested}
\dfrac{1}{2} \Vert \partial^{\bm{\alpha}}_\bx \rho(t) \Vert_{L^2(\Omega_\eta)}^2 = \dfrac{1}{2} \Vert \partial^{\bm{\alpha}}_\bx \rho_0 \Vert_{L^2(\Omega_{\eta_0})}^2  - \dfrac{1}{2}  \int_0^t \int_{\Omega_\eta} |\partial^{\bm{\alpha}}_\bx \rho|^2 \Div\bv \dx\ds +  \int_0^t \int_{\Omega_\eta} \left({\mathlarger{\mathtt{C}_1 }}  + {\mathlarger{\mathtt{C}_2 }}  \right) \partial^{\bm{\alpha}}_\bx \rho \dx\ds.
\end{equation}
However, by the Leibniz rule, 
\[
{\mathlarger{\mathtt{C}_1 }} = - {\mathlarger{ \sum\limits_{\bm{0} < \bm{\beta} \leq \bm{\alpha}} } } \binom{\bm{\alpha}}{\bm{\beta}}  \partial^{\bm{\beta}}_\bx \bv \cdot \nabla  \partial^{\bm{\alpha} - \bm{\beta}}_\bx \rho \quad \text{ and }  \quad  {\mathlarger{\mathtt{C}_2 }} = - {\mathlarger{ \sum\limits_{\bm{0} < \bm{\beta} \leq \bm{\alpha}} } } \binom{\bm{\alpha}}{\bm{\beta}} \left( \partial^{\bm{\beta}}_\bx \Div\bv \right) \partial^{\bm{\alpha} - \bm{\beta}}_\bx \rho .
\]
We now estimate each summand according to the value of $|\bm{\beta}|$.  If $|\bm{\beta}| = 3$, then necessarily  $\bm{\alpha} - \bm{\beta} = \bm{0}$. Thus, 
\begin{equation}\label{eq:betacase1}
\begin{aligned}
\Vert \partial^{\bm{\beta}}_\bx \bv \cdot \nabla \rho \Vert_{L^2(\Omega_\eta)} & \leq \Vert  \bv \Vert_{W^{3,2}(\Omega_\eta)} \Vert  \nabla \rho \Vert_{L^\infty(\Omega_\eta)} ,
\\[0.4em]
\Vert \left( \partial^{\bm{\beta}}_\bx \Div\bv\right)  \rho \Vert_{L^2(\Omega_\eta)} & \leq \Vert \bv \Vert_{W^{4,2}(\Omega_\eta)} \Vert   \rho \Vert_{L^\infty(\Omega_\eta)} .
\end{aligned}
\end{equation}
If $|\bm{\beta}| = 1$, then ${\bm{\alpha} - \bm{\beta}} $ has order at most 2, whence 
\begin{equation*}
\begin{aligned}
\Vert \partial^{\bm{\beta}}_\bx \bv \cdot \nabla \partial^{\bm{\alpha} - \bm{\beta}}_\bx \rho \Vert_{L^2(\Omega_\eta)}  & \leq \Vert \nabla \bv \Vert_{L^\infty(\Omega_\eta)} \Vert  \rho \Vert_{W^{3,2}(\Omega_\eta)} ,
\\[0.4em]
\Vert \left( \partial^{\bm{\beta}}_\bx \Div\bv\right)  \partial^{\bm{\alpha} - \bm{\beta}}_\bx \rho \Vert_{L^2(\Omega_\eta)}  & \leq  \Vert \nabla^2 \bv \Vert_{L^4(\Omega_\eta)} \Vert  \nabla^2\rho \Vert_{L^{4}(\Omega_\eta)} .
\end{aligned}
\end{equation*}
Since 
\begin{equation}\label{eq:W14LW22}
W^{1,4}(\Omega_\eta) = \Big[ L^\infty(\Omega_\eta), W^{2,2}(\Omega_\eta) \Big]_{1/2, 4}, 
\end{equation}
we derive from Young's inequality that 
\begin{equation}\label{eq:betacase2}
\begin{aligned}
\Vert \partial^{\bm{\beta}}_\bx \bv \cdot \nabla \partial^{\bm{\alpha} - \bm{\beta}}_\bx \rho \Vert_{L^2(\Omega_\eta)}  & \leq \Vert \nabla \bv \Vert_{L^\infty(\Omega_\eta)} \Vert  \rho \Vert_{W^{3,2}(\Omega_\eta)} ,
\\[0.4em]
\Vert \left( \partial^{\bm{\beta}}_\bx \Div\bv\right)  \partial^{\bm{\alpha} - \bm{\beta}}_\bx \rho \Vert_{L^2(\Omega_\eta)}  & \leq  \Vert  \bv \Vert_{W^{3,2}(\Omega_\eta)} \Vert  \nabla \rho \Vert_{L^{\infty}(\Omega_\eta)}  +  \Vert  \rho \Vert_{W^{3,2}(\Omega_\eta)} \Vert  \nabla \bv \Vert_{L^{\infty}(\Omega_\eta)}  . 
\end{aligned}
\end{equation}
Finally, if $|\bm{\beta}| = 2$, then $ {\bm{\alpha} - \bm{\beta}}  $ has order at most 1. In this case,  using H\"older's inequality and \eqref{eq:W14LW22}, we obtain  
\begin{equation}\label{eq:betacase3}
\begin{aligned}
\Vert \partial^{\bm{\beta}}_\bx \bv \cdot \nabla \partial^{\bm{\alpha} - \bm{\beta}}_\bx \rho \Vert_{L^2(\Omega_\eta)}  & \leq \Vert  \bv \Vert_{W^{3,2}(\Omega_\eta)} \Vert  \nabla \rho \Vert_{L^{\infty}(\Omega_\eta)}  +  \Vert  \rho \Vert_{W^{3,2}(\Omega_\eta)} \Vert  \nabla \bv \Vert_{L^{\infty}(\Omega_\eta)}  ,
\\[0.4em]
\Vert \left( \partial^{\bm{\beta}}_\bx \Div\bv\right)  \partial^{\bm{\alpha} - \bm{\beta}}_\bx \rho \Vert_{L^2(\Omega_\eta)}  & \leq  \Vert  \bv \Vert_{W^{3,2}(\Omega_\eta)} \Vert  \nabla \rho \Vert_{L^{\infty}(\Omega_\eta)}   . 
\end{aligned}
\end{equation}
Combining \eqref{eq:betacase1}--\eqref{eq:betacase3} and summing over all $\bm{0} < \bm{\beta} \leq \bm{\alpha}$, we conclude that 
\begin{subequations}\label{eq:CommutatorEstim12}
\begin{equation}\label{eq:Commutator1Estim}
\Vert {\mathlarger{\mathtt{C}_1 }}  \Vert_{L^2(\Omega_\eta)} \lesssim \Vert  \bv \Vert_{W^{3,2}(\Omega_\eta)} \Vert  \nabla \rho \Vert_{L^{\infty}(\Omega_\eta)}  +  \Vert  \rho \Vert_{W^{3,2}(\Omega_\eta)} \Vert  \nabla \bv \Vert_{L^{\infty}(\Omega_\eta)} ,
\end{equation}
\\[-0.75cm]
\begin{equation}\label{eq:Commutator2Estim}
\Vert {\mathlarger{\mathtt{C}_2 }}  \Vert_{L^2(\Omega_\eta)} \lesssim \Vert \bv \Vert_{W^{4,2}(\Omega_\eta)} \Vert   \rho \Vert_{L^\infty(\Omega_\eta)}    + \Vert  \bv \Vert_{W^{3,2}(\Omega_\eta)} \Vert  \nabla \rho \Vert_{L^{\infty}(\Omega_\eta)}  +  \Vert  \rho \Vert_{W^{3,2}(\Omega_\eta)} \Vert  \nabla \bv \Vert_{L^{\infty}(\Omega_\eta)}.
\end{equation}
\end{subequations}
Therefore, applying H\"older's inequality to \eqref{eq:ContEqTested} and using  \eqref{eq:CommutatorEstim12}, we deduce that  
\begin{equation}
\begin{aligned}
\sup\limits_{I_*} \Vert \partial^{\bm{\alpha}}_\bx \rho(t) \Vert_{L^2(\Omega_\eta)}^2 & \lesssim  \Vert \partial^{\bm{\alpha}}_\bx \rho_0 \Vert_{L^2(\Omega_{\eta_0})}^2 + \int_{I_*}  \Vert  \bv \Vert_{W^{3,2}(\Omega_\eta)} \Vert  \nabla \rho \Vert_{L^{\infty}(\Omega_\eta)}  \Vert \rho \Vert_{W^{|\bm{\alpha}|, 2}(\Omega_\eta)} \dt 
\\[0.4em]
& \quad +  \int_{I_*}  \Big( \Vert  \rho \Vert_{W^{3,2}(\Omega_\eta)} \Vert  \nabla \bv \Vert_{L^{\infty}(\Omega_\eta)}  + \Vert \bv \Vert_{W^{4, 2}(\Omega_\eta)} \Vert \rho \Vert_{L^{\infty}(\Omega_\eta)} \Big) \Vert \rho \Vert_{W^{|\bm{\alpha}|, 2}(\Omega_\eta)} \dt .
\end{aligned}
\end{equation}
Summing over all $|\bm{\alpha}| \leq 3 $ and using Young's inequality yields the desired result. \\
\end{proof}



\medskip

\noindent {\textbf{Step 2: Conditional higher-order estimates for the displacement} $\eta .$ } \\
Since the viscoelastic shell equation \eqref{eq:ShellEq} contains operators of different differential order in space and in time, a single testing procedure cannot recover all components of the desired regularity asserted in  \cref{theo:Existenceresult}.  However, by \cref{rem:HigherOrderEta}, the  higher-order spatial derivatives of the displacement $\eta$ are already controlled (conditionally),  whereas control of the  corresponding  higher-order time derivatives remains to be established.  \\
For this purpose, we first decompose the time derivative of the fluid--structure coupling term ${\mathlarger{\mathtt{F}}}_\eta $ as
\[
\sum\limits_{\mathsf{k} =1 }^{4} {\mathlarger{\mathtt{F}}}^{(\mathsf{k})}_\eta  := \partial_t {\mathlarger{\mathtt{F}}}_\eta,  
\]
where
\begin{align*}
{\mathlarger{\mathtt{F}}}^{(1)}_\eta
  &= \bn^\intercal \bigl( \partial_t\bm{\tau}\,\bn_\eta \bigr)\circ\bm{\varphi}_\eta
     \,\det(\naby \bm{\varphi}_{\eta}),
&
{\mathlarger{\mathtt{F}}}^{(2)}_\eta
  &= \bn^\intercal \Bigl[ \bigl( (\nabla\bm{\tau})\circ\bm{\varphi}_\eta\,
       \partial_t\bm{\varphi}_\eta \bigr)
       \,\bn_\eta\circ\bm{\varphi}_\eta \Bigr]
     \,\det(\naby \bm{\varphi}_{\eta}),
\\[0.5em]
{\mathlarger{\mathtt{F}}}^{(3)}_\eta
  &= \bn^\intercal \Bigl[ \bm{\tau}\circ\bm{\varphi}_\eta\,
       \partial_t\bigl( \bn_\eta\circ\bm{\varphi}_\eta\bigr) \Bigr]
     \,\det(\naby \bm{\varphi}_{\eta}),
&
{\mathlarger{\mathtt{F}}}^{(4)}_\eta
  &= \bn^\intercal \bigl(\bm{\tau}\,\bn_\eta\bigr)\circ \bm{\varphi}_\eta\,
     \partial_t \det(\naby \bm{\varphi}_{\eta}).
\end{align*}
Differentiating the shell equation \eqref{eq:ShellEq} and the momentum equation $\eqref{eq:ContMomentEq}_2$  with respect to time yields
\begin{subequations}\label{eq:ShellMomEqTime}
\begin{align}
\partial_t^3 \eta
  - \partial_t^2 \Dely \eta
  + \partial_t \Dely^2 \eta
  &= - \partial_t {\mathlarger{\mathtt{F}}}_\eta
  &&\text{in } I \times \omega, \label{eq:ShellEqTime}
\\[0.4em]
(\partial_t \rho) \dot{\bv}
  + \rho\!\left(
      \partial_t^2 \bv
      + \partial_t\bv \cdot \nabla\bv
      + \bv \cdot \partial_t\nabla\bv
    \right)
  &= \partial_t (\Div \bm{\tau})
  &&\text{in } I \times \Omega_\eta .  \label{eq:ContMomentEqTime}
\end{align}
\end{subequations}
In the sequel, we use \eqref{eq:ShellMomEqTime} to obtain energy-type estimates for  higher-order time derivatives of $\eta$ and $\bv$.

\medskip 

\indent {\textbf{Step 2a:}  Testing \eqref{eq:ShellEqTime} and \eqref{eq:ContMomentEqTime} with $\partial_t^3 \eta $ and $\partial_t^2\bv $ respectively, we get on one hand   
\begin{equation}\label{eq:ShellTime1}
\begin{aligned}
& \int_{I_*} \int_{\omega} |\partial_t^3 \eta |^2 \dy\dt   +   \dfrac{1}{2} \int_{I_*} \dfrac{\dd}{\dt} \int_{\omega}  |\partial^2_t\naby\eta|^2 \dy\dt + \int_{I_*} \int_{\omega}  \partial_t \Dely^2\eta \partial^3_t \eta \dy\dt 
\\[0.4em]
& \quad =  - \sum\limits_{\mathsf{k}=1}^{4}  \left( \int_{I_*} \int_{\omega}  {\mathlarger{\mathtt{F}}}^{(\mathsf{k})}_\eta \partial_t^3 \eta \dy\dt \right). 
\end{aligned}
\end{equation}
On the other hand, we have 
\begin{equation} \label{eq:MomTime1}
\begin{aligned}
& \int_{I_*} \int_{\Omega_\eta} \rho |\partial^2_t \bv|^2 \dx\dt -  \int_{I_*} \int_{\Omega_\eta} \Div(\rho\bv) \dot{\bv} \cdot \partial^2_t \bv \dx\dt  
\\[0.4em]
& \quad +  \int_{I_*} \int_{\Omega_\eta} \rho \left( \partial_t\bv\cdot \nabla\bv \right) \cdot \partial^2_t \bv \dx\dt   + \int_{I_*} \int_{\Omega_\eta} \rho \left( \bv\cdot \partial_t\nabla\bv \right) \cdot \partial^2_t \bv \dx\dt  
\\[0.4em]
& \quad =  \int_{I_*} \int_{\Omega_\eta}  \Div(\partial_t \bm{\tau}) \cdot \partial^2_t \bv \dx\dt.   
\end{aligned}
\end{equation} 
Hence, adding \eqref{eq:ShellTime1} and \eqref{eq:MomTime1}, and using the identity  \eqref{eq:identityEta}, we arrive at
 
\begin{align}\label{eq:ShellMomTime0}
& \int_{I_*} \int_{\omega} |\partial_t^3 \eta |^2 \dy\dt   +   \dfrac{1}{2} \int_{I_*} \dfrac{\dd}{\dt} \int_{\omega}  |\partial^2_t\naby\eta|^2 \dy\dt + \int_{I_*} \int_{\Omega_\eta} \rho |\partial^2_t \bv|^2 \dx\dt    \nonumber
\\[0.4em] 
& \quad + \dfrac{\mu}{2} \int_{I_*} \dfrac{\dd}{\dt} \int_{\Omega_\eta} |\partial_t \nabla\bv|^2 \dx\dt  + \dfrac{\lambda + \mu}{2} \int_{I_*} \dfrac{\dd}{\dt} \int_{\Omega_\eta} |\partial_t\Div\bv|^2 \dx\dt   \nonumber
\\[0.4em]
&  = -  \sum\limits_{\mathsf{k}=2}^{4}  \left( \int_{I_*} \int_{\omega}  {\mathlarger{\mathtt{F}}}^{(\mathsf{k})}_\eta \partial_t^3 \eta \dy\dt \right) + \int_{I_*} \int_{\Omega_\eta} \Div(\rho\bv) \dot{\bv} \cdot \partial^2_t \bv \dx\dt    \nonumber
\\[0.4em]
& \quad -  \int_{I_*} \int_{\Omega_\eta} \rho \left( \partial_t\bv\cdot \nabla\bv \right) \cdot \partial^2_t \bv \dx\dt  - \int_{I_*} \int_{\Omega_\eta} \rho \left( \bv\cdot \partial_t\nabla\bv \right) \cdot \partial^2_t \bv \dx\dt     \nonumber
\\[0.4em]
& \quad +  \dfrac{\mu}{2} \int_{I_*}  \int_{\partial\Omega_\eta} (\partial_t\eta \bn) \circ \bm{\varphi}_\eta^{-1} \cdot \bn_\eta  |\partial_t\nabla\bv|^2 \dH\dt  - \mu \int_{I_*}  \int_{\partial\Omega_\eta} \left( \partial^2_t \bv \cdot \partial_t\nabla\bv \right) \cdot \bn_\eta \dH\dt    \nonumber
\\[0.4em]
&\quad +  \mu \int_{I_*}  \int_{\partial\Omega_\eta}  \partial_t(\Div \bv) \partial^2_t\bv \cdot \bn_\eta \dH\dt  +  \dfrac{\lambda + \mu}{2}  \int_{I_*}  \int_{\partial\Omega_\eta} (\partial_t\eta \bn) \circ \bm{\varphi}_\eta^{-1} \cdot \bn_\eta  |\partial_t\Div\bv|^2  \dH\dt  \nonumber
\\[0.4em]
& \quad + \int_{I_*}  \int_{\Omega_\eta}  (\partial_t p) \partial^2_t (\Div\bv) \dx\dt  
\\[0.4em]
& \quad -2\int_{I_*} \int_{\omega} (\partial_t{\bm{\tau}})^\intercal \circ \bm{\varphi}_{\eta} \naby\!\left( \partial_t^2 \eta \bn \right)\left(\naby\bm{\varphi}_{\eta}\right)^{-1} \partial_t\bm{\varphi}_{\eta} \cdot \bn_\eta\circ  \bm{\varphi}_{\eta}  \det(\naby \bm{\varphi}_{\eta}) \dy\dt    \nonumber
\\[0.4em]
& \quad + \int_{I_*} \int_{\omega} (\partial_t{\bm{\tau}})^\intercal \circ \bm{\varphi}_{\eta}  \naby^2(\partial_t\eta \bn) \left(\naby\bm{\varphi}_{\eta}\right)^{-1}\partial_t\bm{\varphi}_{\eta}\left(\naby\bm{\varphi}_{\eta}\right)^{-1}\partial_t\bm{\varphi}_{\eta} \cdot \bn_\eta\circ  \bm{\varphi}_{\eta}  \det(\naby \bm{\varphi}_{\eta}) \dy\dt     \nonumber
\\[0.4em]
&\quad  - \int_{I_*} \int_{\omega} (\partial_t{\bm{\tau}})^\intercal \circ \bm{\varphi}_{\eta}  \naby(\partial_t\eta\bn) \left(\naby\bm{\varphi}_{\eta}\right)^{-1} \partial_t^2 \bm{\varphi}_{\eta} \cdot \bn_\eta\circ  \bm{\varphi}_{\eta}  \det(\naby \bm{\varphi}_{\eta}) \dy\dt   \nonumber
\\[0.4em]
& \quad +  \int_{I_*} \int_{\omega} (\partial_t{\bm{\tau}})^\intercal \circ \bm{\varphi}_{\eta}  \naby(\partial_t\eta\bn) \left(\naby\bm{\varphi}_{\eta}\right)^{-1} \bigg[ 2 \partial_t\naby\bm{\varphi}_{\eta}  \left(\naby\bm{\varphi}_{\eta}\right)^{-1}\partial_t\bm{\varphi}_{\eta}     \nonumber
\\[0.4em]
& \qquad \qquad  \qquad  \qquad  \qquad \; -   \naby^2\bm{\varphi}_\eta   \left(\naby\bm{\varphi}_{\eta}\right)^{-1}\partial_t\bm{\varphi}_{\eta} \left(\naby\bm{\varphi}_{\eta}\right)^{-1}\partial_t\bm{\varphi}_{\eta} \bigg]  \cdot \bn_\eta\circ  \bm{\varphi}_{\eta}  \det(\naby \bm{\varphi}_{\eta}) \dy\dt     \nonumber
\\[0.4em]
& \quad + \int_{I_*} \dfrac{\dd}{\dt}  \int_\omega \big(\partial_t \naby\Dely\eta \cdot \partial^2_{t}\naby\eta \big) \dy\dt  +   \int_{I_*}\int_\omega |\partial^2_{t}\Dely\eta|^2 \dy\dt    \nonumber
\\[0.4em]  
& =: \sum\limits_{\mathsf{k} = 1}^{18}  {\mathlarger{\mathfrak{C}}}_\mathsf{k}.    \nonumber
\end{align}
The structure of the right-hand side is entirely analogous to that encountered in the proof of the acceleration estimate (cf.  \cref{sec:Proof1}).  The only new feature here is the presence of the additional time derivatives. We therefore follow the same strategy and only sketch the essential steps in the estimates of ${\mathlarger{\mathfrak{C}}}_\mathsf{k}$.

We start with the geometric contribution
\[
{\mathlarger{\mathfrak{C}}}_1 =  - \int_{I_*} \int_{\omega}  {\mathlarger{\mathtt{F}}}^{(2)}_\eta \partial_t^3 \eta \dy\dt = - \int_{I_*} \int_{\omega} \bn^\intercal \Bigl[ \bigl( (\nabla\bm{\tau})\circ\bm{\varphi}_\eta\,
       \partial_t\bm{\varphi}_\eta \bigr)
       \,\bn_\eta\circ\bm{\varphi}_\eta \Bigr]
     \,\det(\naby \bm{\varphi}_{\eta})  \partial_t^3 \eta \dy\dt . 
\]
Applying  H\"older and Young's inequalities, we obtain for some arbitrary $\kappa > 0, $
\begin{align*}
{\mathlarger{\mathfrak{C}}}_1 \lesssim \kappa \int_{I_*} \Vert \partial_t^3 \eta \Vert_{L^2(\omega)}^2 \dt + c(\kappa) \int_{I_*} \left(  \Vert \nabla^2\bv \Vert_{L^{8}(\partial\Omega_\eta)}^2  +  \Vert \nabla p \Vert_{L^{8}(\partial\Omega_\eta)}^2  \right) \Vert \partial_t \eta \Vert_{L^{8/3}(\omega)}^2 \dt.
\end{align*}  
However, the embeddings 
\[
W^{3/4, 2}(\partial\Omega_\eta)  \hookrightarrow L^{8}(\partial\Omega_\eta); \qquad W^{1/4, 2}(\omega)  \hookrightarrow L^{8/3}(\omega)
\]
yield 

\begin{align*}
{\mathlarger{\mathfrak{C}}}_1 \lesssim  \kappa \int_{I_*} \Vert \partial_t^3 \eta \Vert_{L^2(\omega)}^2 \dt  + c(\kappa) \int_{I_*} \left(\Vert \nabla^2\bv \Vert_{W^{5/4, 2}(\Omega_\eta)}^2  +  \Vert \nabla p \Vert_{W^{5/4, 2}(\Omega_\eta)}^2  \right) \Vert \partial_t \eta \Vert_{W^{1/4, 2}(\omega)}^2 \dt.
\end{align*}
Using the interpolation identities
\[
W^{5/4,2 }(\Omega_\eta) = \bigl[ W^{1,2}(\Omega_\eta), W^{2,2}(\Omega_\eta)   \bigr]_{1/4}; \qquad W^{1/4,2 }(\Omega_\eta) = \bigl[ L^{2 }(\Omega_\eta), W^{1,2 }(\Omega_\eta)\bigr]_{1/4}
\]
together with Young's inequality, we further obtain

\begin{equation}\label{eq:C1}
\begin{aligned}
{\mathlarger{\mathfrak{C}}}_1 &  \lesssim  \kappa \int_{I_*} \Vert \partial_t^3 \eta \Vert_{L^2(\omega)}^2 \dt  + \kappa \left(\sup\limits_{I_*}\Vert \partial_t \eta \Vert_{W^{1,2}(\omega)}^2 \right) \int_{I_*} \Vert \nabla^2\bv \Vert_{W^{2, 2}(\Omega_\eta)}^2 \dt  
\\[0.4em]
& \quad  + c(\kappa)  \int_{I_*} \Vert \nabla^2\bv \Vert_{W^{1, 2}(\Omega_\eta)}^2 \Vert \partial_t \eta \Vert_{W^{1,2}(\omega)}^2  \dt   + c(\kappa) \int_{I_*} \Vert p \Vert_{W^{3,2}(\Omega_\eta)}^2 \Vert \partial_t \eta \Vert_{W^{1,2}(\omega)}^2  \dt .
\end{aligned}
\end{equation}\\

Before estimating ${\mathlarger{\mathfrak{C}}}_2, $ we first observe that 
\begin{equation}\label{eq:ChainRuleNormal}
\partial_t(\bn_\eta\circ\bm{\varphi}_\eta) = \dfrac{1}{|\partial_1\bm{\varphi}_\eta \times \partial_2 \bm{\varphi}_\eta|} \Big( \mathbb{I}_{3\times3} - \bn_\eta\circ\bm{\varphi}_\eta \otimes  \bn_\eta\circ\bm{\varphi}_\eta \! \Big) \partial_t \left( \partial_1\bm{\varphi}_\eta \times \partial_2 \bm{\varphi}_\eta \right).
\end{equation}
Since the deformation remains non-degenerate, it follows -- using Poincar\'e inequality 
\begin{equation}
\Vert \partial_t(\bn_\eta\circ\bm{\varphi}_\eta) \Vert_{L^4(\omega)} \lesssim \Vert \partial_t \naby \eta \Vert_{L^4(\omega)} .
\end{equation}
Thus, for  
\[
{\mathlarger{\mathfrak{C}}}_2 =  - \int_{I_*} \int_{\omega}  {\mathlarger{\mathtt{F}}}^{(3)}_\eta \partial_t^3 \eta \dy\dt = - \int_{I_*} \int_{\omega} \bn^\intercal \Bigl[ \bm{\tau}\circ\bm{\varphi}_\eta\,
       \partial_t\bigl( \bn_\eta\circ\bm{\varphi}_\eta\bigr) \Bigr]
     \,\det(\naby \bm{\varphi}_{\eta}) \partial_t^3 \eta \dy\dt,
\]
applying H\"older and Young's inequalities together with the embedding 
\[
W^{1/2,2}(\partial\Omega_\eta) \hookrightarrow L^4(\partial\Omega_\eta),
\]
we obtain 
\begin{equation*}
\begin{aligned}
{\mathlarger{\mathfrak{C}}}_2   \lesssim \kappa  \int_{I_*} \Vert \partial_t^3 \eta \Vert_{L^{2}(\omega)} \dt + c(\kappa)   \int_{I_*} \left( \Vert \nabla\bv \Vert_{W^{1,2}(\Omega)}^2 + \Vert p \Vert_{W^{1,2}(\Omega)}^2  \right) \Vert \partial_t \naby\eta \Vert_{W^{1/2,2}(\omega)}^2  \dt .
\end{aligned}
\end{equation*}
Whence 
\begin{equation}\label{eq:C2}
\begin{aligned}
{\mathlarger{\mathfrak{C}}}_2  & \lesssim   \kappa \int_{I_*} \Vert \partial_t^3 \eta \Vert_{L^2(\omega)}^2 \dt  + c(\kappa) \sup\limits_{I_*} \Vert \partial_t \eta \Vert_{W^{3/2,2}(\omega)}^2   \left(  \int_{I_*}\Vert \nabla\bv \Vert_{W^{1,2}(\Omega)}^2  \dt +   \int_{I_*} \Vert p \Vert_{W^{1,2}(\Omega)}^2  \dt  \right).
\end{aligned}
\end{equation}

Recalling the identity 
\begin{equation}\label{eq:DetIdentity}
\partial_t\det(\naby\bm{\varphi}_\eta)
  = \left( \bn_\eta\circ \bm{\varphi}_\eta \right)  \cdot \partial_t \left( \partial_1\bm{\varphi}_\eta \times \partial_2 \bm{\varphi}_\eta \right),
\end{equation}
and noting that 
\[
\Vert \partial_t\det(\naby\bm{\varphi}_\eta) \Vert_{L^{4}(\omega)}  \lesssim \Vert \partial_t \naby\eta \Vert_{L^{4}(\omega)}, 
\]
we obtain for the contribution 
\[
{\mathlarger{\mathfrak{C}}}_3 =  - \int_{I_*} \int_{\omega}  {\mathlarger{\mathtt{F}}}^{(4)}_\eta \partial_t^3 \eta \dy\dt = - \int_{I_*} \int_{\omega} \bn^\intercal \bigl(\bm{\tau}\,\bn_\eta\bigr)\circ \bm{\varphi}_\eta\,
     \partial_t \det(\naby \bm{\varphi}_{\eta}) \partial_t^3 \eta \dy\dt
\]
the following estimate 
\begin{equation}\label{eq:C3}
\begin{aligned}
{\mathlarger{\mathfrak{C}}}_3  & \lesssim   \kappa \int_{I_*} \Vert \partial_t^3 \eta \Vert_{L^2(\omega)}^2 \dt  +  c(\kappa) \sup\limits_{I_*} \Vert \partial_t \eta \Vert_{W^{3/2,2}(\omega)}^2   \left(  \int_{I_*}\Vert \nabla\bv \Vert_{W^{1,2}(\Omega)}^2  \dt +   \int_{I_*} \Vert p \Vert_{W^{1,2}(\Omega)}^2  \dt  \right).
\end{aligned}
\end{equation}

We next consider the convective  term 
\[
{\mathlarger{\mathfrak{C}}}_4 = \int_{I_*} \int_{\Omega_\eta} \Div(\rho\bv) \dot{\bv} \cdot \partial^2_t \bv \dx\dt   =: {\mathlarger{\mathfrak{C}}}_{4, 1} + {\mathlarger{\mathfrak{C}}}_{4, 2}, 
\]
where 
\[
{\mathlarger{\mathfrak{C}}}_{4, 1} = \int_{I_*} \int_{\Omega_\eta} (\nabla\rho\cdot\bv) \dot{\bv} \cdot \partial^2_t \bv \dx\dt, \qquad  {\mathlarger{\mathfrak{C}}}_{4, 2} = \int_{I_*} \int_{\Omega_\eta} (\rho\Div\bv) \dot{\bv} \cdot \partial^2_t \bv \dx\dt.
\]
To estimate  ${\mathlarger{\mathfrak{C}}}_{4, 1}, $ we apply H\"older and Young's inequalities, which yields 
\begin{align*}
{\mathlarger{\mathfrak{C}}}_{4, 1} & \lesssim \kappa \int_{I_*} \Vert \partial_t^2\bv \Vert_{L^2(\Omega_\eta)}^2 \dt   + c(\kappa)\Vert \rho \Vert_{L^\infty\left( I_*; W^{3,2}(\Omega_\eta)  \right)}^2 \int_{I_*} \Vert \bv \Vert_{L^4(\Omega_\eta)}^2 \Vert \partial_t\bv \Vert_{L^4(\Omega_\eta)}^2  \dt
\\[0.4em]
& \quad + c(\kappa)\Vert \rho \Vert_{L^\infty\left( I_*; W^{3,2}(\Omega_\eta)  \right)}^2 \int_{I_*} \Vert \bv \Vert_{L^8(\Omega_\eta)}^2 \Vert \bv \Vert_{L^8(\Omega_\eta)}^2 \Vert \nabla\bv \Vert_{L^4(\Omega_\eta)}^2  \dt . 
\end{align*} 
However, using the embeddings 
\[
W^{1,2}(\Omega_\eta) \hookrightarrow W^{3/4, 2}(\Omega_\eta) \hookrightarrow L^{4}(\Omega_\eta), 
\]
and the interpolation 
\begin{equation}\label{eq:W34L4}
\begin{aligned}
W^{3/4, 2}(\Omega_\eta) & = \Big[ L^{2}(\Omega_\eta), W^{1, 2}(\Omega_\eta)\Big]_{3/4}, 
\\[0.4em]
W^{9/8, 2}(\Omega_\eta) & = \Big[ W^{1, 2}(\Omega_\eta), W^{2, 2}(\Omega_\eta)\Big]_{1/8},
\end{aligned}
\end{equation}
we obtain 

\begin{align*}
{\mathlarger{\mathfrak{C}}}_{4, 1} & \lesssim  c(\kappa)\Vert \rho \Vert_{L^\infty\left( I_*; W^{3,2}(\Omega_\eta)  \right)}^2 \int_{I_*} \left(  \Vert \bv \Vert_{L^2(\Omega_\eta)}^2 \Vert \partial_t\bv \Vert_{L^2(\Omega_\eta)}^2   \right)^{1/4} \left( \Vert \bv \Vert_{W^{1,2}(\Omega_\eta)}^2 \Vert \partial_t\bv \Vert_{W^{1,2}(\Omega_\eta)}^2  \right)^{3/4} \dt
\\[0.4em]
& \quad + \kappa \int_{I_*} \Vert \partial_t^2\bv \Vert_{L^2(\Omega_\eta)}^2 \dt  + c(\kappa) \Vert \rho \Vert_{L^\infty\left( I_*; W^{3,2}(\Omega_\eta)  \right)}^2 \Vert \bv \Vert_{L^\infty\left( I_*; W^{1,2}(\Omega_\eta)  \right)}^4 \int_{I_*} \Vert \bv \Vert_{W^{2,2}(\Omega_\eta)}^2   \dt .
\end{align*}
Young's inequality further implies that 

\begin{equation}\label{eq:C41}
\begin{aligned}
{\mathlarger{\mathfrak{C}}}_{4, 1}  & \lesssim \kappa \int_{I_*} \Vert \partial_t^2\bv \Vert_{L^2(\Omega_\eta)}^2 \dt  + \kappa \Vert \rho \Vert_{L^\infty\left( I_*; W^{3,2}(\Omega_\eta)  \right)}^2 \Vert \bv \Vert_{L^\infty\left( I_*; W^{1,2}(\Omega_\eta)  \right)}^2 \int_{I_*} \Vert \partial_t \bv \Vert_{W^{2,2}(\Omega_\eta)}^2    \dt
\\[0.4em]
& \quad + c(\kappa) \Vert \rho \Vert_{L^\infty\left( I_*; W^{3,2}(\Omega_\eta)  \right)}^2 \Vert \bv \Vert_{L^\infty\left( I_*; L^{2}(\Omega_\eta)  \right)}^2 \int_{I_*} \Vert \partial_t \bv \Vert_{L^{2}(\Omega_\eta)}^2   \dt
\\[0.4em]
& \quad + c(\kappa) \Vert \rho \Vert_{L^\infty\left( I_*; W^{3,2}(\Omega_\eta)  \right)}^2 \Vert \bv \Vert_{L^\infty\left( I_*; W^{1,2}(\Omega_\eta)  \right)}^4 \int_{I_*} \Vert \bv \Vert_{W^{2,2}(\Omega_\eta)}^2   \dt. 
\end{aligned}
\end{equation}\\[-0.8em]

In a similar fashion, we have 
\begin{align*}
{\mathlarger{\mathfrak{C}}}_{4, 2}  & \lesssim \kappa \int_{I_*} \Vert \rho^{1/2}\partial_t^2\bv \Vert_{L^2(\Omega_\eta)}^2 \dt  +  c(\kappa)\Vert \rho \Vert_{L^\infty\left( I_*; L^{\infty}(\Omega_\eta)  \right)} \int_{I_*} \Vert \nabla\bv \Vert_{L^4(\Omega_\eta)}^2 \Vert \partial_t\bv \Vert_{L^4(\Omega_\eta)}^2  \dt
\\[0.4em]
& \quad + c(\kappa) \int_{I_*} \Vert  \rho^{1/2}\bv\cdot \nabla\bv \Vert_{L^2(\Omega_\eta)}^2 \Vert \Div\bv \Vert_{L^\infty(\Omega_\eta)}^2  \dt.
\end{align*}
Hence, using \eqref{eq:W34L4} and Young's inequality we deduce that 

\begin{equation}\label{eq:C42}
\begin{aligned}
{\mathlarger{\mathfrak{C}}}_{4, 2}  & \lesssim \kappa \int_{I_*} \Vert \rho^{1/2}\partial_t^2\bv \Vert_{L^2(\Omega_\eta)}^2 \dt 
+ \kappa \Vert \rho \Vert_{L^\infty\left( I_*; L^{\infty}(\Omega_\eta)  \right)} \left( \sup\limits_{I_*}\Vert \partial_t\bv \Vert_{W^{1,2}(\Omega_\eta)}^2  \right) \int_{I_*} \Vert \bv \Vert_{W^{2,2}(\Omega_\eta)}^2  \dt
\\[0.4em]
& \quad + c(\kappa) \Vert \rho \Vert_{L^\infty\left( I_*; L^{\infty}(\Omega_\eta)  \right)}  \Vert \bv \Vert_{L^\infty\left( I_*; W^{1,2}(\Omega_\eta)  \right)}^2 \int_{I_*} \Vert   \partial_t \bv \Vert_{L^2(\Omega_\eta)}^2 \dt
\\[0.4em]
&\quad + c(\kappa) \int_{I_*} \Vert  \rho^{1/2}\bv\cdot \nabla\bv \Vert_{L^2(\Omega_\eta)}^2 \Vert \bv \Vert_{W^{3,2}(\Omega_\eta)}^2  \dt.
\end{aligned}
\end{equation}
Of note, the last term of \eqref{eq:C42} can be estimated by  nothing else but 
\[
c(\kappa)\left( \sup\limits_{I_*} \Vert \bv \Vert_{W^{3,2}(\Omega_\eta)}^2  \right) {\mathlarger{\mathfrak{R}}}_1,   
\]
where ${\mathlarger{\mathfrak{R}}}_1 $ is from \eqref{eq:AccFirstEstimate}. 

Moving further with the term 
\[
{\mathlarger{\mathfrak{C}}}_5 =   -  \int_{I_*} \int_{\Omega_\eta} \rho \left( \partial_t\bv\cdot \nabla\bv \right) \cdot \partial^2_t \bv \dx\dt,  
\]
we have that 
\begin{align*}
{\mathlarger{\mathfrak{C}}}_5 & \lesssim  \kappa \int_{I_*} \Vert \rho^{1/2}\partial_t^2\bv \Vert_{L^2(\Omega_\eta)}^2 \dt  + c(\kappa)\Vert \rho \Vert_{L^\infty\left( I_*; L^{\infty}(\Omega_\eta)  \right)} \int_{I_*} \Vert \nabla\bv \Vert_{L^4(\Omega_\eta)}^2 \Vert \partial_t\bv \Vert_{L^4(\Omega_\eta)}^2  \dt.
\end{align*}
Using once again interpolation and Young's inequality we deduce that  

\begin{equation}\label{eq:C5}
\begin{aligned}
{\mathlarger{\mathfrak{C}}}_5  & \lesssim \kappa \int_{I_*} \Vert \rho^{1/2}\partial_t^2\bv \Vert_{L^2(\Omega_\eta)}^2 \dt 
+ \kappa \Vert \rho \Vert_{L^\infty\left( I_*; L^{\infty}(\Omega_\eta)  \right)} \left( \sup\limits_{I_*}\Vert \partial_t \nabla\bv \Vert_{L^{2}(\Omega_\eta)}^2  \right) \int_{I_*} \Vert \bv \Vert_{W^{2,2}(\Omega_\eta)}^2  \dt
\\[0.4em]
& \quad + c(\kappa) \Vert \rho \Vert_{L^\infty\left( I_*; L^{\infty}(\Omega_\eta)  \right)}  \Vert \nabla\bv \Vert_{L^\infty\left( I_*; L^{2}(\Omega_\eta)  \right)}^2 \int_{I_*} \Vert   \partial_t \bv \Vert_{L^2(\Omega_\eta)}^2 \dt. 
\end{aligned}
\end{equation} \\

For the term 
\[
{\mathlarger{\mathfrak{C}}}_6 = - \int_{I_*} \int_{\Omega_\eta} \rho \left( \bv\cdot \partial_t\nabla\bv \right) \cdot \partial^2_t \bv \dx\dt  ,  
\]
an application of H\"older and  Young's inequalities yields 
\begin{equation}\label{eq:C5prime}
{\mathlarger{\mathfrak{C}}}_6  \lesssim \kappa \int_{I_*} \Vert \rho^{1/2}\partial_t^2\bv \Vert_{L^2(\Omega_\eta)}^2 \dt  + c(\kappa)\int_{I_*} \Vert  \rho^{1/2}\bv\cdot \partial_t\nabla\bv \Vert_{L^2(\Omega_\eta)}^2  \dt .
\end{equation}
However,  the last term of \eqref{eq:C5prime} is similar to  ${\mathlarger{\mathfrak{R}}}_1$ -- up to omitting the time derivative on $\nabla\bv. $ Therefore, arguing as in the derivation of the estimate for $ {\mathlarger{\mathfrak{R}}}_1 $ (see Section \ref{sec:Proof1}), we arrive at 

\begin{equation}\label{eq:C6}
{\mathlarger{\mathfrak{C}}}_6  \lesssim \kappa  \int_{I_*} \left(  \Vert \rho^{1/2}\partial_t^2\bv \Vert_{L^2(\Omega_\eta)}^2  + \Vert \partial_t \bv \Vert_{W^{2,2}(\Omega_\eta)}^2 \right)\dt   + c(\kappa) \int_{I_*} \Vert  \rho^{1/2}\bv\Vert_{L^\mathtt{r}(\Omega_\eta)}^\mathtt{s} \Vert  \partial_t\bv\Vert_{W^{1,2}(\Omega_\eta)}^2 \dt, 
\end{equation} 
where $\mathtt{s} = \dfrac{2\mathtt{r}}{\mathtt{r} - 3} \in [2, \infty)$. \\

Moreover, the term 
\[
{\mathlarger{\mathfrak{C}}}_7 = \dfrac{\mu}{2} \int_{I_*}  \int_{\partial\Omega_\eta} (\partial_t\eta \bn) \circ \bm{\varphi}_\eta^{-1} \cdot \bn_\eta  |\partial_t\nabla\bv|^2 \dH\dt 
\]
is the analogue of ${\mathlarger{\mathfrak{R}}}_2$. Therefore a similar argument yields
\begin{align}\label{eq:C7}
{\mathlarger{\mathfrak{C}}}_7 \lesssim \kappa \int_{I_*} \Vert \partial_t\bv \Vert_{W^{2, 2}(\Omega_\eta)}^2 \dt + c(\kappa) \int_{I_*} \Vert \partial_t\eta \Vert_{W^{1,2}(\omega)}^2 \Vert \partial_t\nabla\bv \Vert_{L^{ 2}(\Omega_\eta)}^2 \dt.
\end{align} \\[-0.8em]

The boundary term 
\[
{\mathlarger{\mathfrak{C}}}_8 =    - \mu \int_{I_*}  \int_{\partial\Omega_\eta} \left( \partial^2_t \bv \cdot \partial_t\nabla\bv \right) \cdot \bn_\eta \dH\dt ,
\]
is estimated as follows. By  H\"older's inequality, we have 
\begin{equation}\label{eq:C8primary}
\begin{aligned}
{\mathlarger{\mathfrak{C}}}_8 & \lesssim  \int_{I_*} \Vert \partial_t \nabla\bv \Vert_{L^{2}(\partial\Omega_\eta)}  \Vert \left(\partial_t^2 \bv\right)\circ \bm{\varphi}_{\eta} \Vert_{L^2(\omega)} \dt.
\end{aligned}
\end{equation}
However, 
\begin{equation}\label{eq:SecondTimeVelocityEstimateL2}
\begin{aligned}
\Vert \left(\partial_t^2 \bv\right)\circ \bm{\varphi}_{\eta} \Vert_{L^2(\omega)} & \lesssim \Vert \partial_t^3 \eta \Vert_{L^2(\omega)} +  \Vert \partial_t^2  \naby\eta \Vert_{L^2(\omega)} \Vert \partial_t  \eta \Vert_{L^\infty(\omega)}  +  \Vert \partial_t \Dely  \eta \Vert_{L^6(\omega)} \Vert \partial_t  \eta \Vert_{L^6(\omega)}^2 
\\[0.4em]
& \quad + \Vert \partial_t \naby \eta \Vert_{L^8(\omega)} \Vert \Dely  \eta \Vert_{L^8(\omega)} \Vert \partial_t  \eta \Vert_{L^8(\omega)}^2    + \Vert \partial_t \naby \eta \Vert_{L^6(\omega)}^2 \Vert \partial_t  \eta \Vert_{L^6(\omega)}
\\[0.4em]
& \quad + \Vert \partial_t \naby \eta \Vert_{L^4(\omega)}\Vert \partial_t^2  \eta \Vert_{L^4(\omega)}
\\[0.4em]
& =: \sum\limits_{\mathsf{i} = 1}^{6} {\mathlarger{\mathfrak{Z}}}_{\mathsf{i}}  .
\end{aligned}
\end{equation}
That is, 
\[
 {\mathlarger{\mathfrak{C}}}_8 \lesssim \sum\limits_{\mathsf{i} = 1}^{6} {\mathlarger{\mathfrak{C}}}_8^{\mathsf{i}}, \quad \text{with} \quad   {\mathlarger{\mathfrak{C}}}_8^{\mathsf{i}} :=   \int_{I_*} \Vert \partial_t \nabla\bv \Vert_{L^{2}(\partial\Omega_\eta)}{\mathlarger{\mathfrak{Z}}}_{\mathsf{i}}  \dt.
\]
Of particular interest is the term 
\[
 {\mathlarger{\mathfrak{C}}}_8^{1} = \int_{I_*} \Vert \partial_t \nabla\bv \Vert_{L^{2}(\partial\Omega_\eta)}{\mathlarger{\mathfrak{Z}}}_{1} \dt =  \int_{I_*} \Vert \partial_t \nabla\bv \Vert_{L^{2}(\partial\Omega_\eta)} \Vert \partial_t^3 \eta \Vert_{L^2(\omega)} \dt. 
\]
Using interpolation, we have 
\[
 {\mathlarger{\mathfrak{C}}}_8^{1} \lesssim \int_{I_*}  \Vert \partial_t \nabla\bv \Vert_{W^{1/2,2}(\partial\Omega_\eta)}^{1/2}\Vert (\partial_t \bv)\circ \bm{\varphi}_\eta \Vert_{W^{1/2,2}(\omega)}^{1/2}  \Vert \partial_t^3 \eta \Vert_{L^2(\omega)} \dt. 
\]
An application of Young's inequality yields 
\begin{equation}\label{eq:C8FirstTerm}
\begin{aligned}
 {\mathlarger{\mathfrak{C}}}_8^{1} & \lesssim  \kappa \int_{I_*} \Vert \partial_t^3 \eta \Vert_{L^2(\omega)}^2 \dt + \kappa \int_{I_*} \Vert \partial_t \nabla^2 \bv \Vert_{L^{2}(\Omega_\eta)}^2 \dt   + c(\kappa) \int_{I_*} \Vert (\partial_t \bv)\circ \bm{\varphi}_\eta \Vert_{W^{1/2,2}(\omega)}^2 \dt.
 \end{aligned}
\end{equation}
By the interpolation 
\begin{equation}\label{eq:InterpoW12L4}
W^{1/2, 2}(\omega) = \Big[ L^2(\omega), W^{1,2}(\omega)\Big]_{1/2} 
\end{equation}
and the embedding 
\begin{equation}\label{eq:EmbedW12Lp}
W^{1, 2}(\omega) \hookrightarrow  L^{\mathtt{p}}(\omega) \quad \text{for all }\; \mathtt{p} \in [1, \infty),  
\end{equation}
it follows from \eqref{eq:C8FirstTerm}  that
\begin{equation}\label{eq:C8FirstTermEstimate}
\begin{aligned}
 {\mathlarger{\mathfrak{C}}}_8^{1} & \lesssim  \kappa \int_{I_*} \Vert \partial_t^3 \eta \Vert_{L^2(\omega)}^2 \dt + \kappa \int_{I_*} \Vert \partial_t \nabla^2 \bv \Vert_{L^{2}(\Omega_\eta)}^2 \dt   + c(\kappa) \int_{I_*} \Vert \partial_t^2 \eta \Vert_{L^{2}(\omega)}^2 \dt
 \\[0.4em]
 & \quad +  c(\kappa) \int_{I_*} \Vert \partial_t \Dely \eta \Vert_{L^{2}(\omega)}^2 \Vert \partial_t \naby \eta \Vert_{L^{2}(\omega)}^2  \dt + \kappa   \int_{I_*} \Vert \partial_t^2 \Dely \eta \Vert_{L^2(\omega)}^2 \dt  
 \\[0.4em]
 & \quad + \kappa \int_{I_*} \Big( \Vert \partial_t \naby\Dely \eta \Vert_{L^{2}(\omega)}^2 \Vert \partial_t \naby \eta \Vert_{L^{2}(\omega)}^2 +  \Vert \partial_t \Dely \eta \Vert_{L^{2}(\omega)}^2 \Vert \naby\Dely \eta \Vert_{L^{2}(\omega)}^2\Vert \partial_t \naby \eta \Vert_{L^{2}(\omega)}^2 \Big) \dt 
 \\[0.4em]
 &\quad + \kappa \int_{I_*} \Vert \partial_t \naby\Dely \eta \Vert_{L^{2}(\omega)}^2 \Vert \partial_t \Dely \eta \Vert_{L^{2}(\omega)}^2   \dt .
 \end{aligned}
\end{equation}
Estimates for the remaining terms $ {\mathlarger{\mathfrak{C}}}_8^{\mathsf{i}}, \;\mathsf{i} \in \{2, \ldots, 6\} , $ follow by straightforward applications of  H\"older's and Young's inequalities, together with the embeddings \eqref{eq:EmbedW12Lp} and 
\begin{equation}\label{eq:EmbedW22Linfty}
W^{2, 2}(\omega) \hookrightarrow  L^{\infty}(\omega) .  
\end{equation}
Hence,
\begin{equation}\label{eq:C8}
\begin{aligned}
{\mathlarger{\mathfrak{C}}}_8 & \lesssim \kappa \int_{I_*} \Vert \partial^3_t\eta \Vert_{L^{ 2}(\omega)}^2 \dt +\kappa   \int_{I_*} \Vert \partial_t\bv \Vert_{W^{2,2}(\Omega_\eta)}^2 \dt + c(\kappa)  \int_{I_*} \Vert \partial^2_t\eta \Vert_{L^{ 2}(\omega)}^2 \dt
\\[0.4em]
&\quad + c(\kappa) \int_{I_*} \Vert \partial_t \naby\eta \Vert_{L^2(\omega)}^2  \Vert \partial_t\Dely \eta \Vert_{L^2(\omega)}^2 \dt + \kappa\int_{I_*} \Vert \partial^2_t \Dely\eta \Vert_{L^{ 2}(\omega)}^2 \dt 
\\[0.4em]
& \quad + c(\kappa)  \int_{I_*} \Vert \partial_t \naby\Dely\eta \Vert_{L^2(\omega)}^2 \Big( \Vert \partial_t \naby\eta \Vert_{L^2(\omega)}^2 +  \Vert \partial_t \Dely\eta \Vert_{L^2(\omega)}^2 \Big) \dt
\\[0.4em]
&\quad + c(\kappa)  \int_{I_*} \Vert \partial_t \Dely\eta \Vert_{L^2(\omega)}^2  \Vert \partial_t \naby\eta \Vert_{L^2(\omega)}^2  \Vert \naby \Dely\eta \Vert_{L^2(\omega)}^2  \dt
\\[0.4em]
&\quad +  c(\kappa)  \int_{I_*} \Vert \partial_t^2 \naby\eta \Vert_{L^2(\omega)}^2  \Vert \partial_t\Dely \eta \Vert_{L^2(\omega)}^2 \dt + \kappa  \int_{I_*} \Vert \partial_t \nabla^2\bv \Vert_{L^2(\Omega_\eta)}^2  \Vert \partial_t \naby\eta \Vert_{L^2(\omega)}^2 \dt.
\end{aligned}
\end{equation} \\[-0.8em]

Since the term 
\[
{\mathlarger{\mathfrak{C}}}_9 = \mu \int_{I_*}  \int_{\partial\Omega_\eta}  \partial_t(\Div \bv) \partial^2_t\bv \cdot \bn_\eta \dH\dt  
\]
is similar to ${\mathlarger{\mathfrak{C}}}_8$, a similar argument yields
\begin{equation}\label{eq:C9}
\begin{aligned}
{\mathlarger{\mathfrak{C}}}_9 & \lesssim \kappa \int_{I_*} \Vert \partial^3_t\eta \Vert_{L^{ 2}(\omega)}^2 \dt +\kappa   \int_{I_*} \Vert \partial_t\bv \Vert_{W^{2,2}(\Omega_\eta)}^2 \dt + c(\kappa)  \int_{I_*} \Vert \partial^2_t\eta \Vert_{L^{ 2}(\omega)}^2 \dt
\\[0.4em]
&\quad + c(\kappa) \int_{I_*} \Vert \partial_t \naby\eta \Vert_{L^2(\omega)}^2  \Vert \partial_t\Dely \eta \Vert_{L^2(\omega)}^2 \dt + \kappa\int_{I_*} \Vert \partial^2_t \Dely\eta \Vert_{L^{ 2}(\omega)}^2 \dt 
\\[0.4em]
& \quad + c(\kappa)  \int_{I_*} \Vert \partial_t \naby\Dely\eta \Vert_{L^2(\omega)}^2 \Big( \Vert \partial_t \naby\eta \Vert_{L^2(\omega)}^2 +  \Vert \partial_t \Dely\eta \Vert_{L^2(\omega)}^2 \Big) \dt
\\[0.4em]
&\quad + c(\kappa)  \int_{I_*} \Vert \partial_t \Dely\eta \Vert_{L^2(\omega)}^2  \Vert \partial_t \naby\eta \Vert_{L^2(\omega)}^2  \Vert \naby \Dely\eta \Vert_{L^2(\omega)}^2  \dt
\\[0.4em]
&\quad +  c(\kappa)  \int_{I_*} \Vert \partial_t^2 \naby\eta \Vert_{L^2(\omega)}^2  \Vert \partial_t\Dely \eta \Vert_{L^2(\omega)}^2 \dt + \kappa  \int_{I_*} \Vert \partial_t \nabla^2\bv \Vert_{L^2(\Omega_\eta)}^2  \Vert \partial_t \naby\eta \Vert_{L^2(\omega)}^2 \dt.
\end{aligned}
\end{equation} \\[-0.8em]

Furthermore, we get for 
\[
{\mathlarger{\mathfrak{C}}}_{10} =  \dfrac{\lambda + \mu}{2}  \int_{I_*}  \int_{\partial\Omega_\eta} (\partial_t\eta \bn) \circ \bm{\varphi}_\eta^{-1} \cdot \bn_\eta  |\partial_t\Div\bv|^2  \dH\dt , 
\]
which is the divergence analogue of ${\mathlarger{\mathfrak{C}}}_7 $, the following estimate 

\begin{align}\label{eq:C10}
{\mathlarger{\mathfrak{C}}}_{10} \lesssim \kappa \int_{I_*} \Vert \partial_t\bv \Vert_{W^{2, 2}(\Omega_\eta)}^2 \dt + c(\kappa) \int_{I_*} \Vert \partial_t\eta \Vert_{W^{1,2}(\omega)}^2 \Vert \partial_t\nabla\bv \Vert_{L^{ 2}(\Omega_\eta)}^2 \dt.
\end{align} \\[-0.8em]

Consider now the pressure-driven contribution 
\[
{\mathlarger{\mathfrak{C}}}_{11} = \int_{I_*}  \int_{\Omega_\eta}  (\partial_t p) \partial^2_t (\Div\bv) \dx\dt.  
\] 
Using a standard integration-by-parts argument, we obtain 
\begin{align}\label{eq:C11primary}
{\mathlarger{\mathfrak{C}}}_{11} \lesssim  \int_{I_*} \Vert \partial_t^2 \bv \Vert_{L^2(\Omega_\eta)} \Vert \partial_t p  \Vert_{W^{1,2}(\Omega_\eta)}  \dt + \int_{I_*} \Vert \left(\partial_t^2 \bv\right)\circ \bm{\varphi}_\eta \Vert_{L^2(\omega)}  \Vert \partial_t p  \Vert_{W^{1,2}(\Omega_\eta)}  \dt. 
\end{align}
By  the renormalised continuity equation \eqref{eq:RenormContEq}, Young's inequality and  \eqref{eq:SecondTimeVelocityEstimateL2}, we further deduce that 
\begin{equation}\label{eq:C11}
\begin{aligned}
{\mathlarger{\mathfrak{C}}}_{11} & \lesssim \kappa \int_{I_*} \Vert \partial_t^2\bv \Vert_{L^{ 2}(\Omega_\eta)}^2 \dt + c(\kappa) \Vert \rho \Vert_{L^\infty\left(I_*; W^{3,2}(\Omega_\eta)  \right)}^2 \int_{I_*}  \Vert \bv \Vert_{W^{2, 2}(\Omega_\eta)}^2 \dt
\\[0.4em]
&\quad + \kappa \int_{I_*} \Vert \partial^3_t \eta \Vert_{L^2(\omega)}^2 \dt + \kappa \int_{I_*} \Vert \partial^2_t \naby\eta \Vert_{L^2(\omega)}^2 \Big( \Vert \partial_t \naby\eta \Vert_{L^2(\omega)}^2   + \Vert \partial_t \Dely\eta \Vert_{L^2(\omega)}^2  \Big) \dt  
\\[0.4em]
& \quad + \kappa \int_{I_*}  \Vert \partial_t \naby\Dely\eta \Vert_{L^2(\omega)}^2  \Vert \partial_t \naby\eta \Vert_{L^2(\omega)}^2  \Big( 1  +  \Vert \partial_t \Dely\eta \Vert_{L^2(\omega)}^2  \Big) \dt
\\[0.4em]
& \quad + \kappa \int_{I_*}  \Vert \partial_t \Dely\eta \Vert_{L^2(\omega)}^2 \Vert \naby \Dely\eta \Vert_{L^2(\omega)}^2  \Vert \partial_t \naby\eta \Vert_{L^2(\omega)}^4 \dt  .
\end{aligned}
\end{equation} \\[-0.8em]

We next estimate the boundary-geometry term
\begin{align*}
{\mathlarger{\mathfrak{C}}}_{12} & = -2\int_{I_*} \int_{\omega} (\partial_t{\bm{\tau}})^\intercal \circ \bm{\varphi}_{\eta} \naby\!\left( \partial_t^2 \eta \bn \right)\left(\naby\bm{\varphi}_{\eta}\right)^{-1} \partial_t\bm{\varphi}_{\eta} \cdot \bn_\eta\circ  \bm{\varphi}_{\eta}  \det(\naby \bm{\varphi}_{\eta}) \dy\dt,
\end{align*} 
which we split according to the product structure of $\naby\left( \partial^2_t \eta \bn \right)$. That is, 
\[
{\mathlarger{\mathfrak{C}}}_{12}  := {\mathlarger{\mathfrak{C}}}_{12}^{\mathtt{a}} + {\mathlarger{\mathfrak{C}}}_{12}^{\mathtt{b}},
\]
where 
\begin{align*}
{\mathlarger{\mathfrak{C}}}_{12}^{\mathtt{a}} & = -2\int_{I_*} \int_{\omega} (\partial_t{\bm{\tau}})^\intercal \circ \bm{\varphi}_{\eta}  \left( \bn \otimes \partial^2_t  \naby\eta  \right)\left(\naby\bm{\varphi}_{\eta}\right)^{-1} \partial_t\bm{\varphi}_{\eta}  \cdot \bn_\eta\circ  \bm{\varphi}_{\eta}  \det(\naby \bm{\varphi}_{\eta}) \dy\dt,
\\[0.4em]
{\mathlarger{\mathfrak{C}}}_{12}^{\mathtt{b}} & = -2\int_{I_*} \int_{\omega} (\partial_t{\bm{\tau}})^\intercal \circ \bm{\varphi}_{\eta}  \left( \partial^2_t  \eta \naby\bn \right)\left(\naby\bm{\varphi}_{\eta}\right)^{-1} \partial_t\bm{\varphi}_{\eta}  \cdot \bn_\eta\circ  \bm{\varphi}_{\eta}  \det(\naby \bm{\varphi}_{\eta}) \dy\dt.
\end{align*}
Applying  H\"older's inequality and using uniform bounds on the geometric quantities  associated with the deformation map $\bm{\varphi}_\eta$, we obtain
\begin{align*}
{\mathlarger{\mathfrak{C}}}_{12}^{\mathtt{a}} & \lesssim \int_{I_*} \Big(\Vert \partial_t \nabla\bv \Vert_{L^2(\partial\Omega_\eta)} + \Vert \partial_t p\Vert_{L^2(\partial\Omega_\eta)}  \Big)\Vert \partial^2_t \naby \eta \Vert_{L^2(\omega)} \Vert \partial_t  \bm{\varphi}_{\eta} \Vert_{L^\infty(\omega)}  \dt.
\end{align*}
Using Young's inequality together with the renormalised continuity equation \eqref{eq:RenormContEq}, we derive that 
\begin{equation}\label{eq:C12a}
\begin{aligned}
{\mathlarger{\mathfrak{C}}}_{12}^{\mathtt{a}}  & \lesssim \kappa \int_{I_*} \Vert \partial_t \nabla^2\bv \Vert_{L^2(\Omega_\eta)}^2 \dt + c(\kappa)\int_{I_*} \Vert \partial_t^2 \naby\eta \Vert_{L^2(\omega)}^2 \Vert \partial_t \Dely\eta \Vert_{L^2(\omega)}^2 \dt 
\\[0.4em]
& \quad + \kappa \sup\limits_{I_*}\Vert \rho \Vert_{W^{3,2}(\Omega_\eta)}^2 \int_{I_*} \Vert  \nabla^2\bv \Vert_{L^2(\Omega_\eta)}^2 \dt. 
\end{aligned}
\end{equation}
Treating the term ${\mathlarger{\mathfrak{C}}}_{12}^{\mathtt{b}}$ analogously leads to 
\begin{equation}\label{eq:C12b}
\begin{aligned}
{\mathlarger{\mathfrak{C}}}_{12}^{\mathtt{b}}  & \lesssim \kappa \int_{I_*} \Vert \partial_t \nabla^2\bv \Vert_{L^2(\Omega_\eta)}^2 \dt + c(\kappa)\int_{I_*} \Vert \partial_t^2 \naby\eta \Vert_{L^2(\omega)}^2 \Vert \partial_t \Dely\eta \Vert_{L^2(\omega)}^2 \dt 
\\[0.4em]
& \quad + \kappa \sup\limits_{I_*}\Vert \rho \Vert_{W^{3,2}(\Omega_\eta)}^2 \int_{I_*} \Vert  \nabla^2\bv \Vert_{L^2(\Omega_\eta)}^2 \dt. 
\end{aligned}
\end{equation}\\[-0.8em]

Proceeding analogously, we decompose 
\[
{\mathlarger{\mathfrak{C}}}_{13} = \int_{I_*} \int_{\omega} (\partial_t{\bm{\tau}})^\intercal \circ \bm{\varphi}_{\eta}  \naby^2(\partial_t\eta \bn) \left(\naby\bm{\varphi}_{\eta}\right)^{-1}\partial_t\bm{\varphi}_{\eta}\left(\naby\bm{\varphi}_{\eta}\right)^{-1}\partial_t\bm{\varphi}_{\eta} \cdot \bn_\eta\circ  \bm{\varphi}_{\eta}  \det(\naby \bm{\varphi}_{\eta}) \dy\dt
\] 
according to the chain rule. More precisely, 
\[
\naby^2 \left(\partial_t \eta \bn \right) = \bn\otimes \partial_t\naby^2\eta +  \partial_t\naby\eta \otimes \naby\bn +  \partial_t\eta \naby^2\bn +   \naby\bn\otimes \partial_t\naby\eta, 
\]
we define ${\mathlarger{\mathfrak{C}}}_{13}^{\mathtt{a}}, {\mathlarger{\mathfrak{C}}}_{13}^{\mathtt{b}}, {\mathlarger{\mathfrak{C}}}_{13}^{\mathtt{c}} $ and ${\mathlarger{\mathfrak{C}}}_{13}^{\mathtt{d}}$ accordingly. We have for the first contribution 
\begin{align*}
{\mathlarger{\mathfrak{C}}}_{13}^{\mathtt{a}} & \lesssim \int_{I_*} \Big(\Vert \partial_t \nabla\bv \Vert_{L^2(\partial\Omega_\eta)} + \Vert \partial_t p\Vert_{L^2(\partial\Omega_\eta)}  \Big)\Vert \partial_t \naby^2 \eta \Vert_{L^6(\omega)} \Vert \partial_t  \eta \Vert_{L^6(\omega)}^2  \dt.
\end{align*}
Using the renormalised continuity equation \eqref{eq:RenormContEq}, Young's inequality combined with the embedding \eqref{eq:EmbedW12Lp} we obtain 
\begin{equation}\label{eq:C13a}
\begin{aligned}
{\mathlarger{\mathfrak{C}}}_{13}^{\mathtt{a}}  & \lesssim \kappa \int_{I_*} \Vert \partial_t \nabla^2\bv \Vert_{L^2(\Omega_\eta)}^2 \Vert \partial_t \naby\eta \Vert_{L^2(\omega)}^2  \dt + c(\kappa) \int_{I_*} \Vert \partial_t \naby\Dely\eta \Vert_{L^2(\omega)}^2 \Vert \partial_t \naby\eta \Vert_{L^2(\omega)}^2  \dt 
\\[0.4em]
& \quad + \kappa \sup\limits_{I_*}\Vert \rho \Vert_{W^{3,2}(\Omega_\eta)}^2 \int_{I_*} \Vert  \nabla^2\bv \Vert_{L^2(\Omega_\eta)}^2 \Vert \partial_t \naby\eta \Vert_{L^2(\omega)}^2 \dt. 
\end{aligned}
\end{equation}
Similarly, we get  for $\mathfrak{l} \in \left\{\mathtt{b}, \mathtt{c},  \mathtt{d} \right\} $, 
\begin{equation}\label{eq:C13l}
\begin{aligned}
{\mathlarger{\mathfrak{C}}}_{13}^{\mathfrak{l}}  & \lesssim \kappa \int_{I_*} \Vert \partial_t \nabla^2\bv \Vert_{L^2(\Omega_\eta)}^2 \Vert \partial_t \naby\eta \Vert_{L^2(\omega)}^2  \dt + c(\kappa) \int_{I_*} \Vert \partial_t \naby\Dely\eta \Vert_{L^2(\omega)}^2 \Vert \partial_t \naby\eta \Vert_{L^2(\omega)}^2  \dt 
\\[0.4em]
& \quad + \kappa \sup\limits_{I_*}\Vert \rho \Vert_{W^{3,2}(\Omega_\eta)}^2 \int_{I_*} \Vert  \nabla^2\bv \Vert_{L^2(\Omega_\eta)}^2 \Vert \partial_t \naby\eta \Vert_{L^2(\omega)}^2 \dt.  
\end{aligned}
\end{equation}\\

Considering the term 
\[
{\mathlarger{\mathfrak{C}}}_{14} = - \int_{I_*} \int_{\omega} (\partial_t{\bm{\tau}})^\intercal \circ \bm{\varphi}_{\eta}  \left(\bn\otimes\partial_t \naby\eta + \partial_t\eta \naby\bn\right) \left(\naby\bm{\varphi}_{\eta}\right)^{-1}\partial^2_t  \bm{\varphi}_{\eta} \cdot \bn_\eta\circ  \bm{\varphi}_{\eta}  \det(\naby \bm{\varphi}_{\eta}) \dy\dt, 
\]
H\"older's inequality yields
\begin{equation*}
\begin{aligned}
{\mathlarger{\mathfrak{C}}}_{14} &\lesssim  \int_{I_*} \Big(\Vert \partial_t \nabla\bv \Vert_{L^2(\partial\Omega_\eta)} + \Vert \partial_t p\Vert_{L^2(\partial\Omega_\eta)}  \Big)\Vert \partial_t \naby \eta \Vert_{L^\infty(\omega)} \Vert \partial_t^2 \eta \Vert_{L^2(\omega)}  \dt.
\end{aligned}
\end{equation*}
Using successively \eqref{eq:EmbedW22Linfty}, the renormalised continuity equation \eqref{eq:RenormContEq} and Young's inequality, we obtain 
\begin{equation}\label{eq:C14}
\begin{aligned}
{\mathlarger{\mathfrak{C}}}_{14} & \lesssim \kappa \int_{I_*} \Vert \partial_t \nabla^2\bv \Vert_{L^2(\Omega_\eta)}^2 \dt + c(\kappa) \int_{I_*}  \Vert \partial_t \naby \Dely\eta \Vert_{L^2(\omega)}^2 \Vert \partial_t^2 \eta \Vert_{L^2(\omega)}^2  \dt  
\\[0.4em]
& \quad +  \kappa \sup\limits_{I_*}\Vert \rho \Vert_{W^{3,2}(\Omega_\eta)}^2 \int_{I_*} \Vert  \nabla^2\bv \Vert_{L^2(\Omega_\eta)}^2 \dt  . 
\end{aligned}
\end{equation}\\[-0.8em]

Considering the contribution
\begin{align*}
{\mathlarger{\mathfrak{C}}}_{15} &  =  2 \int_{I_*} \int_{\omega} (\partial_t{\bm{\tau}})^\intercal \circ \bm{\varphi}_{\eta}  \naby(\partial_t\eta\bn) \left(\naby\bm{\varphi}_{\eta}\right)^{-1} \partial_t\naby\bm{\varphi}_{\eta}  \left(\naby\bm{\varphi}_{\eta}\right)^{-1}\partial_t\bm{\varphi}_{\eta}    \cdot \bn_\eta\circ  \bm{\varphi}_{\eta}  \det(\naby \bm{\varphi}_{\eta}) \dy\dt,
\end{align*}
H\"older's inequality yields
\begin{equation*}
\begin{aligned}
{\mathlarger{\mathfrak{C}}}_{15} & \lesssim \int_{I_*} \Big(\Vert \partial_t \nabla\bv \Vert_{L^{2}(\partial\Omega_\eta)} + \Vert \partial_t p\Vert_{L^{2}(\partial\Omega_\eta)}  \Big)  \Vert \naby(\partial_t\eta\bn) \Vert_{L^6(\omega)}^2   \Vert \partial_t \eta \Vert_{L^6(\omega)} \dt.
\end{aligned}
\end{equation*}
Relying once again on the renormalised continuity equation \eqref{eq:RenormContEq} and Young's inequality, we derive that 
\begin{equation}\label{eq:C15}
\begin{aligned}
{\mathlarger{\mathfrak{C}}}_{15} & \lesssim \kappa \int_{I_*} \Vert \partial_t \nabla^2\bv \Vert_{L^2(\Omega_\eta)}^2\Vert \partial_t \naby\eta \Vert_{L^2(\omega)}^2 \dt + c(\kappa)  \int_{I_*}  \Vert \partial_t \naby\Dely \eta \Vert_{L^2(\omega)}^2 \Vert \partial_t \Dely\eta \Vert_{L^2(\omega)}^2  \dt   
\\[0.4em]
& \quad +  \kappa \sup\limits_{I_*}\Vert \rho \Vert_{W^{3,2}(\Omega_\eta)}^2 \int_{I_*} \Vert  \nabla^2\bv \Vert_{L^2(\Omega_\eta)}^2 \Vert \partial_t \naby\eta \Vert_{L^2(\omega)}^2 \dt. 
\end{aligned}
\end{equation}

To estimate the last boundary-geometry term 
\begin{align*}
{\mathlarger{\mathfrak{C}}}_{16} &  =  - \int_{I_*} \int_{\omega} (\partial_t{\bm{\tau}})^\intercal \circ \bm{\varphi}_{\eta} \naby(\partial_t\eta\bn) \left(\naby\bm{\varphi}_{\eta}\right)^{-1}   \naby^2\bm{\varphi}_\eta   \left(\naby\bm{\varphi}_{\eta}\right)^{-1}\partial_t\bm{\varphi}_{\eta} \left(\naby\bm{\varphi}_{\eta}\right)^{-1}\partial_t\bm{\varphi}_{\eta} 
\\[0.4em]
& \qquad \qquad  \qquad \qquad \qquad \qquad\qquad \qquad\qquad \qquad \qquad \qquad\qquad \quad   \cdot \bn_\eta\circ  \bm{\varphi}_{\eta}  \det(\naby \bm{\varphi}_{\eta}) \dy\dt,
\end{align*}
we first use H\"older's inequality. This yields
\begin{equation*}
\begin{aligned}
{\mathlarger{\mathfrak{C}}}_{16} &   \lesssim \int_{I_*} \Big(\Vert \partial_t \nabla\bv \Vert_{L^{2}(\partial\Omega_\eta)} + \Vert \partial_t p\Vert_{L^{2}(\partial\Omega_\eta)}  \Big) \Vert \partial_t\naby\eta \Vert_{L^{4}(\omega)} \Vert \Dely\eta \Vert_{L^{6}(\omega)} \Vert \partial_t \eta \Vert_{L^{6}(\omega)}^2    \dt. 
\end{aligned}
\end{equation*} 
Using \eqref{eq:RenormContEq} and Young's inequality, we deduce that 
\begin{equation}\label{eq:C16}
\begin{aligned}
{\mathlarger{\mathfrak{C}}}_{16} & \lesssim \kappa \int_{I_*} \Vert \partial_t \nabla^2\bv \Vert_{L^2(\Omega_\eta)}^2 \Vert \partial_t \naby\eta \Vert_{L^2(\omega)}^2 \dt +  \kappa \sup\limits_{I_*}\Vert \rho \Vert_{W^{3,2}(\Omega_\eta)}^2 \int_{I_*} \Vert  \nabla^2\bv \Vert_{L^2(\Omega_\eta)}^2 \Vert \partial_t \naby\eta \Vert_{L^2(\omega)}^2  \dt 
\\[0.4em]
& \quad + c(\kappa)  \int_{I_*}  \Vert \partial_t \Dely\eta \Vert_{L^2(\omega)}^2 \Vert \partial_t \naby\eta \Vert_{L^2(\omega)}^2 \Vert  \naby\Dely\eta \Vert_{L^2(\omega)}^2  \dt   . 
\end{aligned}
\end{equation}

Regarding  the term 
\[
{\mathlarger{\mathfrak{C}}}_{17} = \int_{I_*} \dfrac{\dd}{\dt}  \int_\omega \big(\partial_t \naby\Dely\eta \cdot \partial^2_{t}\naby\eta \big) \dy\dt,  
\]
it is straightforward that 
\begin{align}\label{eq:C17}
{\mathlarger{\mathfrak{C}}}_{17} \lesssim \kappa \sup\limits_{I_*}\Vert \partial_t^2 \naby\eta \Vert_{L^2(\omega)}^2 + c(\kappa) \sup\limits_{I_*}\Vert \partial_t \naby\Dely\eta \Vert_{L^2(\omega)}^2.
\end{align} 
To gain control of  $\partial_t^2\Dely\eta$ and $\partial_t\naby\Dely\eta$, we now choose a higher-order test function for the shell equation.\\


\medskip 

\indent \textbf{Step 2b:}  We now test  \eqref{eq:ShellEqTime}  with $-\partial_t^2 \Dely\eta$. This  yields 

\begin{equation}\label{eq:EstimStep2b}
\begin{aligned}
&\dfrac{1}{2} \int_{I_*} \dfrac{\dd}{\dt} \int_{\omega} |\partial^2_t\naby\eta|^2 \dy\dt + \int_{I_*} \int_{\omega}  |\partial_t^2\Dely\eta|^2 \dx\dt + \dfrac{1}{2} \int_{I_*} \dfrac{\dd}{\dt} \int_{\omega} |\partial_t\naby\Dely\eta|^2 \dy\dt 
\\[0.4em]
& =   \sum\limits_{k =1}^{4}  \int_{I_*} \int_{\omega} {\mathlarger{\mathtt{F}}}^{(k)}_\eta \partial^2_t \Dely\eta \dy\dt
\\[0.4em]
& =: \sum\limits_{k =1}^{4} {\mathlarger{\mathfrak{I}}}_k .
\end{aligned}
\end{equation}\\

For the first  term 
\[
{\mathlarger{\mathfrak{I}}}_1 =  \int_{I_*} \int_{\omega}    \bn^\intercal \bigl( \partial_t\bm{\tau}\,\bn_\eta \bigr)\circ\bm{\varphi}_\eta \,\det(\naby \bm{\varphi}_{\eta})    \partial^2_t \Dely\eta \dy\dt, 
\]
it holds that 
\begin{equation*}
\begin{aligned}
{\mathlarger{\mathfrak{I}}}_1 & \lesssim   \int_{I_*}  \Vert \partial_t\nabla\bv \Vert_{W^{1/2, 2}(\partial\Omega_\eta)} \Vert \partial_t^2 \naby\eta \Vert_{W^{1/2,2}(\omega)}  \dt  + \int_{I_*} \Vert \partial_t p \Vert_{L^{2}(\partial\Omega_\eta)}  \Vert \partial_t^2 \Dely\eta \Vert_{L^{2}(\omega)}  \dt .
\end{aligned}
\end{equation*}
Using the renormalised continuity equation \eqref{eq:RenormContEq}, interpolation, and Young's inequality, we derive that  
\begin{equation}\label{eq:I1}
\begin{aligned}
{\mathlarger{\mathfrak{I}}}_1 &\lesssim \kappa\int_{I_*}  \left(   \Vert \partial^2_t \Dely\eta   \Vert_{L^2(\omega)}^2 \dt  +   \Vert \partial_t \bv \Vert_{W^{2, 2}(\Omega_\eta)}^2   \right) \dt  + c(\kappa) \int_{I_*}  \Vert \partial^2_t \naby\eta   \Vert_{L^2(\omega)}^2  \dt
\\[0.4em]
& \quad + c(\kappa) \int_{I_*} \Vert \rho \Vert_{W^{3,2}(\Omega_\eta) }^2 \Vert \bv \Vert_{W^{2, 2}(\Omega_\eta)}^2 \dt.
\end{aligned}
\end{equation}

Now consider 
\[
{\mathlarger{\mathfrak{I}}}_2 =  \int_{I_*} \int_{\omega}    \bn^\intercal \Bigl[ \bigl( (\nabla\bm{\tau})\circ\bm{\varphi}_\eta\,   \partial_t\bm{\varphi}_\eta \bigr) \,\bn_\eta\circ\bm{\varphi}_\eta \Bigr]    \,\det(\naby \bm{\varphi}_{\eta})    \partial^2_t \Dely\eta \dy\dt, 
\]
then a duality argument together with the embedding \[ L^2(\omega) \hookrightarrow W^{-1/2, 2}(\omega) \]   yields 
\begin{align*}
{\mathlarger{\mathfrak{I}}}_2  &\lesssim \int_{I_*} \left( \Vert \nabla^2\bv \Vert_{W^{1/2, 2}(\partial\Omega_\eta)}  + \Vert \nabla p \Vert_{W^{1/2, 2}(\partial\Omega_\eta)}    \right) \Vert \partial_t\eta \Vert_{L^{\infty}(\omega)}  \Vert \partial^2_t \Dely\eta\Vert_{L^{2}(\omega)} \dt.
\end{align*}
Thus,  from the trace theorem and Young's inequality, we arrive at 
\begin{equation}\label{eq:I2}
\begin{aligned}
{\mathlarger{\mathfrak{I}}}_2 &\lesssim \kappa \int_{I_*} \Vert \partial^2_t \Dely\eta   \Vert_{L^2(\omega)}^2 \dt  +  c(\kappa)\int_{I_*}  \Vert \bv \Vert_{W^{3, 2}(\Omega_\eta)}^2 \Vert \partial_t\Dely\eta \Vert_{L^2(\omega)}^2  \dt
\\[0.4em]
& \quad + c(\kappa) \int_{I_*} \Vert \rho \Vert_{W^{3,2}(\Omega_\eta) }^2 \Vert \partial_t\Dely\eta \Vert_{L^2(\omega)}^2 \dt.
\end{aligned}
\end{equation}

To estimate the term 

\[
{\mathlarger{\mathfrak{I}}}_3 =  \int_{I_*} \int_{\omega}   \bn^\intercal \Bigl[ \bm{\tau}\circ\bm{\varphi}_\eta\,
       \partial_t\bigl( \bn_\eta\circ\bm{\varphi}_\eta\bigr) \Bigr]    \,\det(\naby \bm{\varphi}_{\eta})   \partial^2_t \Dely\eta \dy\dt, 
\]
we rely on \eqref{eq:ChainRuleNormal}.  Applying H\"older inequality with $L^\infty -$ norm on $\partial_t\bigl( \bn_\eta\circ\bm{\varphi}_\eta\bigr)$, we deduce  from Young's inequality and the norm equivalence
\[
\Vert  \partial_t\bigl( \bn_\eta\circ\bm{\varphi}_\eta\bigr) \Vert_{L^{\infty}(\omega)} \approx \Vert \partial_t\naby \eta  \Vert_{L^{\infty}(\omega)},
\]
that 
\begin{equation}\label{eq:I3}
\begin{aligned}
{\mathlarger{\mathfrak{I}}}_3 &\lesssim \kappa \int_{I_*} \Vert \partial^2_t \Dely\eta   \Vert_{L^2(\omega)}^2 \dt +  c(\kappa)\int_{I_*}  \Vert \bv \Vert_{W^{2, 2}(\Omega_\eta)}^2 \Vert \partial_t\naby \Dely\eta \Vert_{L^2(\omega)}^2  \dt 
\\[0.4em]
& \quad + c(\kappa) \int_{I_*} \Vert \nabla p \Vert_{L^{2}(\Omega_\eta) }^2 \Vert \partial_t\naby \Dely\eta \Vert_{L^2(\omega)}^2  \dt.
\end{aligned}
\end{equation}

Similarly, using \eqref{eq:DetIdentity} and arguing as before, we obtain for the term 
\[
{\mathlarger{\mathfrak{I}}}_4 =  \int_{I_*} \int_{\omega}   \bn^\intercal \bigl(\bm{\tau}\,\bn_\eta\bigr)\circ \bm{\varphi}_\eta\,
     \partial_t \det(\naby \bm{\varphi}_{\eta})  \partial^2_t \Dely\eta \dy\dt, 
\]
 the following estimate 

\begin{equation}\label{eq:I4}
\begin{aligned}
{\mathlarger{\mathfrak{I}}}_4 &\lesssim \kappa \int_{I_*} \Vert \partial^2_t \Dely\eta   \Vert_{L^2(\omega)}^2 \dt +  c(\kappa)\int_{I_*}  \Vert \bv \Vert_{W^{2, 2}(\Omega_\eta)}^2 \Vert \partial_t\naby \Dely\eta \Vert_{L^2(\omega)}^2  \dt 
\\[0.4em]
& \quad + c(\kappa) \int_{I_*} \Vert \nabla p \Vert_{L^{2}(\Omega_\eta) }^2 \Vert \partial_t\naby \Dely\eta \Vert_{L^2(\omega)}^2  \dt.
\end{aligned}
\end{equation}\\

The estimates derived in {\bf{Steps~2a--2b}}, together with \cref{rem:HigherOrderEta}, show that  higher-order derivatives of the shell displacement $\eta$  can be bounded on $I_*$ in terms of the initial data, the density estimates from {\bf{Step~1}}, and higher-order Sobolev norms of the fluid variables $\bv$ and $p$.  In particular, all shell terms can be controlled provided the quantities 
\[
\Vert \bv \Vert_{L^2(I_*;W^{4,2}(\Omega_\eta))}, \qquad
\|\partial_t\bv\|_{L^2(I_*;W^{2,2}(\Omega_\eta))},
\qquad
\sup_{I_*}\|\rho\|_{W^{3,2}(\Omega_\eta)},
\quad 
\|p\|_{L^2(I_*;W^{3,2}(\Omega_\eta))},
\]
remain finite.    \\
Therefore, it remains  to complement the above conditional estimates for $\eta$ with corresponding higher-order estimates for the fluid velocity $\bv$ and pressure $p$.  The relevant result is stated  in the next lemma, which allows to close the argument.

\medskip


\begin{lemma}\label{lem:VelocityEstim}
Let $(\rho, \bv, \eta) $ be a strong solution to \eqref{eq:ContMomentEq}--\eqref{eq:interfaceCond} in the sense of  \cref{def:StrongSol}. Moreover, let Assumptions \ref{A1}--\ref{A3}, \ref{B} and the compatibility condition \eqref{eq:CC} hold.
Then  $(\bv, \rho)$ satisfy the  estimate 
\begin{equation}\label{eq:VelocityCondEstimate}
\begin{aligned}
& \int_{I_*} \left( \Vert \partial_t \bv \Vert_{W^{2,2}(\Omega_\eta)}^2   + \Vert \bv \Vert_{W^{4,2}(\Omega_\eta)}^2  + \Vert p \Vert_{W^{3,2}(\Omega_\eta)}^2 \right) \dt  + \sup\limits_{I_*} \Vert \rho \Vert_{W^{3,2}(\Omega_\eta)}^2 
\\[0.4em]
& \quad \lesssim {\mathlarger{\mathtt{E}}}_{\mathrm{acc}}  +   \Vert \bv_0 \Vert_{W^{3,2}(\Omega_{\eta_0})}^2 +  \Vert \rho_0 \Vert_{W^{3,2}(\Omega_{\eta_0})}^2 + \Vert \eta_0 \Vert_{W^{5,2}(\omega)}^2 +  \Vert \eta_* \Vert_{W^{3,2}(\omega)}^2,
\end{aligned}
\end{equation}
with the hidden constant depending on \ref{A1}--\ref{A3}, \ref{B},  $T_*$,   and  $C_0$.
\end{lemma}

\medskip 

\begin{proof}
Throughout the proof we work on the fixed domain $\Omega_{\eta_0}$.  For this purpose, we introduce the time-dependent  diffeomorphism 
\[
\bfPsi_{\eta \to \eta_0 }  := \bfPsi_\eta \circ \left(\bfPsi_{\eta_0} \right)^{-1} \colon \Omega_{\eta_0} \to   \Omega_{\eta},  
\]
and define the pull-back density and velocity by 
\[
\underline{\rho} = \rho \circ \bfPsi_{\eta \to \eta_0 }, \qquad \underline{\bv} := \bv\circ \bfPsi_{\eta \to \eta_0 }. 
\]
Moreover, we define
\begin{equation*}\label{matrices}
\begin{aligned}
\mathbf{A}_{\eta \to \eta_0 } & =J_{\eta \to \eta_0 } \big(\nabx \bfPsi_{\eta \to \eta_0 }^{-1}\circ \bfPsi_{\eta \to \eta_0 }\big)\big( \nabx \bfPsi_{\eta \to \eta_0 }^{-1}\circ \bfPsi_{\eta \to \eta_0 } \big)^\intercal,
\\[0.4em]
\mathbf{B}_{\eta \to \eta_0 } & =J_{\eta \to \eta_0 } \left(\nabx \bfPsi_{\eta \to \eta_0 }^{-1}\circ \bfPsi_{\eta \to \eta_0 }\right)^\intercal,
\end{aligned}
\end{equation*}
where $J_{\eta \to \eta_0 } = \mathrm{det}(\nabla\bfPsi_{\eta \to \eta_0 })$.   \\

We next reduce to homogeneous boundary conditions by setting  
\[ \underline{\bu} = \underline{\bv} - \mathcal{E}_{\eta_0}(\partial_t\eta\bn). \]
For convenience, we denote by $\mathcal{L}$ the Lam\'e operator 
\[
\mathcal{L}\,\underline{\bu} := \mu\Delta\underline{\bu} + (\lambda + \mu)\nabla\Div\underline{\bu},  
\] 
and further introduce  the operators 
\begin{equation*}
\begin{aligned}
\underline{\mathsf{B}}\, \underline{\bu} & :=  - J_{\eta \to \eta_0 } \underline{\rho} \nabla\underline{\bu} \cdot \partial_t \bfPsi_{\eta \to \eta_0 }^{-1}\circ \bfPsi_{\eta \to \eta_0 }    -   \underline{\rho} \mathcal{E}_{\eta_0}(\partial_t\eta\bn)  \big( \nabla \underline{\bu} \colon \mathbf{B}_{\eta \to \eta_0 } \big) -  \underline{\rho} \underline{\bu}  \big( \nabla\mathcal{E}_{\eta_0}(\partial_t\eta\bn) \colon  \mathbf{B}_{\eta \to \eta_0 }    \big),
\end{aligned}
\end{equation*}
and 
\begin{equation*}
\begin{aligned}
\underline{\mathsf{f}} & :=  - J_{\eta \to \eta_0 } \underline{\rho} \partial_t \mathcal{E}_{\eta_0}(\partial_t\eta\bn)  -  J_{\eta \to \eta_0 } \underline{\rho} \nabla\mathcal{E}_{\eta_0}(\partial_t\eta\bn) \cdot \partial_t \bfPsi_{\eta \to \eta_0 }^{-1}\circ \bfPsi_{\eta \to \eta_0 }  - \underline{\rho} \underline{\bu} \big( \nabla \underline{\bu} \colon \mathbf{B}_{\eta \to \eta_0 } \big)  
\\[0.4em]
& \quad \;\,  - \underline{\rho} \mathcal{E}_{\eta_0}(\partial_t\eta\bn) \big( \nabla\mathcal{E}_{\eta_0}(\partial_t\eta\bn)  \colon\mathbf{B}_{\eta \to \eta_0 }  \big)  - \mathbf{B}_{\eta \to \eta_0 } \nabla p(\underline{\rho})     
\\[0.4em]
& \quad \;\, + \Div\Bigg[  \mu \left(\mathbf{A}_{\eta \to \eta_0 }  - \mathbb{I}_{3 \times 3}   \right) \nabla\underline{\bu}  + (\lambda + \mu)\biggl( \Bigl(  \left(\mathbf{B}_{\eta \to \eta_0 } - \mathbb{I}_{3 \times 3} \right) \colon \nabla\underline{\bu} \Bigr)  \mathbb{I}_{3 \times 3}  
\\[0.4em]
& \qquad \qquad + \left( \mathbf{B}_{\eta \to \eta_0 } \colon \nabla\underline{\bu}  \right)  \left( \dfrac{1}{J_{\eta \to \eta_0 }} \mathbf{B}_{\eta \to \eta_0 } -  \mathbb{I}_{3 \times 3}  \right)                 \biggr)      \Bigg]      
\\[0.4em]
& \quad \;\, + \Div\Bigg[ \mu  \mathbf{A}_{\eta \to \eta_0 }  \nabla\mathcal{E}_{\eta_0}(\partial_t\eta\bn) +   \dfrac{(\lambda + \mu)}{J_{\eta \to \eta_0}}   \Big( \mathbf{B}_{\eta \to \eta_0 } \colon \nabla \mathcal{E}_{\eta_0}(\partial_t\eta\bn)     \Big)     \mathbf{B}_{\eta \to \eta_0 }     \Bigg].
\end{aligned}
\end{equation*}
Then $\underline{\bu}$ solves the initial--boundary value problem (IBVP) (cf. \cite[Section 3, Lemma 3.2]{ngougoue2025local})

\begin{equation}\label{eq:IBVP-u}
\left\{
\begin{aligned}
J_{\eta \to \eta_0}\underline{\rho}\partial_t\underline{\bu} - \mathcal{L}\,\underline{\bu}  - \underline{\mathsf{B}} \, \underline{\bu} & = \underline{\mathsf{f}}  && \text{ in } I_*\times \Omega_{\eta_0}
\\[0.4em]
\underline{\bu} & = 0 && \text{ on }  I_* \times \partial\Omega_{\eta_0} 
\\[0.4em]
 \underline{\bu}(0) &= \underline{\bu}_0 = \underline{\bv}_0 - \mathcal{E}_{\eta_0} (\eta_*\bn) && \text{ in } \Omega_{\eta_0} .
\end{aligned}\right.
\end{equation}

We define 
\[
D\!\left(\mathcal{L}\right)  = \left\{  \underline{\bu} \in W^{2,2}(\Omega_{\eta_0})  \;\colon\;  \mathcal{L}\,\underline{\bu} \in  L^{2}(\Omega_{\eta_0}) \text{ and } \underline{\bu}\Big |_{\partial\Omega_{\eta_0}} = 0   \right\} 
\]
to be the domain of the Lam\'e operator ${\mathlarger{\mathcal{L}}}$.
We further introduce  the multiplicative perturbation
\[
\mathcal{M}(t) \;\colon\;  W^{2,2}(\Omega_{\eta_0}) \xrightarrow{\hspace*{0.7cm}} W^{2,2}(\Omega_{\eta_0});  \quad \bu \longmapsto \bigl(J_{\eta \to \eta_0}\underline{\rho}\bigr)^{-1}(t)\underline{\bu} .
\] 
Since the density $\rho$ is uniformly positive and bounded -- a consequence of Assumption \eqref{eq:PositivityAssumption}  -- 
and due to the regularity of $\rho$, it follows that $\mathcal{M}(t) $ is well-defined with $\mathcal{M}(t) \in \mathscr{L}\left( W^{2,2}(\Omega_{\eta_0})\right)$\footnote{For a Banach space $X$, we denote by $\mathscr{L}(X)$ the space of linear bounded operators $T \colon X \to X$.  }. \\
For every $t \in I_*$, we define 
\[
  \mathsf{A}(t):=   \mathcal{M}(t)\mathcal{L}, \qquad   \mathsf{B}(t):=   \mathcal{M}(t)\underline{\mathsf{B}} \quad \text{and} \quad \mathsf{f}(t) :=  \mathcal{M}(t) \underline{\mathsf{f}}. 
\]
Thus, \eqref{eq:IBVP-u} is nothing else but an IBVP of the form 

\begin{equation}\label{eq:CauchyPb-u}
\tag{$\mathtt{PB}$}\left\{
\begin{aligned}
\partial_t\underline{\bu} & =  \mathsf{A} \underline{\bu} + \mathsf{B}\underline{\bu}  + \mathsf{f}  && \text{ in } I_* \times \Omega_{\eta_0},
\\[0.4em]
\underline{\bu} & = 0 && \text{ on }  I_* \times \partial\Omega_{\eta_0} ,
\\[0.4em]
\underline{\bu}(0) & = \underline{\bu}_0   && \text{ in } \Omega_{\eta_0}.  
\end{aligned}\right.
\end{equation}\\
Importantly, we have for all $t \in [0, T_*)$, 
\[
D(\mathsf{A}(t)) = D(\mathcal{L}) = W^{2,2}(\Omega_{\eta_0})\cap W^{1,2}_0(\Omega_{\eta_0}). 
\]
Furthermore, note that the acceleration estimate \eqref{eq:AccelEstimate} yields  precisely the boundary regularity required for  maximal regularity in $W^{4,2}$. More precisely,  since $\partial\Omega_{\eta_0}$ is uniformly of class $W^{7/2, 2}$ (cf. \cref{rem:HigherOrderEta}), the elliptic problem 
\begin{equation}
\gamma \underline{\bu} + \mathsf{A}(0)\underline{\bu} =  \mathsf{h}
\end{equation}
with $\gamma \in \R$, admits a unique solution $\underline{\bu} \in W^{4,2}(\Omega_{\eta_0})\cap W^{1,2}_0(\Omega_{\eta_0})$ for  every $\mathsf{h} \in W^{2,2}(\Omega_{\eta_0})\cap W^{1,2}_0(\Omega_{\eta_0})$. Moreover,  
\begin{equation}
\Vert \underline{\bu} \Vert_{W^{4,2}(\Omega_{\eta_0})\cap W^{1,2}_0(\Omega_{\eta_0})}   \leq C \Big( \Vert \mathsf{h} \Vert_{W^{2,2}(\Omega_{\eta_0})\cap W^{1,2}_0(\Omega_{\eta_0})}  + \Vert \underline{\bu} \Vert_{W^{2,2}(\Omega_{\eta_0})\cap W^{1,2}_0(\Omega_{\eta_0})}  \Big) 
\end{equation}  
with $C > 0$. Hence, for each fixed  $t$, $\mathsf{A}(t)$ defines a closed operator with graph norm equivalent to the $W^{4,2}$-norm  (see, e.g., \cite[Chapter 14, Section 14.5, Theorem 14.5.4]{maz2009theory}). \\[-0.8em]

We now aim to apply Solonnikov's anisotropic $L^2-$theory for linear parabolic systems  of general form \cite[Chapter 5, Section 21]{solonnikov1965}. Accordingly, we verify the required hypotheses for \eqref{eq:CauchyPb-u}. \\
Since $\underline{\rho} \in W^{1,\infty}\left(   I_*; W^{2,2}(\Omega_{\eta_0})    \right)$  and satisfies the uniform positivity condition of Assumption \eqref{eq:PositivityAssumption}, the identity 
\[
\partial_t  \underline{\rho}^{-1}  = - \underline{\rho}^{-2} \partial_t \underline{\rho}
\]
together with the Banach algebra property of $W^{2,2}(\Omega_{\eta_0})$ implies 
\[
\underline{\rho}^{-1} \in W^{1,\infty}\left(   I_*; W^{2,2}(\Omega_{\eta_0})    \right).
\]
Owing to the non-degeneracy of the deformation, $J_{\eta \to \eta_0}$ is uniformly bounded away from zero.  Moreover, the Hanzawa transform (see  \cref{sec:Prelim}) yields
\[
J_{\eta \to \eta_0} \in W^{1,\infty}\left(   I_*; W^{2,2}(\Omega_{\eta_0})    \right) \qquad \Bigl(\text{ since} \;\; \partial_t\eta \in L^{\infty}\left(   I_*; W^{3,2}(\Omega_{\eta_0})    \right) \Bigr), 
\]
whence,
\[
\dfrac{1}{J_{\eta \to \eta_0 } }  \in W^{1,\infty}\left(   I_*; W^{2,2}(\Omega_{\eta_0})    \right). 
\]
In particular, the map 
\[
t \longmapsto \left( J_{\eta \to \eta_0}\underline{\rho} \right)^{-1}(t)
\] 
is Lipschitz continuous in time. Hence, 
\[
\mathsf{A} \in C\Big( I_*; \mathscr{L}\big( D(\mathsf{A} (0)),  W^{2,2}(\Omega_{\eta_0})\big)       \Big).
\]
Moreover,  $\forall\, \bx \in \overline{\Omega_{\eta_0}}$, let  $\mathcal{P}_{\mathsf{A}} $ be the principal part of  $- \mathsf{A}(t)$. That is,  $\forall\, \bm{\xi} \in \R^3,$ we have 
\[
\mathcal{P}_{\mathsf{A}} (t, \bx, \bm{\xi}) = \left(J_{\eta \to \eta_0}\underline{\rho}\right)^{-1} \Big( \mu|\bm{\xi}|^2 \mathbb{I}_{3\times3} + (\lambda + \mu)  \bm{\xi}\otimes\bm{\xi} \Big). 
\]
The operator $- \mathsf{A}(t)$ is normally elliptic in the sense that $\forall\, \bm{\xi} \in \R^3, \, |\bm{\xi}| = 1$ its spectrum 
\[
\sigma\left(\mathcal{P}_{\mathsf{A}} (t, \bx, \bm{\xi}) \right) \subset \mathbb{C}_+  := \{ z \in \mathbb{C} \colon \text{Re}(z) > 0 \}.
\]
Indeed, it follows from the rank-nullity theorem that 
\[
\sigma\left(\mathcal{P}_{\mathsf{A}} (t, \bx, \bm{\xi}) \right)  = \sigma_{\mathtt{p}}\left(\mathcal{P}_{\mathsf{A}} (t, \bx, \bm{\xi}) \right)  := \big\{ \lambda \in \mathbb{C} \;\colon\; \lambda \mathbb{I}_{3\times3} - \mathcal{P}_{\mathsf{A}} (t, \bx, \bm{\xi}) \text{ is not injective}  \big\}.
\]
If ${\bm{\mathsf{w}}}_0 \in \R^3 \setminus  \{0\}$  is an eigenvector of $\mathcal{P}_{\mathsf{A}} (t, \bx, \bm{\xi})$ with associated eigenvalue $\lambda_0$, that is,
\[
\mu\left(J_{\eta \to \eta_0}\underline{\rho}\right)^{-1}|\bm{\xi}|^2 {\bm{\mathsf{w}}}_0  + (\lambda + \mu) \left(J_{\eta \to \eta_0}\underline{\rho}\right)^{-1} \left(\bm{\xi}\otimes\bm{\xi} \right) {\bm{\mathsf{w}}}_0 = \lambda_0 {\bm{\mathsf{w}}}_0,
\]
then up to assuming $|{\bm{\mathsf{w}}}_0| = 1$, it follows from the symmetry and the  Courant--Fischer min-max  theorem 
\cite[Chapter 4, Theorem 4.2.2]{HornJohnson2012} that 
\[
\lambda_{0, \min} = \left(J_{\eta \to \eta_0}\underline{\rho}\right)^{-1} \Big( \mu |\bm{\xi}|^2 + \min\limits_{|{\bm{\mathsf{w}}}_0| = 1} (\lambda + \mu) (\bm{\xi}\cdot {\bm{\mathsf{w}}}_0)^2 \Big). 
\]
Whence, $\sigma\left(\mathcal{P}_{\mathsf{A}} (t, \bx, \bm{\xi}) \right) \subset \mathbb{C}_+ $  holds.   Hence, \eqref{eq:CauchyPb-u} is parabolic in the sense of Petrovsky. \\[-0.8em]

We next verify the Shapiro--Lopatinskii complementing condition for the Dirichlet boundary operator.  That is,  
 $\forall\,  \by \in \partial\Omega_{\eta_0}, \; \bm{\xi} \in \R^3  $ with $\bm{\xi} \cdot \bn_{\eta_0}(\by) = 0, \; \forall\, \bm{h} \in \mathbb{C}^3, \zeta \in \overline{\mathbb{C}}_+ $  with $|\bm{\xi}| + |\zeta| \neq 0$, the ODE system 
 \begin{equation}\label{eq:ODEsyst}
\tag{$\mathtt{SL}$}\left\{
\begin{aligned}
\zeta {\bm{\mathsf{w}}}(z) +  \mathcal{P}_{\mathsf{A}} \left(t, \by, \bm{\xi} + i\bn_{\eta_0}(\by)\partial_z \right) {\bm{\mathsf{w}}}(z) & = 0  && \text{ in } \R_+,
\\[0.4em]
{\bm{\mathsf{w}}}(0) & = \bm{h} ,  &&  
\end{aligned}\right.
\end{equation}
with $\partial_z := \dfrac{\partial}{\partial z}$,  admits a unique solution 
\[
{\bm{\mathsf{w}}} \in C_0 \left( \R_+ ; \mathbb{C}^3 \right) := \left\{ {\bm{\mathsf{w}}} \colon \R_+ \to \mathbb{C}^3 \text{ continuous, such that  } {\bm{\mathsf{w}}}(z) \to 0 \text { as } z \to \infty         \right\} .
\]
Indeed, put 
\[
\mathcal{S}_{(\mathtt{SL})} =  \Big\{   {\bm{\mathsf{w}}} \in C_0 \left( \R_+ ; \mathbb{C}^3 \right) \;\colon \;    \zeta {\bm{\mathsf{w}}}(z) +  \mathcal{P}_{\mathsf{A}} \left(t, \by, \bm{\xi} + i\bn_{\eta_0}(\by)\partial_z \right) {\bm{\mathsf{w}}}(z)  = 0 \;\; \forall\, z \in \R_+          \Big\}
\]
and consider the map 
\[ \mathcal{S} \;\colon\;  \mathcal{S}_{(\mathtt{SL})}  \xrightarrow{\hspace*{0.5cm}}  \mathbb{C}^3; \qquad {\bm{\mathsf{w}}} \longmapsto {\bm{\mathsf{w}}}(0) .
\]
Then $\ker\!\left( \mathcal{S}  \right) =  \left\{ 0_{C_0 \left( \R_+ ; \mathbb{C}^3 \right)} \right\} $. Importantly, by normal ellipticity, $\mathrm{dim}\left(\mathcal{S}_{(\mathtt{SL})}\right) = 3$.

Moving further with the smoothness assumptions on the coefficients of $\mathsf{A}$, we first note that 
 $\forall\, t \in I_*, $    
\begin{align*}
\mathsf{A} (t) = \sum\limits_{\mathsf{i},\mathsf{j}  = 1}^{3} a_{\mathsf{i},\mathsf{j}} \mathsf{D}_{\mathsf{i}\mathsf{j}} = \left( J_{\eta \to \eta_0}\underline{\rho} \right)^{-1}(t) \Big( \mu \mathbb{I}_{3\times 3}\Delta  + (\lambda + \mu)\nabla\otimes\nabla \Big),
\end{align*}
where 
\[
(\nabla\otimes\nabla) \underline{\bu} := \left(  \sum\limits_{\mathsf{j} = 1}^{3}    \mathsf{D}_{\mathsf{i}\mathsf{j}} \underline{u}_\mathsf{j}        \right)_{\mathsf{i} = 1}^{3}  = \nabla\Div\underline{\bu},
 \quad  \text{ with } \;\;   \mathsf{D}_{\mathsf{i}\mathsf{j}}  = \dfrac{\partial}{\partial x_\mathsf{i}} \left(  \dfrac{\partial}{\partial x_\mathsf{j}} \right) 
\]
and 
\[
[a_{\mathsf{i},\mathsf{j}}]  =  (\lambda + 2\mu)\left( J_{\eta \to \eta_0}\underline{\rho} \right)^{-1} \mathbb{I}_{3 \times 3} + \begin{bmatrix}
0 &   & 1 && 1 \\\\
1 &  & 0  && 1 \\ \\
1 && 1 && 0
\end{bmatrix} .
\]
Thus,   
\[
a_{\mathsf{i},\mathsf{i}} \in W^{1,\infty}\left(   I_*; W^{2,2}(\Omega_{\eta_0})    \right) \quad \forall\,  \mathsf{i} = \left\{ 1,2,3 \right\}. 
\]
In particular, by Sobolev embedding 
\[
W^{2,2}(\Omega_{\eta_0})   \hookrightarrow    C^{0, 1/2}(\overline{\Omega_{\eta_0}}\,), 
\]
it holds that 
\[
a_{\mathsf{i},\mathsf{i}} \in C^{0,1}\left(I_*; C^{0, 1/2}(\overline{\Omega_{\eta_0}}\,)  \right) \quad \forall\,  \mathsf{i} = \left\{ 1,2,3 \right\}. 
\]
Noteworthily, the compatibility condition \eqref{eq:CC} yields 
\[
\partial_t \underline{\bu}\circ  \bm{\varphi}_{\eta_0}\raisebox{-1.6ex}{$\Big|_{t=0}$}  = 0  \quad \text{on } \omega. 
\]
Hence, by \cite[Chapter 5, Section 21, Theorem 5.4]{solonnikov1965},  the solution $\underline{\bu}$ of \eqref{eq:CauchyPb-u} satisfies 
\begin{equation}\label{eq:uMRestimatePrior}
\begin{aligned}
&  \Vert \partial_t \underline{\bu} \Vert^2_{L^2\left(I_*;  W^{2,2}(\Omega_{\eta_0})   \right)} + \Vert \underline{\bu} \Vert^2_{L^2\left(I_*;  W^{4,2}(\Omega_{\eta_0})   \right)}  
 \\[0.4em]
&   \lesssim  \Vert \underline{\bu}_0 \Vert^2_{W^{3,2}(\Omega_{\eta_0}) } +  \Vert  \mathsf{B}\underline{\bu} \Vert^2_{L^2\left(I_*;  W^{2,2}(\Omega_{\eta_0})   \right)} +  \Vert  \mathsf{f} \Vert^2_{L^2\left(I_*;  W^{2,2}(\Omega_{\eta_0})   \right)} + \Vert  \partial_t \big( \mathsf{B}\underline{\bu} + \mathsf{f} \big) \Vert^2_{L^2\left(I_*;  L^{2}(\Omega_{\eta_0})   \right)} ,
\end{aligned}
\end{equation}
where, for arbitrary $\kappa > 0$,   the following estimates hold:
\begin{align}
\Vert  \mathsf{B}\underline{\bu}  \Vert^2_{L^2\left(I_*;  W^{2,2}(\Omega_{\eta_0})   \right)}
&\lesssim \kappa \int_{I_*} \Vert \underline{\bu} \Vert_{W^{4,2}(\Omega_{\eta_0})}^2  \dt
+ \int_{I_*} \Vert \underline{\bu} \Vert_{W^{3,2}(\Omega_{\eta_0})}^2
\Vert \partial_t \eta \Vert_{W^{5/2,2}(\omega)}^2  \dt
+ {\mathlarger{\mathtt{E}}}_{\mathrm{acc}} ,
\label{eq:BuEstim}
\\[1em]
\Vert  \partial_t \mathsf{B}\underline{\bu}  \Vert^2_{L^2\left(I_*;  L^{2}(\Omega_{\eta_0})   \right)}
&\lesssim  \kappa \int_{I_*} \Vert \partial_t \underline{\bu} \Vert_{W^{2,2}(\Omega_{\eta_0})}^2  \dt
+ {\mathlarger{\mathtt{E}}}_{\mathrm{acc}}  \nonumber
\\[0.4em]
&\quad + \int_{I_*} \Vert \underline{\bu} \Vert_{W^{3,2}(\Omega_{\eta_0})}^2
\left(
\Vert \partial_t^2 \eta \Vert_{W^{1/2, 2}(\omega)}^2
+ \Vert \partial_t \eta \Vert_{W^{5/2, 2}(\omega)}^2
\right) \dt ,   \label{eq:TimeBuEstim}
\end{align}
\\[-0.6em]
\begin{equation}\label{eq:fEstim}
\begin{aligned}
&  \Vert  \mathsf{f} + \mathcal{E}_{\eta_0}(\partial_t\eta\bn)  \Vert^2_{L^2\left(I_*;  W^{2,2}(\Omega_{\eta_0})   \right)} 
\\[0.4em]
& \lesssim    \int_{I_*} \left( \Vert \nabla  \underline{\rho} \Vert_{L^\infty(\Omega_{\eta_0})}^2  + \Vert \partial_t \Dely  \eta \Vert_{L^2(\omega)}^2 \right) \Vert   \partial_t\eta\Vert_{W^{5/2,2}(\omega)}^2  \dt  + \int_{I_*} \Vert \nabla  \underline{\rho} \Vert_{L^\infty(\Omega_{\eta_0})}^2  \Vert   \underline{\rho} \Vert_{W^{3,2}(\Omega_{\eta_0})}^2  \dt 
\\[0.4em]
& \quad  + (T_* + \kappa) \sup\limits_{I_*} \Vert   \partial_t \eta \Vert_{W^{5/2,2}(\omega)}^2  + T_*^{1/2} \int_{I_*}  \Vert   \partial_t \eta \Vert_{W^{7/2,2}(\omega)}^2  \dt  
\\[0.4em]
& \quad    + \int_{I_*} \left( \Vert \nabla  \underline{\rho} \Vert_{L^\infty(\Omega_{\eta_0})}^2  + \Vert \nabla  \underline{\bu} \Vert_{L^\infty(\Omega_{\eta_0})}^2 + \Vert \underline{\bu} \Vert_{W^{2,2}(\Omega_{\eta_0})}^2 \right) \Vert   \underline{\bu} \Vert_{W^{3,2}(\Omega_{\eta_0})}^2  \dt    + {\mathlarger{\mathtt{E}}}_{\mathrm{acc}}
\\[0.4em]
& \quad + \int_{I_*} \Vert \nabla  \underline{\rho} \Vert_{L^\infty(\Omega_{\eta_0})}^2  \Vert   \underline{\rho} \Vert_{W^{3,2}(\Omega_{\eta_0})}^2  \dt  + T_* \sup\limits_{I_*} \Vert   \underline{\bu} \Vert_{W^{3,2}(\Omega_{\eta_0})}^2  + T_*^{1/2} \int_{I_*}  \Vert   \underline{\bu} \Vert_{W^{4,2}(\Omega_{\eta_0})}^2 \dt
\\[0.4em]
& \quad   +  \int_{I_*}  \Vert \eta \Vert_{W^{9/2, 2}(\omega)}^2 \Vert \nabla\underline{\bu}\Vert_{L^\infty(\Omega_{\eta_0})}^2 \dt  + \int_{I_*} \Vert \nabla \underline{\rho} \Vert_{L^{\infty}(\Omega_{\eta_0})}^2 \Vert  p(\underline{\rho}) \Vert_{W^{3,2}(\Omega_{\eta_0})}^2 \dt  
\\[0.4em]
& \quad  +  \int_{I_*} \Vert  p(\underline{\rho}) \Vert_{W^{3,2}(\Omega_{\eta_0})}^2 \dt + \int_{I_*}  \left(  1 +   \Vert \underline{\rho} \Vert_{W^{2,2}(\Omega_{\eta_0})}^2 \right) \Vert  \nabla \underline{\rho} \Vert_{L^{\infty}(\Omega_{\eta_0})}^2    \dt ,
\end{aligned}
\end{equation}
and 

\begin{equation}\label{eq:TimefEstim}
\begin{aligned}
& \Vert \partial_t \big( \mathsf{f} + \mathcal{E}_{\eta_0}(\partial_t\eta\bn) \big)  \Vert^2_{L^2\left(I_*;  L^{2}(\Omega_{\eta_0})   \right)} 
\\[0.4em]
&\lesssim {\mathlarger{\mathtt{E}}}_{\mathrm{acc}}  +  \int_{I_*} \left( \Vert   \underline{\bv} \Vert_{W^{2,2}(\Omega_{\eta_0})}^2  +  \Vert   \partial_t \eta \Vert_{W^{5/2,2}(\omega)}^2 +  \Vert  \nabla \underline{\bv} \Vert_{L^{\infty}(\Omega_{\eta_0})}^2  + \Vert  \nabla \underline{\rho} \Vert_{L^{\infty}(\Omega_{\eta_0})}^2   \right)  \Vert \underline{\rho} \Vert_{W^{3,2}(\Omega_{\eta_0})}^2 \dt  
\\[0.4em]
&\quad +  \int_{I_*}  \Vert \underline{\bu} \Vert_{W^{2,2}(\Omega_{\eta_0})}^2 \Big( \Vert \underline{\bv} \Vert_{W^{3,2}(\Omega_{\eta_0})}^2  +  \Vert \underline{\rho} \Vert_{W^{3,2}(\Omega_{\eta_0})}^2 \Big) \dt    
\\[0.4em]
&\quad +  \int_{I_*}  \Big( \Vert \underline{\bv} \Vert_{W^{3,2}(\Omega_{\eta_0})}^2  +   \Vert \underline{\bu} \Vert_{W^{3,2}(\Omega_{\eta_0})}^2 \Big)\Vert \nabla\underline{\rho} \Vert_{L^{\infty}(\Omega_{\eta_0})}^2 \dt     + T_*  \int_{I_*} \Vert \partial_t^2 \Dely\eta \Vert_{L^{2}(\omega)}^2   \dt
\\[0.4em]
&\quad +  \int_{I_*} \Big(  \Vert   \partial_t \eta \Vert_{W^{5/2,2}(\omega)}^2  +  \Vert \partial_t \underline{\bu} \Vert_{L^{2}(\Omega_{\eta_0})}^2 + \Vert \nabla\underline{\bu} \Vert_{L^{\infty}(\Omega_{\eta_0})}^2   \Big) \Vert \underline{\bu} \Vert_{W^{3,2}(\Omega_{\eta_0})}^2    \dt
\\[0.4em]
&\quad + \kappa  \int_{I_*} \Big(  \Vert \partial_t \eta \Vert_{W^{5/2,2}(\omega)}^2 +  \Vert\partial_t^2 \eta \Vert_{L^{2}(\omega)}^2  \Big)  \Vert\partial_t \eta \Vert_{W^{3,2}(\omega)}^2 \dt  + (T_* + \kappa)\int_{I_*} \Vert \partial_t \underline{\bu} \Vert_{W^{2,2}(\Omega_{\eta_0})}^2  \dt  .
\end{aligned}
\end{equation}
For details, we refer to \cref{app:MRsourcetermEstim}.\\
Furthermore, it follows from \cite[Section 2, p. 273, (2.4)]{moser1966rapidly}  that 
\begin{equation}\label{eq:PressureDensityEstim}
\Vert p(\underline{\rho}) \Vert_{W^{3,2}(\Omega_{\eta_0})} \lesssim  \Big(1 + \Vert \underline{\rho}\Vert_{W^{3,2}(\Omega_{\eta_0})}\Big),
\end{equation}
with the hidden constant depending on $\Vert \underline{\rho}_0 \Vert_{W^{3,2}(\Omega_{\eta_0})}$. \\
Hence, by {\bf{Steps~1--2}}, \cref{rem:HigherOrderEta,prop:DensityW32BoundV},  choosing $\kappa > 0$ sufficiently small and applying Gr\"onwall's lemma,  the estimate \eqref{eq:uMRestimatePrior}  closes on sufficiently short time intervals. \\
Indeed,  Let 
\[
{\mathlarger{\mathtt{X}}}_{I_*} := \int_{I_*} \left( \Vert \partial_t \bv \Vert_{W^{2,2}(\Omega_\eta)}^2   + \Vert \bv \Vert_{W^{4,2}(\Omega_\eta)}^2  + \Vert p \Vert_{W^{3,2}(\Omega_\eta)}^2 \right) \dt  + \sup\limits_{I_*} \Vert \rho \Vert_{W^{3,2}(\Omega_\eta)}^2 . 
\]
For convenience, we define  $J_{\tau, t_0} := (t_0 , t_0 + \tau   ) \subset I_*,$
\[
   \mathtt{S}(t_0) :=   \Vert \bv(t_0) \Vert_{W^{3,2}(\Omega_{\eta(t_0)})}^2 +  \Vert \rho(t_0) \Vert_{W^{3,2}(\Omega_{\eta(t_0)})}^2 +  \Vert \eta (t_0) \Vert_{W^{5,2}(\omega)}^2 + \Vert \partial_t\eta (t_0) \Vert_{W^{3,2}(\omega)}^2, 
\]
and denote by $ {\mathlarger{\mathtt{E}}}_{\mathrm{acc}}(\tau; t_0) $   the acceleration energy   on  $J_{\tau, t_0}$.  \\
For $t_0 = 0$ and $\tau$ sufficiently small,  we get that  
\begin{equation}
{\mathlarger{\mathtt{X}}}_{J_{\tau, 0}}   \lesssim {\mathlarger{\mathtt{E}}}_{\mathrm{acc}}(\tau; 0)  +  \mathtt{S}(0) .
\end{equation}
Now let  $t_0 = \tau$.  The estimate on $J_{\tau, \tau}$ reads
\[
{\mathlarger{\mathtt{X}}}_{J_{\tau, \tau}}   \lesssim  {\mathlarger{\mathtt{E}}}_{\mathrm{acc}}(\tau; \tau) +  \mathtt{S}(\tau). 
\]
However, it follows from the continuity of the maps 
\[
t \mapsto \Vert \bv \Vert_{W^{3,2}(\Omega_{\eta })}^2, \quad t \mapsto \Vert \rho \Vert_{W^{3,2}(\Omega_{\eta })}^2,  \quad  t \mapsto \Vert  \eta \Vert_{W^{5,2}(\omega)}^2  \quad \text{and } \quad t \mapsto \Vert \partial_t \eta \Vert_{W^{3,2}(\omega)}^2 ,
\]
that 
\[
\mathtt{S}(\tau)  \leq  \sup\limits_{J_{\tau, 0}} \left(  \Vert \bv \Vert_{W^{3,2}(\Omega_{\eta })}^2 +   \Vert \rho \Vert_{W^{3,2}(\Omega_{\eta })}^2 +  \Vert \partial_t \eta \Vert_{W^{3,2}(\omega)}^2 + \Vert \eta \Vert_{W^{5,2}(\omega)}^2 \right).
\]
By interpolation and \cref{rem:HigherOrderEta}, we obtain   $\mathtt{S}(\tau) \lesssim {\mathlarger{\mathtt{E}}}_{\mathrm{acc}}(\tau; 0) +   \mathtt{S}(0)$.  Consequently,
\begin{equation}
{\mathlarger{\mathtt{X}}}_{J_{\tau, \tau}}   \lesssim {\mathlarger{\mathtt{E}}}_{\mathrm{acc}}(\tau; 0) + {\mathlarger{\mathtt{E}}}_{\mathrm{acc}}(\tau; \tau) +  \mathtt{S}(0) .
\end{equation} \\
Hence, covering $I_*$ by finitely many such intervals  yields the desired result.  
This completes the proof of  \cref{lem:VelocityEstim}. 
\end{proof}
\begin{remark}
Since 
\[
D(\mathsf{A}(t))  \xhookrightarrow{d} W^{2,2}(\Omega_{\eta_0})\cap W^{1,2}_0(\Omega_{\eta_0}) := \mathbb{X}
\]
it follows from \cite[Section 3, Theorem 3.1]{LionsMagenes1} that 
\[
L^2\left( I_*, D(\mathsf{A}(t)) \right) \cap  W^{1,2}\left( I_*, \mathbb{X}\right)   \hookrightarrow C\left(I_* ;  D(\mathsf{A}^{1/2}(t)) \right).
\]
In particular 
\begin{equation}\label{eq:ThirdOrderVelocitySup}
\sup\limits_{I_*} \Vert \underline{\bv}\Vert_{W^{3,2}(\Omega_{\eta_0})}^2  \lesssim \int_{I_*} \Vert \partial_t \underline{\bv}\Vert_{W^{2,2}(\Omega_{\eta_0})}^2   \dt +   \int_{I_*} \Vert  \underline{\bv}\Vert_{W^{4,2}(\Omega_{\eta_0})}^2   \dt ,
 \end{equation}
where the hidden constant is independent of $T_*$.

\end{remark}

\noindent By combining all preceding estimates, choosing $\kappa > 0$ sufficiently small, and applying Gr\"onwall's lemma, we conclude the proof of   \cref{theo:MainResult2}.   
\begin{remark}
 \cref{theo:MainResultFinal} is  an immediate  consequence of  \cref{theo:MainResult} and \cref{theo:MainResult2}.
\end{remark}

\begin{remark}
The derivation of the continuation criterion relies on Assumption \ref{B}, namely 
\[
\int_{I_*} \Vert \nabla\bv \Vert_{L^\infty(\Omega_{\eta})}^2 \dt  +  \int_{I_*} \Vert \nabla\rho \Vert_{L^\infty(\Omega_{\eta})}^2 \dt < \infty .
\] 
We point out, however, that the argument can also be carried out under the alternative condition 
\begin{enumerate}[label={$(B^\sharp)$}]

    \item \label{Bnew}
    (\textbf{Material acceleration condition})
 \[     \int_{I_*} \Vert \rho \dot{\bv} \Vert_{W^{1,2}(\Omega_\eta)}^2  \dt < \infty . \]   
     
\end{enumerate}
Indeed, by \cref{lem:MRlowerOrderVp}, a  maximal regularity argument yields
\begin{equation}\label{eq:MRstep1}
\begin{aligned}
\int_{I_*} \left( \Vert \bv \Vert_{W^{3,2}(\Omega_\eta)}^2 +  \Vert p \Vert_{W^{2,2}(\Omega_\eta)}^2  \right) \dt \lesssim  \int_{I_*} \Vert \rho\dot{\bv} \Vert_{W^{1,2}(\Omega_\eta)}^2 \dt + \int_{I_*} \Vert \partial_t\eta \Vert_{W^{5/2,2}(\omega)}^2 \dt. 
\end{aligned}
\end{equation}
While \eqref{eq:MRstep1}  implies, through Sobolev embedding,  that 
\[
\int_{I_*} \Vert \nabla\bv \Vert_{L^\infty(\Omega_{\eta})}^2 \dt < \infty,
\]
the corresponding control of the density, that is,
\[
 \int_{I_*} \Vert \nabla\rho \Vert_{L^\infty(\Omega_{\eta})}^2 \dt < \infty
\]
is only recovered a posteriori. Nevertheless, the estimate in   \cref{lem:VelocityEstim} can still be closed,  as the $W^{2,2}-$control of the density provided by    \eqref{eq:MRstep1} is sufficient.

\end{remark}


\section{Weak--Strong Uniqueness } \label{sec:WeakStrong}


The aim of this section is to establish the weak-strong uniqueness property for the coupled fluid--structure system \eqref{eq:ContMomentEq}--\eqref{eq:interfaceCond}. More precisely, we prove that, given the same initial data,  a finite-energy weak solution $(\rho_1, \bv_1, \eta_1)$ in the sense of  \cite[Section 2, Definition 2.1]{macha2022existence}, satisfying \eqref{eq:BasicEnergy} must coincide with a strong solution $(\rho_2, \bv_2, \eta_2)$ satisfying \eqref{eq:HigherEstimate}. The overall proof strategy follows the relative entropy framework developed in \cite{feireisl2012relative} and its extension to fluid--structure interaction problems in  \cite[Section 2]{trifunovic2023compressible}, with appropriate modifications to account for the present model.  
Since both solutions are defined on different moving domains, we follow the approach of \cite{schwarzacher2022weak} and  map the strong solution  $(\rho_2, \bv_2, \eta_2)$ onto the  domain $\Omega_{\eta_1}$ of the weak solution. To this end, we introduce the Hanzawa map   
\[
\bfPsi_{\eta_2 \to \eta_1 } :=   \bfPsi_{\eta_2}\circ \left( \bfPsi_{\eta_1 }\right)^{-1} \colon \Omega_{\eta_1} \to \Omega_{\eta_2} ,
\]
and define 
\begin{equation}
\bv^{\sharp}_2 = \bv_2 \circ \bfPsi_{\eta_2 \to \eta_1 }, \quad  \rho^{\sharp}_2 = \rho_2 \circ \bfPsi_{\eta_2 \to \eta_1 } . 
\end{equation}   
We then compare $(\rho_1, \bv_1, \eta_1)$ with the triple $(\rho^{\sharp}_2, \bv^{\sharp}_2, \eta_2)$. 
Accordingly,  for  $t \in I$, we define the relative entropy with respect to $(\rho^{\sharp}_2, \bv^{\sharp}_2, \eta_2)$  as 
\begin{equation}\label{eq:EntropyRelativeInitial}
\begin{aligned}
{\mathlarger{\EuScript{E}} }_{\mathrm{rel}}\! \left( (\rho_1, \bv_1, \eta_1) \bm{\Big|} (\rho^{\sharp}_2, \bv^{\sharp}_2, \eta_2) \right) (t) & :=  \dfrac{1}{2}  \int_{\Omega_{\eta_1(t)}} \rho_1 |\bv_1 - \bv^{\sharp}_2|^2 \dx  +  \int_{\Omega_{\eta_1(t)}}  {\mathlarger{\mathcal{H}}}\!\left(\rho_1 \bm{\big|}\rho^\sharp_2\right) \dx  
\\[0.4em]
& \quad + \dfrac{1}{2} \int_{\omega} \bigl(  |\partial_t \eta_1 - \partial_t \eta_2|^2  + |\Dely \eta_1 - \Dely \eta_2|^2 \bigr) \dy ,  
\end{aligned}
\end{equation}
where 
\[
{\mathlarger{\mathcal{H}}}\!\left(\rho_1 \bm{\big|}\rho^\sharp_2\right) :=  H(\rho_1) - H(\rho^\sharp_2) - H'(\rho_2^\sharp)(\rho_1 - \rho^\sharp_2).
\]
The following theorem shows that the relative entropy  \eqref{eq:EntropyRelativeInitial} satisfies a suitable inequality, which forms the basis of the weak-strong uniqueness  argument.

\begin{theorem}
Let  $(\rho_1, \bv_1, \eta_1)$ be a finite-energy weak solution of \eqref{eq:ContMomentEq}--\eqref{eq:interfaceCond} with initial data $(\rho_1^0, \bv_1^0, \eta_1^0, \eta_{*, 1})$  in the sense of  \cite[Section 2, Definition 2.1]{macha2022existence}. Moreover, let $(\rho_2, \bv_2, \eta_2)$ be a strong solution of \eqref{eq:ContMomentEq}--\eqref{eq:interfaceCond} with initial data  $(\rho_2^0, \bv_2^0, \eta_2^0, \eta_{*, 2})$ in the sense of \cref{def:StrongSol}.  Suppose, in addition,  that 
\begin{equation}
\eta_1 \in L^\infty\left(I, C^1(\omega) \right) .
\end{equation}
Then the following relative energy inequality holds:
\begin{equation}\label{eq:RelativEnergyWeakL}
\begin{aligned}
&{\mathlarger{\EuScript{E}} }_{\mathrm{rel}}\! \left( (\rho_1, \bv_1, \eta_1) \bm{\Big|} (\rho^{\sharp}_2, \bv^{\sharp}_2, \eta_2) \right)\!(t) + \int_0^t  \int_{\Omega_{\eta_1(t)}}  \mathbb{S}\left(\nabla\bv_1 - \nabla\bv^\sharp_2\right)\colon \left(\nabla\bv_1 - \nabla\bv^\sharp_2\right)  \dx \ds 
\\
& +   \int_0^t \int_\omega \bigl|\partial_t\naby\eta_1 - \partial_t\naby\eta_2\bigr|^2 \dy\ds  
\\
&\leq  {\mathlarger{\EuScript{E}} }_{\mathrm{rel}}\! \left( (\rho_1, \bv_1, \eta_1) \bm{\Big|} (\rho^{\sharp}_2, \bv^{\sharp}_2, \eta_2) \right)\!(0)   + \int_0^t  {\mathlarger{\EuScript{R}}}\!\left(\rho_1, \bv_1, \eta_1 \bm{\big|} \rho^{\sharp}_2, \bv^{\sharp}_2, \eta_2 \right)   \ds ,
\end{aligned}
\end{equation}
where 
\begin{align*}
{\mathlarger{\EuScript{R}}}\!\left(\rho_1, \bv_1, \eta_1 \bm{\big|} \rho^{\sharp}_2, \bv^{\sharp}_2, \eta_2 \right) & :=  \int_{\Omega_{\eta_1(t)}} \!\!\!\!\! \mathbb{S}\left(\nabla\bv^\sharp_2\right)\!\colon \! \left(\nabla\bv^\sharp_2 - \nabla\bv_1\!\right)  \dx  +  \int_{\Omega_{\eta_1(t)}} \!\!\!\!\! \rho_1 \left(\partial_s \bv^\sharp_2 + \bv_1\cdot \nabla\bv^\sharp_2\right) \cdot \left(\!\bv^\sharp_2 - \bv_1\!\right) \dx 
\\[0.2em]
& \quad + \dfrac{a\gamma}{\gamma -1}  \int_{\Omega_{\eta_1(t)}}  \left[  \left(\rho^\sharp_2\bv^\sharp_2 - \rho_1\bv_1\right)\cdot \nabla\left(\rho^\sharp_2\right)^{\gamma -1}  -  \left(\rho_1 - \rho^\sharp_2\right)\partial_s \left(\rho_2^\sharp\right)^{\gamma -1} \right] \dx
\\[0.2em]
&\quad +   \int_{\Omega_{\eta_1(t)}} \left(p(\rho^\sharp_2)  - p(\rho_1)\right) \Div\!\left(\bv^\sharp_2\right) \dx  
\\[0.2em]
&\quad - \int_{\omega} \!(\partial_s\eta_2 - \partial_s\eta_1)\bn\cdot \!\left(p(\rho^\sharp_2)\bn_{\eta_1}\!\right)\!\circ \bm{\varphi}_{\eta_1}  \mathrm{det}(\naby\bm{\varphi}_{\eta_1}) \dy   + \int_{\omega}\! (\partial_s\eta_2 - \partial_s\eta_1)\Dely^2\eta_2 \dy
\\[0.2em]
&\quad  + \int_{\omega} (\partial_s\eta_2 - \partial_s\eta_1)\partial_s^2\eta_2 \dy    -  \int_{\omega} (\partial_s\eta_2 - \partial_s\eta_1)\partial_t \Dely\eta_2 \dy .
\end{align*}

\end{theorem}

\begin{proof}
The proof closely mirrors that of \cite[Section 3.2]{trifunovic2023compressible}, with a slight modification accounting for the more general  geometric deformation considered here. For completeness, we present the argument in detail.  

Since 
\(\eta_1 \in W^{1,\infty}\left(I; L^{2}(\omega)\right) \cap L^\infty\left(I, W^{2,2}(\omega)\right) \),  
the Hanzawa map $\bfPsi_{\eta_2 \to \eta_1 }$ is a  $W^{2,2}-$diffeomorphism for a.e. $t \in I$. Together with the regularity of $(\rho_2, \bv_2, \eta_2)$,  it follows that $(\rho^{\sharp}_2, \bv^{\sharp}_2, \eta_2)$  lies in the closure of the admissible smooth test class for the weak solution $(\rho_1, \bv_1, \eta_1)$.  
Moreover, by construction, 
\begin{equation}\label{eq:PullbackKinematicCond}
\bv^{\sharp}_2\circ \bm{\varphi}_{\eta_1} = \bv_2\circ\bm{\varphi}_{\eta_2} =  (\partial_t\eta_2) \bn ,
\end{equation}
so that $(\bv^{\sharp}_2, \partial_t\eta_2)$ satisfies the compatibility conditions required for admissible test functions in the weak formulation.   Therefore, by a density argument, we may use $(\bv^{\sharp}_2, \partial_t\eta_2), \; |\bv^\sharp_2|^2$, and $\left(\rho^\sharp_2\right)^{\gamma -1}$ as test functions in the weak formulation. \\
Thus, testing the coupled momentum-structure equation $\eqref{eq:ContMomentEq}_2-$\eqref{eq:ShellEq} with  $(\bv^{\sharp}_2, \partial_s\eta_2)$, and using Reynolds' transport theorem, we derive that 
\begin{equation}\label{eq:TestFunc1}
\begin{aligned}
&\int_{\Omega_{\eta_1(t)}} \left(\rho_1 \bv_1\cdot \bv^\sharp_2 \right)(t) \dx + \int_\omega \left(\partial_t\eta_1 \partial_t\eta_2\right) (t) \dy
\\[0.4em]
& = \int_{\Omega_{\eta_1(0)}} \left(\rho_1 \bv_1\cdot \bv^\sharp_2 \right)(0) \dx + \int_\omega \left(\partial_t\eta_1 \partial_t\eta_2\right) (0) \dy +  \int_0^t \int_{\Omega_{\eta_1(t)}} \left(\rho_1 \bv_1\cdot \partial_s \bv^\sharp_2\right) \dx\ds
\\[0.4em]
&\quad +  \int_0^t \int_{\Omega_{\eta_1(t)}} \!\!\! \left( \rho_1 \bv_1\otimes\bv_1 \right) \colon \nabla\bv^\sharp_2  \dx\ds   +  \int_0^t \int_{\Omega_{\eta_1(t)}} \!\!\!  p(\rho_1)\Div\bv^\sharp_2 \dx\ds   -  \int_0^t \!\int_{\Omega_{\eta_1(t)}} \!\!\!\!\!  \mathbb{S}(\nabla\bv_1)\colon \nabla\bv^\sharp_2  \dx\ds
\\[0.4em]
&\quad + \int_0^t \int_\omega \partial_s\eta_1\partial_s^2 \eta_2 \dy\ds -  \int_0^t \int_\omega \Dely\eta_1 \partial_s\Dely\eta_2 \dy\ds -  \int_0^t \int_\omega \partial_s\naby\eta_1\cdot\partial_s\naby\eta_2 \dy\ds .
\end{aligned}
\end{equation}
Moreover, by choosing $\dfrac{1}{2}|\bv^\sharp_2|^2$  as test function for the continuity equation $\eqref{eq:ContMomentEq}_1$, we obtain 
\begin{equation}\label{eq:TestFunc2}
\begin{aligned}
\dfrac{1}{2} \int_{\Omega_{\eta_1(t)}} \left(\rho_1 |\bv^\sharp_2|^2\right) (t) \dx & = \dfrac{1}{2} \int_{\Omega_{\eta_1(0)}} \left(\rho_1 |\bv^\sharp_2|^2\right) (0) \dx  + \int_0^t \int_{\Omega_{\eta_1(t)}} \rho_1 \partial_s \bv^\sharp_2 \cdot \bv^\sharp_2 \dx\ds 
\\[0.4em]
&\quad + \int_0^t \int_{\Omega_{\eta_1(t)}} \rho_1 \left( \bv_1\cdot \nabla\bv^\sharp_2 \right)\cdot \bv^\sharp_2 \dx\ds .
\end{aligned}
\end{equation}
We deduce from \eqref{eq:TestFunc2} that, for the choice of  $\dfrac{a\gamma}{\gamma -1} \left(\rho^\sharp_2\right)^{\gamma-1} $ as test function, it holds that 
\begin{equation}\label{eq:TestFunc3}
\begin{aligned}
\dfrac{a\gamma}{\gamma-1} \int_{\Omega_{\eta_1(t)}} \!\! \left(\rho_1\left(\rho^\sharp_2\right)^{\gamma-1} \right)(t) \dx & =  \dfrac{a\gamma}{\gamma-1} \Bigg[ \int_{\Omega_{\eta_1(0)}} \! \left(\rho_1\left(\rho^\sharp_2\right)^{\gamma-1} \right)(0) \dx 
\\[0.4em]
&\quad + \int_0^t\int_{\Omega_{\eta_1(t)}} \left( \rho_1 \partial_s \!\left(\rho^\sharp_2\right)^{\gamma-1} +  \rho_1\bv_1 \cdot \nabla\!\left(\rho^\sharp_2\right)^{\gamma-1} \right) \!\dx\ds  \Bigg] .
\end{aligned}
\end{equation} 
However, observe that 
\begin{equation*}
\begin{aligned}
\dfrac{a\gamma}{\gamma-1} \int_0^t \int_{\Omega_{\eta_1(t)}} \!\! \rho^\sharp_2\bv^\sharp_2\cdot \nabla \!\left(\rho^\sharp_2\right)^{\gamma-1}  \!\! \!\dx\ds  & =  \int_0^t \int_{\Omega_{\eta_1(t)}}  \bv^\sharp_2 \nabla p(\rho^\sharp_2) \dx\ds 
\\[0.4em]
& = -  \int_0^t \int_{\Omega_{\eta_1(t)}} \!\! p(\rho^\sharp_2) \Div\bv^\sharp_2 \dx\ds + \int_0^t \int_{\partial\Omega_{\eta_1(t)}} \!\!p(\rho^\sharp_2) \bv^\sharp_2 \cdot \bn_{\eta_1} \dH\ds ,
\end{aligned} 
\end{equation*}
which after using the kinematic boundary condition \eqref{eq:PullbackKinematicCond}, 
further reduces to 
\begin{equation}\label{eq:DensityPullTest1}
\begin{aligned}
\dfrac{a\gamma}{\gamma-1} \int_0^t \int_{\Omega_{\eta_1(t)}} \rho^\sharp_2\bv^\sharp_2\cdot \nabla \!\left(\rho^\sharp_2\right)^{\gamma-1}  \!\! \!\dx\ds  
& = -  \int_0^t \int_{\Omega_{\eta_1(t)}}  p(\rho^\sharp_2) \Div\bv^\sharp_2 \dx\ds 
\\[0.4em]
&\quad + \int_0^t \int_\omega  \partial_s\eta_2\, \bn \cdot \left( p(\rho^\sharp_2)\bn_{\eta_1}\right)\circ \bm{\varphi}_{\eta_1}  \mathrm{det}(\naby\bm{\varphi}_{\eta_1}) \dy\ds .
\end{aligned} 
\end{equation}
Furthermore,  the Reynolds' transport theorem yields 
\begin{equation*}
\begin{aligned}
\int_0^t \dfrac{\dd}{\ds} \int_{\Omega_{\eta_1(t)}} p(\rho^\sharp_2) \dx\ds = \int_0^t \int_{\Omega_{\eta_1(t)}} \partial_s p(\rho^\sharp_2) \dx\ds   + \int_0^t \int_{\partial\Omega_{\eta_1(t)}} p(\rho^\sharp_2) \bv_1 \cdot \bn_{\eta_1} \dH\ds ,
\end{aligned}
\end{equation*}
which together with the identity 
\[
a\gamma \left(\rho^\sharp_2\right)^{\gamma-1} \partial_s  \rho^\sharp_2  = \partial_s p(\rho^\sharp_2)  =a \rho^\sharp_2\partial_s \left(\rho^\sharp_2\right)^{\gamma-1} + a\left(\rho^\sharp_2\right)^{\gamma-1} \partial_s  \rho^\sharp_2 ,
\]
imply that 
\begin{equation}\label{eq:DensityPullTest2}
\begin{aligned}
 \int_{\Omega_{\eta_1(t)}} p(\rho^\sharp_2) (t) \dx & = \int_{\Omega_{\eta_1(0)}} p(\rho^\sharp_2) (0) \dx +  \dfrac{a\gamma}{\gamma -1}  \int_0^t \int_{\Omega_{\eta_1(t)}}  \rho^\sharp_2\partial_s \left(\rho^\sharp_2\right)^{\gamma-1}  \dx\ds 
 \\[0.4em]
 & \quad +  \int_0^t \int_\omega  \partial_s\eta_1\,\bn\cdot \left(p(\rho^\sharp_2)\bn_{\eta_1}\right)\circ \bm{\varphi}_{\eta_1}  \mathrm{det}(\naby\bm{\varphi}_{\eta_1})\dy\ds .
\end{aligned}
\end{equation}
Combining \eqref{eq:DensityPullTest1} and \eqref{eq:DensityPullTest2}, we deduce that  
\begin{equation}\label{eq:TestFunc4}
\begin{aligned}
 \int_{\Omega_{\eta_1(t)}} \! p(\rho^\sharp_2) (t) \dx & = \int_{\Omega_{\eta_1(0)}} \!p(\rho^\sharp_2) (0) \dx 
 \\[0.4em]
 & \quad + \int_0^t \int_\omega  \left( \partial_s\eta_1 - \partial_s\eta_2 \right)\bn\cdot \left(p(\rho^\sharp_2)\bn_{\eta_1}\!\right) \circ \bm{\varphi}_{\eta_1}  \mathrm{det}(\naby\bm{\varphi}_{\eta_1}) \dy\ds 
 \\[0.4em]
 & \quad +    \dfrac{a\gamma}{\gamma -1}  \int_0^t \int_{\Omega_{\eta_1(t)}} \left( \rho^\sharp_2\bv^\sharp_2\cdot \nabla \left(\rho^\sharp_2\right)^{\gamma-1}  +   \rho^\sharp_2\partial_s \left(\rho^\sharp_2\right)^{\gamma-1}  \right) \dx\ds 
 \\[0.4em]
 & \quad + \int_0^t \int_{\Omega_{\eta_1(t)}}  p(\rho^\sharp_2) \Div\bv^\sharp_2 \dx\ds .
\end{aligned}
\end{equation}
Importantly,  the last three terms of \eqref{eq:TestFunc1} can be rewritten so as to isolate some relevant energy  terms.\\
Indeed, for the first term, it holds that 
\begin{equation}
\int_0^t \int_\omega \partial_s\eta_1\partial_s^2 \eta_2 \dy\ds = \int_0^t \int_\omega \left(\partial_s\eta_1 - \partial_s\eta_2\right)\partial_s^2 \eta_2 \dy\ds + \dfrac{1}{2} \int_\omega |\partial_t\eta_2|^2 (t) \dy -  \dfrac{1}{2} \int_\omega |\partial_t\eta_2|^2 (0) \dy.
\end{equation}
Considering the bending term, we get that 
\begin{equation*}
\begin{aligned}
 -  \int_0^t \int_\omega \Dely\eta_1 \partial_s\Dely\eta_2 \dy\ds & = -  \int_0^t \int_\omega \eta_1 \partial_s\Dely^2\eta_2 \dy\ds - \dfrac{1}{2} \int_0^t \dfrac{\dd}{\ds} \int_\omega |\Dely\eta_2|^2 \dy\ds 
 \\
 &\quad + \dfrac{1}{2} \int_0^t \dfrac{\dd}{\ds} \int_\omega |\Dely\eta_2|^2 \dy\ds
 \\[0.4em]
 & = \int_0^t \int_\omega \partial_s\eta_1 \Dely^2\eta_2 \dy\ds  -  \int_0^t \dfrac{\dd}{\ds} \int_\omega \eta_1\Dely^2\eta_2 \dy\ds 
 \\[0.4em]
 &\quad - \int_0^t \int_\omega \partial_s\Dely\eta_2\Dely\eta_2 \dy\ds + \dfrac{1}{2} \int_\omega |\Dely\eta_2|^2 (t)\dy  -  \dfrac{1}{2} \int_\omega |\Dely\eta_2|^2 (0)\dy .
\end{aligned}
\end{equation*}
That is,
 \begin{equation}
\begin{aligned}
-  \int_0^t \int_\omega \Dely\eta_1 \partial_s\Dely\eta_2 \dy\ds & = \int_0^t\int_\omega \left( \partial_s\eta_1 - \partial_s\eta_2 \right) \Dely^2\eta_2 \dy\ds   - \int_\omega  \left(\Dely\eta_1\Dely\eta_2 \right)(t)\dy 
\\[0.4em]
&\quad +  \int_\omega \! \left(\Dely\eta_1\Dely\eta_2\right) \!(0)\dy  
\\[0.4em]
&\quad +  \dfrac{1}{2} \int_\omega |\Dely\eta_2|^2 (t)\dy  -  \dfrac{1}{2} \int_\omega \!|\Dely\eta_2|^2 (0)\dy .
\end{aligned}
\end{equation}
Finally,   for the viscous damping term,  we derive that  
\begin{equation*}
\begin{aligned}
 -  \int_0^t \int_\omega \partial_s\naby\eta_1\cdot\partial_s\naby\eta_2 \dy\ds & = \int_0^t \int_\omega \partial_s\Dely\eta_1\partial_s\eta_2 \dy\ds  +  \int_0^t \int_\omega  \partial_s\Dely\eta_2 \partial_s\eta_2   \dy\ds 
 \\[0.4em]
 &\quad +   \int_0^t \int_\omega |\partial_s\naby\eta_2|^2  \dy\ds
 \\[0.4em]
 &=   \int_0^t \int_\omega \partial_s\Dely\eta_2 (\partial_s\eta_2 -\partial_s\eta_1)\dy\ds + 2\int_0^t \int_\omega \partial_s\Dely\eta_1\partial_s\eta_2 \dy\ds 
 \\[0.4em]
 &\quad +   \int_0^t \int_\omega |\partial_s\naby\eta_2|^2  \dy\ds. 
\end{aligned}
\end{equation*}
That is,
 \begin{equation}
\begin{aligned}
-  \int_0^t \int_\omega \partial_s\naby\eta_1\cdot\partial_s\naby\eta_2 \dy\ds & =    \int_0^t \int_\omega |\partial_s\naby\eta_2|^2  \dy\ds +  \int_0^t \int_\omega \partial_s\Dely\eta_2 (\partial_s\eta_2 -\partial_s\eta_1)\dy\ds 
\\[0.4em]
&\quad - 2\int_0^t \int_\omega \partial_s\naby\eta_1\cdot \partial_s\naby\eta_2 \dy\ds .
\end{aligned}
\end{equation}
Moreover, the finite-energy weak solution $(\rho_1, \bv_1, \eta_1)$ satisfies \eqref{eq:BasicEnergy}, that is, 
\begin{equation}\label{eq:BasicE}
\begin{aligned}
&\dfrac{1}{2} \int_{\Omega_{\eta_1(t)}} \left(\rho_1|\bv_1|^2\right) (t)\dx +  \int_{\Omega_{\eta_1(t)}}  H(\rho_1)(t) \dx +  \int_0^t \int_{\Omega_{\eta_1(t)}}  \mathbb{S}(\nabla\bv_1)\colon \nabla\bv_1 \dx\ds 
\\[0.4em]
&\quad +  \dfrac{1}{2} \int_\omega |\partial_t\eta_1 (t)|^2\dy +  \dfrac{1}{2} \int_\omega |\Dely\eta_1(t)|^2\dy + \int_0^t \int_\omega |\partial_s\naby\eta_1|^2\dy\ds
\\[0.4em]
& \leq \dfrac{1}{2} \int_{\Omega_{\eta_1^0}} \rho_1^0 \left(\bv_1^0\right)^2  \dx + \int_{\Omega_{\eta_1^0}} H(\rho_1^0) \dx + \dfrac{1}{2} \int_\omega |\eta_{1,*}|^2 \dy  + \dfrac{1}{2} \int_\omega |\Delta\eta_1^0|^2 \dy.
\end{aligned}
\end{equation}
Therefore, summing  \eqref{eq:BasicE} - \eqref{eq:TestFunc1} + \eqref{eq:TestFunc2} - \eqref{eq:TestFunc3} + \eqref{eq:TestFunc4}, we get the desired result.

\end{proof}

Although the relative energy inequality  \eqref{eq:RelativEnergyWeakL}   allows for a direct comparison with  the transported strong solution $(\rho^{\sharp}_2, \bv^{\sharp}_2, \eta_2)$, it is not suitable for weak-strong uniqueness. Indeed, the triple $(\rho^{\sharp}_2, \bv^{\sharp}_2, \eta_2)$ no longer retains the regularity required to qualify as a strong solution in the sense of  \cref{def:StrongSol}.  Moreover, both  $(\rho^{\sharp}_2, \bv^{\sharp}_2, \eta_2)$ and $(\rho_1, \bv_1, \eta_1)$  do not satisfy the same system of equations   on $\Omega_{\eta_1}$.  For these reasons, we transfer both solutions as well as the relative energy inequality  \eqref{eq:RelativEnergyWeakL} to the fixed reference domain $\Omega$.
Let
\[
\underline{\rho}_{\mathsf{k}} := \rho_\mathsf{k} \circ \bfPsi_{\eta_\mathsf{k}} , \quad  \underline{\bv}_\mathsf{k} := \bv_\mathsf{k}\circ \bfPsi_{\eta_\mathsf{k}}, \quad J_{\eta_{\mathsf{k}}} := \mathrm{det}(\nabla\bfPsi_{\eta_{\mathsf{k}}})  \qquad \forall\; \mathsf{k} \in \{1,2\}.
\]
In order to simplify the subsequent analysis, we introduce a compact notation for derivatives in the reference configuration.
Let  $\mathtt{u}$ be a sufficiently smooth scalar or vector  field defined on $I\times\Omega_{\eta_{\mathsf{k}}}$, and denote its pullback by  $\mathtt{u}_{\mathsf{k}} = \mathtt{u}\circ\bfPsi_{\eta_\mathsf{k}}$. We define the corresponding  spatial derivatives  expressed in the reference configuration by (for all $\mathsf{j}, \mathsf{k} \in \{1,2\}$)
\begin{equation*}
\begin{aligned}
\nabla_\mathtt{k}\mathtt{u}_{\mathsf{j}} & := \nabla \mathtt{u}_{\mathsf{j}}\left(\nabla  \bfPsi_{\eta_\mathsf{k}} \right)^{-1}\circ  \bfPsi_{\eta_\mathsf{k}},
\\[0.4em]
\Div_\mathtt{k}\mathtt{u}_{\mathsf{j}} = \nabla_\mathtt{k}\cdot \mathtt{u}_{\mathsf{j}} &:= \mathrm{tr}\left( \nabla_\mathtt{k}\mathtt{u}_{\mathsf{j}} \right) . 
\end{aligned}
\end{equation*}
Moreover, the chain rule yields 
\[
\partial_t \mathtt{u}_{\mathsf{k}} = \left(\partial_t \mathtt{u}\right)\circ \bfPsi_{\eta_\mathsf{k}} + \partial_t \bfPsi_{\eta_\mathsf{k}} \cdot \nabla_\mathtt{k}\mathtt{u}_{\mathsf{k}} .
\]
Using the above notation, the relative energy inequality \eqref{eq:RelativEnergyWeakL}  can be rewritten on the fixed reference domain $\Omega$ through a standard change of variables as follows  

\begin{equation}\label{eq:RelativEnergyFixedD}
\begin{aligned}
&{\mathlarger{\EuScript{E}} }_{\mathrm{rel}}\! \left( \big(\,\underline{\rho}_1, \underline{\bv}_1, \eta_1\big) \bm{\Big|} \big(\,\underline{\rho}_2, \underline{\bv}_2, \eta_2\big) \right)\!(t)  + {\mathlarger{\mathcal{D}}} \!\left(\underline{\bv}_1, \eta_1 \bm{\big|} \underline{\bv}_2, \eta_2  \right)\!(t) 
\\[0.4em]
&\leq  {\mathlarger{\EuScript{E}} }_{\mathrm{rel}}\! \left( \big(\,\underline{\rho}_1, \underline{\bv}_1, \eta_1\big) \bm{\Big|} \big(\,\underline{\rho}_2, \underline{\bv}_2, \eta_2\big) \right)\!(0)   + \int_0^t  {\mathlarger{\EuScript{R}}}\!\left(\underline{\rho}_1, \underline{\bv}_1, \eta_1 \bm{\big|}\underline{\rho}_2, \underline{\bv}_2, \eta_2 \right)   \ds ,
\end{aligned}
\end{equation}
where the relative entropy on the fixed domain  is given by 
\begin{equation}\label{eq:EntropyRelativeFixed}
\begin{aligned}
{\mathlarger{\EuScript{E}} }_{\mathrm{rel}}\! \left( \big(\,\underline{\rho}_1, \underline{\bv}_1, \eta_1\big) \bm{\Big|} \big(\,\underline{\rho}_2, \underline{\bv}_2, \eta_2\big)  \right) (t) & :=  \dfrac{1}{2}  \int_{\Omega} \underline{\rho}_1 |\underline{\bv}_1 - \underline{\bv}_2|^2  J_{\eta_1}\dx  +  \int_{\Omega} J_{\eta_1} {\mathlarger{\mathcal{H}}}\!\left(\underline{\rho}_1 \bm{\big|}\underline{\rho}_2\right) \dx  
\\[0.4em]
& \quad + \dfrac{1}{2} \int_{\omega} \left(  |\partial_t \eta_1 - \partial_t \eta_2|^2  + |\Dely \eta_1 - \Dely \eta_2|^2 \right) \dy ,  
\end{aligned}
\end{equation}
the dissipation term 
\begin{equation}
\begin{aligned}
{\mathlarger{\mathcal{D}}} \left(\underline{\bv}_1, \eta_1 \bm{\big|} \underline{\bv}_2, \eta_2  \right)(t) & :=  \int_0^t  \int_{\Omega} J_{\eta_{1}}  \mathbb{S}\left(\nabla_1\underline{\bv}_1 - \nabla_1\underline{\bv}_2\right)\colon \left(\nabla_1\underline{\bv}_1 - \nabla_1\underline{\bv}_2\right)  \dx \ds 
\\[0.4em]
&\quad  +  \int_0^t \int_\omega \bigl|\partial_t\naby\eta_1 - \partial_t\naby\eta_2\bigr|^2 \dy\ds  ,
\end{aligned}
\end{equation}
and 
\begin{align}
{\mathlarger{\EuScript{R}}}\!\left(\underline{\rho}_1, \underline{\bv}_1, \eta_1 \bm{\big|}\underline{\rho}_2, \underline{\bv}_2, \eta_2 \right) & :=  \int_{\Omega} J_{\eta_1} \mathbb{S}\left(\nabla_1\underline{\bv}_2\right) \colon  \left(\nabla_1\underline{\bv}_2 - \nabla_1\underline{\bv}_1 \right)  \dx     \nonumber
\\[0.4em]
&\quad + \int_{\Omega} J_{\eta_1} \underline{\rho}_1 \left(\partial_s \underline{\bv}_2 + \underline{\bv}_1\cdot \nabla_1\underline{\bv}_2\right) \cdot \left(\underline{\bv}_2 - \underline{\bv}_1\right) \dx     \nonumber
\\[0.4em]
&\quad - \int_{\Omega} J_{\eta_1} \underline{\rho}_1 \left( \partial_s \bfPsi_{\eta_1} \cdot \nabla_1\underline{\bv}_2\right) \cdot \left(\underline{\bv}_2 - \underline{\bv}_1\right) \dx    \nonumber
\\[0.4em]
&\quad +  \dfrac{a\gamma}{\gamma -1}  \int_\Omega \bigg[  J_{\eta_1}\left( \underline{\rho}_2\underline{\bv}_2 - \underline{\rho}_1\underline{\bv}_1 \right)\cdot \nabla_1(\,\underline{\rho}_2^{\gamma -1})    \nonumber
\\[0.4em]
&\quad\quad \qquad \qquad \;\;\;\;    +  J_{\eta_1} (\,\underline{\rho}_2 - \underline{\rho}_1 ) \Big( \partial_s(\,\underline{\rho}_2^{\gamma -1}) -  \partial_s \bfPsi_{\eta_1} \cdot \nabla_1(\,\underline{\rho}_2^{\gamma -1} )  \Big)  \bigg] \dx   \label{eq:RemainderFix}
\\[0.4em]
&\quad +   \int_\Omega J_{\eta_1} \left( p(\,\underline{\rho}_2) - p(\,\underline{\rho}_1)  \right) \nabla_1\cdot \underline{\bv}_2 \dx  \nonumber
\\[0.4em]
&\quad  - \int_{\omega} (\partial_s\eta_2 - \partial_s\eta_1) p(\,\underline{\rho}_2) \bn\cdot\bn_{\eta_1}\circ \bm{\varphi}_{\eta_1}  \mathrm{det}(\naby\bm{\varphi}_{\eta_1}) \dy    \nonumber
\\[0.4em]
&\quad + \int_{\omega} (\partial_s\eta_2 - \partial_s\eta_1)\Dely^2\eta_2 \dy + \int_{\omega} (\partial_s\eta_2 - \partial_s\eta_1)\partial_s^2\eta_2 \dy    \nonumber
\\[0.4em]
&\quad    -  \int_{\omega} (\partial_s\eta_2 - \partial_s\eta_1)\partial_t \Dely\eta_2 \dy .  \nonumber
\end{align}
Having introduced the necessary framework and notation, we are now ready to prove \cref{theo:WeakStrongUniqueMain}. This is the purpose of the next section.


\subsection{Proof of \cref{theo:WeakStrongUniqueMain}}

\phantom{Just to make it start on the next line}\\
For brevity, we henceforth suppress the dependence on the arguments $\big(\,\underline{\rho}_1, \underline{\bv}_1, \eta_1\big)$ and $ \big(\,\underline{\rho}_2, \underline{\bv}_2, \eta_2\big)$, and simply write ${\mathlarger{\EuScript{E}} }_{\mathrm{rel}} (t)$, ${\mathlarger{\mathcal{D}}}(t)$ and ${\mathlarger{\EuScript{R}}} (t)$  for the relative entropy, dissipation,  and remainder terms, respectively. \\
Recall that 
\begin{equation}
{\mathlarger{\EuScript{E}} }_{\mathrm{rel}}(t)  + {\mathlarger{\mathcal{D}}}(t) 
\leq  {\mathlarger{\EuScript{E}} }_{\mathrm{rel}}(0)   + \int_0^t  {\mathlarger{\EuScript{R}}}(s)  \ds  . \tag{\ref{eq:RelativEnergyFixedD}}
\end{equation}
Therefore, it suffices to show that for a.e. $t \in I$, 
\begin{equation}\label{eq:RemainderEstimFinal}
 \int_0^t  {\mathlarger{\EuScript{R}}}(s)  \ds  \leq C(\kappa) \int_0^t \mathtt{G}(s)  {\mathlarger{\EuScript{E}} }_{\mathrm{rel}}(s)  \ds + \kappa {\mathlarger{\mathcal{D}}}(t) , 
\end{equation}
with $\mathtt{G} \in L^1(I)$ and arbitrary $\kappa > 0$.  This would imply uniqueness by Gr\"onwall's inequality, provided the initial data coincide. \\
In the spirit of \cite[Section 4.2]{trifunovic2023compressible}, we split the proof of   \eqref{eq:RemainderEstimFinal} into several steps for clarity.

\medskip

\noindent {\textbf{Step 1: Reformulation of the remainder} ${\mathlarger{\EuScript{R}}}$.  } \\
We start by recalling that the pullback $\big(\,\underline{\rho}_2, \underline{\bv}_2, \eta_2\big)$ of the strong solution satisfies 
\begin{equation}\label{eq:StrongSystFix}
(\mathtt{S})\left\{\begin{aligned}
&\partial_t \underline{\rho}_2 - \partial_t  \bfPsi_{\eta_2} \cdot \nabla_2\underline{\rho}_2  + \nabla_2\cdot \left(\,\underline{\rho}_2\underline{\bv}_2 \right)  = 0 \quad \text{ in } I\times \Omega ,
\\[0.4em]
&\partial_t \underline{\bv}_2  - \partial_t  \bfPsi_{\eta_2} \cdot \nabla_2\underline{\bv}_2 + \underline{\bv}_2 \cdot \nabla_2\underline{\bv}_2 =  -\dfrac{1}{\underline{\rho}_2} \nabla_2 p(\, \underline{\rho}_2)  +  \dfrac{1}{\underline{\rho}_2} \nabla_2\cdot \mathbb{S}\left( \nabla_2\underline{\bv}_2 \right)   \quad \text{ in } I\times \Omega ,
\\[0.4em]
& \partial_t^2\eta_2 - \partial_t\Dely\eta_2 + \Dely^2\eta_2 = -\bn^\intercal \bigg[ \left( \mathbb{S}\left( \nabla_2\underline{\bv}_2 \right) - p(\, \underline{\rho}_2)\mathbb{I}_{3\times3}  \right) \bigg] \bn_{\eta_2} \circ \bm{\varphi}_{\eta_2}  \mathrm{det}(\naby\bm{\varphi}_{\eta_2}) \quad \text{on } I \times \omega .
\end{aligned}
\right.
\end{equation}
Next, we decompose the remainder ${\mathlarger{\EuScript{R}}}$ as 
\[
{\mathlarger{\EuScript{R}}}(s) = {\mathlarger{\EuScript{R}}}_{\mathtt{visc}}(s) + {\mathlarger{\EuScript{R}}}_{\mathtt{mom}}(s)  + {\mathlarger{\EuScript{R}}}_{\mathtt{geo}}(s)  +{\mathlarger{\EuScript{R}}}_{\mathtt{press}}(s)  +  {\mathlarger{\EuScript{R}}}_{\mathtt{shell}}(s), 
\]
where 
\begin{subequations}
\begin{equation}\label{eq:Rvisc}
 {\mathlarger{\EuScript{R}}}_{\mathtt{visc}}(s) =  \int_{\Omega} J_{\eta_1} \mathbb{S}\left(\nabla_1\underline{\bv}_2\right) \colon  \left(\nabla_1\underline{\bv}_2 - \nabla_1\underline{\bv}_1 \right)  \dx  ,
\end{equation}
\begin{equation}\label{eq:Rmom}
 {\mathlarger{\EuScript{R}}}_{\mathtt{mom}}(s) =  \int_{\Omega} J_{\eta_1} \underline{\rho}_1 \left(\partial_s \underline{\bv}_2 + \underline{\bv}_1\cdot \nabla_1\underline{\bv}_2\right) \cdot \left(\underline{\bv}_2 - \underline{\bv}_1\right) \dx   ,
\end{equation}
\begin{equation}
 {\mathlarger{\EuScript{R}}}_{\mathtt{geo}}(s) =   - \int_{\Omega} J_{\eta_1} \underline{\rho}_1 \left( \partial_s \bfPsi_{\eta_1} \cdot \nabla_1\underline{\bv}_2\right) \cdot \left(\underline{\bv}_2 - \underline{\bv}_1\right) \dx  ,
\end{equation}
\begin{equation}\label{eq:Rpress}
\begin{aligned}
 {\mathlarger{\EuScript{R}}}_{\mathtt{press}}(s) & =   \dfrac{a\gamma}{\gamma -1}  \int_\Omega \Bigg[  J_{\eta_1}\left( \underline{\rho}_2\underline{\bv}_2 - \underline{\rho}_1\underline{\bv}_1 \right)\cdot \nabla_1(\,\underline{\rho}_2^{\gamma -1})
\\[0.4em]
&\quad\quad \qquad \qquad \;\;\;\;    +  J_{\eta_1} (\,\underline{\rho}_2 - \underline{\rho}_1 ) \Big( \partial_s(\,\underline{\rho}_2^{\gamma -1}) -  \partial_s \bfPsi_{\eta_1} \cdot \nabla_1(\,\underline{\rho}_2^{\gamma -1})  \Big)  \Bigg] \dx  
\\[0.4em]
&\quad + \int_\Omega J_{\eta_1} \left( p(\,\underline{\rho}_2) - p(\,\underline{\rho}_1)  \right) \nabla_1\cdot \underline{\bv}_2 \dx ,
\end{aligned}
\end{equation}
and 
\begin{equation}\label{eq:Rshell}
\begin{aligned}
 {\mathlarger{\EuScript{R}}}_{\mathtt{shell}}(s) & =     - \int_{\omega} (\partial_s\eta_2 - \partial_s\eta_1) p(\,\underline{\rho}_2) \bn\cdot\bn_{\eta_1}\circ \bm{\varphi}_{\eta_1}  \mathrm{det}(\naby\bm{\varphi}_{\eta_1}) \dy  + \int_{\omega} (\partial_s\eta_2 - \partial_s\eta_1)\Dely^2\eta_2 \dy
\\[0.4em]
&\quad  + \int_{\omega} (\partial_s\eta_2 - \partial_s\eta_1)\partial_s^2\eta_2 \dy   -  \int_{\omega} (\partial_s\eta_2 - \partial_s\eta_1)\partial_s \Dely\eta_2 \dy .
\end{aligned}
\end{equation}
\end{subequations}
Considering the momentum part, it holds that
\begin{equation*}
\begin{aligned}
\int_0^t {\mathlarger{\EuScript{R}}}_{\mathtt{mom}}(s)  \ds & =  \int_0^t \int_{\Omega} J_{\eta_1} \underline{\rho}_1 \left(\partial_s \underline{\bv}_2  -  \partial_s \bfPsi_{\eta_2} \cdot \nabla_2\underline{\bv}_2 + \underline{\bv}_2\cdot \nabla_2\underline{\bv}_2\right) \cdot \left(\underline{\bv}_2 - \underline{\bv}_1\right) \dx\ds
\\[0.4em]
&\quad +  \int_0^t \int_{\Omega} J_{\eta_1} \underline{\rho}_1 \Big[ \underline{\bv}_2\cdot \left( \nabla_1 - \nabla_2\right)\underline{\bv}_2\Big] \cdot \left(\underline{\bv}_2 - \underline{\bv}_1\right) \dx\ds
\\[0.4em]
&\quad +  \int_0^t \int_{\Omega} J_{\eta_1} \underline{\rho}_1 \Big[\left( \underline{\bv}_1 -  \underline{\bv}_2\right) \cdot \nabla_1\underline{\bv}_2\Big] \cdot \left(\underline{\bv}_2 - \underline{\bv}_1\right) \dx\ds
\\[0.4em]
&\quad + \int_0^t \int_{\Omega} J_{\eta_1} \underline{\rho}_1 \left( \partial_s \bfPsi_{\eta_2} \cdot \nabla_2\underline{\bv}_2 \right)\cdot \left(\underline{\bv}_2 - \underline{\bv}_1\right) \dx\ds . 
\end{aligned}
\end{equation*}
Using the momentum equation $\eqref{eq:StrongSystFix}_2$, we further obtain 
\begin{equation}\label{eq:RmomIntegr}
\begin{aligned}
&\int_0^t {\mathlarger{\EuScript{R}}}_{\mathtt{mom}}(s)  \ds 
\\
& =  - \int_0^t \int_{\Omega}  \dfrac{J_{\eta_1} \underline{\rho}_1}{\underline{\rho}_2} \nabla_2 p(\, \underline{\rho}_2) \cdot \left(\underline{\bv}_2 - \underline{\bv}_1\right) \dx\ds  + \int_0^t \int_{\Omega} \dfrac{J_{\eta_1} \underline{\rho}_1}{\underline{\rho}_2} \Big[ \nabla_2\cdot \mathbb{S}(\nabla_2\underline{\bv}_2) \Big]\cdot  \left(\underline{\bv}_2 - \underline{\bv}_1\right)         \dx\ds
\\[0.4em]
&\quad +  \int_0^t \int_{\Omega} J_{\eta_1} \underline{\rho}_1 \Big[ \underline{\bv}_2\cdot \left( \nabla_1 - \nabla_2\right)\underline{\bv}_2\Big] \cdot \left(\underline{\bv}_2 - \underline{\bv}_1\right) \dx\ds 
\\[0.4em]
&\quad +  \int_0^t \int_{\Omega} J_{\eta_1} \underline{\rho}_1 \Big[\left( \underline{\bv}_1 -  \underline{\bv}_2\right) \cdot \nabla_1\underline{\bv}_2\Big] \cdot \left(\underline{\bv}_2 - \underline{\bv}_1\right) \dx\ds
\\[0.4em]
&\quad + \int_0^t \int_{\Omega} J_{\eta_1} \underline{\rho}_1 \left( \partial_s \bfPsi_{\eta_2} \cdot \nabla_2\underline{\bv}_2 \right)\cdot \left(\underline{\bv}_2 - \underline{\bv}_1\right) \dx\ds .
\end{aligned}
\end{equation}
However, observe that the second term of  \eqref{eq:RmomIntegr}, that is,
\begin{equation*}
\begin{aligned}
\int_0^t {\mathlarger{\EuScript{R}}}_{\mathtt{mom}, 2}(s)  \ds & :=  \int_0^t \int_{\Omega} \dfrac{J_{\eta_1} \underline{\rho}_1}{\underline{\rho}_2} \Big[ \nabla_2\cdot \mathbb{S}(\nabla_2\underline{\bv}_2) \Big]\cdot  \left(\underline{\bv}_2 - \underline{\bv}_1\right)         \dx\ds
\\[0.4em]
& = \int_0^t \int_{\Omega} J_{\eta_1} \dfrac{ 1}{\underline{\rho}_2} (\, \underline{\rho}_1 - \underline{\rho}_2) \Big[ \nabla_2\cdot \mathbb{S}(\nabla_2\underline{\bv}_2) \Big]\cdot  \left(\underline{\bv}_2 - \underline{\bv}_1\right)         \dx\ds
\\[0.4em]
&\quad +  \int_0^t \int_{\Omega} \left(J_{\eta_1} - J_{\eta_2} \right) \Big[ \nabla_2\cdot \mathbb{S}(\nabla_2\underline{\bv}_2) \Big]\cdot  \left(\underline{\bv}_2 - \underline{\bv}_1\right)         \dx\ds
\\[0.4em]
&\quad +   \int_0^t \int_{\Omega} J_{\eta_2} \Big[ \nabla_2\cdot \mathbb{S}(\nabla_2\underline{\bv}_2) \Big]\cdot  \left(\underline{\bv}_2 - \underline{\bv}_1\right)         \dx\ds . 
\end{aligned}
\end{equation*}
Integrating the last term of the above identity  by parts, that is, 
\begin{align*}
  \int_0^t \int_{\Omega} J_{\eta_2} \Big[ \nabla_2\cdot \mathbb{S}(\nabla_2\underline{\bv}_2) \Big]\cdot  \left(\underline{\bv}_2 - \underline{\bv}_1\right)         \dx\ds & = \int_0^t \int_{\Omega_{\eta_2}(t)}  \Div\left(\mathbb{S}(\nabla\bv_2)\right) \cdot \left( \bv_2 - \underline{\bv}_1\circ \bfPsi_{\eta_2}^{-1}  \right) \dx\ds  
  \\[0.4em]
  & = - \int_0^t \int_{\Omega_{\eta_2}(t)} \mathbb{S}(\nabla\bv_2) \colon \Big( \nabla\bv_2 - \nabla\left(\underline{\bv}_1\circ \bfPsi_{\eta_2}^{-1} \right)  \Big) \dx\ds  
\\[0.2em]  
& \quad +   \int_0^t \int_{\partial\Omega_{\eta_2}(t)} \left( \mathbb{S}(\nabla\bv_2)\bn_{\eta_2}\right) \cdot \left( \bv_2 - \underline{\bv}_1\circ \bm{\varphi}_{\eta_2}^{-1}  \right)  \dH\ds
\\[0.2em]
& = -  \int_0^t \int_{\Omega} J_{\eta_2}  \mathbb{S}(\nabla_2\underline{\bv}_2) \colon \left(\nabla_2\underline{\bv}_2 - \nabla_2\underline{\bv}_1  \right)    \dx\ds 
\\[0.2em]
& \quad + \int_0^t \!\int_\omega \! \Big[\mathbb{S}(\nabla\bv_2)\bn_{\eta_2}\!\Big]\!\circ\! \bm{\varphi}_{\eta_2}\! \cdot \bn (\partial_t\eta_2 - \partial_t\eta_1) \mathrm{det}(\naby\bm{\varphi}_{\eta_2})  \! \dy\ds ,
\end{align*}
we deduce that 
\begin{equation}\label{eq:RmomIntegr2T}
\begin{aligned}
& \int_0^t {\mathlarger{\EuScript{R}}}_{\mathtt{mom}, 2}(s)  \ds 
\\[0.4em]
& =  \int_0^t \!\! \int_\Omega J_{\eta_1} \dfrac{ 1}{\underline{\rho}_2} (\, \underline{\rho}_1 - \underline{\rho}_2) \Big[ \nabla_2\cdot \mathbb{S}(\nabla_2\underline{\bv}_2) \Big]\cdot  \left(\underline{\bv}_2 - \underline{\bv}_1\right)         \dx\ds
\\[0.4em]
&\quad +  \int_0^t  \int_\Omega  \bigg(\! \left(J_{\eta_1} - J_{\eta_2} \right) \Big[ \nabla_2\cdot \mathbb{S}(\nabla_2\underline{\bv}_2) \Big]\cdot  \left(\underline{\bv}_2 - \underline{\bv}_1\right)         -   J_{\eta_2}  \mathbb{S}(\nabla_2\underline{\bv}_2) \colon \left(\nabla_2\underline{\bv}_2 - \nabla_2\underline{\bv}_1  \right) \!\bigg)   \dx\ds 
\\[0.4em]
&\quad +    \int_0^t  \int_\Omega \Big[\mathbb{S}(\nabla\bv_2)\bn_{\eta_2}\!\Big]\!\circ\! \bm{\varphi}_{\eta_2}\! \cdot  \!(\partial_t\eta_2 - \partial_t\eta_1)\bn\, \mathrm{det}(\naby\bm{\varphi}_{\eta_2})   \dy\ds. 
\end{aligned}
\end{equation}
Adding  the first term of \eqref{eq:Rpress} and the first term of \eqref{eq:RmomIntegr}, we  derive that 
\begin{align*}
&\int_0^t {\mathlarger{\EuScript{R}}}_{\mathtt{press}, 1{\scalebox{0.35}{+}}}(s) \ds
\\[0.4em]
&:= \dfrac{a\gamma}{\gamma -1}  \int_0^t \!\! \int_\Omega \bigg[  J_{\eta_1}\left( \underline{\rho}_2\underline{\bv}_2 - \underline{\rho}_1\underline{\bv}_1 \right)\cdot \nabla_1(\,\underline{\rho}_2^{\gamma -1})
  +  J_{\eta_1} (\,\underline{\rho}_2 - \underline{\rho}_1 ) \Big( \partial_s(\,\underline{\rho}_2^{\gamma -1} ) -  \partial_s \bfPsi_{\eta_1} \cdot \nabla_1(\,\underline{\rho}_2^{\gamma -1} )  \!\Big)  \bigg] \dx  \ds
\\[0.4em]
&\quad -   \int_0^t \!\! \int_\Omega  \dfrac{J_{\eta_1} \underline{\rho}_1}{\underline{\rho}_2} \nabla_2 p(\, \underline{\rho}_2) \cdot \left(\underline{\bv}_2 - \underline{\bv}_1\right) \dx\ds 
\\[0.4em]
& = \dfrac{a\gamma}{\gamma -1}  \int_0^t \!\! \int_\Omega J_{\eta_1}(\, \underline{\rho}_2 - \underline{\rho}_1) \Big[ \underline{\bv}_2 \cdot \nabla_1(\,\underline{\rho}_2^{\gamma -1})
  + \partial_s(\,\underline{\rho}_2^{\gamma -1} ) -  \partial_s \bfPsi_{\eta_1} \cdot \nabla_1(\,\underline{\rho}_2^{\gamma -1})   \Big] \dx \ds
  \\[0.4em]
 &\quad +  \dfrac{a\gamma}{\gamma -1} \int_0^t \!\! \int_\Omega  J_{\eta_1} \underline{\rho}_1 (\nabla_1 - \nabla_2) \underline{\rho}_2^{\gamma -1} \cdot (\underline{\bv}_2 - \underline{\bv}_1) \dx\ds
 \\[0.4em]
 &\quad + \dfrac{a\gamma}{\gamma -1}  \int_0^t \!\! \int_\Omega  J_{\eta_1} (\,\underline{\rho}_2 - \underline{\rho}_1 ) \left( \underline{\bv}_2 \cdot \nabla_2(\,\underline{\rho}_2^{\gamma -1})  -   \partial_s \bfPsi_{\eta_2} \cdot \nabla_2(\,\underline{\rho}_2^{\gamma -1})  \right)\dx\ds
 \\[0.4em]
 &\quad - \dfrac{a\gamma}{\gamma -1}  \int_0^t \!\! \int_\Omega  J_{\eta_1} (\,\underline{\rho}_2 - \underline{\rho}_1 ) \left( \underline{\bv}_2 \cdot \nabla_2(\,\underline{\rho}_2^{\gamma -1})  -   \partial_s \bfPsi_{\eta_2} \cdot \nabla_2(\,\underline{\rho}_2^{\gamma -1})  \right)\dx\ds 
\end{align*}
where we have used the fact that 
\[
\dfrac{1}{\underline{\rho}_2}\nabla_2 p(\,\underline{\rho}_2 ) =  \dfrac{a\gamma}{\gamma -1}  \nabla_2\left(\underline{\rho}_2^{\gamma -1}\right) .
\]
Moreover, using the renormalised continuity equation
\begin{equation}
\partial_t\left(\underline{\rho}_2^{\gamma -1}\right)  + \underline{\bv}_2\cdot \nabla_2\left(\underline{\rho}_2^{\gamma -1}\right)  - \partial_t \bfPsi_{\eta_2} \cdot \nabla_2\left(\underline{\rho}_2^{\gamma -1}\right) = -(\gamma -1) \left( \nabla_2\cdot \underline{\bv}_2\right)\underline{\rho}_2^{\gamma -1} ,
\end{equation}
it follows that 
\begin{align}\label{RpressIntegr1}
&\int_0^t {\mathlarger{\EuScript{R}}}_{\mathtt{press}, 1{\scalebox{0.35}{+}}}(s) \ds
\\[0.4em]
& = -a\gamma \int_0^t  \int_\Omega J_{\eta_1} (\,\underline{\rho}_2 - \underline{\rho}_1 ) \left( \nabla_2\cdot \underline{\bv}_2\right)\underline{\rho}_2^{\gamma -1} \dx\ds   \nonumber
\\[0.4em]
&\quad + \dfrac{a\gamma}{\gamma -1} \int_0^t  \int_\Omega  J_{\eta_1} (\,\underline{\rho}_2 - \underline{\rho}_1 ) \bigg[ \underline{\bv}_2\cdot (\nabla_1 - \nabla_2) \underline{\rho}_2^{\gamma -1} - \left(  \partial_s \bfPsi_{\eta_1} -  \partial_s \bfPsi_{\eta_2} \right)\cdot \nabla_1(\,\underline{\rho}_2^{\gamma -1})   \nonumber
\\[0.4em]
&\qquad \qquad\qquad\qquad\qquad \qquad\qquad \qquad\qquad\qquad \qquad\qquad \qquad -  \partial_s \bfPsi_{\eta_2} \cdot  (\nabla_1 - \nabla_2) \underline{\rho}_2^{\gamma -1} \bigg] \dx\ds
\\[0.4em]
&\quad +  \dfrac{a\gamma}{\gamma -1} \int_0^t  \int_\Omega   J_{\eta_1} \underline{\rho}_1 (\nabla_1 - \nabla_2) \underline{\rho}_2^{\gamma -1}  \cdot ( \underline{\bv}_2 -  \underline{\bv}_1) \dx\ds.  \nonumber
\end{align}
Combining the second term  of \eqref{eq:Rpress} with the first term of \eqref{RpressIntegr1}, we get 
\begin{equation}\label{RpressIntegr2}
\begin{aligned}
&\int_0^t {\mathlarger{\EuScript{R}}}_{\mathtt{press}, 2{\scalebox{0.35}{+}}}(s) \ds
\\[0.4em]
& :=   \int_0^t \int_\Omega J_{\eta_1} \left( p(\,\underline{\rho}_2) - p(\,\underline{\rho}_1)  \right) \nabla_1\cdot \underline{\bv}_2 \dx\ds   -a\gamma \int_0^t  \int_\Omega J_{\eta_1} (\,\underline{\rho}_2 - \underline{\rho}_1 ) \left( \nabla_2\cdot \underline{\bv}_2\right)\underline{\rho}_2^{\gamma -1} \dx\ds
\\[0.4em]
&= - \int_0^t \int_\Omega   J_{\eta_1} {\mathlarger{\mathcal{H}}}\!\left(\underline{\rho}_1 \bm{\big|}\underline{\rho}_2\right)  \nabla_1\cdot \underline{\bv}_2   \dx\ds    + a\gamma \int_0^t  \int_\Omega J_{\eta_1} (\,\underline{\rho}_2 - \underline{\rho}_1 )\underline{\rho}_2^{\gamma -1} ( \nabla_1 - \nabla_2) \cdot \underline{\bv}_2 \dx\ds. 
\end{aligned}
\end{equation}
Finally, testing the shell equation $\eqref{eq:StrongSystFix}_3$ with $(\partial_s\eta_2 - \partial_s\eta_1)$ yields 
\begin{equation}\label{eq:RshellIntegr}
\begin{aligned}
&\int_0^t \int_{\omega} \Big( \partial_s^2\eta_2  -  \partial_s \Dely\eta_2 +  \Dely^2\eta_2  \Big) (\partial_s\eta_2 - \partial_s\eta_1) \dy\ds
\\[0.3em]
& =  \int_0^t  \int_{\omega} (\partial_s\eta_2 - \partial_s\eta_1) p(\,\underline{\rho}_2) \bn\cdot\bn_{\eta_2}\circ \bm{\varphi}_{\eta_2}  \mathrm{det}(\naby\bm{\varphi}_{\eta_2}) \dy  \ds
\\[0.3em]
&\quad - \int_0^t \int_\omega  \Big[\mathbb{S}(\nabla\bv_2)\bn_{\eta_2}\!\Big]\!\circ\! \bm{\varphi}_{\eta_2}\! \cdot  \!(\partial_s\eta_2 - \partial_s\eta_1)\bn\, \mathrm{det}(\naby\bm{\varphi}_{\eta_2})   \dy\ds  . 
\end{aligned}
\end{equation}
Hence, it follows from \eqref{eq:Rvisc}--\eqref{eq:Rshell}, \eqref{eq:RmomIntegr2T},  and \eqref{RpressIntegr1}--\eqref{eq:RshellIntegr}, that 
\begin{align}
& \int_0^t  {\mathlarger{\EuScript{R}}}(s)  \ds    \nonumber
\\[0.2em]
& =   \int_0^t \int_{\Omega} J_{\eta_1} \mathbb{S}\left(\nabla_1\underline{\bv}_2\right) \colon  \left(\nabla_1\underline{\bv}_2 - \nabla_1\underline{\bv}_1 \right)  \dx\ds  -  \int_0^t  \int_{\Omega} J_{\eta_2} \mathbb{S}\left(\nabla_2\underline{\bv}_2\right) \colon  \left(\nabla_2\underline{\bv}_2 - \nabla_2\underline{\bv}_1 \right)  \dx\ds   \nonumber
\\[0.2em]
&\quad + \int_0^t \int_{\Omega} J_{\eta_1} \dfrac{ 1}{\underline{\rho}_2} (\, \underline{\rho}_1 - \underline{\rho}_2) \Big[ \nabla_2\cdot \mathbb{S}(\nabla_2\underline{\bv}_2) \Big]\cdot  \left(\underline{\bv}_2 - \underline{\bv}_1\right)         \dx\ds   \nonumber
\\[0.2em]
&\quad + \int_0^t \int_{\Omega} \left(J_{\eta_1} - J_{\eta_2} \right) \Big[ \nabla_2\cdot \mathbb{S}(\nabla_2\underline{\bv}_2) \Big]\cdot  \left(\underline{\bv}_2 - \underline{\bv}_1\right)         \dx\ds    \nonumber
\\[0.2em]
&\quad + \dfrac{a\gamma}{\gamma -1} \int_0^t  \int_\Omega  J_{\eta_1} (\,\underline{\rho}_2 - \underline{\rho}_1 ) \bigg[ \underline{\bv}_2\cdot (\nabla_1 - \nabla_2) \underline{\rho}_2^{\gamma -1} - \left(  \partial_s \bfPsi_{\eta_1} -  \partial_s \bfPsi_{\eta_2} \right)\cdot \nabla_1(\,\underline{\rho}_2^{\gamma -1})   \nonumber
\\[0.2em]
&\qquad \qquad\qquad\qquad\qquad \qquad\qquad \qquad\qquad\qquad \qquad\qquad \qquad -  \partial_s \bfPsi_{\eta_2} \cdot  (\nabla_1 - \nabla_2) \underline{\rho}_2^{\gamma -1} \bigg] \dx\ds    
\\[0.2em]
&\quad +  \int_0^t \int_{\Omega} J_{\eta_1} \underline{\rho}_1 \Big[ \underline{\bv}_2\cdot \left( \nabla_1 - \nabla_2\right)\underline{\bv}_2\Big] \cdot \left(\underline{\bv}_2 - \underline{\bv}_1\right) \dx\ds     \nonumber
\\[0.2em]
&\quad +  \int_0^t \int_{\Omega} J_{\eta_1} \underline{\rho}_1 \Big[\left( \underline{\bv}_1 -  \underline{\bv}_2\right) \cdot \nabla_1\underline{\bv}_2\Big] \cdot \left(\underline{\bv}_2 - \underline{\bv}_1\right) \dx\ds   \nonumber
\\[0.2em]
&\quad +  \dfrac{a\gamma}{\gamma -1} \int_0^t  \int_\Omega   J_{\eta_1} \underline{\rho}_1 (\nabla_1 - \nabla_2) \underline{\rho}_2^{\gamma -1}  \cdot ( \underline{\bv}_2 -  \underline{\bv}_1) \dx\ds    \nonumber
\\[0.2em]
&\quad +   a\gamma \int_0^t  \int_\Omega J_{\eta_1} (\,\underline{\rho}_2 - \underline{\rho}_1 )\underline{\rho}_2^{\gamma -1} ( \nabla_1 - \nabla_2) \cdot \underline{\bv}_2 \dx\ds    -  \int_0^t \int_\Omega   J_{\eta_1} {\mathlarger{\mathcal{H}}}\!\left(\underline{\rho}_1 \bm{\big|}\underline{\rho}_2\right)  \nabla_1\cdot \underline{\bv}_2   \dx\ds    \nonumber
\\[0.2em]
&\quad + \int_0^t  \int_\Omega J_{\eta_1} \underline{\rho}_1 \Big[  \partial_s \bfPsi_{\eta_2} \cdot  \nabla_2\underline{\bv}_2  - \partial_t \bfPsi_{\eta_1} \cdot  \nabla_1\underline{\bv}_2  \Big] \cdot ( \underline{\bv}_2 -  \underline{\bv}_1) \dx\ds     \nonumber
\\[0.2em]
&\quad +  \int_0^t \int_\omega  (\partial_s\eta_2 - \partial_s\eta_1) p(\,\underline{\rho}_2) \bn\cdot  \big(  \bn_{\eta_2}\circ \bm{\varphi}_{\eta_2}  \mathrm{det}(\naby\bm{\varphi}_{\eta_2})  - \bn_{\eta_1}\circ \bm{\varphi}_{\eta_1}  \mathrm{det}(\naby\bm{\varphi}_{\eta_1}) \big) \dy\ds \nonumber
\\[0.2em]
& \quad =: \sum\limits_{\mathsf{k} =1}^{12}  {\mathlarger{\EuScript{T}}}_{\mathsf{k}}.   \nonumber
\end{align}
In the next step, we show that each $ {\mathlarger{\EuScript{T}}}_{\mathsf{k}}, \, \mathsf{k} \in \{1,\ldots, 11\}$  satisfies \eqref{eq:RemainderEstimFinal}. Before proceeding, we make the  following observations, which will be used repeatedly.

\begin{remark}
By the uniform ellipticity of the mapping
\[
\R^{3\times 3} \ni \mathbf{A} \mapsto \mathbb{S}(\mathbf{A}) ,
\]
there exist constants $c, C > 0$ such that 
\[
c |\mathbf{A}|^2 \leq \mathbb{S}(\mathbf{A}) \colon \mathbf{A} \leq C |\mathbf{A}|^2 , \quad \text{ and }\quad   c |\mathbf{A}|^2 \leq |\mathbb{S}(\mathbf{A})|^2 \leq C |\mathbf{A}|^2 .
\] 
In particular, setting  $\underline{\bv} = \underline{\bv}_1 - \underline{\bv}_2$, the quantities  
\[
\int_\Omega J_{\eta_1}\mathbb{S}(\nabla_1\underline{\bv})\colon \nabla_1\underline{\bv} \dx,  \qquad  \int_\Omega J_{\eta_1} \left|\mathbb{S}(\nabla_1\underline{\bv})\right|^2 \dx
\]
are equivalent up to multiplicative constants.
\end{remark}

\begin{remark}\label{rem:DensityRelativeEntrop}
Let $2 \wedge \gamma := \min\{2, \gamma\}$. Then
\begin{equation}\label{eq:DensityRelativeEntrop}
\Vert \underline{\rho}_1(t) - \underline{\rho}_2(t) \Vert_{L^{2\wedge \gamma}(\Omega)}^2 \leq C {\mathlarger{\EuScript{E}} }_{\mathrm{rel}}(t)  \qquad \forall\; t \in I.
\end{equation}
This estimate follows from the convexity of the pressure potential $H(\rho) = \dfrac{p(\rho)}{\gamma -1 }$, together with  the uniform lower and upper bounds on the density $\underline{\rho}_2$.  For a detailed proof, see e.g., \cite{trifunovic2023compressible, feireisl2012relative}.
\end{remark}

\medskip

\noindent {\textbf{Step 2: Estimate of } ${\mathlarger{\EuScript{T}}}_{\mathsf{k}}$.  } \\
We first consider the first two  terms, that is,
\begin{equation*}
\begin{aligned}
{\mathlarger{\EuScript{T}}}_1 + {\mathlarger{\EuScript{T}}}_2 & =   \int_0^t \int_{\Omega} J_{\eta_1} \mathbb{S}\left(\nabla_1\underline{\bv}_2\right) \colon  \left(\nabla_1\underline{\bv}_2 - \nabla_1\underline{\bv}_1 \right)  \dx \ds -   \int_0^t \int_{\Omega} J_{\eta_2} \mathbb{S}\left(\nabla_2\underline{\bv}_2\right) \colon  \left(\nabla_2\underline{\bv}_2 - \nabla_2\underline{\bv}_1 \right)  \dx\ds
\\[0.4em]
& = \int_0^t \int_{\Omega} \left(J_{\eta_1} - J_{\eta_2} \right) \mathbb{S}\left(\nabla_1\underline{\bv}_2\right) \colon  \left(\nabla_1\underline{\bv}_2 - \nabla_1\underline{\bv}_1 \right)  \dx \ds
\\[0.4em]
&\quad +  \int_0^t \int_{\Omega}  J_{\eta_2}  \mathbb{S}\big((\nabla_1 -\nabla_2 )\underline{\bv}_2\big) \colon  \left(\nabla_1\underline{\bv}_2 - \nabla_1\underline{\bv}_1 \right)  \dx \ds
\\[0.4em]
&\quad + \int_0^t \int_{\Omega} J_{\eta_2}  \mathbb{S}\left(\nabla_2\underline{\bv}_2\right) \colon  \left(\nabla_1 - \nabla_2 \right) (\underline{\bv}_2 - \underline{\bv}_1)  \dx \ds.
\end{aligned}
\end{equation*} 
Using H\"older's inequality, together with the uniform boundedness of $J_{\eta_2}$ and $\left(\nabla\bfPsi_{\eta_2} \right)^{-1}\circ \bfPsi_{\eta_2} $, we obtain  
\begin{equation}
\begin{aligned}
{\mathlarger{\EuScript{T}}}_1 + {\mathlarger{\EuScript{T}}}_2 & \lesssim  \int_0^t \Vert \eta_1 - \eta_2 \Vert_{W^{1,4}(\omega)} \Vert \nabla \underline{\bv}_2\Vert_{L^\infty(\Omega)}  \Vert \left( \nabla\bfPsi_{\eta_1} \right)^{-1}\circ \bfPsi_{\eta_1}   \Vert_{L^{24}(\Omega)}^2 \Vert \underline{\bv}_2 - \underline{\bv}_1\Vert_{W^{1, 3/2}(\Omega)} \ds
\\[0.4em]
& \quad +  \int_0^t \Vert \eta_1 - \eta_2 \Vert_{W^{1,4}(\omega)}  \Vert \nabla \underline{\bv}_2\Vert_{L^\infty(\Omega)}\Vert \underline{\bv}_2 - \underline{\bv}_1\Vert_{W^{1, 3/2}(\Omega)} \Vert \left( \nabla\bfPsi_{\eta_1} \right)^{-1}\circ \bfPsi_{\eta_1}   \Vert_{L^{12}(\Omega)}  \ds 
\\[0.4em]
&\quad +  \int_0^t  \! \Vert \nabla \underline{\bv}_2\Vert_{L^\infty(\Omega)}  \Vert\! \left( \nabla\bfPsi_{\eta_1} \right)^{-1}\circ \bfPsi_{\eta_1} - \left(\nabla\bfPsi_{\eta_2} \right)^{-1}\circ \bfPsi_{\eta_2}   \Vert_{L^{3}(\Omega)}  \Vert \underline{\bv}_2 - \underline{\bv}_1\Vert_{W^{1, 3/2}(\Omega)} \ds .
\end{aligned}
\end{equation}
Moreover, by the a priori bounds for the weak and strong solutions, notably  
\begin{subequations}
\begin{equation}
\sup\limits_{(0,t)} \Vert  \underline{\bv}_2 \Vert_{W^{3,2}(\Omega)}^2 \leq C  ,
\end{equation}
\begin{equation}\label{eq:Eta1control}
\sup\limits_{(0,t)}  \Vert \left( \nabla\bfPsi_{\eta_1} \right)^{-1}\circ \bfPsi_{\eta_1}   \Vert_{L^{24}(\Omega) } \lesssim 1 + \sup\limits_{(0,t)}  \Vert \eta_1 \Vert_{ W^{2,2}(\omega) }  \leq C ,
\end{equation}
\end{subequations}
we further obtain 
\begin{equation}\label{eq:priorT12Estim}
{\mathlarger{\EuScript{T}}}_1 + {\mathlarger{\EuScript{T}}}_2  \lesssim  \int_0^t \Vert \Dely\eta_1 - \Dely\eta_2 \Vert_{L^2(\omega)} \Vert \underline{\bv}_2 - \underline{\bv}_1\Vert_{W^{1, 3/2}(\Omega)}  \ds .
\end{equation}
However, observe that 
\begin{equation*}
\Vert \underline{\bv}_2 - \underline{\bv}_1\Vert_{W^{1, 3/2}(\Omega)}^2  \leq   \Vert \nabla\bfPsi_{\eta_1} \Vert_{L^{12}(\Omega) }^3  \Vert \nabla_1(\underline{\bv}_2 - \underline{\bv}_1)\Vert_{L^2(\Omega)}^2  \lesssim  \Vert \eta_1\Vert_{W^{2,2}(\omega)}^3 \int_\Omega J_{\eta_1} |\mathbb{S}\left(\nabla_1(\underline{\bv}_2 - \underline{\bv}_1)\right)|^2 \dx.
\end{equation*}
Therefore, by Young's inequality and \eqref{eq:Eta1control}, we deduce from \eqref{eq:priorT12Estim} that 
\begin{equation}\label{eq:T1T2Estim}
{\mathlarger{\EuScript{T}}}_1 + {\mathlarger{\EuScript{T}}}_2  \lesssim  C(\kappa) \int_0^t {\mathlarger{\EuScript{E}} }_{\mathrm{rel}}(s) \ds + \kappa  \int_0^t \int_\Omega J_{\eta_1} |\mathbb{S}\left(\nabla_1(\underline{\bv}_2 - \underline{\bv}_1)\right)|^2 \dx \ds ,
\end{equation}
for arbitrary $\kappa >0$. 

We next consider 
\begin{equation*}
\begin{aligned}
{\mathlarger{\EuScript{T}}}_3 & =     \int_0^t \int_{\Omega} J_{\eta_1} \dfrac{ 1}{\underline{\rho}_2} (\, \underline{\rho}_1 - \underline{\rho}_2) \Big[ \nabla_2\cdot \mathbb{S}(\nabla_2\underline{\bv}_2) \Big]\cdot  \left(\underline{\bv}_2 - \underline{\bv}_1\right)         \dx\ds  
\\[0.4em]
&\lesssim    \int_0^t  \Vert J_{\eta_1} \Vert_{L^{12}(\Omega)}  \left\Vert \dfrac{1}{J_{\eta_2}} \right\Vert_{L^{\infty}(\Omega)}  \Vert \underline{\rho}_2 - \underline{\rho}_1  \Vert_{L^{2\wedge\gamma}(\Omega)}  \Vert \nabla_2\cdot \mathbb{S}(\nabla_2\underline{\bv}_2) \Vert_{L^\mathtt{p}(\Omega)} \Vert \underline{\bv}_2 - \underline{\bv}_1 \Vert_{L^3(\Omega)}   \ds  .
\end{aligned}
\end{equation*} 
Since 
\[
\sup\limits_{(0, t)}  \Vert J_{\eta_1} \Vert_{L^{12}(\Omega)} \lesssim   1 + \sup\limits_{(0, t)} \Vert \Dely\eta_1 \Vert_{L^2(\omega)} \leq C , \quad \text{and } \;\;  \mathtt{p} > \max\left\{  12, \dfrac{12\gamma}{7\left(\gamma - 12/7 \right)}  \right\} ,
\]
Young's inequality, together with \cref{rem:DensityRelativeEntrop}, yields
\begin{equation}\label{eq:T3Estim}
{\mathlarger{\EuScript{T}}}_3 \leq   C(\kappa) \int_0^t  \Vert \nabla_2\cdot \mathbb{S}(\nabla_2\underline{\bv}_2) \Vert_{L^\mathtt{p}(\Omega)}^2 {\mathlarger{\EuScript{E}} }_{\mathrm{rel}}(s)   \ds  + \kappa  \int_0^t  \Vert \underline{\bv}_2 - \underline{\bv}_1 \Vert_{L^3(\Omega)}^2  \ds.
\end{equation} \\
We now show  that 
\begin{equation}\label{eq:DivL1integr}
\Vert \nabla_2\cdot \mathbb{S}(\nabla_2\underline{\bv}_2) \Vert_{L^\mathtt{p}(\Omega)}^2 \in L^1\left( (0, t) \right).
\end{equation}
Indeed, we have 
\begin{equation*}
\begin{aligned}
\Vert \nabla_2\cdot \mathbb{S}(\nabla_2\underline{\bv}_2) \Vert_{L^\mathtt{p}(\Omega)}  & \lesssim  \Vert \nabla^2 \underline{\bv}_2 \Vert_{L^{\mathtt{p}}(\Omega)}  \Vert \nabla\bfPsi_{\eta_2}^{-1}\circ \bfPsi_{\eta_2}   \Vert_{L^{\infty}(\Omega)}^2  
\\[0.4em]
&\quad +  \Vert \nabla \underline{\bv}_2 \Vert_{L^{\infty}(\Omega)}  \Vert \nabla\left( \nabla\bfPsi_{\eta_2}^{-1}\circ \bfPsi_{\eta_2}  \right) \Vert_{L^{\infty}(\Omega)}\Vert \nabla\bfPsi_{\eta_2}^{-1}\circ \bfPsi_{\eta_2}   \Vert_{L^{\infty}(\Omega)}.
\end{aligned}
\end{equation*}
By  Sobolev embeddings, and  the properties of the Hanzawa transform, this yields 
\begin{equation*}
\begin{aligned}
\Vert \nabla_2\cdot \mathbb{S}(\nabla_2\underline{\bv}_2) \Vert_{L^\mathtt{p}(\Omega)}  & \lesssim  \Vert\underline{\bv}_2 \Vert_{W^{4,2}(\Omega)} \left( 1 +  \Vert \eta_2   \Vert_{W^{3, 2}(\omega)} \right)^2  
\\[0.4em]
&\quad +  \Vert\underline{\bv}_2 \Vert_{W^{3,2}(\Omega)} \Vert \eta_2   \Vert_{W^{4, 2}(\omega)} \ \left( 1 +  \Vert \eta_2   \Vert_{W^{3, 2}(\omega)} \right).
\end{aligned}
\end{equation*}
Consequently, the regularity of the strong solution  yields the desired result. 

For the term  
\[
{\mathlarger{\EuScript{T}}}_4 =  \int_0^t \int_{\Omega} \left(J_{\eta_1} - J_{\eta_2} \right) \Big[ \nabla_2\cdot \mathbb{S}(\nabla_2\underline{\bv}_2) \Big]\cdot  \left(\underline{\bv}_2 - \underline{\bv}_1\right)         \dx\ds ,
\]
it follows immediately from H\"older's inequality that 
\begin{equation*}
{\mathlarger{\EuScript{T}}}_4 \lesssim \int_0^t \Vert \eta_1 - \eta_2 \Vert_{W^{1,3}(\omega)} \Vert \nabla_2\cdot \mathbb{S}(\nabla_2\underline{\bv}_2) \Vert_{L^\infty(\Omega)}   \Vert \underline{\bv}_2 - \underline{\bv}_1 \Vert_{L^{3/2}(\Omega)}      \ds .
\end{equation*}
Applying Young's inequality yields
\begin{equation}\label{eq:T4Estim}
{\mathlarger{\EuScript{T}}}_4 \lesssim C(\kappa) \int_0^t  \Vert \nabla_2\cdot \mathbb{S}(\nabla_2\underline{\bv}_2) \Vert_{L^\infty(\Omega)}^2 {\mathlarger{\EuScript{E}} }_{\mathrm{rel}}(s)   \ds  + \kappa  \int_0^t  \Vert \underline{\bv}_2 - \underline{\bv}_1 \Vert_{L^{3/2}(\Omega)}^2  \ds .
\end{equation}
Importantly, it follows from \eqref{eq:DivL1integr}, that  
\[
\Vert \nabla_2\cdot \mathbb{S}(\nabla_2\underline{\bv}_2) \Vert_{L^\infty(\Omega)}^2 \in L^1\left( (0, t) \right).
\]

Considering the term 
\begin{equation*}
\begin{aligned}
{\mathlarger{\EuScript{T}}}_5 & =   \dfrac{a\gamma}{\gamma -1} \int_0^t  \int_\Omega  J_{\eta_1} (\,\underline{\rho}_2 - \underline{\rho}_1 ) \bigg[ \underline{\bv}_2\cdot (\nabla_1 - \nabla_2) \underline{\rho}_2^{\gamma -1} - \left(  \partial_s \bfPsi_{\eta_1} -  \partial_s \bfPsi_{\eta_2} \right)\cdot \nabla_1(\,\underline{\rho}_2^{\gamma -1}) 
\\[0.4em]
&\qquad \qquad\qquad\qquad\qquad \qquad\qquad \qquad\qquad\qquad \qquad\qquad \qquad -  \partial_s \bfPsi_{\eta_2} \cdot  (\nabla_1 - \nabla_2) \underline{\rho}_2^{\gamma -1} \bigg] \dx\ds,
\end{aligned}
\end{equation*}
we apply once again H\"older's inequality to derive that 
\begin{equation*}
\begin{aligned}
{\mathlarger{\EuScript{T}}}_5 & \lesssim \int_0^t \Vert J_{\eta_1} \Vert_{L^{12}(\Omega)} \Vert \underline{\rho}_2 -  \underline{\rho}_1  \Vert_{L^{2\wedge\gamma}(\Omega)} \bigg[   \Vert\! \left( \nabla\bfPsi_{\eta_1} \right)^{-1}\circ \bfPsi_{\eta_1} - \left(\nabla\bfPsi_{\eta_2} \right)^{-1}\circ \bfPsi_{\eta_2}   \Vert_{L^{3}(\Omega)} \Vert \underline{\bv}_2 \cdot \nabla(\, \underline{\rho}_2^{\gamma -1}) \Vert_{L^\mathtt{p}(\Omega)}
\\[0.4em]
& \qquad \qquad  + \Vert \partial_s \eta_1 - \partial_s \eta_2 \Vert_{L^3(\omega)}  \Vert \nabla(\, \underline{\rho}_2^{\gamma -1}) \Vert_{L^\mathtt{p}(\Omega)}   
\\[0.4em]
& \quad \qquad \qquad \qquad \qquad +   \Vert\! \left( \nabla\bfPsi_{\eta_1} \right)^{-1}\circ \bfPsi_{\eta_1} - \left(\nabla\bfPsi_{\eta_2} \right)^{-1}\circ \bfPsi_{\eta_2}   \Vert_{L^{3}(\Omega)} \Vert \partial_s\bfPsi_{\eta_2} \cdot \nabla(\, \underline{\rho}_2^{\gamma -1}) \Vert_{L^\mathtt{p}(\Omega)}  \bigg] \ds.
\end{aligned}
\end{equation*}
By Young's inequality, we deduce that 
\begin{equation}\label{eq:T5Estim}
\begin{aligned}
{\mathlarger{\EuScript{T}}}_5 & \lesssim C(\kappa)\! \int_0^t \!\Big(   \Vert \underline{\bv}_2\! \cdot \nabla(\, \underline{\rho}_2^{\gamma -1}) \Vert_{L^\mathtt{p}(\Omega)}^2 + \Vert \nabla(\, \underline{\rho}_2^{\gamma -1}) \Vert_{L^\mathtt{p}(\Omega)}^2 +  \Vert \partial_s\bfPsi_{\eta_2} \!\cdot \nabla(\, \underline{\rho}_2^{\gamma -1}) \Vert_{L^\mathtt{p}(\Omega)}^2 + 1 \!\Big) {\mathlarger{\EuScript{E}} }_{\mathrm{rel}}(s)   \ds
\\[0.4em]
& \quad + \kappa \int_0^t  \Vert \partial_s\naby\eta_1 - \partial_s\naby\eta_2 \Vert_{L^2(\omega)}^2  \ds .
\end{aligned}
\end{equation}
We note that,
\begin{equation}
 \Vert \underline{\bv}_2\! \cdot \nabla(\, \underline{\rho}_2^{\gamma -1}) \Vert_{L^\mathtt{p}(\Omega)}, \quad   \Vert \nabla(\, \underline{\rho}_2^{\gamma -1}) \Vert_{L^\mathtt{p}(\Omega)}, 
 \quad    \Vert \partial_s\bfPsi_{\eta_2} \!\cdot \nabla(\, \underline{\rho}_2^{\gamma -1}) \Vert_{L^\mathtt{p}(\Omega)}  \in L^2((0,t)). 
\end{equation}
Indeed, it holds that 
\begin{subequations}
\begin{equation}
\Vert \underline{\bv}_2\! \cdot \nabla(\, \underline{\rho}_2^{\gamma -1}) \Vert_{L^\mathtt{p}(\Omega)}  \lesssim  \Vert \underline{\bv}_2\Vert_{W^{2,2}(\Omega)} \Vert \underline{\rho}_2 \Vert_{W^{3,2}(\Omega)},
\end{equation}
\begin{equation}
 \Vert \partial_s\bfPsi_{\eta_2} \!\cdot \nabla(\, \underline{\rho}_2^{\gamma -1}) \Vert_{L^\mathtt{p}(\Omega)}  \lesssim  \left( 1 + \Vert \partial_s \eta_2 \Vert_{W^{2,2}(\omega)} \right)  \Vert \underline{\rho}_2 \Vert_{W^{3,2}(\Omega)}  .
\end{equation}
\end{subequations}

In a similar fashion, we obtain 
\begin{equation*}
\begin{aligned}
{\mathlarger{\EuScript{T}}}_6 & \lesssim  \int_0^t  \Vert J_{\eta_1} \Vert_{L^{24}(\Omega)} \Vert \underline{\rho}_1 \Vert_{L^{ \gamma}(\Omega)}  \Vert \underline{\bv} \Vert_{L^\infty(\Omega)}  
\\[0.4em]
& \quad \quad \qquad  \cdot \Vert\! \left( \nabla\bfPsi_{\eta_1} \right)^{-1}\circ \bfPsi_{\eta_1} - \left(\nabla\bfPsi_{\eta_2} \right)^{-1}\circ \bfPsi_{\eta_2}   \Vert_{L^{24}(\Omega)}   \Vert \nabla\underline{\bv}_2 \Vert_{L^\mathtt{p}(\Omega)} \Vert \underline{\bv}_2 - \underline{\bv}_1 \Vert_{L^3(\Omega)} \ds .
\end{aligned} 
\end{equation*}
Hence, by Young's inequality and \eqref{eq:BasicEnergy}, it follows that 
\begin{equation}\label{eq:T6Estim}
\begin{aligned}
{\mathlarger{\EuScript{T}}}_6 & \lesssim C(\kappa)  \int_0^t  \Vert \nabla \underline{\bv}_2\Vert_{L^\mathtt{p}(\Omega)}^2 {\mathlarger{\EuScript{E}} }_{\mathrm{rel}}(s)   \ds  + \kappa \int_0^t \Vert \underline{\bv}_2 - \underline{\bv}_1 \Vert_{L^3(\Omega)}^2  \ds ,
\end{aligned}
\end{equation}
with 
\[ \Vert \nabla \underline{\bv}_2\Vert_{L^\mathtt{p}(\Omega)} \in L^2\left( (0, t) \right) .\]

Considering the term ${\mathlarger{\EuScript{T}}}_7$, we have 
\begin{equation*}
{\mathlarger{\EuScript{T}}}_7  \lesssim \int_0^t     \Vert \left(\nabla\bfPsi_{\eta_1} \right)^{-1}\circ \bfPsi_{\eta_1}   \Vert_{L^{\infty}(\Omega)}  \Vert \nabla\underline{\bv}_2\Vert_{L^\infty(\Omega)} \int_\Omega J_{\eta_1} \underline{\rho}_1 |\underline{\bv}_2 - \underline{\bv}_1|^2 \dx\ds .
\end{equation*}
Whence, 
\begin{equation}\label{eq:T7Estim}
{\mathlarger{\EuScript{T}}}_7 \lesssim \int_0^t   \Vert \nabla\underline{\bv}_2\Vert_{L^\infty(\Omega)} {\mathlarger{\EuScript{E}} }_{\mathrm{rel}}(s)  \ds .
\end{equation}

Arguing as in the estimate of ${\mathlarger{\EuScript{T}}}_6$, we obtain
\begin{equation}\label{eq:T8Estim}
\begin{aligned}
{\mathlarger{\EuScript{T}}}_8 & \lesssim C(\kappa)  \int_0^t  \Vert \nabla \underline{\rho}_2\Vert_{L^\mathtt{p}(\Omega)}^2 {\mathlarger{\EuScript{E}} }_{\mathrm{rel}}(s)   \ds  + \kappa \int_0^t \Vert \underline{\bv}_2 - \underline{\bv}_1 \Vert_{L^3(\Omega)}^2  \ds .
\end{aligned}
\end{equation}

We next estimate 
\[
{\mathlarger{\EuScript{T}}}_9 =  a\gamma \int_0^t  \int_\Omega J_{\eta_1} (\,\underline{\rho}_2 - \underline{\rho}_1 )\underline{\rho}_2^{\gamma -1} ( \nabla_1 - \nabla_2) \cdot \underline{\bv}_2 \dx\ds.
\]
By H\"older's inequality, it holds that 
\begin{equation*}
\begin{aligned}
{\mathlarger{\EuScript{T}}}_9 & \lesssim \int_0^t   \Vert J_{\eta_1}\Vert_{L^{24}(\Omega)} \Vert \underline{\rho}_2 - \underline{\rho}_1 \Vert_{L^{2\wedge\gamma}(\Omega)}  \Vert \underline{\rho}_2 \Vert_{L^\infty(\Omega)}^{\gamma -1} 
\\[0.4em]
& \qquad \qquad \cdot \Vert \left( \nabla\bfPsi_{\eta_1} \right)^{-1}\circ \bfPsi_{\eta_1} - \left(\nabla\bfPsi_{\eta_2} \right)^{-1}\circ \bfPsi_{\eta_2}   \Vert_{L^{24}(\Omega)}  \Vert \nabla \underline{\bv}_2 \Vert_{L^{\mathtt{p}}(\Omega)} \ds .
\end{aligned}
\end{equation*} 
Hence,
\begin{equation}\label{eq:T9Estim}
\begin{aligned}
{\mathlarger{\EuScript{T}}}_9 & \lesssim   \int_0^t  \Vert \nabla \underline{\bv}_2 \Vert_{L^{\mathtt{p}}(\Omega)}^2 {\mathlarger{\EuScript{E}} }_{\mathrm{rel}}(s)   \ds  .
\end{aligned}
\end{equation}

For the term 
\[
{\mathlarger{\EuScript{T}}}_{10} = -  \int_0^t \int_\Omega   J_{\eta_1} {\mathlarger{\mathcal{H}}}\!\left(\underline{\rho}_1 \bm{\big|}\underline{\rho}_2\right)  \nabla_1\cdot \underline{\bv}_2   \dx\ds ,
\]
it is immediate that 
\begin{equation*}
{\mathlarger{\EuScript{T}}}_{10} \lesssim \int_0^t  \Vert \left(\nabla\bfPsi_{\eta_1} \right)^{-1}\circ \bfPsi_{\eta_1}   \Vert_{L^{\infty}(\Omega)}  \Vert \nabla\underline{\bv}_2\Vert_{L^\infty(\Omega)} \int_\Omega {\mathlarger{\mathcal{H}}}\!\left(\underline{\rho}_1 \bm{\big|}\underline{\rho}_2\right) \dx\ds.
\end{equation*}
Whence,
\begin{equation}\label{eq:T10Estim}
{\mathlarger{\EuScript{T}}}_{10} \lesssim \int_0^t \Vert \nabla\underline{\bv}_2\Vert_{L^\infty(\Omega)} {\mathlarger{\EuScript{E}} }_{\mathrm{rel}}(s)  \ds .
\end{equation}

Next, we estimate 
\begin{equation*}
\begin{aligned}
{\mathlarger{\EuScript{T}}}_{11} & = \int_0^t  \int_\Omega J_{\eta_1} \underline{\rho}_1 \Big[  \partial_s \bfPsi_{\eta_2} \cdot  \nabla_2\underline{\bv}_2  - \partial_s \bfPsi_{\eta_1} \cdot  \nabla_1\underline{\bv}_2  \Big] \cdot ( \underline{\bv}_2 -  \underline{\bv}_1) \dx\ds 
\\[0.3em]
& =  \int_0^t  \int_\Omega J_{\eta_1} (\underline{\rho}_1 - \underline{\rho}_2) \Big[  \partial_s \bfPsi_{\eta_2} \cdot  \nabla_2\underline{\bv}_2  - \partial_s \bfPsi_{\eta_1} \cdot  \nabla_1\underline{\bv}_2  \Big] \cdot ( \underline{\bv}_2 -  \underline{\bv}_1) \dx\ds 
\\[0.3em]
& \quad +  \int_0^t  \int_\Omega J_{\eta_1}  \underline{\rho}_2 \Big[  \partial_s \bfPsi_{\eta_2} \cdot  \nabla_2\underline{\bv}_2  - \partial_s \bfPsi_{\eta_1} \cdot  \nabla_1\underline{\bv}_2  \Big] \cdot ( \underline{\bv}_2 -  \underline{\bv}_1) \dx\ds .
\end{aligned}
\end{equation*}
Using H\"older's inequality, we deduce that 
\begin{equation*}
\begin{aligned}
 {\mathlarger{\EuScript{T}}}_{11} & \lesssim \int_0^t   \Vert J_{\eta_1} \Vert_{L^{24}(\Omega)} \Vert \underline{\rho}_1 - \underline{\rho}_2 \Vert_{L^{2\wedge \gamma}(\Omega)} \Vert \nabla\underline{\bv}_2 \Vert_{L^\infty(\Omega)}  \bigg(   \Vert \partial_s \bfPsi_{\eta_1} \left(\nabla \bfPsi_{\eta_1} \right)^{-1}\circ \bfPsi_{\eta_1} \Vert_{L^{24}(\Omega)} 
 \\[0.3em]
 &  \qquad \qquad \qquad \qquad \qquad  \qquad \qquad \qquad \qquad  +  \Vert \partial_s \bfPsi_{\eta_2} \left(\nabla \bfPsi_{\eta_2} \right)^{-1}\circ \bfPsi_{\eta_2} \Vert_{L^{24}(\Omega)}    \bigg) \Vert \underline{\bv}_2  - \underline{\bv}_1 \Vert_{L^3(\Omega)}  \ds 
 \\[0.4em]
 & \quad +  \int_0^t   \Vert J_{\eta_1} \Vert_{L^{\infty}(\Omega)} \Vert  \underline{\rho}_2 \Vert_{L^{\infty}(\Omega)} \Vert \nabla\underline{\bv}_2 \Vert_{L^\infty(\Omega)}  \Bigg(  \Vert  \left( \partial_t \bfPsi_{\eta_1}  - \partial_s \bfPsi_{\eta_2}  \right)   \left(\nabla \bfPsi_{\eta_1} \right)^{-1}\circ \bfPsi_{\eta_1}\Vert_{L^2(\Omega)}  
 \\[0.3em]
 & \qquad \qquad \qquad \qquad \quad    + \left \Vert \partial_s \bfPsi_{\eta_2} \left(  \left(\nabla \bfPsi_{\eta_1} \right)^{-1}\circ \bfPsi_{\eta_1} - \left(\nabla \bfPsi_{\eta_2} \right)^{-1}\circ \bfPsi_{\eta_2}   \right)   \right\Vert_{L^2(\Omega)} \Bigg)  \Vert \underline{\bv}_2  - \underline{\bv}_1 \Vert_{L^2(\Omega)}  \ds  .
 \end{aligned}
\end{equation*}
By Young's inequality and the  regularity properties of the weak and strong solutions, it follows that 
\begin{equation}\label{eq:T11Estim}
\begin{aligned}
{\mathlarger{\EuScript{T}}}_{11} & \lesssim C(\kappa)\! \int_0^t \!  \Big(\! \Vert \partial_s \bfPsi_{\eta_1} \left(\nabla \bfPsi_{\eta_1} \right)^{-1}\!\circ \bfPsi_{\eta_1} \Vert_{L^{24}(\Omega)}^2 +   \Vert \partial_s \bfPsi_{\eta_2}\! \left(\nabla \bfPsi_{\eta_2} \right)^{-1}\!\circ \bfPsi_{\eta_2} \Vert_{L^{24}(\Omega)}^2 + 1 \!\Big) {\mathlarger{\EuScript{E}} }_{\mathrm{rel}}(s)  \ds  
\\[0.3em]
& \quad + \kappa \int_0^t  \Vert \underline{\bv}_2  - \underline{\bv}_1 \Vert_{L^2(\Omega)}^2   \ds   ,
\end{aligned}
\end{equation}
where 
\[
\Vert \partial_s \bfPsi_{\eta_1} \left(\nabla \bfPsi_{\eta_1} \right)^{-1}\!\circ \bfPsi_{\eta_1} \Vert_{L^{24}(\Omega)} , \;\;    \Vert \partial_s \bfPsi_{\eta_2}\! \left(\nabla \bfPsi_{\eta_2} \right)^{-1}\!\circ \bfPsi_{\eta_2} \Vert_{L^{24}(\Omega)}  \; \in  L^2 \left( (0, t) \right).  
\]

Finally, we consider the term 
\begin{equation}\label{eq:T12Old}
{\mathlarger{\EuScript{T}}}_{12} = \int_0^t \int_\omega  (\partial_s\eta_2 - \partial_s\eta_1) p(\,\underline{\rho}_2) \bn\cdot  \big(  \bn_{\eta_2}\circ \bm{\varphi}_{\eta_2}  \mathrm{det}(\naby\bm{\varphi}_{\eta_2})  - \bn_{\eta_1}\circ \bm{\varphi}_{\eta_1}  \mathrm{det}(\naby\bm{\varphi}_{\eta_1}) \big) \dy\ds .
\end{equation}
We first rewrite  ${\mathlarger{\EuScript{T}}}_{12}$ by introducing  
\[
\mathtt{W}(\eta) := \bn\cdot\bn_\eta \mathrm{det}(\naby\bm{\varphi}_\eta).
\]
Using properties of the triple scalar product, it follows that 
\begin{equation}\label{eq:DetDiff}
\mathtt{W}(\eta_2) - \mathtt{W}(\eta_1) = (\eta_2 - \eta_1) \Big[ \bn\cdot \big( (\partial_1\bm{\varphi} \times \partial_2\bn) + (\partial_1\bn \times \partial_2 \bm{\varphi} ) \big) + (\eta_2 + \eta_1)(\partial_1\bn \times \partial_2\bn) \Big]  . 
\end{equation}
Therefore, substituting  \eqref{eq:DetDiff} into \eqref{eq:T12Old}, and applying  H\"older's inequality, the trace theorem  and Young's inequality, we obtain  
\begin{equation}
{\mathlarger{\EuScript{T}}}_{12} \lesssim  C(\kappa) \int_{0}^t \Vert p(\,\underline{\rho}_2) \Vert_{W^{1,2}(\Omega)}^2 {\mathlarger{\EuScript{E}} }_{\mathrm{rel}}(s) \ds  + \kappa \int_0^t   \Vert \partial_s \naby\eta_2 - \partial_s \naby\eta_1\Vert_{L^2(\omega)}^2 \ds .
\end{equation} 
This concludes the proof of \cref{theo:WeakStrongUniqueMain}.



 \appendix
\crefalias{section}{appendix}
\section{ }
\label{app:MRsourcetermEstim}
In order to prove the estimates \eqref{eq:BuEstim}, \eqref{eq:TimeBuEstim}, \eqref{eq:fEstim}, and \eqref{eq:TimefEstim},   we first recall  that the Hanzawa transform depends smoothly on the interface deformation $\eta$. In particular, for the  relative mapping $\bfPsi_{\eta \to \eta_0}$  and for all   $\mathtt{k} \in \{1 , \ldots, 5 \}, \; \mathtt{p} \in [1, \infty] $, the following estimates hold
\begin{subequations}\label{eq:GeoEstimate}
\begin{equation}
\Vert  \mathbf{A}_{\eta \to \eta_0 } - \mathbb{I}_{3\times 3} \Vert_{W^{\mathtt{k}, \mathtt{p}}(\Omega_{\eta_0})}  + \Vert \mathbf{B}_{\eta \to \eta_0 } - \mathbb{I}_{3\times 3} \Vert_{W^{\mathtt{k}, \mathtt{p}}(\Omega_{\eta_0})}   \lesssim   \Vert \eta - \eta_0 \Vert_{W^{\mathtt{k}+1, \mathtt{p}}(\omega) } ,  
\end{equation}
\begin{equation}
 \left\Vert \dfrac{1}{J_{\eta \to \eta_0} } -1 \right\Vert_{W^{\mathtt{k}, \mathtt{p}}(\Omega_{\eta_0})} +  \Vert \nabla\bfPsi_{\eta \to \eta_0 } - \mathbb{I}_{3\times 3} \Vert_{W^{\mathtt{k}, \mathtt{p}}(\Omega_{\eta_0})}   \lesssim   \Vert \eta - \eta_0 \Vert_{W^{\mathtt{k}+1, \mathtt{p}}(\omega) }. 
\end{equation}
\end{subequations}
Here the implicit constant $C(\eta_0) >0$ (cf. \cite[Section 2]{breit2024regularity}). 

We start with the additive perturbation operator   $\mathsf{B}$ acting on the velocity filed $\underline{\bu}$, while the forcing term $\mathsf{f}$ will be treated subsequently.
More precisely, we first consider 
\begin{equation*}
\begin{aligned}
\mathsf{B}\underline{\bu} & =  -  \nabla\underline{\bu} \cdot \partial_t \bfPsi_{\eta \to \eta_0 }^{-1}\circ \bfPsi_{\eta \to \eta_0 }   -  \dfrac{1}{J_{\eta \to \eta_0 }} \Big( \underline{\bu}  \bigl( \nabla\mathcal{E}_{\eta_0}(\partial_t\eta\bn) \colon  \mathbf{B}_{\eta \to \eta_0 }    \bigr)  +  \mathcal{E}_{\eta_0}(\partial_t\eta\bn)  \big( \nabla \underline{\bu} \colon \mathbf{B}_{\eta \to \eta_0 } \big) \Big).
\end{aligned}
\end{equation*}
Upon rewriting $\mathsf{B}\underline{\bu}$ as a perturbation of the reference operator, we deduce that 
\begin{align*}
\int_{I_*} \!\!\Vert \mathsf{B}\underline{\bu} \Vert_{W^{2,2}(\Omega_{\eta_0}\!)}^2 \!\dt & \lesssim \int_{I_*} \Vert \nabla\underline{\bu} \Vert_{W^{2,2}(\Omega_{\eta_0})}^2   \Vert \left(\nabla\bfPsi_{\eta \to \eta_0} \right)^{-1}  - \mathbb{I}_{3 \times 3} \Vert_{W^{2,2}(\Omega_{\eta_0})}^2   \Vert \partial_t\bfPsi_{\eta \to \eta_0} \Vert_{W^{2,2}(\Omega_{\eta_0})}^2  \dt
\\[0.4em]
& \quad + \int_{I_*} \Vert \nabla\underline{\bu} \Vert_{W^{2,2}(\Omega_{\eta_0})}^2  \Vert \partial_t\bfPsi_{\eta \to \eta_0} \Vert_{W^{2,2}(\Omega_{\eta_0})}^2  \dt
\\[0.4em]
& \quad + \! \int_{I_*}\! \left\Vert \!\frac{1}{J_{\eta \to \eta_0}} \!-\!1\! \right \Vert_{W^{2,2}(\Omega_{\eta_0}\!)}^2 \!\!\!\!\! \Vert \underline{\bu} \Vert_{W^{2,2}(\Omega_{\eta_0}\!)}^2  \!\Vert\!\nabla\mathcal{E}_{\eta_0}\!(\partial_t\eta\bn) \!\Vert_{W^{2, 2}(\Omega_{\eta_0}\!)}^2  \!  \Vert \!\mathbf{B}_{\eta \to \eta_0 }\! - \mathbb{I}_{3\times 3}\! \Vert_{W^{2, 2}(\Omega_{\eta_0}\!)}^2 \!\dt
\\[0.4em]
& \quad + \int_{I_*} \left\Vert \frac{1}{J_{\eta \to \eta_0}} - 1  \right \Vert_{W^{2,2}(\Omega_{\eta_0})}^2 \Vert \underline{\bu} \Vert_{W^{2,2}(\Omega_{\eta_0})}^2  \Vert\Div\left( \mathcal{E}_{\eta_0}(\partial_t\eta\bn) \right) \Vert_{W^{2, 2}(\Omega_{\eta_0})}^2 \dt
\\[0.4em]
& \quad +  \int_{I_*} \Vert \underline{\bu} \Vert_{W^{2,2}(\Omega_{\eta_0})}^2 \Vert\nabla \mathcal{E}_{\eta_0}(\partial_t\eta\bn) \Vert_{W^{2, 2}(\Omega_{\eta_0})}^2  \Vert \mathbf{B}_{\eta \to \eta_0 } - \mathbb{I}_{3\times 3} \Vert_{W^{2, 2}(\Omega_{\eta_0} )}^2 \dt 
\\[0.4em]
& \quad +  \int_{I_*} \Vert \underline{\bu} \Vert_{W^{2,2}(\Omega_{\eta_0})}^2 \ \Vert\Div\left( \mathcal{E}_{\eta_0}(\partial_t\eta\bn) \right) \Vert_{W^{2, 2}(\Omega_{\eta_0})}^2 \dt 
\\[0.4em]
& \quad + \! \int_{I_*}\! \left\Vert \!\frac{1}{J_{\eta \to \eta_0}} \!-\!1\! \right \Vert_{W^{2,2}(\Omega_{\eta_0}\!)}^2 \!\!\!\!\! \Vert \!\nabla\underline{\bu} \Vert_{W^{2,2}(\Omega_{\eta_0}\!)}^2  \!\Vert\mathcal{E}_{\eta_0}\!(\partial_t\eta\bn) \!\Vert_{W^{2, 2}(\Omega_{\eta_0}\!)}^2  \!  \Vert \!\mathbf{B}_{\eta \to \eta_0 }\! - \mathbb{I}_{3\times 3}\! \Vert_{W^{2, 2}(\Omega_{\eta_0}\!)}^2 \!\dt
\\[0.4em]
& \quad + \int_{I_*} \left\Vert \frac{1}{J_{\eta \to \eta_0}} - 1  \right \Vert_{W^{2,2}(\Omega_{\eta_0})}^2 \Vert \Div\underline{\bu} \Vert_{W^{2,2}(\Omega_{\eta_0})}^2  \Vert\mathcal{E}_{\eta_0}(\partial_t\eta\bn) \Vert_{W^{2, 2}(\Omega_{\eta_0})}^2 \dt
\\[0.4em]
& \quad +  \int_{I_*} \Vert \nabla\underline{\bu} \Vert_{W^{2,2}(\Omega_{\eta_0})}^2 \Vert \mathcal{E}_{\eta_0}(\partial_t\eta\bn) \Vert_{W^{2, 2}(\Omega_{\eta_0})}^2  \Vert \mathbf{B}_{\eta \to \eta_0 } - \mathbb{I}_{3\times 3} \Vert_{W^{2, 2}(\Omega_{\eta_0} )}^2 \dt 
\\[0.4em]
& \quad +   \int_{I_*} \Vert \Div\underline{\bu} \Vert_{W^{2,2}(\Omega_{\eta_0})}^2 \Vert \mathcal{E}_{\eta_0}(\partial_t\eta\bn) \Vert_{W^{2, 2}(\Omega_{\eta_0})}^2 \dt  
\\[0.4em]
& \lesssim \sum\limits_{\mathsf{i} = 1}^{10} {\mathlarger{\mathfrak{B}}}_\mathsf{i} .
\end{align*}

Using the geometric estimates \eqref{eq:GeoEstimate}, we derive that 
\begin{equation}\label{eq:B1}
{\mathlarger{\mathfrak{B}}}_1 \lesssim  \int_{I_*} \Vert \underline{\bu} \Vert_{W^{3,2}(\Omega_{\eta_0})}^2 \Vert \partial_t\eta \Vert_{W^{5/2,2}(\omega)}^2 \dt. 
\end{equation} 

Since ${\mathlarger{\mathfrak{B}}}_2$ is similar to ${\mathlarger{\mathfrak{B}}}_1$ up to an integrand factor, it holds 
\begin{equation}\label{eq:B2}
{\mathlarger{\mathfrak{B}}}_2 \lesssim \int_{I_*} \Vert \underline{\bu} \Vert_{W^{3,2}(\Omega_{\eta_0})}^2 \Vert \partial_t\eta \Vert_{W^{5/2,2}(\omega)}^2 \dt. 
\end{equation}

For the contributions ${\mathlarger{\mathfrak{B}}}_\mathsf{i}, \mathsf{i} \in \{3,\dots, 6\} $, it is straightforward -- using \eqref{eq:GeoEstimate} -- that  
 \begin{equation}\label{eq:B3456}
{\mathlarger{\mathfrak{B}}}_\mathsf{i} \lesssim   \int_{I_*} \Vert \underline{\bu} \Vert_{W^{3,2}(\Omega_{\eta_0})}^2 \Vert \partial_t \eta \Vert_{W^{5/2, 2}(\omega)}^2\dt    . 
\end{equation}

Considering the term ${\mathlarger{\mathfrak{B}}}_7$,  we get using \eqref{eq:GeoEstimate} and interpolation 
\begin{equation*}
{\mathlarger{\mathfrak{B}}}_7 \lesssim \sup\limits_{I_*} \Vert \eta - \eta_0 \Vert_{W^{3,2}(\omega)}^4 \int_{I_*} \Vert \nabla\underline{\bu} \Vert_{W^{1,2}(\Omega_{\eta_0})} \Vert \nabla\underline{\bu} \Vert_{W^{3,2}(\Omega_{\eta_0})} \Vert \partial_t\eta \Vert_{W^{3/2,2}(\omega)}^2  \dt. 
\end{equation*}
An application of Young's inequality further yields to 
\begin{equation}\label{eq:B7}
{\mathlarger{\mathfrak{B}}}_7 \lesssim \kappa \sup\limits_{I_*} \Vert \partial_t\eta \Vert_{W^{3/2, 2}(\omega)}^2 \int_{I_*} \Vert \underline{\bu} \Vert_{W^{4,2}(\Omega_{\eta_0})}^2 \dt + c(\kappa) \sup\limits_{I_*} \Vert \partial_t\eta \Vert_{W^{3/2, 2}(\omega)}^2 \int_{I_*} \Vert \nabla\underline{\bu} \Vert_{W^{1,2}(\Omega_{\eta_0})}^2 \dt. 
\end{equation} 

The remaining terms ${\mathlarger{\mathfrak{B}}}_\mathsf{j}, \, \mathsf{j} \in \{8, 9, 10\}$, being similar to ${\mathlarger{\mathfrak{B}}}_7$ up to an integrand factor, we deduce that 
\begin{equation}\label{eq:B8910}
{\mathlarger{\mathfrak{B}}}_\mathsf{j} \lesssim  \kappa \sup\limits_{I_*} \Vert \partial_t\eta \Vert_{W^{3/2, 2}(\omega)}^2 \int_{I_*} \Vert \underline{\bu} \Vert_{W^{4,2}(\Omega_{\eta_0})}^2 \dt + c(\kappa) \sup\limits_{I_*} \Vert \partial_t\eta \Vert_{W^{3/2, 2}(\omega)}^2 \int_{I_*} \Vert \nabla\underline{\bu} \Vert_{W^{1,2}(\Omega_{\eta_0})}^2 \dt. 
\end{equation}
Hence, \eqref{eq:BuEstim} holds, with  the hidden constant  depending on the value of the acceleration energy  ${\mathlarger{\mathtt{E}}}_{\mathrm{acc}}$.\\[-0.8em]


We proceed with the estimate for  $\partial_t \mathsf{B}\underline{\bu}$,   which we decompose as 
\[
\partial_t \mathsf{B}\underline{\bu} =   \big(\partial_t \mathsf{B}\underline{\bu} \big)_1  +  \big(\partial_t \mathsf{B}\underline{\bu} \big)_2 +  \big(\partial_t \mathsf{B}\underline{\bu} \big)_3 ,
\]
where 
\begin{align*}
 \big(\partial_t \mathsf{B}\underline{\bu} \big)_1 & :=  -\partial_t \nabla\underline{\bu} \cdot \partial_t  \bfPsi_{\eta \to \eta_0 }^{-1}\circ \bfPsi_{\eta \to \eta_0 }  + \nabla\underline{\bu} \cdot \partial_t \left( \partial_t  \bfPsi_{\eta \to \eta_0 }^{-1}\circ \bfPsi_{\eta \to \eta_0 }   \right) ,
 \\[1em]
  \big(\partial_t \mathsf{B}\underline{\bu} \big)_2 & := -\partial_t \left(\dfrac{1}{J_{\eta \to \eta_0 }} \right)  \mathcal{E}_{\eta_0}(\partial_t\eta\bn)  \left(\nabla\underline{\bu} \colon \mathbf{B}_{\eta \to \eta_0 }  \right)  - \dfrac{1}{J_{\eta \to \eta_0 }}  \partial_t  \mathcal{E}_{\eta_0}(\partial_t\eta\bn)  \left(\nabla\underline{\bu} \colon \mathbf{B}_{\eta \to \eta_0 }  \right) 
\\[0.4em]
&\quad -   \dfrac{1}{J_{\eta \to \eta_0 }}  \mathcal{E}_{\eta_0}(\partial_t\eta\bn)  \Bigl(  \partial_t \nabla\underline{\bu} \colon \mathbf{B}_{\eta \to \eta_0 }  +  \nabla\underline{\bu} \colon \partial_t\mathbf{B}_{\eta \to \eta_0 }  \Bigr) , 
\\[0.4em] 
\intertext{and}
\\[-1em]
\big(\partial_t \mathsf{B}\underline{\bu} \big)_3 & := -\partial_t \left(\dfrac{1}{J_{\eta \to \eta_0 }} \right)  \underline{\bu} \left(\nabla\mathcal{E}_{\eta_0}(\partial_t\eta\bn) \colon \mathbf{B}_{\eta \to \eta_0 }  \right)  - \dfrac{1}{J_{\eta \to \eta_0 }}\partial_t \underline{\bu} \left(\nabla\mathcal{E}_{\eta_0}(\partial_t\eta\bn) \colon \mathbf{B}_{\eta \to \eta_0 }  \right)  
\\[0.4em]
&\quad - \dfrac{1}{J_{\eta \to \eta_0 }}\underline{\bu} \Big( \partial_t \nabla\mathcal{E}_{\eta_0}(\partial_t\eta\bn) \colon \mathbf{B}_{\eta \to \eta_0 }  + \nabla\mathcal{E}_{\eta_0}(\partial_t\eta\bn) \colon \partial_t \mathbf{B}_{\eta \to \eta_0 }  \Big)  .
\end{align*}
For the first term, we have 
\begin{align*}
\int_{I_*} \left\Vert \big(\partial_t \mathsf{B}\underline{\bu} \big)_1 \right\Vert_{L^2(\Omega_{\eta_0})}^2 \dt & \lesssim \int_{I_*}  \Vert \partial_t \nabla\underline{\bu} \Vert_{L^2(\Omega_{\eta_0})}^2 \Vert \left(\nabla\bfPsi_{\eta \to \eta_0} \right)^{-1}  - \mathbb{I}_{3 \times 3} \Vert_{L^{\infty}(\Omega_{\eta_0})}^2   \Vert \partial_t\bfPsi_{\eta \to \eta_0} \Vert_{L^{\infty}(\Omega_{\eta_0})}^2  \dt
\\[0.4em]
&\quad + \int_{I_*} \Vert \partial_t \nabla\underline{\bu} \Vert_{L^2(\Omega_{\eta_0})}^2 \Vert \partial_t\bfPsi_{\eta \to \eta_0} \Vert_{L^{\infty}(\Omega_{\eta_0})}^2  \dt
\\[0.4em]
& \quad +   \int_{I_*} \Vert  \nabla\underline{\bu} \Vert_{L^\infty(\Omega_{\eta_0})}^2  \Vert \partial_t^2\bfPsi_{\eta \to \eta_0 }^{-1}\circ \bfPsi_{\eta \to \eta_0 }  \Vert_{L^2(\Omega_{\eta_0})}^2  \dt
\\[0.4em]
& \quad +   \int_{I_*} \Vert  \nabla\underline{\bu} \Vert_{L^\infty(\Omega_{\eta_0})}^2  \Vert \partial_t\nabla \bfPsi_{\eta \to \eta_0 }^{-1}\circ \bfPsi_{\eta \to \eta_0 }  \cdot \partial_t \bfPsi_{\eta \to \eta_0}  \Vert_{L^2(\Omega_{\eta_0})}^2  \dt .
\end{align*}
Using  \eqref{eq:GeoEstimate}, Sobolev embeddings, \cref{rem:HigherOrderEta},  interpolation and Young's inequality, we deduce that  
\begin{equation}\label{eq:TimeBuEstim1}
\int_{I_*} \left\Vert \big(\partial_t \mathsf{B}\underline{\bu} \big)_1 \right\Vert_{L^2(\Omega_{\eta_0})}^2 \dt  \lesssim \kappa \int_{I_*} \Vert \partial_t \underline{\bu} \Vert_{W^{2,2}(\Omega_{\eta_0})}^2  \dt  +  \int_{I_*} \Vert  \underline{\bu} \Vert_{W^{3,2}(\Omega_{\eta_0})}^2  \Vert \partial_t^2 \eta \Vert_{L^{2}(\omega)}^2 \dt   +   {\mathlarger{\mathtt{E}}}_{\mathrm{acc}}.
\end{equation}
A similar argument applies to $ \big(\partial_t \mathsf{B}\underline{\bu} \big)_2$. Indeed, it holds that 
\begin{equation}\label{eq:TimeBuEstim2prior}
\begin{aligned}
& \int_{I_*} \left\Vert \big(\partial_t \mathsf{B}\underline{\bu} \big)_2 \right\Vert_{L^2(\Omega_{\eta_0})}^2 \dt 
\\[0.4em]
& \lesssim   \int_{I_*}\Big(\Vert \partial_t \naby\eta\Vert_{L^\infty(\omega)}^2   +  \Vert \partial_t  \mathcal{E}_{\eta_0}(\partial_t\eta\bn) \Vert_{L^2(\Omega_{\eta_0})}^2  \Big) \Vert  \nabla\underline{\bu} \Vert_{L^\infty(\Omega_{\eta_0})}^2    \dt 
\\[0.4em]
&\quad +  \int_{I_*} \Vert  \mathcal{E}_{\eta_0}(\partial_t\eta\bn) \Vert_{L^\infty(\Omega_{\eta_0})}^2  \Big( \Vert \partial_t \nabla\underline{\bu} \Vert_{L^2(\Omega_{\eta_0})}^2  + \Vert \nabla\underline{\bu} \Vert_{L^\infty(\Omega_{\eta_0})}^2 \Vert  \partial_t \mathbf{B}_{\eta \to \eta_0 }  \Vert_{L^2(\Omega_{\eta_0})}^2   \Big)   \dt  . 
\end{aligned}
\end{equation}
Whence,
\begin{equation}\label{eq:TimeBuEstim2}
\begin{aligned}
& \int_{I_*} \left\Vert \big(\partial_t \mathsf{B}\underline{\bu} \big)_2 \right\Vert_{L^2(\Omega_{\eta_0})}^2 \dt  
\\[0.4em]
&\lesssim  \kappa \int_{I_*} \Vert \partial_t \underline{\bu} \Vert_{W^{2,2}(\Omega_{\eta_0})}^2  \dt +  \int_{I_*} \Big(\Vert \partial_t \eta \Vert_{W^{5/2, 2}(\omega)}^2  +  \Vert \partial_t^2 \eta \Vert_{L^{2}(\omega)}^2 \Big) \Vert  \underline{\bu} \Vert_{W^{3,2}(\Omega_{\eta_0})}^2 \dt  +   {\mathlarger{\mathtt{E}}}_{\mathrm{acc}}.
\end{aligned}
\end{equation}
By symmetry in $\underline{\bu}$ and $\mathcal{E}_{\eta_0}(\partial_t\eta\bn)$, \eqref{eq:TimeBuEstim2} holds for $\big(\partial_t \mathsf{B}\underline{\bu} \big)_3$. That is, 
\begin{equation}\label{eq:TimeBuEstim3}
\begin{aligned}
& \int_{I_*} \left\Vert \big(\partial_t \mathsf{B}\underline{\bu} \big)_3 \right\Vert_{L^2(\Omega_{\eta_0})}^2 \dt  
\\[0.4em]
&\lesssim  \kappa \int_{I_*} \Vert \partial_t \underline{\bu} \Vert_{W^{2,2}(\Omega_{\eta_0})}^2  \dt +  \int_{I_*} \Big(\Vert \partial_t \eta \Vert_{W^{5/2, 2}(\omega)}^2  +  \Vert \partial_t^2 \eta \Vert_{L^{2}(\omega)}^2 \Big) \Vert  \underline{\bu} \Vert_{W^{3,2}(\Omega_{\eta_0})}^2 \dt  +   {\mathlarger{\mathtt{E}}}_{\mathrm{acc}}.
\end{aligned}
\end{equation} 
Thus, \eqref{eq:TimeBuEstim} follows from \eqref{eq:TimeBuEstim1}--\eqref{eq:TimeBuEstim3}. \\[-0.8em]


Considering the forcing term $\mathsf{f}$, which we write as  
\begin{align*}
\mathsf{f} + \partial_t \mathcal{E}_{\eta_0}(\partial_t\eta\bn) & =     - \nabla\mathcal{E}_{\eta_0}(\partial_t\eta\bn) \cdot \partial_t \bfPsi_{\eta \to \eta_0 }^{-1}\circ \bfPsi_{\eta \to \eta_0 }  -\dfrac{1}{J_{\eta \to \eta_0}}\underline{\bu} \big( \nabla \underline{\bu} \colon \mathbf{B}_{\eta \to \eta_0 } \big)  
\\[0.3em]
& \quad \;\,  -\dfrac{1}{J_{\eta \to \eta_0}} \mathcal{E}_{\eta_0}(\partial_t\eta\bn) \big( \nabla\mathcal{E}_{\eta_0}(\partial_t\eta\bn)  \colon\mathbf{B}_{\eta \to \eta_0 }  \big)  - \left(J_{\eta \to \eta_0}\underline{\rho} \right)^{-1} \mathbf{B}_{\eta \to \eta_0 } \nabla p(\underline{\rho})
\\[0.3em]
& \quad \;\, + \left(J_{\eta \to \eta_0}\underline{\rho} \right)^{-1}\Div\Biggl[  \mu \left(\mathbf{A}_{\eta \to \eta_0 }  - \mathbb{I}_{3 \times 3}   \right) \nabla\underline{\bu}  + (\lambda + \mu)\biggl( \Bigl(  \left(\mathbf{B}_{\eta \to \eta_0 } - \mathbb{I}_{3 \times 3} \right) \colon \nabla\underline{\bu} \Bigr)  \mathbb{I}_{3 \times 3}    
\\[0.3em]
& \qquad \qquad\qquad \qquad \qquad \qquad\qquad \qquad\qquad + \left( \mathbf{B}_{\eta \to \eta_0 } \colon \nabla\underline{\bu}  \right)  \left( \dfrac{1}{J_{\eta \to \eta_0 }} \mathbf{B}_{\eta \to \eta_0 } -  \mathbb{I}_{3 \times 3}  \right)                 \biggr)      \Biggr]      
\\[0.3em]
& \quad \;\, + \left(J_{\eta \to \eta_0}\underline{\rho} \right)^{-1}\Div\Bigg[ \mu  \mathbf{A}_{\eta \to \eta_0 }  \nabla\mathcal{E}_{\eta_0}(\partial_t\eta\bn) +   \dfrac{(\lambda + \mu)}{J_{\eta \to \eta_0}}   \Big( \mathbf{B}_{\eta \to \eta_0 } \colon \nabla \mathcal{E}_{\eta_0}(\partial_t\eta\bn)     \Big)     \mathbf{B}_{\eta \to \eta_0 }     \Bigg]
\\[0.3em]
& = \sum\limits_{l = 1}^{6} {\mathlarger{\mathfrak{S}}}_l.
\end{align*}
Using chain rule and \eqref{eq:GeoEstimate}, we derive that 
\begin{equation*}
\begin{aligned}
\int_{I_*} \Vert {\mathlarger{\mathfrak{S}}}_1 \Vert_{W^{2,2}(\Omega_{\eta_0})}^2 \dt & \lesssim \int_{I_*} \Vert \nabla\mathcal{E}_{\eta_0}(\partial_t\eta\bn) \Vert_{W^{2,2}(\Omega_{\eta_0})}^2 \Vert \partial_t\eta \Vert_{W^{2,2}(\omega)}^2   \dt
\\[0.4em]
& \quad \lesssim \int_{I_*} \Vert \partial_t \eta \Vert_{W^{5/2,2}}^2 \Vert \partial_t\eta \Vert_{W^{2,2}(\omega)}^2 \dt .
\end{aligned}
\end{equation*}
A standard interpolation argument yields
\begin{equation}\label{eq:S1}
\int_{I_*} \Vert {\mathlarger{\mathfrak{S}}}_1 \Vert_{W^{2,2}(\Omega_{\eta_0})}^2 \dt \lesssim   \sup\limits_{I_*}\Vert \partial_t  \eta \Vert_{W^{3/2, 2}(\omega)}   \int_{I_*} \Vert \partial_t  \eta \Vert_{W^{5/2, 2}(\omega)}^3 \dt.
\end{equation}
For the term ${\mathlarger{\mathfrak{S}}}_2$, it holds that 
\begin{equation*}
\int_{I_*} \Vert {\mathlarger{\mathfrak{S}}}_2 \Vert_{W^{2,2}(\Omega_{\eta_0})}^2 \dt \leq \int_{I_*} \left\Vert \dfrac{1}{ J_{\eta \to \eta_0} } \right\Vert_{W^{2,2}(\Omega_{\eta_0})}^2     \Vert \underline{\bu}\cdot \nabla  \underline{\bu} \Vert_{W^{2,2}(\Omega_{\eta_0})}^2 \Vert \mathbf{B}_{\eta \to \eta_0 } \Vert_{W^{2,2}(\Omega_{\eta_0})}^2   \dt.
\end{equation*}
Therefore,  we deduce from \eqref{eq:GeoEstimate}  that
\begin{equation}\label{eq:S2}
\int_{I_*} \Vert {\mathlarger{\mathfrak{S}}}_2 \Vert_{W^{2,2}(\Omega_{\eta_0})}^2 \dt \lesssim    \int_{I_*}  \Vert \underline{\bu} \Vert_{W^{3,2}(\Omega_{\eta_0})}^2 \Vert \underline{\bu} \Vert_{W^{2,2}(\Omega_{\eta_0})}^2  \dt.
\end{equation}
Moreover, it follows immediately from  \eqref{eq:GeoEstimate} that
 \begin{equation}\label{eq:S3}
\int_{I_*} \Vert {\mathlarger{\mathfrak{S}}}_3 \Vert_{W^{2,2}(\Omega_{\eta_0})}^2 \dt \lesssim  \left( \int_{I_*} \Vert \partial_t \eta \Vert_{W^{5/2, 2}(\omega)}^2 \dt \right) \sup\limits_{I_*} \Vert \partial_t\eta \Vert_{W^{3/2, 2}(\omega)}^2 .
\end{equation}
Furthermore, it holds that (see e.g., \cite[Chapter 3, Lemma 3.4]{MajdaBertozzi2002})  
\begin{equation}\label{eq:PressureMoserIneq}
\begin{aligned}
\Vert \underline{\rho}^{-1}\nabla p(\underline{\rho}) \Vert_{W^{2,2}(\Omega_{\eta_0})}^2 & \lesssim   \Vert \nabla\left(\underline{\rho}^{-1} \right) \Vert_{L^{\infty}(\Omega_{\eta_0})}^2 \Vert  p(\underline{\rho}) \Vert_{W^{2,2}(\Omega_{\eta_0})}^2  +  \Vert \underline{\rho}^{-1}  \Vert_{W^{2,2}(\Omega_{\eta_0})}^2 \Vert  \nabla p(\underline{\rho}) \Vert_{L^{\infty}(\Omega_{\eta_0})}^2
\\[0.4em]
& \qquad +   \Vert \underline{\rho}^{-1}  \Vert_{L^{\infty}(\Omega_{\eta_0})}^2 \Vert  p(\underline{\rho}) \Vert_{W^{3,2}(\Omega_{\eta_0})}^2. 
\end{aligned}
\end{equation}
However,  
\begin{equation*}
 \left\Vert \underline{\rho}^{-1} \right\Vert_{W^{2,2}(\Omega_{\eta_0})}  \lesssim 1 +   \Vert \underline{\rho} \Vert_{W^{2,2}(\Omega_{\eta_0})} + \Vert \nabla\underline{\rho} \Vert_{L^{4}(\Omega_{\eta_0})}^2 , 
\end{equation*}
which, after the  interpolation 
\begin{equation}\label{W14LW22FixDomain}
W^{1,4}(\Omega_{\eta_0}) = \Big[ L^\infty(\Omega_{\eta_0}), W^{2,2}(\Omega_{\eta_0}) \Big]_{1/2, 4}, 
\end{equation}
yields the estimate, with a constant $C = C\left(\Vert \underline{\rho} \Vert_{L^\infty\left(I_*; L^\infty(\Omega_{\eta_0}) \right)} \right) > 0$, given by
\begin{equation}\label{eq:MultiPertubationEstim}
 \left\Vert  \underline{\rho}^{-1} \right\Vert_{W^{2,2}(\Omega_{\eta_0})}  \lesssim 1 +   \Vert \underline{\rho} \Vert_{W^{2,2}(\Omega_{\eta_0})}. 
\end{equation}
Thus, \eqref{eq:PressureMoserIneq} becomes 
\begin{equation}\label{eq:PressureMoserIneqFinal}
\begin{aligned}
\Vert \underline{\rho}^{-1}\nabla p(\underline{\rho}) \Vert_{W^{2,2}(\Omega_{\eta_0})}^2 & \lesssim   \Vert \nabla \underline{\rho} \Vert_{L^{\infty}(\Omega_{\eta_0})}^2 \Vert  p(\underline{\rho}) \Vert_{W^{3,2}(\Omega_{\eta_0})}^2  + \left(  1 +   \Vert \underline{\rho} \Vert_{W^{2,2}(\Omega_{\eta_0})}^2 \right) \Vert  \nabla \underline{\rho} \Vert_{L^{\infty}(\Omega_{\eta_0})}^2  
\\[0.4em]
& \qquad +   \Vert  p(\underline{\rho}) \Vert_{W^{3,2}(\Omega_{\eta_0})}^2, 
\end{aligned}
\end{equation}
with hidden constant depending on $\Vert \underline{\rho} \Vert_{L^\infty\left(I_*; L^\infty(\Omega_{\eta_0})\right)} $.
We therefore, deduce  from \eqref{eq:GeoEstimate} and  \eqref{eq:PressureMoserIneqFinal} that
\begin{equation}\label{eq:S4}
\begin{aligned}
\int_{I_*} \Vert {\mathlarger{\mathfrak{S}}}_4 \Vert_{W^{2,2}(\Omega_{\eta_0})}^2 \dt & \lesssim  \int_{I_*} \Vert \nabla \underline{\rho} \Vert_{L^{\infty}(\Omega_{\eta_0})}^2 \Vert  p(\underline{\rho}) \Vert_{W^{3,2}(\Omega_{\eta_0})}^2 \dt +  \int_{I_*} \Vert  p(\underline{\rho}) \Vert_{W^{3,2}(\Omega_{\eta_0})}^2 \dt
\\[0.4em]
& \quad + \int_{I_*}  \left(  1 +   \Vert \underline{\rho} \Vert_{W^{2,2}(\Omega_{\eta_0})}^2 \right) \Vert  \nabla \underline{\rho} \Vert_{L^{\infty}(\Omega_{\eta_0})}^2    \dt.
\end{aligned}
\end{equation}
To estimate  ${\mathlarger{\mathfrak{S}}}_5$, we first observe that 
\begin{equation}\label{eq:AbsCont}
\eta(t) - \eta_0 = \int_{0}^t \partial_s\eta (s) \ds \quad \forall\, t \in I_* .
\end{equation}
This implies that 
\begin{equation}\label{eq:EtaEta0Norm}
\sup\limits_{I_*} \Vert \eta - \eta_0 \Vert_{W^{k,2}(\omega)} \leq T_* \sup\limits_{I_*}\Vert \partial_t\eta \Vert_{W^{k,2}(\omega)} \qquad \forall\, k  > 0.
\end{equation}
Thus, we have 
\begin{equation}\label{eq:S5prime}
\begin{aligned}
\int_{I_*} \Vert {\mathlarger{\mathfrak{S}}}_5 \Vert_{W^{2,2}(\Omega_{\eta_0})}^2 \dt  & \lesssim  \int_{I_*}    \left\Vert \underline{\rho}^{-1} \Div \Big(\left(  \mathbf{A}_{\eta \to \eta_0 } - \mathbb{I}_{3\times 3}  \right) \nabla\underline{\bu} \Big) \right\Vert_{W^{2, 2}(\Omega_{\eta_0})}^2\dt
\\[0.4em]
&\lesssim \sum\limits_{\mathsf{j} = 0}^{2}  {\mathlarger{ \mathfrak{U} }}_{\mathsf{j}} ,
\end{aligned}
\end{equation}
where 
\[
{\mathlarger{ \mathfrak{U} }}_{\mathsf{j}}  := \int_{I_*} \left\Vert \nabla^{\mathsf{j}} \left(  \underline{\rho}^{-1} \Div \Big(\left(  \mathbf{A}_{\eta \to \eta_0 } - \mathbb{I}_{3\times 3}  \right) \nabla\underline{\bu} \Big)  \right) \right\Vert_{L^{2}(\Omega_{\eta_0})}^2   \dt , \qquad \mathsf{j} \in \{ 0, 1, 2\}.  
\]
We first estimate the $\mathsf{j} = 2$ contribution, 
\begin{equation}
\begin{aligned}
{\mathlarger{ \mathfrak{U} }}_{2} &\lesssim  \int_{I_*} \left\Vert \nabla^{2} \left(  \underline{\rho}^{-1} \right) \Div \Big(\left(  \mathbf{A}_{\eta \to \eta_0 } - \mathbb{I}_{3\times 3}  \right) \nabla\underline{\bu} \Big)  \right\Vert_{L^{2}(\Omega_{\eta_0})}^2  \dt  
\\[0.4em]
& \quad +    \int_{I_*} \left\Vert \nabla \left(  \underline{\rho}^{-1} \right) \nabla\Div \Big(\left(  \mathbf{A}_{\eta \to \eta_0 } - \mathbb{I}_{3\times 3}  \right) \nabla\underline{\bu} \Big)  \right\Vert_{L^{2}(\Omega_{\eta_0})}^2  \dt 
\\[0.4em]
&\quad +  \int_{I_*} \left\Vert  \underline{\rho}^{-1}  \nabla^2\Div \Big(\left(  \mathbf{A}_{\eta \to \eta_0 } - \mathbb{I}_{3\times 3}  \right) \nabla\underline{\bu} \Big)  \right\Vert_{L^{2}(\Omega_{\eta_0})}^2  \dt 
\\[0.4em]
&\lesssim {\mathlarger{ \mathfrak{U} }}_{2, \mathtt{a}} +  {\mathlarger{ \mathfrak{U} }}_{2, \mathtt{b}} + {\mathlarger{ \mathfrak{U} }}_{2, \mathtt{c}} .
\end{aligned}
\end{equation} 
Using H\"older's inequality,  we obtain 
\begin{equation}\label{eq:U2aprimary}
\begin{aligned}
{\mathlarger{ \mathfrak{U} }}_{2, \mathtt{a}} & \lesssim  \int_{I_*}  \Vert \nabla^{2} \left(  \underline{\rho}^{-1} \right) \Vert_{L^4(\Omega_{\eta_0})}^2  \Vert  \mathbf{A}_{\eta \to \eta_0 } - \mathbb{I}_{3\times 3} \Vert_{W^{1,\infty}(\Omega_{\eta_0})}^2 \Vert \nabla^2 \underline{\bu} \Vert_{L^4(\Omega_{\eta_0})}^2 \dt
\\[0.4em]
&\lesssim   \int_{I_*} \left(1  + \Vert \nabla^2   \underline{\rho} \Vert_{L^4(\Omega_{\eta_0})}^2    \right)   \Vert  \mathbf{A}_{\eta \to \eta_0 } - \mathbb{I}_{3\times 3} \Vert_{W^{1,\infty}(\Omega_{\eta_0})}^2 \Vert \nabla^2 \underline{\bu} \Vert_{L^4(\Omega_{\eta_0})}^2 \dt ,
\end{aligned}
\end{equation}
where the last inequality follows from Gagliardo--Nirenberg inequality 
\begin{equation}\label{L8W24}
 \Vert \nabla   \underline{\rho} \Vert_{L^8(\Omega_{\eta_0})}^2  \lesssim  \Vert \underline{\rho} \Vert_{L^\infty(\Omega_{\eta_0})}\Vert \nabla^2   \underline{\rho} \Vert_{L^4(\Omega_{\eta_0})}.
\end{equation}
By  \eqref{W14LW22FixDomain}, \eqref{eq:GeoEstimate}, and Sobolev embedding, we deduce from \eqref{eq:U2aprimary} that 
\begin{equation}\label{eq:U2a}
\begin{aligned}
{\mathlarger{ \mathfrak{U} }}_{2, \mathtt{a}} & \lesssim  \sup\limits_{I_*} \Vert \eta -\eta_0 \Vert_{W^{7/2,2}(\omega)}^2 \int_{I_*} \left( \Vert \nabla  \underline{\rho} \Vert_{L^\infty(\Omega_{\eta_0})}^2 \Vert   \underline{\rho} \Vert_{W^{3,2}(\Omega_{\eta_0})}^2    +   \Vert \nabla  \underline{\bu} \Vert_{L^\infty(\Omega_{\eta_0})}^2 \Vert   \underline{\bu} \Vert_{W^{3,2}(\Omega_{\eta_0})}^2\right) \dt
\\[0.4em]
&\quad + T_*    \sup\limits_{I_*} \Vert \eta -\eta_0 \Vert_{W^{7/2,2}(\omega)}^2 \sup\limits_{I_*} \Vert   \underline{\bu} \Vert_{W^{3,2}(\Omega_{\eta_0})}^2  .
\end{aligned}
\end{equation}
It follows immediately from  the Banach algebra property of $W^{2,2}(\Omega_{\eta_0})$ and  \eqref{eq:GeoEstimate} that  
\begin{equation}\label{eq:U2b}
\begin{aligned}
{\mathlarger{ \mathfrak{U} }}_{2, \mathtt{b}} & \lesssim  \sup\limits_{I_*} \Vert \eta -\eta_0 \Vert_{W^{3,2}(\omega)}^2 \int_{I_*} \Vert \nabla  \underline{\rho} \Vert_{L^\infty(\Omega_{\eta_0})}^2 \Vert   \underline{\bu} \Vert_{W^{3,2}(\Omega_{\eta_0})}^2 \dt.
\end{aligned}
\end{equation}
Furthermore, we have 
\begin{equation*}
\begin{aligned}
{\mathlarger{ \mathfrak{U} }}_{2, \mathtt{c}}  & \lesssim  \sup\limits_{I_*} \Vert  \underline{\rho} \Vert_{L^\infty(\Omega_{\eta_0})}^2 \int_{I_*}    \left\Vert  \nabla^3 \Big(\left(  \mathbf{A}_{\eta \to \eta_0 } - \mathbb{I}_{3\times 3}  \right) \nabla\underline{\bu} \Big) \right\Vert_{L^2(\Omega_{\eta_0})}^2\dt
\\[0.4em]
& \lesssim  \Biggl( \int_{I_*}    \left\Vert  \nabla^3 \left(  \mathbf{A}_{\eta \to \eta_0 } - \mathbb{I}_{3\times 3}  \right) \nabla\underline{\bu}  \right\Vert_{L^2(\Omega_{\eta_0})}^2\dt   +  \int_{I_*}    \left\Vert  \nabla^2 \left(  \mathbf{A}_{\eta \to \eta_0 } - \mathbb{I}_{3\times 3}  \right) \nabla^2\underline{\bu}  \right\Vert_{L^2(\Omega_{\eta_0})}^2\dt
\\[0.4em]
& \qquad \qquad + \int_{I_*}    \left\Vert  \nabla \left(  \mathbf{A}_{\eta \to \eta_0 } - \mathbb{I}_{3\times 3}  \right) \nabla^3\underline{\bu}  \right\Vert_{L^2(\Omega_{\eta_0})}^2\dt + \int_{I_*}    \left\Vert \left(  \mathbf{A}_{\eta \to \eta_0 } - \mathbb{I}_{3\times 3}  \right) \nabla^4\underline{\bu}  \right\Vert_{L^2(\Omega_{\eta_0})}^2\dt    \Biggr)
\\[0.4em]
& =: \sum\limits_{\mathsf{j} = 1}^{4}  {\mathlarger{\mathfrak{O}}}_{\mathsf{j}}.
\end{aligned}
\end{equation*}
Using  \eqref{eq:GeoEstimate},  it follows that 
\begin{equation*}
\begin{aligned}
{\mathlarger{\mathfrak{O}}}_{1} & \lesssim  \int_{I_*}  \Vert \Dely^2\eta \Vert_{L^2(\omega)}^2 \Vert \nabla\underline{\bu}\Vert_{L^\infty(\Omega_{\eta_0})}^2 \dt  + \int_{I_*} \Vert \naby^4\eta_0 \Vert_{L^4(\omega)}^2  \Vert \nabla\underline{\bu}\Vert_{L^4(\Omega_{\eta_0})}^2 \dt.
\end{aligned}
\end{equation*}
We further obtain  -- after applying Sobolev embeddings -- the estimate 
\begin{equation}
\begin{aligned}
{\mathlarger{\mathfrak{O}}}_{1} & \lesssim \int_{I_*}  \Vert \eta \Vert_{W^{9/2, 2}(\omega)}^2 \Vert \nabla\underline{\bu}\Vert_{L^\infty(\Omega_{\eta_0})}^2 \dt   +  \Vert \eta_0 \Vert_{W^{5,2}(\omega)}^2 \int_{I_*}   \Vert \nabla^2\underline{\bu}\Vert_{L^2(\Omega_{\eta_0})}^2 \dt.
\end{aligned}
\end{equation}
For the term ${\mathlarger{\mathfrak{O}}}_{2}$, it holds that 
\begin{equation*}
\begin{aligned}
{\mathlarger{\mathfrak{O}}}_{2} & \lesssim \sup\limits_{I_*} \Vert \eta - \eta_0 \Vert_{W^{3,2}(\omega)}^2 \int_{I_*} \Vert \nabla^4\underline{\bu}\Vert_{L^2(\Omega_{\eta_0})}^2 \dt .
\end{aligned}
\end{equation*}
Combining standard interpolation with \eqref{eq:EtaEta0Norm}, we obtain  
\begin{equation}
\begin{aligned}
{\mathlarger{\mathfrak{O}}}_{2} & \lesssim T_*^{1/2}\sup\limits_{I_*} \Vert \eta -\eta_0 \Vert_{W^{ 7/2, 2}(\omega)}^{3/2}  \sup\limits_{I_*}\Vert \partial_t \eta \Vert_{W^{3/2, 2}(\omega)}^{1/2}  \int_{I_*}   \Vert \nabla^4\underline{\bu}\Vert_{L^2(\Omega_{\eta_0})}^2 \dt.
\end{aligned}
\end{equation}
Considering the term ${\mathlarger{\mathfrak{O}}}_{3}$,  a combination of H\"older's inequality and Sobolev embeddings yields
\begin{equation*}
\begin{aligned}
{\mathlarger{\mathfrak{O}}}_{3} & \lesssim \sup\limits_{I_*} \Vert \eta - \eta_0 \Vert_{W^{3,2}(\omega)}^2 \int_{I_*} \Vert \nabla^4\underline{\bu}\Vert_{L^2(\Omega_{\eta_0})}^2 \dt .
\end{aligned}
\end{equation*}
Therefore, arguing as for ${\mathlarger{\mathfrak{O}}}_{2}$, it follows that 
 \begin{equation}
\begin{aligned}
{\mathlarger{\mathfrak{O}}}_{3} & \lesssim T_*^{1/2}\sup\limits_{I_*} \Vert \eta -\eta_0 \Vert_{W^{ 7/2, 2}(\omega)}^{3/2}  \sup\limits_{I_*}\Vert \partial_t \eta \Vert_{W^{3/2, 2}(\omega)}^{1/2}  \int_{I_*}   \Vert \nabla^4\underline{\bu}\Vert_{L^2(\Omega_{\eta_0})}^2 \dt.
\end{aligned}
\end{equation}
Similar estimate holds for ${\mathlarger{\mathfrak{O}}}_{4}$, that is, 
\begin{equation}
\begin{aligned}
{\mathlarger{\mathfrak{O}}}_{4} & \lesssim T_*^{1/2}\sup\limits_{I_*} \Vert \eta -\eta_0 \Vert_{W^{ 7/2, 2}(\omega)}^{3/2}  \sup\limits_{I_*}\Vert \partial_t \eta \Vert_{W^{3/2, 2}(\omega)}^{1/2}  \int_{I_*}   \Vert \nabla^4\underline{\bu}\Vert_{L^2(\Omega_{\eta_0})}^2 \dt.
\end{aligned}
\end{equation}
Hence, 
\begin{equation}\label{eq:U2c}
\begin{aligned}
{\mathlarger{ \mathfrak{U} }}_{2, \mathtt{c}}    & \lesssim  \int_{I_*}  \Vert \eta \Vert_{W^{9/2, 2}(\omega)}^2 \Vert \nabla\underline{\bu}\Vert_{L^\infty(\Omega_{\eta_0})}^2 \dt   +  \int_{I_*}   \Vert \nabla^2\underline{\bu}\Vert_{L^2(\Omega_{\eta_0})}^2 \dt  + T_*^{1/2}\int_{I_*}   \Vert \nabla^4\underline{\bu}\Vert_{L^2(\Omega_{\eta_0})}^2 \dt   .
\end{aligned}
\end{equation}
We now consider the  $\mathsf{j} = 1$ contribution, 
\begin{equation}
\begin{aligned}
{\mathlarger{ \mathfrak{U} }}_{1} &\lesssim  \int_{I_*} \left\Vert \nabla \left(  \underline{\rho}^{-1} \right) \Div \Big(\left(  \mathbf{A}_{\eta \to \eta_0 } - \mathbb{I}_{3\times 3}  \right) \nabla\underline{\bu} \Big)  \right\Vert_{L^{2}(\Omega_{\eta_0})}^2  \dt  
\\[0.4em]
& \quad +    \int_{I_*} \left\Vert \underline{\rho}^{-1} \nabla\Div \Big(\left(  \mathbf{A}_{\eta \to \eta_0 } - \mathbb{I}_{3\times 3}  \right) \nabla\underline{\bu} \Big)  \right\Vert_{L^{2}(\Omega_{\eta_0})}^2  \dt 
\\[0.4em]
&\lesssim {\mathlarger{ \mathfrak{U} }}_{1, \mathtt{a}} +  {\mathlarger{ \mathfrak{U} }}_{1, \mathtt{b}}.
\end{aligned}
\end{equation} 
Using  \eqref{eq:GeoEstimate}, it is straightforward that 
\begin{equation}\label{eq:U1a}
{\mathlarger{ \mathfrak{U} }}_{1, \mathtt{a}} \lesssim \sup\limits_{I_*} \Vert \eta -\eta_0 \Vert_{W^{3,2}(\omega)}^2 \int_{I_*} \Vert \nabla  \underline{\rho} \Vert_{L^\infty(\Omega_{\eta_0})}^2 \Vert   \underline{\bu} \Vert_{W^{3,2}(\Omega_{\eta_0})}^2 \dt.
\end{equation}
Moreover, upon expanding the operator $\nabla \Div(\cdot)$, we obtain 
\begin{equation}\label{eq:U1b}
\begin{aligned}
{\mathlarger{ \mathfrak{U} }}_{1, \mathtt{b}} & \lesssim \sup\limits_{I_*} \Vert \eta -\eta_0 \Vert_{W^{3,2}(\omega)}^2  \left( \int_{I_*} \Vert \nabla  \underline{\bu} \Vert_{L^\infty(\Omega_{\eta_0})}^2 \dt  +  T_* \sup\limits_{I_*} \Vert   \underline{\bu} \Vert_{W^{3,2} (\Omega_{\eta_0})}^2 \right) 
\\[0.4em]
&\quad +  \sup\limits_{I_*} \Vert \eta -\eta_0 \Vert_{W^{7/2, 2}(\omega)}^2  \int_{I_*} \Vert \nabla^2  \underline{\bu} \Vert_{L^2(\Omega_{\eta_0})}^2 \dt .
\end{aligned}
\end{equation}
Considering the $\mathsf{j} = 0$ term, it holds that 
\begin{equation}\label{eq:U0}
\begin{aligned}
{\mathlarger{ \mathfrak{U} }}_{0} & \lesssim  \sup\limits_{I_*} \Vert \eta -\eta_0 \Vert_{W^{7/2, 2}(\omega)}^2  \int_{I_*} \Vert \nabla^2  \underline{\bu} \Vert_{L^2(\Omega_{\eta_0})}^2 \dt .
\end{aligned}
\end{equation}
Consequently, 
\begin{equation}\label{eq:S5}
\begin{aligned}
\int_{I_*} \Vert {\mathlarger{\mathfrak{S}}}_5 \Vert_{W^{2,2}(\Omega_{\eta_0})}^2 \dt  & \lesssim  \int_{I_*} \left( \Vert \nabla  \underline{\rho} \Vert_{L^\infty(\Omega_{\eta_0})}^2  + \Vert \nabla  \underline{\bu} \Vert_{L^\infty(\Omega_{\eta_0})}^2 \right) \Vert   \underline{\bu} \Vert_{W^{3,2}(\Omega_{\eta_0})}^2  \dt  
\\[0.4em]
& \quad + \int_{I_*} \Vert \nabla  \underline{\rho} \Vert_{L^\infty(\Omega_{\eta_0})}^2  \Vert   \underline{\rho} \Vert_{W^{3,2}(\Omega_{\eta_0})}^2  \dt  + T_* \sup\limits_{I_*} \Vert   \underline{\bu} \Vert_{W^{3,2}(\Omega_{\eta_0})}^2  
\\[0.4em]
& \quad + T_*^{1/2} \int_{I_*}  \Vert   \underline{\bu} \Vert_{W^{4,2}(\Omega_{\eta_0})}^2 \dt  +  \int_{I_*}  \Vert \eta \Vert_{W^{9/2, 2}(\omega)}^2 \Vert \nabla\underline{\bu}\Vert_{L^\infty(\Omega_{\eta_0})}^2 \dt   + {\mathlarger{\mathtt{E}}}_{\mathrm{acc}} .
\end{aligned}
\end{equation}
Finally, we have 
\begin{equation}\label{eq:S6prime}
\begin{aligned}
\int_{I_*} \Vert {\mathlarger{\mathfrak{S}}}_6 \Vert_{W^{2,2}(\Omega_{\eta_0})}^2 \dt  & \lesssim  \int_{I_*}    \left\Vert \underline{\rho}^{-1} \Div \Big( \mathbf{A}_{\eta \to \eta_0 } \nabla\mathcal{E}_{\eta_0}(\partial_t\eta\bn)\Big) \right\Vert_{W^{2, 2}(\Omega_{\eta_0})}^2\dt . 
\end{aligned}
\end{equation}
However, observe that the right-hand side of \eqref{eq:S6prime} is structurally identical to that of \eqref{eq:S5prime}. Accordingly, one may proceed exactly as in the derivation of  \eqref{eq:S5} to obtain
 \begin{equation}\label{eq:S6}
\begin{aligned}
\int_{I_*} \Vert {\mathlarger{\mathfrak{S}}}_6 \Vert_{W^{2,2}(\Omega_{\eta_0})}^2 \dt  & \lesssim  \int_{I_*} \left( \Vert \nabla  \underline{\rho} \Vert_{L^\infty(\Omega_{\eta_0})}^2  + \Vert \partial_t \Dely  \eta \Vert_{L^2(\omega)}^2 \right) \Vert   \partial_t\eta\Vert_{W^{5/2,2}(\omega)}^2  \dt  
\\[0.4em]
& \quad + \int_{I_*} \Vert \nabla  \underline{\rho} \Vert_{L^\infty(\Omega_{\eta_0})}^2  \Vert   \underline{\rho} \Vert_{W^{3,2}(\Omega_{\eta_0})}^2  \dt  + T_* \sup\limits_{I_*} \Vert   \partial_t \eta \Vert_{W^{5/2,2}(\omega)}^2  
\\[0.4em]
& \quad + T_*^{1/2} \int_{I_*}  \Vert   \partial_t \eta \Vert_{W^{7/2,2}(\omega)}^2  \dt  +  \int_{I_*}  \Vert \eta \Vert_{W^{9/2, 2}(\omega)}^2 \Vert \partial_t \Dely  \eta \Vert_{L^2(\omega)}^2 \dt   + {\mathlarger{\mathtt{E}}}_{\mathrm{acc}} .
\end{aligned}
\end{equation}
Hence, \eqref{eq:fEstim} follows by combining the above estimate with \cref{rem:HigherOrderEta}. \\[-0.8em]


We conclude with the estimate of $\partial_t \mathsf{f}$, which we split  as 
\[
\partial_t  \big(\mathsf{f}+ \partial_t \mathcal{E}_{\eta_0}(\partial_t\eta\bn) \big) =  \sum\limits_{\mathsf{j} = 1}^{8}  \big(\partial_t \mathsf f \big)_\mathsf{j},
\]
where  
\begin{align*}
\big(\partial_t \mathsf f \big)_1 & :=   -\partial_t \Big( \nabla\mathcal{E}_{\eta_0}(\partial_t\eta\bn) \cdot \partial_t \bfPsi_{\eta \to \eta_0 }^{-1}\circ \bfPsi_{\eta \to \eta_0 }  \Big),
\\[1em]
\big(\partial_t \mathsf f \big)_2 & :=   - \partial_t \bigg( \dfrac{1}{J_{\eta \to \eta_0}}\underline{\bu} \big( \nabla \underline{\bu} \colon \mathbf{B}_{\eta \to \eta_0 } \big)    \bigg),
\\[1em]
\big(\partial_t \mathsf f \big)_3 & :=  -\partial_t \bigg( \dfrac{1}{J_{\eta \to \eta_0}} \mathcal{E}_{\eta_0}(\partial_t\eta\bn) \big( \nabla\mathcal{E}_{\eta_0}(\partial_t\eta\bn)  \colon\mathbf{B}_{\eta \to \eta_0 }  \big)   \bigg),
\\[1em]
\big(\partial_t \mathsf f \big)_4 & :=  - \partial_t \bigg( \dfrac{1}{J_{\eta \to \eta_0} \underline{\rho}} \mathbf{B}_{\eta \to \eta_0 }\nabla p(\underline{\rho})   \bigg) ,
\\[1em]
\big(\partial_t \mathsf f \big)_5 & :=  - \partial_t \!\left( \dfrac{1}{J_{\eta \to \eta_0} \underline{\rho}} \,\right) \Div\Biggl[  \mu \left(\mathbf{A}_{\eta \to \eta_0 }  - \mathbb{I}_{3 \times 3}   \right) \nabla\underline{\bu}  + (\lambda + \mu)\biggl( \Bigl(  \left(\mathbf{B}_{\eta \to \eta_0 } - \mathbb{I}_{3 \times 3} \right) \colon \nabla\underline{\bu} \Bigr)  \mathbb{I}_{3 \times 3}    
\\[0.4em]
& \qquad \qquad\qquad \qquad \qquad \qquad\qquad \qquad\qquad\quad  + \left( \mathbf{B}_{\eta \to \eta_0 } \colon \nabla\underline{\bu}  \right)  \left( \dfrac{1}{J_{\eta \to \eta_0 }} \mathbf{B}_{\eta \to \eta_0 } -  \mathbb{I}_{3 \times 3}  \right)               \biggr)      \Biggr]  ,
\\[1em]
\big(\partial_t \mathsf f \big)_6 & := -  \dfrac{1}{J_{\eta \to \eta_0} \underline{\rho}} \, \partial_t\Div\Biggl[  \mu \left(\mathbf{A}_{\eta \to \eta_0 }  - \mathbb{I}_{3 \times 3}   \right) \nabla\underline{\bu}  + (\lambda + \mu)\biggl( \Bigl(  \left(\mathbf{B}_{\eta \to \eta_0 } - \mathbb{I}_{3 \times 3} \right) \colon \nabla\underline{\bu} \Bigr)  \mathbb{I}_{3 \times 3}    
\\[0.4em]
& \qquad \qquad\qquad \qquad \qquad \qquad\qquad \qquad\qquad\quad  + \left( \mathbf{B}_{\eta \to \eta_0 } \colon \nabla\underline{\bu}  \right)  \left( \dfrac{1}{J_{\eta \to \eta_0 }} \mathbf{B}_{\eta \to \eta_0 } -  \mathbb{I}_{3 \times 3}  \right)               \biggr)      \Biggr]  ,
\\[1em]
\big(\partial_t \mathsf f \big)_7 & := \partial_t \!\left( \dfrac{1}{J_{\eta \to \eta_0} \underline{\rho}} \,\right)\Div\Bigg[ \mu  \mathbf{A}_{\eta \to \eta_0 }  \nabla\mathcal{E}_{\eta_0}(\partial_t\eta\bn) +   \dfrac{(\lambda + \mu)}{J_{\eta \to \eta_0}}   \Big( \mathbf{B}_{\eta \to \eta_0 } \colon \nabla \mathcal{E}_{\eta_0}(\partial_t\eta\bn)     \Big)     \mathbf{B}_{\eta \to \eta_0 }     \Bigg],
\\[1em]
\big(\partial_t \mathsf f \big)_8 & := \dfrac{1}{J_{\eta \to \eta_0} \underline{\rho}} \,\partial_t\Div\Bigg[ \mu  \mathbf{A}_{\eta \to \eta_0 }  \nabla\mathcal{E}_{\eta_0}(\partial_t\eta\bn) +   \dfrac{(\lambda + \mu)}{J_{\eta \to \eta_0}}   \Big( \mathbf{B}_{\eta \to \eta_0 } \colon \nabla \mathcal{E}_{\eta_0}(\partial_t\eta\bn)     \Big)     \mathbf{B}_{\eta \to \eta_0 }     \Bigg].
\end{align*}
Starting with the first term, we have, by \eqref{eq:GeoEstimate},  
\begin{align*}
\int_{I_*} \left\Vert \big(\partial_t \mathsf f \big)_1 \right\Vert_{L^2(\Omega_{\eta_0})}^2 \dt & \lesssim  \int_{I_*} \Vert \partial_t \nabla\mathcal{E}_{\eta_0}(\partial_t\eta\bn) \Vert_{L^2(\Omega_{\eta_0})}^2 \Vert \partial_t\bfPsi_{\eta \to \eta_0} \Vert_{L^{\infty}(\Omega_{\eta_0})}^2  \dt
\\[0.4em]
& \quad +   \int_{I_*} \Vert  \nabla\mathcal{E}_{\eta_0}(\partial_t\eta\bn) \Vert_{L^\infty(\Omega_{\eta_0})}^2  \Vert \partial_t^2\bfPsi_{\eta \to \eta_0 }^{-1}\circ \bfPsi_{\eta \to \eta_0 }  \Vert_{L^2(\Omega_{\eta_0})}^2  \dt
\\[0.4em]
& \quad +   \int_{I_*} \Vert  \nabla\mathcal{E}_{\eta_0}(\partial_t\eta\bn) \Vert_{L^\infty(\Omega_{\eta_0})}^2  \Vert \partial_t\nabla \bfPsi_{\eta \to \eta_0 }^{-1}\circ \bfPsi_{\eta \to \eta_0 }  \cdot \partial_t \bfPsi_{\eta \to \eta_0}  \Vert_{L^2(\Omega_{\eta_0})}^2  \dt .
\end{align*}
Using the trace theorem and Sobolev embeddings, we obtain 
 \begin{align*}
\int_{I_*} \left\Vert \big(\partial_t \mathsf f \big)_1 \right\Vert_{L^2(\Omega_{\eta_0})}^2 \dt & \lesssim  \int_{I_*} \Vert \partial_t^2 \eta \Vert_{W^{1/2, 2}(\omega)}^2 \Vert \partial_t\eta \Vert_{W^{3/2, 2}(\omega)}^2  \dt
\\[0.4em]
& \quad +   \int_{I_*} \Vert \partial_t \eta \Vert_{W^{5/2,2}(\omega)}^2 \Big(  \Vert \partial_t^2\eta \Vert_{L^{2}(\omega)}^2   +  \Vert \partial_t\naby\eta \Vert_{L^{2}(\omega)}^2 \Vert \partial_t\eta \Vert_{W^{3/2, 2}(\omega)}^2 \Big) \dt .
\end{align*}
Thus, it follows from \cref{rem:HigherOrderEta}, interpolation, and Young's inequality that  
\begin{equation}\label{eq:TimefEstim1}
\int_{I_*} \left\Vert \big(\partial_t \mathsf f \big)_1 \right\Vert_{L^2(\Omega_{\eta_0})}^2 \dt  \lesssim  {\mathlarger{\mathtt{E}}}_{\mathrm{acc}} + \kappa \int_{I_*} \Vert \partial_t \eta \Vert_{W^{3,2}(\omega)}^2  \Vert \partial_t^2\eta \Vert_{L^{2}(\omega)}^2  \dt    .
\end{equation}
Noting that $\big(\partial_t \mathsf f \big)_2$ coincides with $\big(\partial_t \mathsf{B}\underline{\bu} \big)_2$  upon replacing $\mathcal{E}_{\eta_0}(\partial_t\eta\bn)$ by $\underline{\bu}$, we derive from\eqref{eq:TimeBuEstim2prior} that 
\begin{equation*}
\begin{aligned}
\int_{I_*} \left\Vert\big(\partial_t \mathsf f \big)_2 \right\Vert_{L^2(\Omega_{\eta_0})}^2 \dt & \lesssim   \int_{I_*}\Big(\Vert \partial_t \naby\eta\Vert_{L^\infty(\omega)}^2   +  \Vert \partial_t \underline{\bu} \Vert_{L^2(\Omega_{\eta_0})}^2  \Big) \Vert  \nabla\underline{\bu} \Vert_{L^\infty(\Omega_{\eta_0})}^2    \dt 
\\[0.4em]
&\quad +  \int_{I_*} \Vert \underline{\bu} \Vert_{L^\infty(\Omega_{\eta_0})}^2  \Big( \Vert \partial_t \nabla\underline{\bu} \Vert_{L^2(\Omega_{\eta_0})}^2  +  \Vert \nabla\underline{\bu} \Vert_{L^\infty(\Omega_{\eta_0})}^2 \Vert  \partial_t \mathbf{B}_{\eta \to \eta_0 }  \Vert_{L^2(\Omega_{\eta_0})}^2   \Big)   \dt  . 
\end{aligned}
\end{equation*}
By Sobolev embeddings and interpolation, we deduce that 
\begin{equation}\label{eq:TimefEstim2}
\begin{aligned}
 \int_{I_*} \left\Vert \big(\partial_t \mathsf f \big)_2 \right\Vert_{L^2(\Omega_{\eta_0})}^2 \dt 
& \lesssim  \int_{I_*}\Big(  \Vert \partial_t\eta \Vert_{W^{5/2,2}(\omega)}^2   +  \Vert \partial_t \underline{\bu} \Vert_{L^{2}(\Omega_{\eta_0})}^2 +  \Vert \nabla\underline{\bu} \Vert_{L^\infty(\Omega_{\eta_0})}^2  \Big)\Vert \underline{\bu} \Vert_{W^{3,2}(\Omega_{\eta_0})}^2 \dt    
\\[0.4em]
&\quad +  \kappa \int_{I_*} \Vert \partial_t \underline{\bu} \Vert_{W^{2,2}(\Omega_{\eta_0})}^2  \dt   + {\mathlarger{\mathtt{E}}}_{\mathrm{acc}} .
\end{aligned}
\end{equation}
Similarly, we have 
\begin{equation*}
\begin{aligned}
\int_{I_*} \left\Vert\big(\partial_t \mathsf f \big)_3 \right\Vert_{L^2(\Omega_{\eta_0})}^2 \dt & \lesssim   \int_{I_*}\Big(\Vert \partial_t \naby\eta\Vert_{L^\infty(\omega)}^2   +  \Vert \partial_t \mathcal{E}_{\eta_0}(\partial_t\eta\bn) \Vert_{L^2(\Omega_{\eta_0})}^2  \Big) \Vert  \nabla\mathcal{E}_{\eta_0}(\partial_t\eta\bn) \Vert_{L^\infty(\Omega_{\eta_0})}^2    \dt 
\\[0.4em]
&\quad +  \int_{I_*} \Vert \mathcal{E}_{\eta_0}(\partial_t\eta\bn) \Vert_{L^\infty(\Omega_{\eta_0})}^2  \Vert \partial_t \nabla\mathcal{E}_{\eta_0}(\partial_t\eta\bn) \Vert_{L^2(\Omega_{\eta_0})}^2  \dt
\\[0.4em]
&\quad + \int_{I_*} \Vert \mathcal{E}_{\eta_0}(\partial_t\eta\bn) \Vert_{L^\infty(\Omega_{\eta_0})}^2 \Vert \nabla\mathcal{E}_{\eta_0}(\partial_t\eta\bn) \Vert_{L^\infty(\Omega_{\eta_0})}^2 \Vert  \partial_t \mathbf{B}_{\eta \to \eta_0 }  \Vert_{L^2(\Omega_{\eta_0})}^2   \dt  . 
\end{aligned}
\end{equation*}
Using \cref{rem:HigherOrderEta}, interpolation and Young's inequality, we deduce that 
\begin{equation}\label{eq:TimefEstim3}
\begin{aligned}
 \int_{I_*} \left\Vert \big(\partial_t \mathsf f \big)_3 \right\Vert_{L^2(\Omega_{\eta_0})}^2 \dt 
& \lesssim  \kappa \int_{I_*}\Big(  \Vert \partial_t\eta \Vert_{W^{5/2,2}(\omega)}^2   +  \Vert \partial_t^2 \eta\Vert_{L^{2}(\omega)}^2 \Big)\Vert \partial_t\eta \Vert_{W^{3,2}(\omega)}^2 \dt   + {\mathlarger{\mathtt{E}}}_{\mathrm{acc}} .
\end{aligned}
\end{equation}
For the pressure term $ \big(\partial_t \mathsf f \big)_4$, it is straightforward that 
\begin{equation}\label{eq:TimefEstim4prior}
\begin{aligned}
 \int_{I_*} \left\Vert \big(\partial_t \mathsf f \big)_4 \right\Vert_{L^2(\Omega_{\eta_0})}^2 \dt  & \lesssim  \int_{I_*}  \left\Vert \partial_t \big(\, \underline{\rho}^{-1}\nabla p(\underline{\rho})   \big) \right\Vert_{L^2(\Omega_{\eta_0})}^2 \dt
 \\[0.4em]
 &\lesssim   \int_{I_*} \Vert \partial_t  \underline{\rho}  \Vert_{L^2(\Omega_{\eta_0})}^2 \Vert \nabla \underline{\rho}  \Vert_{L^\infty(\Omega_{\eta_0})}^2   \dt +     \int_{I_*} \Vert \partial_t \nabla \underline{\rho}  \Vert_{L^2(\Omega_{\eta_0})}^2   \dt .
 \end{aligned}
\end{equation} 
Moreover, $\underline{\rho}$ satisfies the transport-type equation  
\begin{equation}
 \partial_t  \underline{\rho}   =  - \Big(  \partial_t\bfPsi_{\eta \to \eta_0}^{-1}\circ\bfPsi_{\eta \to \eta_0}  + \dfrac{1}{J_{\eta \to \eta_0}}\mathbf{B}_{\eta \to \eta_0 }^\intercal \underline{\bv}  \Big) \cdot \nabla \underline{\rho}   - \Big(\dfrac{1}{J_{\eta \to \eta_0}} \mathbf{B}_{\eta \to \eta_0 } \colon \nabla\underline{\bv}\Big) \underline{\rho} . 
\end{equation}
 This yields
 \begin{align}
 \Vert \partial_t  \underline{\rho}  \Vert_{L^\infty(\Omega_{\eta_0})} & \lesssim  \Big( \Vert  \partial_t \eta \Vert_{W^{3/2, 2}(\omega)}  + \Vert \underline{\bv} \Vert_{L^\infty(\Omega_{\eta_0})}  \Big)\Vert \nabla\underline{\rho} \Vert_{L^\infty(\Omega_{\eta_0})}  + \Vert \nabla\underline{\bv} \Vert_{L^\infty(\Omega_{\eta_0})} ,
 \label{eq:TransportRHOinf}
 \\[1em]
  \Vert \partial_t  \underline{\rho}  \Vert_{L^2(\Omega_{\eta_0})} & \lesssim  \Big( \Vert  \partial_t \eta \Vert_{L^{2}(\omega)}  + \Vert \underline{\bv} \Vert_{L^2(\Omega_{\eta_0})}  \Big)\Vert \nabla\underline{\rho} \Vert_{L^\infty(\Omega_{\eta_0})}   + \Vert \nabla \underline{\bv} \Vert_{L^2(\Omega_{\eta_0})} \Vert \underline{\rho} \Vert_{L^\infty(\Omega_{\eta_0})} ,
  \label{eq:TransportRHO1}
 \\[1em]
   \Vert \partial_t  \underline{\rho}  \Vert_{W^{1,2}(\Omega_{\eta_0})} & \lesssim  \Big( \Vert  \partial_t \eta \Vert_{W^{1,\infty}(\omega)}  + \Vert \underline{\bv} \Vert_{W^{1,\infty}(\Omega_{\eta_0})}  \Big)\Vert \nabla\underline{\rho} \Vert_{W^{1,2}(\Omega_{\eta_0})} + \Vert \nabla \underline{\bv} \Vert_{W^{1,2}(\Omega_{\eta_0})} \Vert \nabla\underline{\rho} \Vert_{L^\infty(\Omega_{\eta_0})} .
    \label{eq:TransportRHO2}
 \end{align} 
Substituting  \eqref{eq:TransportRHO1} and  \eqref{eq:TransportRHO2} into \eqref{eq:TimefEstim4prior}, we deduce that 
\begin{equation}\label{eq:TimefEstim4}
\begin{aligned}
&\int_{I_*} \left\Vert \big(\partial_t \mathsf f \big)_4 \right\Vert_{L^2(\Omega_{\eta_0})}^2 \dt  
\\[0.4em]
& \lesssim  \int_{I_*} \Big( \Vert \underline{\bv} \Vert_{W^{2,2}(\Omega_{\eta_0})}^2   +  \Vert \partial_t\eta \Vert_{W^{5/2,2}(\omega)}^2  +  \Vert \nabla\underline{\bv} \Vert_{L^{\infty}(\Omega_{\eta_0})}^2 + \Vert \nabla\underline{\rho} \Vert_{L^\infty(\Omega_{\eta_0})}^2  \Big)  \Vert \underline{\rho} \Vert_{W^{3,2}(\Omega_{\eta_0})}^2   \dt  +  {\mathlarger{\mathtt{E}}}_{\mathrm{acc}} .
\end{aligned}
\end{equation}
By \eqref{eq:GeoEstimate} and \cref{rem:HigherOrderEta}, we immediately have 
\begin{equation}\label{eq:TimefEstim5prior}
\int_{I_*} \left\Vert \big(\partial_t \mathsf f \big)_5 \right\Vert_{L^2(\Omega_{\eta_0})}^2 \dt  \lesssim \int_{I_*}   \Big( \Vert \partial_t\naby\eta\Vert_{L^\infty(\omega)}^2 +  \Vert \partial_t  \underline{\rho}  \Vert_{L^\infty(\Omega_{\eta_0})}^2 \Big) \Vert \underline{\bu} \Vert_{W^{2,2}(\Omega_{\eta_0})}^2 \dt . 
\end{equation}
Using \eqref{eq:TransportRHOinf}, interpolation, and Young's inequality, we obtain   
\begin{equation}\label{eq:TimefEstim5}
\begin{aligned}
& \int_{I_*} \left\Vert \big(\partial_t \mathsf f \big)_5 \right\Vert_{L^2(\Omega_{\eta_0})}^2 \dt 
\\[0.4em]
&  \lesssim  \int_{I_*} \Vert \partial_t\eta \Vert_{W^{5/2,2}(\omega)}^2 \Vert \underline{\bu} \Vert_{W^{3,2}(\Omega_{\eta_0})}^2   \dt  + \int_{I_*}  \Vert \underline{\bu} \Vert_{W^{2,2}(\Omega_{\eta_0})}^2 \Big(  \Vert \underline{\bv} \Vert_{W^{3,2}(\Omega_{\eta_0})}^2  + \Vert \underline{\rho} \Vert_{W^{3,2}(\Omega_{\eta_0})}^2 \Big)  \dt 
\\[0.4em]
&\quad + \int_{I_*}  \Big(  \Vert \underline{\bv} \Vert_{W^{3,2}(\Omega_{\eta_0})}^2  +   \Vert \underline{\bu} \Vert_{W^{3,2}(\Omega_{\eta_0})}^2     \Big)  \Vert \nabla\underline{\rho} \Vert_{L^\infty(\Omega_{\eta_0})}^2    \dt .
\end{aligned}
\end{equation}
Proceeding as for $ \big(\partial_t \mathsf f \big)_5$, and using that time derivative and spatial derivative commute, we  derive that 
\begin{equation*}
\begin{aligned}
 \int_{I_*} \left\Vert \big(\partial_t \mathsf f \big)_6 \right\Vert_{L^2(\Omega_{\eta_0})}^2 \dt & \lesssim \int_{I_*}  \Vert \partial_t \mathbf{A}_{\eta \to \eta_0 } \nabla\underline{\bu} \Vert_{W^{1,2}(\Omega_{\eta_0})}^2 \dt  + \int_{I_*}  \Vert  \left(\mathbf{A}_{\eta \to \eta_0 }  - \mathbb{I}_{3 \times 3}   \right) \partial_t\nabla\underline{\bu} \Vert_{W^{1, 2}(\Omega_{\eta_0})}^2  \dt 
 \\[0.4em]
 &\lesssim  \int_{I_*}  \Vert \partial_t \mathbf{A}_{\eta \to \eta_0 } \Vert_{W^{1,4}(\Omega_{\eta_0})}^2  \Vert \nabla\underline{\bu} \Vert_{W^{1,4}(\Omega_{\eta_0})}^2 \dt      
 \\[0.4em]
 & \quad +    \int_{I_*}  \Vert  \left(\mathbf{A}_{\eta \to \eta_0 }  - \mathbb{I}_{3 \times 3}   \right) \Vert_{L^\infty(\Omega_{\eta_0})}^2 \Vert \partial_t\nabla\underline{\bu} \Vert_{W^{1, 2}(\Omega_{\eta_0})}^2  \dt 
 \\[0.4em]
 &\quad +  \int_{I_*}  \Vert  \left(\mathbf{A}_{\eta \to \eta_0 }  - \mathbb{I}_{3 \times 3}   \right) \Vert_{W^{1,\infty}(\Omega_{\eta_0})}^2 \Vert \partial_t\nabla\underline{\bu} \Vert_{L^{2}(\Omega_{\eta_0})}^2  \dt .
 \end{aligned}
\end{equation*}
Using   \eqref{eq:GeoEstimate} together with Sobolev embeddings, the above estimate reduces to 
\begin{equation}\label{eq:TimefEstim6prior}
\begin{aligned}
 \int_{I_*} \left\Vert \big(\partial_t \mathsf f \big)_6 \right\Vert_{L^2(\Omega_{\eta_0})}^2 \dt  &\lesssim  \int_{I_*}  \Vert \partial_t \eta  \Vert_{W^{5/2, 2}(\omega)}^2  \Vert \underline{\bu} \Vert_{W^{3, 2}(\Omega_{\eta_0})}^2 \dt      
 \\[0.4em]
 & \quad +    \int_{I_*}  \Vert \eta - \eta_0 \Vert_{W^{1, \infty}(\omega)}^2 \Vert \partial_t\underline{\bu} \Vert_{W^{2, 2}(\Omega_{\eta_0})}^2  \dt 
 \\[0.4em]
 &\quad +  \int_{I_*}  \Vert  \eta - \eta_0  \Vert_{W^{7/2, 2}(\omega)}^2 \Vert \partial_t\nabla\underline{\bu} \Vert_{L^{2}(\Omega_{\eta_0})}^2  \dt .
 \end{aligned}
\end{equation}
However, by \eqref{eq:AbsCont}, it holds that  
\begin{equation}\label{eq:AbstContBound}
\begin{aligned}
\sup\limits_{I_*} \Vert \eta - \eta_0 \Vert_{W^{1, \infty}(\omega)}^2 & \leq T_* \int_{I_*} \Vert \partial_t \eta \Vert_{W^{1, \infty}(\omega)}^2 \dt \; \leq T_*  \int_{I_*}  \Vert \partial_t \eta \Vert_{W^{5/2, 2}(\omega)}^2 \dt .
\end{aligned}
\end{equation}
Thus, upon substituting \eqref{eq:AbstContBound} into \eqref{eq:TimefEstim6prior}, we deduce, by \cref{rem:HigherOrderEta}, interpolation, and Young's inequality,  that 
\begin{equation}\label{eq:TimefEstim6}
 \int_{I_*} \left\Vert \big(\partial_t \mathsf f \big)_6 \right\Vert_{L^2(\Omega_{\eta_0})}^2 \dt  \lesssim  {\mathlarger{\mathtt{E}}}_{\mathrm{acc}}  +  (T_* + \kappa) \int_{I_*}\Vert \partial_t\underline{\bu} \Vert_{W^{2,2}(\Omega_{\eta_0})}^2  \dt   +  \int_{I_*}  \Vert \partial_t \eta  \Vert_{W^{5/2, 2}(\omega)}^2  \Vert \underline{\bu} \Vert_{W^{3, 2}(\Omega_{\eta_0})}^2 \dt    .
\end{equation}
Since $\big(\partial_t \mathsf f \big)_7$ is analogous to $\big(\partial_t \mathsf f \big)_5$, the same argument applies and yields
\begin{equation}\label{eq:TimefEstim7}
\int_{I_*} \left\Vert \big(\partial_t \mathsf f \big)_7 \right\Vert_{L^2(\Omega_{\eta_0})}^2 \dt  \lesssim  {\mathlarger{\mathtt{E}}}_{\mathrm{acc}}  + \int_{I_*} \Vert \underline{\bv} \Vert_{W^{3,2}(\Omega_{\eta_0})}^2 \Vert \nabla\underline{\rho} \Vert_{L^\infty(\Omega_{\eta_0})}^2    \dt  .
\end{equation}
Finally,  for the term $\big(\partial_t \mathsf f \big)_8$, up to replacing $\underline{\bu}$ by $ \mathcal{E}_{\eta_0}(\partial_t\eta\bn)$, we may proceed as in the estimate of  $\big(\partial_t \mathsf f \big)_6$. Consequently, we infer  from \eqref{eq:TimefEstim6prior} that 
\begin{equation}\label{eq:TimefEstim8prior}
\begin{aligned}
 \int_{I_*} \left\Vert \big(\partial_t \mathsf f \big)_8 \right\Vert_{L^2(\Omega_{\eta_0})}^2 \dt  &\lesssim  \int_{I_*}  \Vert \partial_t \eta  \Vert_{W^{5/2, 2}(\omega)}^2  \Vert \partial_t\eta \Vert_{W^{5/2, 2}(\omega)}^2 \dt      
 \\[0.4em]
 & \quad +    \int_{I_*}  \Vert \eta - \eta_0 \Vert_{W^{1, \infty}(\omega)}^2 \Vert \partial_t^2 \eta \Vert_{W^{3/2, 2}(\Omega_{\eta_0})}^2  \dt 
 \\[0.4em]
 &\quad +  \int_{I_*}  \Vert  \eta - \eta_0  \Vert_{W^{7/2, 2}(\omega)}^2 \Vert \partial_t^2 \eta \Vert_{W^{1/2, 2}(\Omega_{\eta_0})}^2  \dt  .
 \end{aligned}
\end{equation}
Using  \eqref{eq:AbstContBound}, interpolation and Young's inequality, we deduce that 
\begin{equation} \label{eq:TimefEstim8}
 \int_{I_*} \left\Vert \big(\partial_t \mathsf f \big)_8 \right\Vert_{L^2(\Omega_{\eta_0})}^2 \dt  \lesssim  {\mathlarger{\mathtt{E}}}_{\mathrm{acc}}  + T_* \int_{I_*}  \Vert \partial_t^2 \Dely\eta \Vert_{L^2(\omega)}^2 \dt  + \kappa \int_{I_*}  \Vert \partial_t \eta  \Vert_{W^{5/2, 2}(\omega)}^2  \Vert \partial_t\eta \Vert_{W^{3, 2}(\omega)}^2 \dt .
\end{equation}
Hence, \eqref{eq:TimefEstim} follows by combining \eqref{eq:TimefEstim1}--\eqref{eq:TimefEstim8}.



 \section*{Acknowledgements}
The authors sincerely acknowledge  Dominic Breit for inspiring discussions and constructive feedback related to the present work.  
This work was funded by the Deutsche Forschungsgemeinschaft (DFG) -- Projektnummer 543675748.

\section*{Compliance with Ethical Standards}
\smallskip
\par\noindent
{\bf Conflict of Interest}. The authors declare that they have no conflict of interest.

\smallskip
\par\noindent
{\bf Data Availability}. Data sharing is not applicable to this article as no datasets were generated or analysed during the current study.

\end{document}